\documentclass[11pt,a4paper,leqno]{amsart}

\usepackage{common}
\newtheorem{mainthm}{Theorem}

\usepackage{tikz}

\title[Flag multiresolutions]{Flag multiresolutions: $T1$, paraproducts and commutators}
\author{Ben Foster}
\author{Kangwei Li}
\author{Henri Martikainen}

\address[B.F.]{Department of Mathematics, Washington University
in St. Louis, 1 Brookings Drive, St. Louis, MO 63130, USA}
\email{bfoster@wustl.edu}

\address[K.L.]{School of Mathematical Sciences, Zhejiang Normal University,\newline Jinhua 321004, China}
\email{kangwei.li@zjnu.edu.cn}

\address[H.M.]{Department of Mathematics, Washington University
in St. Louis,\newline 1 Brookings Drive, St. Louis, MO 63130, USA}
\email{henri@wustl.edu}


\begin{document}

\allowdisplaybreaks

\begin{abstract}
	We develop a rich theory of dyadic-probabilistic
	multiresolution methods in the context of flag type singular integrals.
	We apply these methods to prove a full $T1$ theorem, even a representation theorem,
	for bi-parameter flag singular integrals. This is the first time a complete $T1$
	theorem appears in a multi-parameter setting with an entangled nature -- even the related recent
	advances with Zygmund dilations have been in the cancellative $T1 = 0$
	or convolution situations. This theorem comes with a non-trivial necessity argument
	related to the flag type product $\BMO$ assumption on $T1$.
	We also present a version of our theorem in the cancellative tri-parameter
	flag setting showing that our methods are perfectly adaptable to general multi-parameter
	flag settings. Importantly, our $T1$ representation theorem implies
	weighted estimates for every flag $A_p$ weight, showing that our dyadic multiresolutions
	carefully retain the flag invariance of the original singular integrals.
	Our multiresolution methods also yield flag style paraproduct decompositions
	of products of functions and lead to a general and efficient
	method for proving commutator estimates in various flag settings.
\end{abstract}

\maketitle
\tableofcontents

\section{Introduction}
A flag is a nested sequence of linear subspaces of increasing dimension,
such as, $\R^{d_3} \subset \R^{d_2 + d_3} \subset \R^{d_1+d_2+d_3}$ (with the obvious
identifications).
To describe how stereotypical flag singular integrals look like,
we recall first that the $j$th Riesz transform $R^j$ on some Euclidean space $\R^k$ is
the classical one-parameter convolution form
singular integral operator (SIO)
with the kernel $(x_j-y_j)/|x-y|^{k+1}$, $j \in \{1, \ldots, k\}$.
An  example of a tri-parameter flag type singular integral is
$R_{123} R_{23} R_3$, where $R_{123}, R_{23}, R_3$ are \emph{some} Riesz transforms
in $\R^{d}, \R^{d_{23}}$ and $\R^{d_3}$, respectively -- here $d = d_{123} := d_1+d_2+d_3$, $d_{23} := d_2 + d_3$.
More generally, a Fourier multiplier $\wh{T_m f} = m\wh f$ whose symbol $m$ satisfies
\begin{equation*}
	|\partial^\alpha m(\xi)| \lesssim
	(|\xi_1|+|\xi_2|+|\xi_3|)^{-|\alpha_1|}  (|\xi_2|+|\xi_3|)^{-|\alpha_2|}|\xi_3|^{-|\alpha_3|}
\end{equation*}
for $\xi = (\xi_1, \xi_2, \xi_3) \in \R^d = \R^{d_1} \times \R^{d_2} \times \R^{d_3}$
and multi-indices $\alpha = (\alpha_1, \alpha_2, \alpha_3) \in \N^{d_1} \times \N^{d_2} \times \N^{d_3}$,
gives rise to a convolution form tri-parameter flag SIO in $\R^d$.

Flag type SIOs arose in the work of M\"uller, Ricci and Stein \cites{MRS, MRS2}
in connection with certain Marcinkiewicz multipliers on the Heisenberg group; see
also the related later work of Nagel, Ricci and Stein \cite{NRS2001}.
Flag singular integrals in more general homogeneous groups then appeared in \cite{NRSW}.
The theory of product type flag Hardy spaces, $\BMO$ spaces and Littlewood--Paley theory,
together with various different characterizations, has been developed
by several authors, see e.g. \cites{HLS, HLLW}.

Despite the significant body of work in the flag setting, some of which were mentioned above,
some fundamental tools have been noted to be missing.
Indeed, recently in \cite{DLOPW-FLAG} X. Duong, J. Li, Y. Ou, J. Pipher and B. Wick write that
\emph{``In the flag setting, we lack a suitable wavelet basis and this approach is not available.
	Essentially, the wavelet basis requires
	the construction of a suitable multi-resolution analysis, which we do not have in this flag setting.''}
The approach that the authors are referring to is a particular way to prove commutator estimates for SIOs.
More fundamental still is the alluded lack of a flag type multiresolution analysis in the form of
a wavelet or Haar type basis that this approach would require.
In this paper we will develop the flag adapted discrete Haar wavelet methodology
and show that there is a related powerful dyadic--probabilistic multiresolution approach to flag type problems,
including the commutators. In particular,
we can represent a new class of non-convolution form flag type SIOs in a dyadic manner
so that the discrete multiresolution faithfully respects the original flag structure of the operator.
That the decompositions indeed preserve the flag structure in a sharp manner can be seen
from the fact that the resulting dyadic model operators satisfy weighted estimates
with the optimal flag $A_p$ weights (which is a bigger class of weights than
the usual tri-parameter weights). In particular, our dyadic representation theorem
is a refined instance of a $T1$ theorem in this context and directly gives optimal
weighted estimates for our flag SIOs. As we explain later this is also the first time a full
David--Journ\'e style $T1$ theorem appears in this type of an adapted multi-parameter setting
that has an entangled spirit.

The flag setting is not the only specialized multi-parameter setting of interest.
In this regard this paper is a continuation of our recent work in the so-called
Zygmund dilations setting \cites{HLMV, ALM24}.
Zygmund dilations are the particular entangled dilations
$x \mapsto (sx_1, tx_2, stx_3)$, $x \in \R^3$, $s,t > 0$ -- essentially the simplest
``entangled'' dilations. While the flag setting is not as clearly associated
with an underlying dilation, it is in many other ways similar in spirit to the Zygmund case
as will become clear.
Our general goal is to study
these more specialized multi-parameter settings, where there is some kind
of entanglement or interaction between the different parameters, and
try to carry the related invariances of singular integrals to their new dyadic
counterparts.

A general theme of the research then is
obtaining optimal $L^p$ estimates allowed by the invariances.
Let $\calR_Z$ denote all of the Zygmund rectangles
$I \times J \times K$ with $\ell(K) = \ell(I)\ell(J)$ in $\R^3$
and $\calR_F$ denote all of the (tri-parameter) flag rectangles $I \times J \times K$
with $\ell(I) \le \ell(J) \le \ell(K)$. The entanglement is evident in
the form of the rectangles -- the side lengths in different parameters are not independent of each other.
For a class $\mathcal E$ of sets of finite positive measure, we define
\begin{equation*}
	\begin{split}
		M_{\mathcal E}f        & :=\sup_{E\in\mathcal E}1_E\ave{\abs{f}}_E, \qquad \ave{f}_E :=
		|E|^{-1} \int_E f,                                                                                          \\
		[w]_{A_p^{\mathcal E}} & :=\sup_{E\in\mathcal E}\ave{w}_E \ave{w^{-1/(p-1)}}_E^{p-1},\qquad p\in(1,\infty).
	\end{split}
\end{equation*}
Then $M_{\mathcal{E}}$ is $L^p(w)$-bounded if and only if
$[w]_{A_p^{\mathcal E}}<\infty $ in, for instance, each of these cases:
\begin{enumerate}
	\item $\mathcal{E}=\mathcal{Q}:=\{\text{cubes in }\R^d\}$: Muckenhoupt's theorem 1972;
	\item $\mathcal{E}=\mathcal{R}:=\{\text{rectangles in }\R^d\}$: easy iterative product space extension;
	\item $\mathcal{E} = \calR_F$: straightforward iterative argument;
	\item $\mathcal{E}=\mathcal{R}_{Z}$: R. Fefferman \& Pipher \cite{FEPI}.
\end{enumerate}
Beyond the one-parameter theory, in the above maximal function level, the only clearly
challenging result is (4) requiring the Cordoba--Fefferman algorithm; even
the flag case (3) is rather simple in the maximal function level (we record this for the convenience
of the reader in Proposition \ref{prop:maxfunc}).
All of the related singular integral questions are, however, significantly harder -- for instance,
using modern tools the weighted estimates in the pure multi-parameter case follow
from the dyadic representation theorem \cite{Ma1} as first observed in \cite{HPW2018}.
The weighted estimates for Zygmund SIOs, in the generality of SIOs with minimal smoothness and decay,
were only recently solved by two of the current authors together with T. Hyt\"onen and E. Vuorinen \cite{HLMV}.
Weighted estimates for some particular convolution form flag type SIOs
are recorded by Wu \cite{WuFlag}.
As in the Zygmund case, the weighted estimates provide
the most important benchmark to see that our newly developed flag adapted
multiresolution methods are sharp enough. That is why the first corollary of our
representation theorems must be the boundedness of our new flag SIO classes
with the flag weights -- and we achieve this.

The operators
in the Zygmund and flag settings are essentially more
special than usual multi-parameter operators in that they have some
additional decay, invariance or structure adapted to the given setting.
What this means is most obvious in the maximal function level as we saw above.
How the invariance should be captured in the singular integral level -- especially
if we venture outside of multipliers -- is not generally clear at all.
We will, nevertheless, carefully formulate and motivate a natural new non-convolution flag setting in this paper
that, in particular, captures the flag multipliers mentioned above.
To reiterate one of the main points, our singular integrals (just like the maximal functions above) are
on the kernel estimate level essentially
\emph{better} than the classical pure multi-parameter
variants -- but we want a lot out of this, namely, multiple different
enhanced estimates that are not true for general
multi-parameter singular integrals. We also want our $T1$ assumptions
to be flag adapted. This generally speaking means \emph{weaker} assumptions,
for instance, a weak boundedness property just for flag rectangles as opposed
to for general rectangles. The so-called product flag $\BMO$ space
relevant in the $T1$ conditions offers a very non-trivial instance of this latter
philosophy (we discuss this point carefully later).

The restriction in the form of the rectangles
in the Zygmund and flag settings is one manifestation
on how the parameters interact in these specialized settings, and developing
dyadic representations theorems that retain
the invariance of the given SIO is non-trivial. The benefits are clear:
as noted above our Zygmund theory \cite{HLMV} allowed us to, e.g., prove
optimal Zygmund weighted estimates for general Zygmund invariant SIOs with minimal smoothness and decay.
This turned out to be very delicate, and only the case of smooth multipliers
was previously understood by Fefferman--Pipher \cite{FEPI}.
The other previous generalistic methods from more classical
representation theorems \cites{Hy1, Ma1} to sparse domination \cite{Lerner2013ASP},
and their many, many generalizations, do
not provide tools to detect the finer structure granted by these more sophisticated invariances.
Our Zygmund result \cite{HLMV} was, in particular, the first entangled representation theorem.
However, it was still not general as far as representation theorems or $T1$ theorems go as it
was for cancellative SIOs only -- we explain what this means next.
Before that, though, we note another recent result with the same kind of spirit:
in \cite{HyRo} certain special causal SIOs are shown to satisfy ``causal sparse domination'', which is sparse
domination preserving a completely different invariance, namely, causality.

Recall that the $T1$ theorem of David--Journ\'e \cite{DJ1984} is the fundamental $L^2 \to L^2$ boundedness
characterization of classical SIOs $T$
requiring the weak boundedness property (WBP), like $|\langle T1_Q, 1_Q\rangle| \lesssim |Q|$
for all cubes $Q$, and the famous $T1$ conditions $T1, T^*1 \in \BMO$. The main $T1$ assumptions are automatic
for translation invariant convolution operators $Tf = K * f$ in the form $T1 = T^*1 = 0$.
The $T1 = T^*1 = 0$ case is called the cancellative or paraproduct free case and already covers many key operators
(it does not necessarily require convolution form).
Of course, the mere $L^2$ boundedness for convolution operators $K * f$ can also sometimes
be simply established via Plancherel by checking $\widehat K \in L^{\infty}$.
However, the more recent dyadic representation theorems, such as \cites{Hy1, Ma1},
transform David--Journ\'e type assumptions into an exact dyadic identity for $T$
as an average of dyadic model operators
and thus, in addition, allow reducing most further questions to easier, purely dyadic variants.
These much deeper applications include various weighted bounds, like the $A_2$ conjecture in \cite{Hy1}
or the optimal Zygmund weighted estimates in \cite{HLMV}. They also include
commutator estimates, but more on these later.
Because of these key applications that are often non-trivial even for the simplest
of SIOs, cancellative versions
of dyadic representation theorems remain powerful. But they are certainly not completely general and
sometimes it is a very significant challenge to push further -- many obstacles
appear already in pure multi-parameter settings from hybrid operators, like partial paraproducts,
to complicated product $\BMO$ type spaces of Chang and Fefferman \cites{CF1, CF2}. In entangled
situations these obstacles are severe.
There are also situations, like non-doubling analysis,
where convolution operators do not justify the usefulness of the cancellative case $T1 = 0$.
Precisely the dyadic multiresolution methods that we develop here
are the only known attack vector to such non-doubling results, see
\cite{Nazarov2003} for the original one-parameter theory or \cite{Hytonen2014}
for non-doubling analysis in the pure bi-parameter setting. These potential
applications to non-doubling analysis in entangled settings add value to our methods.

We now explain very carefully what we develop in this paper and all of the key philosophies behind them,
as we see them.

\subsection{Basic dyadic flag multiresolution analysis}
We outline briefly the structure of our basic multiresolution analysis.
This is merely the starting point for the proof of a dyadic representation theorem, for instance.
For this discussion say we are in the tri-parameter flag setting in $\R^d = \R^{d_1 + d_2 + d_3}$
and $\calD = \calD^1 \times \calD^2 \times \calD^3$, where $\calD^m$ is a dyadic lattice
in $\R^{d_m}$. The dyadic flag rectangles are $\calD_F = \{I \in \calD \colon
	\ell(I^1) \le \ell(I^2) \le \ell(I^3)\}$ and we often need
subcollections like $\calD_{<, =}$ meaning that $\ell(I^1) < \ell(I^2) = \ell(I^3)$.
Our most basic flag decomposition
is
$$
	f=\sum_{I\in \calD_F} \Lambda_{I^1}\Lambda_{I^2}\Delta_{I^3}f=:\sum_{I\in \calD_F}\Delta_{I, F}f,
$$
where $\Delta$ is a one-parameter martingale difference and, for $m = 1,2$,
$\Lambda_{I^m}$ is defined
so that it implicitly depends on $\ell(I^{m+1})$:
$\Lambda_{I^m} f = \Delta_{I^m} f$ if $\ell(I^m) < \ell(I^{m+1})$
and $\Lambda_{I^m} f = \Delta_{I^m} f + E_{I^m} f$ if $\ell(I^m) = \ell(I^{m+1})$,
where $E$ is the averaging operator $E_A f := 1_A \langle f \rangle_A$. So in the
flag martingale difference $\Delta_{I, F}$ we trade
some cancellation (compared to pure tri-parameter theory)
in the first or second parameters in the cases $\ell(I^1) = \ell(I^2)$
and $\ell(I^2) = \ell(I^3)$, respectively, to be able to restrict the summation to flag
rectangles only. Being later able to deal with this lack of cancellation in connection
with singular integrals is a key part of our theory.
The following weighted square function estimate is the first key indication that
we can continue to build on this multiresolution.
\begin{mainthm}
	For the dyadic flag square function
	\[
		S_F f = S_{\calD_F} f := \Big(\sum_{I\in \calD_F}|\Delta_{I, F} f|^2\Big)^{1/2}
	\]
	we have
	\begin{equation*}
		\| S_F f\|_{L^p(w)}\lesssim \|f\|_{L^p(w)}
	\end{equation*}
	for all $1<p<\infty$ and $w\in A_{p, F}$.
	We also have the following lower bound for all $0<q<\infty$ and $v\in A_{\infty, F}$:
	\begin{equation*}
		\|f\|_{L^q(v)}\lesssim \| S_F f\|_{L^q(v)}.
	\end{equation*}
\end{mainthm}

Suppose $(f,g) \mapsto B(f, g)$ is a bilinear form. We could now formally decompose
$$
	B(f, g) = \sum_{I, J \in \calD_F} B(\Delta_{I, F} f, \Delta_{J, F} g)
$$
simply by decomposing $f$ and $g$ in the flag sense like above. However, we
want to, in addition, have $\ell(I^i) = \ell(J^i)$ for all $i=1,2,3$ --
we denote this requirement by $\ell(I) = \ell(J)$. This is kind of
important even just from a philosophical perspective -- now some
partial dual can mix the rectangles to $I^1 \times J^2 \times I^3$, say,
but this remains flag (this is certainly not the only reason why this is useful).
We can achieve this with the price of
losing some further cancellation -- we show that we can write
\begin{equation}\label{eq:flagB}
	B(f, g) = \sum_{s \in \calS} \Sigma_s, \qquad \calS := \{(s_1, s_2, s_3) \colon s_i \in \{\Delta, E, D\}\},
\end{equation}
where we have precise rules, which we do not carefully explain here, on how to decipher what $\Sigma_s$ is for a given tuple
$s$. But basically (there are some intricate details) the symbol describes the symbol on the first function
and the other one can be inferred (and $D$ denotes the case where there is a $\Lambda$ in both).
For instance, we have
$$
	\Sigma_{(\Delta, E, E)} :=
	\sum_{\substack{ I, J \in \calD_{<, <} \\ \ell(I) = \ell(J) }} B(\Delta_{I^1} E_{I^2} E_{I^3} f, E_{J^1}\Delta_{J^2} \Delta_{J^3} g).
$$
(One intricate detail is are we summing over $\calD_F$ or, say, over $\calD_{<, <}$ as here.)
This is our basic multiresolution decomposition. As the first step in the proof
of a dyadic representation theorem this can be applied with $B(f, g) = \langle Tf, g\rangle$, where $T$ is a
flag singular integral. It can also be applied with $B(f, g) = fg$ to obtain a flag style paraproduct decomposition of
products relevant in commutator theory. Once the rules are explained, the notation is extremely convenient,
and this basic decomposition is a critical first step to many of our more refined decompositions.

\subsection*{The dyadic flag $T1$ representation theorem}
This singular integral part of this paper splits to the bi-parameter and tri-parameter flag settings.
One of our goals is to show that our methods are powerful enough to deal with the completely general
multi-parameter flag setting and restricting to the simplest bi-parameter theory would
not necessarily achieve this. On the other hand, the tri-parameter setting already
serves as a very good model of the general case and there we develop a cancellative $T1$ theorem.
However, we do even more in the bi-parameter flag setting: we develop the full flag analog
of the David--Journ\'e $T1$ theorem.
We do not see an obstacle (except perhaps notational) for a full $T1$
representation theorem also in the tri-parameter (or even multi-parameter) flag setting
-- but we felt that our strategy of writing this paper conveys all of the key ideas
(how to do general multi-parameter and how to do general $T1$) in the cleanest possible way.
Again, the fact that we do an actual full $T1$ in this type of an ``entangled like''
multi-parameter setting is completely new as they have only previously appeared in the
pure product theory \cites{Jou1985, Ma1, Ou}.
On the other hand, any
type of flag native representation theorem, cancellative or not, is also already novel.
On a very basic level this encompasses developing the multiresolution analysis that we
just discussed and was reported
missing in \cite{DLOPW-FLAG}. But it is quite a bit more demanding than just that.
Our methods do not only yield this more general
singular integral framework, but also give a critical new attack vector to commutator
problems -- and we exploit and develop these already in this paper.

The following sharp multiresolution decomposition of SIOs is our first main theorem.
\begin{mainthm}\label{thm:main1}
	Let $T$ be a bi-parameter flag CZO satisfying flag adapted
	$T1$ assumption as in Section \ref{sec:biparCZO}. Then we have
	\begin{align*}
		\langle Tf, g\rangle = C \E \Big(
		 & \sum_{k^1,\, k^2 = 0}^{\infty} 2^{-\alpha_1 k^1} 2^{-\alpha_2 k^2}
		\Big[ \langle Q_k f, g \rangle + \sum_{l=1}^{k^1} \langle R_{k, l}f, g\rangle \Big]
		+ \sum_{k^1,\, k^2 = 0}^{\infty} 2^{-\alpha_2 k^2} \langle R_{k, 0} f, g\rangle                  \\
		 & + \sum_{k^1 = 0}^{\infty} 2^{-\alpha_1 k^1} \langle P_{k^1}^1f, g \rangle +
		\sum_{k^2 = 0}^{\infty} 2^{-\alpha_2 k^2} \langle P_{k^2}^2f, g \rangle\Big)                     \\
		 & + \E \langle \Pi_1(T1, f) + \Pi_1^*(T^*1, f) + \Pi_2(T_11, f) + \Pi_2^*(T_1^*1, f), g\rangle,
	\end{align*}
	where the average is over dyadic lattices and $Q_k$ is a flag shift (in the random lattice)
	as in Definition \ref{def:flagshift},
	$R_{k, l}$ is an almost one-parameter shift as in Definition \ref{def:almostone},
	the operators $P$ are flag partial paraproducts as in Definition \ref{def:partialp}
	and the $\Pi$ operators are full flag paraproducts as in Section \ref{sec:fullpararep}.
	Here $C$ depends only on the kernel estimates and the various CZO assumptions.
	Alternatively, let $T$ be a cancellative tri-parameter flag CZO as in Section \ref{sec:triparCZO}.
	Then $T$ obeys a similar representation involving various shifts but no paraproducts.
	In both cases we have
	$$
		\|Tf\|_{L^p(w)} \lesssim \|f\|_{L^p(w)}
	$$
	for all $p \in (1, \infty)$ and optimal flag weights $w \in A_{p, F}$.
\end{mainthm}

There are many things to unpack and discuss related to this.
We first discuss the cancellative part of the decomposition -- the so-called flag
shifts and almost one-parameter shifts. These can never be avoided. We limit
the discussion to the bi-parameter case.
Abstractly, the bilinear forms of all
of the bi-parameter shifts will have, for a fixed $k = (k^1, k^2)$, the form
$$
	(f,g) \mapsto \sum_{K \in \calF} \sum_{I^{(k)} = J^{(k)} = K} a_{IJK} \langle f, \varphi_{I,J}\rangle
	\langle g, \psi_{I, J}\rangle,
$$
where some collection of dyadic rectangles $\calF$ appears, some abstract functions
$\varphi_{I, J}$, $\psi_{I, J}$ appear and
the coefficients $a_{IJK}$ need to satisfy a certain flag specific normalization.
The minimum requirement on the collection $\calF$ is that $K \in \calF$ guarantees
that $I, J$ are flag -- this means that $2^{-k^1}\ell(K^1) \le 2^{-k^2}\ell(K^2)$
(notice that $K$ does not need to be flag).
The functions $\varphi_{I,J}, \psi_{I, J}$
have some cancellation properties, some constancy properties, some size conditions
and some support conditions. In the most ideal situation $\varphi_{I, J}$ really
only depends on $I$ and is some appropriate Haar function related to $I$. But
we critically need to, at least sometimes when we need to ``improve the cancellation'',
allow that $\supp \varphi_{I, J}
	\subset I \cup J \cup (J^1 \times I^2) \cup (I^1 \times J^2)$. But some
of these ``mixed supports'' cases, like $I^1 \times J^2$, can lead
to a situation in the proof where an average like $\langle f \rangle_{K^1 \times J^2}$
appears. But this is not necessarily flag -- while $\ell(I^1) \le \ell(I^2) = \ell(J^2)$
we have now enlarged the smaller cube $I^1$ to $K^1$ and it can be that $\ell(K^1) > \ell(J^2)$.
So a flag maximal function cannot be used to control such an average, leading to an
obvious issue with weighted estimates.
If we can guarantee that $\calF$ is such that $\ell(I^1) \ll \ell(I^2)$ (i.e., $I^1$ is much smaller
than $I^2$) then this problem does not arise.
Here are where the ``almost one-parameter'' shifts enter the picture.
We need to use, depending on the complexity $k$, finitely many of these
to arrange this restriction (they also serve another purpose --
to arrange more cancellation to our partial paraproducts). The key point is that in these almost one-parameter shifts we have
$\ell(I^1) \sim_l \ell(I^2)$ and that in this one-parameter boundary we need \emph{less} cancellation.
The need for the cancellation, on the other hand, is what sometimes requires blowing up
the supports of the functions $\varphi_{I, J}$ and $\psi_{I, J}$. So, we do not
``improve the cancellation'' in this range leading to the fact that the support conditions
of these functions are better and, thus, the problematic averages do not appear. Instead,
we use the one-parameter philosophy here to deal with the poorer cancellation. With this
maneuver we can in other places impose the additional structure to $\calF$ (guaranteeing that $\ell(I^1) \ll \ell(I^2)$)
allowing us to access the required enhanced cancellation in this pure bi-parameter range
but also not suffer from the support conditions.

There are other interesting features regarding the cancellative part. One is that the kernel estimates
are tailored to imply a flag specific normalization
$|a_{IJK}| \le |I| / |K_F|$, where $K_F \supset K$ is
some appropriate flag enlargement of $K$. This only
provides extra decay over the more familiar ratio $|I| / |K|$ (appearing in other representation theorems)
if $K$ is not flag.
This occasional extra decay is, nevertheless, critical to counterbalance some of the lost
cancellation, a natural feature of the flag multiresolutions, in the $\ell(I^1) = \ell(I^2)$ case.
This is related to the term $\sum_{k^1,\, k^2 = 0}^{\infty} 2^{-\alpha_2 k^2} \langle R_{k, 0} f, g\rangle$
in the above representation theorem -- notice that there is no decay in $k^1$.
Utilizing the decay coming from this ratio correctly is a bit demanding
requiring, in particular, interpolation with a change of measure.

Next, we discuss the various paraproducts and the related $T1$ condition.
What is the right condition on $T1$? As we know
from multi-parameter $T1$ theorems, it should be some kind of ``product''
$\BMO$ space. In this setting there is a natural
flag product $\BMO$ space -- call it $\BMO_{\pro, F}$ -- and we
need to assume as part of our definition of a flag CZO
that $T1 \in \BMO_{\pro, F}$ for the full paraproduct associated with $T1$ to be bounded
with flag weights. This simply means that, uniformly over all dyadic grids, we have
the dual square function estimate
$$
	|\langle b, f\rangle| \lesssim \Big\|\Big(\sum_{I \in \calD_F}
	|\Delta_{I, F} f|^2 \Big)^{1/2}\Big\|_{L^1}.
$$
We discuss this dual viewpoint at length later, but it is our favorite angle
on $\BMO$ spaces (however, this can also be formulated as a Carleson condition, if desired).
In particular, this is not the same space as the classical bi-parameter product $\BMO$ space $\BMO_{\pro}$
of Chang and Fefferman that appears in pure bi-parameter $T1$/representation theorems \cites{Jou1985, Ma1},
and this space is not even obviously comparable to it.
This is because while there are less flag rectangles than general rectangles,
the flag martingale differences of a given function
in a flag rectangle $I^1 \times I^2$ only agree
with the bi-parameter martingale difference if $\ell(I^1) < \ell(I^2)$ and
can be less cancellative in the ``one-parameter boundary'' $\ell(I^1) = \ell(I^2)$.
Critically, is this assumption even necessary? In addition, is $T1 \in \BMO_{\pro, F}$ a weaker assumption than
$T1 \in \BMO_{\pro}$ even if the $\BMO$ spaces are not directly comparable?
We answer both of these questions carefully in Section \ref{sec:BMO}. First, the condition
$T1 \in \BMO_{\pro, F}$ is necessary for the $L^2$ boundedness of Flag SIOs. This is because
of the new necessity result Theorem \ref{thm:T1nec} stating that in the bi-parameter
boundary $\ell(I^1) = \ell(I^2)$ flag SIOs satisfy a ``partially one-parameter''
$\BMO$ condition. This combined with the usual Journ\'e covering lemma \cite{Journe1986}
in the $\ell(I^1) < \ell(I^2)$ regime gives the necessity. We restate Theorem \ref{thm:T1nec} here:
\begin{mainthm}
	Assume that $T$ is a bi-parameter flag SIO satisfying the cube testing condition
	\[
		\|1_Q T1_Q\|_{L^2} \lesssim |Q|^{1/2}
	\]
	for all cubes $Q$. Then we have
	\[
		\sup_{Q \in \calD_=} |Q|^{-\frac 12}
		\Big(\sum_{\substack{I\in \calD_= \\ I\subset Q} }|\langle T1, h_{I^1}^0\otimes h_{I^2}\rangle|^2
		\Big)^{1/2}<\infty.
	\]
\end{mainthm}
Beyond necessity the above also shows that for flag SIOs satisfying the very weak \textbf{one-parameter} style testing
$\|1_Q T1_Q\|_{L^2} \lesssim |Q|^{1/2}$ for \textbf{cubes} $Q$ the assumption $T1 \in \BMO_{\pro, F}$
is a weaker assumption than $T1 \in \BMO_{\pro}$ (which would, by the way, also be necessary
by Journ\'e's covering theorem \cite{Journe1986} as flag SIOs are also classical bi-parameter SIOs).

Recall that the recent Zygmund representation \cite{HLMV} was cancellative so no paraproducts or $T1$ conditions arose.
Recent progress towards understanding
paraproducts in the Zygmund setting was made by two of us and E. Airta in \cite{ALM24},
where the corresponding Zygmund full paraproducts were bounded. This was related to commutators $[b, T]f
	= bTf - T(bf)$; paraproducts are also fundamental for decomposing
products, such as, $bTf$ (hence the name paraproduct). However,
this did not require bounding partial paraproducts as they only appear in
the multiresolution decomposition of the bilinear form $\langle Tf, g\rangle$
related to $T1$ theorems/representation theorems. And it is, indeed, these hybrid
operators -- the partial paraproducts -- that are the most difficult model operators.
We did not know how to handle them in the entangled Zygmund setting, and
this was the main reason our Zygmund theory was restricted to the cancellative case.
The flag setting is the first setting beyond the pure multi-parameter setting
where we are now able to develop a satisfactory theory of partial paraproducts,
and thus a complete $T1$ theorem.
We have reason to believe that this flag result can help
us to develop a full $T1$ theorem also in the Zygmund setting -- it is not
just about bounding the partial paraproducts, it is also about the choice of the multiresolution
so that the structure and the available cancellation play nicely with kernel estimates.

\subsection*{Commutators}
We discuss commutators next, since
they offer another way to capture these multi-parameter invariances in a sharp way.
In their simplest form commutators are defined by $[b,T]f := bTf-T(bf)$.
Here $b$ is a function and $T$ can be an SIO of any type (one-parameter, multi-parameter, flag or Zygmund, say).
We study these commutators and the role of flag-adapted $\BMO$ spaces.
The general idea is that optimal commutator characterizations, in terms of the space
where the symbol $b$ should belong to, should involve only the natural invariance
specific rectangles so that the structure of the SIO is really exploited to get optimal results
better than those that hold for pure multi-parameter operators.
Most naturally this is via some adapted $\BMO$ style spaces.
This is exactly the same type of question than the weighted question:
can we beat generic multi-parameter estimates given the extra invariances?

The natural $\BMO$ for these simple commutators is the ``little'' $\BMO$.
This is the most obvious analog of the classical one-parameter $\BMO$ -- just
replace the supremum over cubes by the supremum over the relevant rectangles (bi-parameter,
flag, Zygmund, \ldots) -- for instance,
$$
	\|b\|_{\bmo_F} := \sup_{I \textup{ flag}} \frac{1}{|I|} \int_I |b - \langle b \rangle_I|.
$$
To get a handle of this space, this definition is often not the most useful --
we almost always prefer the corresponding dual square function estimates.
For instance, the bi-parameter flag little $\BMO$ is equivalent to being in the standard one-parameter
$\BMO(\R^d)$ and, uniformly over $x_1$, in $\BMO(\R^{d_2})$. Therefore, on the dual side it is equivalent
to having both of the following estimates \emph{uniformly over all} dyadic grids $\calD^m$ in $\R^{d_m}$:
\begin{equation*}
	|\langle b, f\rangle| \lesssim \Big\|\Big(\sum_{I \in \calD_{=}} |\Delta_I f|^2 \Big)^{1/2}\Big\|_{L^1}
	\qquad \textup{and} \qquad
	|\langle b, f\rangle| \lesssim \Big\|\Big(\sum_{I^2 \in \calD^2} |\Delta_{I^2} f|^2 \Big)^{1/2}\Big\|_{L^1},
\end{equation*}
where we recall that $\calD = \calD^{12} := \calD^1 \times \calD^2$ and $\calD_=$ consists of the cubes $
	I = I^{12} = I^1 \times I^2 \in \calD$ with
$\ell(I^1) = \ell(I^2)$.

The boundedness proof of the
basic commutator $[b,T]$ is often, not always, trivialized as it is amenable to the ``Cauchy integral trick''.
This is a complex analysis trick exploiting a known connection of $\BMO$ with $A_p$ weights.
It works in most multi-parameter settings -- even in the flag setting, see \cite{DLOPW-FLAG}.
This means that one does not have to bound the commutator $[b,T]$ at all -- its boundedness simply
follows from this trick directly via the weighted boundedness of the original operator $T$.
But if you want
the optimal flag little $\BMO$ you do need to have the optimal flag weighted
estimates, precisely the ones we discussed in Theorem \ref{thm:main1}.
Thus, since our flag CZO class is widely more general than ever before, and they are all weighted bounded by our theory,
we can directly get the basic commutator estimates more generally than in \cite{DLOPW-FLAG} (where
they rely on \cite{WuFlag} for the weighted bounds of some particular SIOs). So we get
$$
	\|[b, T]f\|_{L^p(w)} \lesssim \|b\|_{\bmo_F}\|f\|_{L^p(w)}
$$
for all $1 < p < \infty$ and $w \in A_{p, F}$ whenever
$T$ is a general bi-parameter flag CZO or a cancellative tri-parameter
flag CZO.

We also show that our flag multiresolutions handle these commutators in the cancellative
bi-parameter case directly, without relying on the Cauchy integral trick. This is not a cosmetic point --
this is the base case that demonstrates one can use these methods to study
iterated commutators, such as $[T_1, [b, T_2]]$ for two flag CZOs $T_i$, using these methods.
Such commutators would be genuinely beyond the Cauchy trick.
This relies on the flag paraproduct decomposition obtained from (the bi-parameter
version) of \eqref{eq:flagB} with $B(b,f) = bf$. We write the result of that as
$bf = \sum_{s \in \calS} \pi_s(b, f)$. The worst case is when
$b$ has no cancellation and no $\BMO$ philosophy can save the day -- this
is $s = (E, E)$ in our notation. This case has to be handled inside the
commutator proof. We decompose $\calS \setminus \{(E,E)\} = \calS_1 \cup \calS_2$,
where $s_m \ne E$ for both $m=1,2$ if $s \in \calS_1$.
Bounding these new
flag paraproducts is interesting and one also critically needs
the non-trivial embedding $\bmo_F \subset \BMO_{\pro, F}$ (proved in Theorem \ref{thm:bmoembedding}).
We prove the following boundedness result for our flag paraproducts.
\begin{mainthm}
	Let $1 < p < \infty$ and $w \in A_{p, F}$.
	If $s \in \calS_1$ we have that
	$$
		\|\pi_s(b, f)\|_{L^p(w)} \lesssim \|b\|_{\BMO_{\pro, F}} \|f\|_{L^p(w)}.
	$$
	If $s \in \calS_2$ we have that
	$$
		\|\pi_s(b, f)\|_{L^p(w)} \lesssim \|b\|_{\bmo_F} \|f\|_{L^p(w)}.
	$$
	In particular, for $s \in \calS_1 \cup \calS_2$ we have
	$$
		\|\pi_s(b, f)\|_{L^p(w)} \lesssim \|b\|_{\bmo_F} \|f\|_{L^p(w)}.
	$$
\end{mainthm}
Also, the
case $\calS_1$ above covers all the full paraproducts
that are relevant in the $T1$ representation theorem. The case
$\calS_2$ is only needed in the commutator part.

\vspace{0.3cm}

\noindent \textbf{Acknowledgements.}
K. Li was supported by the National Natural Science Foundation of China through
project numbers 12671125 and 12222114. This material is based upon work supported
by the National Science Foundation (NSF) under Grant No. 2247234 (H. Martikainen).
H.M. was, in addition, supported by the Simons Foundation through MP-TSM-00002361
(travel support for mathematicians).

H.M. thanks Cody Waters for very useful comments.

No LLMs were used in the development of this manuscript apart from the final language and consistency checks.

\section{Preliminaries on flag weights and maximal functions}
We discuss the maximal functions in the tri-parameter setting (the bi-parameter flag
setting is easily understood from this).
We say that a non-negative locally integrable function $w$ belongs to $A_{p, F} = A_{p, F}(\R^d)$,
where $d = d_1 + d_2 + d_3$ and $1 < p < \infty$, if
\[
	[w]_{A_{p, F}}:=  \sup_{I\in \mathcal R_F}
	\langle w\rangle_{I} \langle w^{-\frac 1{p-1}}\rangle_{I}^{p-1}<\infty.
\]
Here $\mathcal R_F$ stands for the collection of \emph{all} flag rectangles in $\R^d$, i.e., the rectangles
$I=I^1\times I^2\times I^3$ with $I^i \subset \R^{d_i}$ and $\ell(I^1)\le \ell(I^2)\le \ell(I^3)$.
We go through a few easy facts about flag weights -- these are well-known but we prove
some of them for the convenience of the reader.
Using the Lebesgue differentiation theorem it is easy to see the part (i) of the following theorem,
where $A_p(\R^d)$ denotes the usual one-parameter Muckenhoupt class. Denote also,
for instance, $d_{23} = d_2 + d_3$.
\begin{prop}\label{prop:weight}\leavevmode
	\begin{enumerate}
		\item[(i)]Let $w\in A_{p, F} = A_{p, F}(\R^d)$, $1<p< \infty$. Then
		      \begin{enumerate}
			      \item[$\bullet$]  $w\in A_{p}(\R^{d})$ with $[w]_{A_p(\R^{d})}\le [w]_{A_{p, F}}$.
			      \item[$\bullet$] $w(x_1, \cdot )\in A_{p}(\R^{d_{23}})$ for a.e. $x_1\in \R^{d_1}$
			            with $$
				            \esssup_{x_1\in \R^{d_1}}[w(x_1, \cdot )]_{A_p(\R^{d_{23}})}\le [w]_{A_{p, F}}.
			            $$
			      \item[$\bullet$]  $w(x_1, x_2, \cdot )\in A_{p}(\R^{d_3})$ for a.e.
			            $x_1\in \R^{d_1}, x_2\in \R^{d_2}$ with
			            $$
				            \esssup_{ x_1\in \R^{d_1}, x_2\in \R^{d_2}}[w(x_1, x_2, \cdot )]_{A_p(\R^{d_3})}
				            \le [w]_{A_{p, F}}.
			            $$
		      \end{enumerate}
		\item[(ii)] Conversely, if
		      \[
			      [w]_{A_p(\R^{d})}+ \esssup_{x_1\in \R^{d_1}}[w(x_1, \cdot )]_{A_p(\R^{d_{23}})}+
			      \esssup_{ x_1\in \R^{d_1}, x_2\in \R^{d_2}}[w(x_1, x_2, \cdot )]_{A_p(\R^{d_3})}<\infty,
		      \] then $w\in A_{p, F}$.
	\end{enumerate}
\end{prop}

\begin{rem}\label{rem:weight}
	We may define $A_{\infty, F}:=\cup_{p>1} A_{p, F}$. Then if $v\in A_{\infty, F}$, by definition
	there exists some $p_0>1$ such that $v\in A_{p_0, F}$.
	Then we see that a version of Proposition \ref{prop:weight} also holds for the case $p=\infty$.
\end{rem}

\begin{proof}[Proof of Proposition \ref{prop:weight} (ii)]
	To see this, fix $x$ and an arbitrary $I \in \calR_F$ with $x \in I$. Then,
	let $J^3$ be a cube in $\R^{d_3}$ such that $J^3\subset I^3$, $\ell(J^3)=\ell(I^2)$
	and $x_3 \in J^3$ (here we used that $\ell(I^2) \le \ell(I^3)$).
	It is clear that
	\[
		\langle |f|\rangle_I \le \langle M_{3} f\rangle_{I^1\times I^2 \times J^3},
	\]where $ M_{3}$ is the one-parameter maximal function in $\R^{d_3}$.
	Next, let $L^{23}$ be a cube in $\R^{d_{23}}$ such that $L^{23}\subset I^2 \times J^3$, $\ell(L^2)=\ell(I^1)$
	and $(x_2, x_3) \in L^{23}$ (here we used that $\ell(I^1) \le \ell(I^2)$). We now have
	\[
		\langle M_{3} f\rangle_{I^1\times I^2 \times J^3}\le \langle M_{23}M_{3} f\rangle_{I^1\times L^2 \times L^3},
	\]
	where $ M_{23}$ is the one-parameter maximal function in $\R^{d_{23}}$.
	Now, $I^1\times L^2 \times L^3$ is a cube in $\R^d$ containing the point $x$, and so
	\begin{equation}\label{eq:e0206}
		M_F f(x)\le M_{123} M_{23}M_{3} f(x),
	\end{equation}
	where $M_{123}$ stands for the one-parameter maximal function in $\R^{d}$ and
	\[
		M_F f(x):=  \sup_{I\in \calR_F}1_I(x) \langle |f|\rangle_I.
	\]
	Then  $M_F$ is bounded on $L^p(w)$, and the claim follows via the following Proposition \ref{prop:maxfunc}.
\end{proof}

Muckenhoupt weights are closely related with the corresponding  maximal function,
which is precisely the following
proposition.
\begin{prop}\label{prop:maxfunc}
	Let $1<p<\infty$. Then $\| M_F f\|_{L^p(w)}\lesssim \|f\|_{L^p(w)}$ if and only if $w\in A_{p, F}$.
\end{prop}
\begin{proof}
	The ``if'' part is a direct consequence of \eqref{eq:e0206} and Proposition \ref{prop:weight} (i).

	To see the ``only if'' part, suppose that $M_F$ is bounded on $L^p(w)$. Set
	$$
		\sigma_N:=w^{-\frac1{p-1}}1_{\{N^{-1}<w<N\}},\qquad N>1.
	$$
	For any $I\in\calR_F$, testing with $\sigma_N1_I$ gives
	\begin{align*}
		w(I)\Big(\frac{\sigma_N(I)}{|I|}\Big)^p
		 &\le \|1_I M_F(\sigma_N1_I)\|_{L^p(w)}^p \\
		 &\lesssim \int_I\sigma_N^p w=\sigma_N(I).
	\end{align*}
	Hence
	\[
		\frac{w(I)}{|I|}\Big(\frac{\sigma_N(I)}{|I|}\Big)^{p-1}\lesssim 1.
	\]
	Letting $N\to\infty$, monotone convergence gives $w\in A_{p,F}$.
\end{proof}

Proposition \ref{prop:maxfunc} tells us exactly that $\calR_F$ is a Muckenhoupt basis.
Therefore, by the well-known Muckenhoupt basis theory (see \cite{CUMP-book}*{Theorem 3.9})
we have the following.

\begin{lem}[Extrapolation]\label{lem:extra}
	Let $(f, g)$ be a pair of non-negative functions.
	Suppose that for some $p_0\in (1, \infty)$ and every $w\in A_{p_0, F}$
	we have
	\[
		\int f^{p_0} w\le \varphi([w]_{A_{p_0, F}}) \int g^{p_0} w,
	\]
	where $\varphi$ is a non-decreasing function. Then for every $p\in (1, \infty)$ and every $w\in A_{p, F}$
	we have
	\[
		\int f^{p} w\le \wt\varphi([w]_{A_{p, F}}) \int g^{p} w,
	\]
	where $\wt\varphi$ is some non-decreasing function.

\end{lem}

On the other hand, the proof of Proposition \ref{prop:maxfunc} also reveals an important property of $A_{p, F}$
weights. We record it as the following
\begin{prop}\label{prop:open}
	Let $1<p<\infty$ and $w\in A_{p, F}$. Then there exists some $\eps:=\eps([w]_{A_{p, F}})>0$  such that
	$w^{1+\eps}\in A_{p, F}$.
\end{prop}
\begin{proof}
	Take
	\[
		\eps:= \frac 1{c_d[w]_{A_{p, F}}^{\max\{1, \frac 1{p-1}\}}},
	\]
	where $c_d$ is a large constant depending only on the dimension.
	By the well-known one parameter theory and Proposition \ref{prop:weight} (i) we have
	\begin{align*}
		[w^{1+\eps}]_{A_p(\R^{d})} & \lesssim  [w]_{A_p(\R^{d})}^{1+\eps}\le [w]_{A_{p, F}}^{1+\eps}\lesssim  [w]_{A_{p, F}}.
	\end{align*}
	In a similar fashion we also have
	$$
		\esssup_{x_1\in \R^{d_1}}[w^{1+\eps}(x_1, \cdot )]_{A_p(\R^{d_{23}})}\lesssim [w]_{A_{p, F}}
	$$
	and
	$$
		\esssup_{ x_1\in \R^{d_1}, x_2\in \R^{d_2}}[w^{1+\eps}(x_1, x_2, \cdot )]_{A_p(\R^{d_3})}
		\lesssim [w]_{A_{p, F}}.
	$$
	Then the conclusion is a consequence of Proposition \ref{prop:weight} (ii).
\end{proof}

\section{Basic dyadic flag multiresolutions}
\subsection{Martingale differences and square functions}
Given a dyadic grid of cubes $\calD$ in $\R^d$, $I \in \calD$ and $k \in \Z$, $k \ge 0$,
we use the following notation:
\begin{enumerate}
	\item $\ell(I)$ is the side length of $I$.
	\item $I^{(k)} \in \calD$ is the $k$th parent of $I$, i.e.,
	      $I \subset I^{(k)}$ and $\ell(I^{(k)}) = 2^k \ell(I)$.
	\item $\ch(I)$ is the collection of the children of $I$, i.e.,
	      $\ch(I) = \{J \in \calD \colon J^{(1)} = I\}$.
	\item $E_I f=\langle f \rangle_I 1_I$ is the averaging operator on $I$,
	      where $\langle f \rangle_I := \fint_{I} f := \frac{1}{|I|} \int _I f$.
	\item $\Delta_I f$ is the one-parameter martingale difference
	      $\Delta_I f= \sum_{J \in \ch (I)} E_{J} f - E_{I} f$.
	\item $\Delta_{I,k} f$ is the martingale difference block
	      $$
		      \Delta_{I,k} f=\sum_{\substack{J \in \calD \\ J^{(k)}=I}} \Delta_{J} f.
	      $$
\end{enumerate}
Now, suppose for the moment that we are in the bi-parameter product space
$\R^{d_1} \times \R^{d_2}$ and we are given
a dyadic grid $\calD^m$ in each $\R^{d_m}$.
We can form the collection of dyadic rectangles $\calD = \calD^{12} := \calD^1 \times \calD^2$,
and the dyadic grid
$$
	\calD_{=} := \{I = I^1 \times I^2 \in \calD \colon \ell(I^1) = \ell(I^2) \}.
$$
Given $I = I^1 \times I^2 \in \calD_{=}$ (or even just in $\calD$) we have the
(one-parameter) martingale difference like defined above:
$$
	\Delta_I f = \sum_{\substack{ J^1 \in \ch(I^1) \\ J^2 \in \ch(I^2) }} E_{J^1 \times J^2} f
	- E_{I^1 \times I^2}f.
$$
On the other hand, we can also form the bi-parameter martingale difference
$$
	\Delta_{I^1} \Delta_{I^2} f,
$$
where we, for example, understand that $\Delta_{I^1} f(x) = \Delta_{I^1}^1 f(x)
	:= (\Delta_{I^1} f(\cdot, x_2) )(x_1)$ for $x = (x_1, x_2) \in \R^{d_1} \times \R^{d_2}$.
That is, lower dimensional operators act parameter wise like this. Our notation will
be consistent -- $\Delta_{I^1 \times I^2}$ is always
the one-parameter martingale difference (and not a shorthand for the bi-parameter
martingale difference $\Delta_{I^1}\Delta_{I^2}$). But what is the explicit connection
between $\Delta_{I^1\times I^2}$ and the bi-parameter martingale difference $\Delta_{I^1}\Delta_{I^2}$?
Notice that
\begin{align*}
	\Delta_{I^1} \Delta_{I^2} f & = \Delta_{I^1} \Big( \sum_{J^2 \in \ch(I^2)} E_{J^2} f - E_{I^2}f \Big)                                             \\
	                            & = \sum_{\substack{ J^1 \in \ch(I^1)                                     \\ J^2 \in \ch(I^2) }} E_{J^1 \times J^2} f
	- \sum_{J^2 \in \ch(I^2)} E_{I^1 \times J^2} f
	- \sum_{J^1 \in \ch(I^1)} E_{J^1 \times I^2} f + E_{I^1 \times I^2} f,
\end{align*}
$$
	E_{I^1} \Delta_{I^2} f = \sum_{J^2 \in \ch(I^2)} E_{I^1 \times J^2} f - E_{I^1 \times I^2} f
$$
and
$$
	\Delta_{I^1} E_{I^2} f = \sum_{J^1 \in \ch(I^1)} E_{J^1 \times I^2} f - E_{I^1 \times I^2} f
$$
so that
$$
	\Delta_{I^1 \times I^2}f = \Delta_{I^1} \Delta_{I^2} f + E_{I^1} \Delta_{I^2} f
	+ \Delta_{I^1} E_{I^2} f.
$$
These types of connections are useful in various calculations later together with the fact that
one-parameter martingale differences in a fixed lattice have the key property $\Delta_I \Delta_J f
	= \Delta_I f$ if $I = J$ and $\Delta_I \Delta_J f = 0$ otherwise.

We now move on to some basic dyadic flag decompositions.
We build these in the tri-parameter flag setting -- of course, the simpler bi-parameter
setting is then obvious from this.

\subsubsection{Flag resolution of functions}
We now work in $\R^d := \R^{d_1+d_2+d_3}=\R^{d_1}\times \R^{d_2}\times \R^{d_3} $.
Let $\calD = \prod_{m=1}^3 \calD^m$, where
$\calD^m$ is a dyadic grid in $\R^{d_m}$, $m=1,2,3$. We define the dyadic flag rectangles
$\calD_F\subset \calD$  by setting
\[
	\calD_F = \Big\{I = \prod_{m=1}^3 I^m \in \calD \colon \ell(I^1)\le \ell(I^2) \le \ell(I^3)\Big\}.
\]
Now, given a function $f$ in $\R^d$, we may first expand it in the third parameter as
\[
	f=\sum_{I^3\in \calD^3} \Delta_{I^3}f.
\]
Next, for a fixed $I^3$, we may expand $\Delta_{I^3}f$ in the second parameter as follows:
\[
	\Delta_{I^3}f= \sum_{\substack{I^2\in \calD^2\\ \ell(I^2)\le \ell(I^3)}}\Delta_{I^2} \Delta_{I^3} f
	+ \sum_{\substack{I^2\in \calD^2\\ \ell(I^2)= \ell(I^3)}} E_{I^2} \Delta_{I^3} f.
\]
We then define $\Lambda_{I^2} f = \Delta_{I^2} f$ if $\ell(I^2) < \ell(I^3)$
and $\Lambda_{I^2} f = \Delta_{I^2} f + E_{I^2} f$ if $\ell(I^2) = \ell(I^3)$ --
so the operator does depend on $\ell(I^3)$ as well. However, we only use this in a context
where it is clear to which side length we are referring to, and so we do not emphasize this
in the notation.
Thus, we abbreviate the above to
$$
	\Delta_{I^3} f = \sum_{\substack{I^2\in \calD^2\\ \ell(I^2) \le \ell(I^3)}} \Lambda_{I^2}\Delta_{I^3} f.
$$
Likewise, we can do a similar decomposition in $\R^{d_1}$ (with a fixed $I^2$). This gives us
\begin{equation*}
	f=\sum_{I\in \calD_F} \Lambda_{I^1}\Lambda_{I^2}\Delta_{I^3}f=:\sum_{I\in \calD_F}\Delta_{I, F}f,
\end{equation*}
where $\Lambda_{I^1}$ is defined similarly as above
so that it implicitly depends on $\ell(I^2)$:
$\Lambda_{I^1} f = \Delta_{I^1} f$ if $\ell(I^1) < \ell(I^2)$
and $\Lambda_{I^1} f = \Delta_{I^1} f + E_{I^1} f$ if $\ell(I^1) = \ell(I^2)$.
Consequently, $\Delta_{I, F} f$ has worse cancellation properties than usual tri-parameter
martingale differences: while we always have
\[
	\int_{\R^{d_3}}\Delta_{I, F}f=0,
\]
the property
\[
	\int_{\R^{d_2}}\Delta_{I, F}f=0
\]
requires that $\ell(I^2)<\ell(I^3)$. A similar phenomenon happens in $\R^{d_1}$.

The flag martingale differences inherit the following orthogonality relationship.
\begin{lem}\label{lem:orthogonal}
	For $I, J \in \calD_F$ we have
	$$
		\Delta_{I,F} \Delta_{J,F} f = \left\{ \begin{array}{ll}
			\Delta_{I,F} f & \textup{if } I = J,   \\
			0              & \textup{if } I \ne J.
		\end{array} \right.
	$$
\end{lem}

\begin{proof}
	Suppose that $\Delta_{I,F} \Delta_{J,F} f\neq 0$. Then automatically $I^3=J^3$
	by the properties of the one-parameter martingale differences.
	Now, if $I^2 \ne J^2$, then we can by symmetry assume that $I^2 \subsetneq J^2$.
	Then it follows that $\ell(I^2) < \ell(J^2) \le \ell(J^3) = \ell(I^3)$,
	so that $\Lambda_{I^2} \Lambda_{J^2}$
	equals either $\Delta_{I^2} \Delta_{J^2} = 0$ or $\Delta_{I^2} (\Delta_{J^2} + E_{J^2}) = 0$.
	So we must have $I^2 = J^2$. Similarly, we see that $I^1 = J^1$ after this. It is also
	clear that $\Delta_{I, F} \Delta_{I, F} = \Delta_{I, F}$, since
	e.g. $\Delta_{I^2} \Delta_{I^2} = \Delta_{I^2}$, $E_{I^2} E_{I^2} = E_{I^2}$ and
	$\Delta_{I^2} E_{I^2} = E_{I^2} \Delta_{I^2} = 0$.
\end{proof}

\subsubsection{Flag square function}
We are now in the position to define the dyadic flag square function
\[
	S_F f = S_{\calD_F} f := \Big(\sum_{I\in \calD_F}|\Delta_{I, F} f|^2\Big)^{1/2}
\]
and study its weighted boundedness. First, we will find it convenient to establish the boundedness of certain variants of one-parameter square functions.
\begin{lem}\label{lem:variantbounds}
	Define the square function
	\[
		\tilde{S}(f)=\left(\sum_{I\in\mathcal{D}_=}|\Upsilon_{I^1}\Upsilon_{I^2}\Upsilon_{I^3} f|^2\right)^{1/2},
	\]
	where each $\Upsilon_{I^j}\in\{\Delta_{I^j},E_{I^j}\}$ and $\Upsilon_{I^j}=\Delta_{I^j}$ for at least one $j$. Then for any $w\in A_p(\R^d)$, we have that $\tilde{S}$ maps $L^p(w)$ to itself continuously.
\end{lem}
\begin{proof}
	First, observe that
	\[
		\Upsilon_{I^1}\Upsilon_{I^2}\Upsilon_{I^3} f=\Upsilon_{I^1}\Upsilon_{I^2}\Upsilon_{I^3} (\Delta_If).
	\]
	Call $g=\Delta_I f$. We have for any $x\in I$ that
	\[
		|\Upsilon_{I^1}\Upsilon_{I^2}\Upsilon_{I^3} g(x)|\lesssim \langle |g|\rangle_{I}
	\]
	Hence, we deduce that
	\[
		\Vert \tilde{S}(f)\Vert_{L^p(w)}\lesssim \left\Vert \left(\sum_{I\in\mathcal{D}_=} \langle |\Delta_If|\rangle_{I}^2 1_I\right)^{1/2}\right\Vert_{L^p(w)}.
	\]
	The claim then follows by using the boundedness of $M_{\calD_=}$ and the square function bound.
\end{proof}
We are now ready to deduce the boundedness of the flag square function.
\begin{thm}\label{thm:square}
	Suppose that $1<p<\infty$ and $w\in A_{p, F}$. Then we have
	\begin{equation}\label{Eq:upper}
		\| S_F f\|_{L^p(w)}\lesssim \|f\|_{L^p(w)}.
	\end{equation}
	We also have the following lower bound for all $0<q<\infty$ and $v\in A_{\infty, F}$:
	\begin{equation}\label{Eq:lower}
		\|f\|_{L^q(v)}\lesssim \| S_F f\|_{L^q(v)}.
	\end{equation}
\end{thm}
\begin{proof}
	We first consider the upper bound. In what follows always $I = I^1 \times I^2 \times I^3 \in \calD_F$.
	Define the subcollections $\calD_{<} = \calD_{<, <}$, $\calD_{<,=}$, $\calD_{=, <}$ and
	$\calD_{=} = \calD_{=, =}$ of $\calD_F$, where the first subscript describes the relationship between
	$\ell(I^1)$ and $\ell(I^2)$ and the second the relationship between $\ell(I^2)$ and $\ell(I^3)$.
	For instance, $\calD_{<, =} := \{I \in \calD \colon \ell(I^1) < \ell(I^2) = \ell(I^3)\}$.
	Next, define $S_{F}^{=} f$ to be like $S_F f$ but where we sum only over $I \in \calD_{=}$,
	and similarly for the other subcollections.
	Then we have
	$$
		(S_F f)^2 = (S_F^{<} f)^2 + (S_F^{<, =} f)^2 + (S_F^{=, <} f)^2 + (S_F^{=} f)^2.
	$$
	Clearly, $S_F^= f$ is a variant of the one-parameter dyadic square function, because
	$\calD_{=}$ is a dyadic lattice and
	$$
		\Delta_{I, F} = (\Delta_{I^1} + E_{I^1})(\Delta_{I^2} + E_{I^2})\Delta_{I^3}, \qquad I \in \calD_{=},
	$$
	while $\Delta_{I}$ is the sum of $A_1 A_2 A_3$, $A_i \in \{\Delta_{I^i}, E_{I^i}\}$, $A_1A_2A_3 \ne E_{I^1} E_{I^2} E_{I^3}$.
	Using that $w \in A_p(\R^d)$ we then have by Lemma \ref{lem:variantbounds} that
	\[
		\| S_F^{=} f\|_{L^p(w)}\lesssim \|f\|_{L^p(w)}.
	\]

	Next, we consider $S_F^{=, <}f$. The key observation is that
	\begin{align*}
		(S_F^{=, <} f)^2 & = \sum_{I \in \calD_{=, <}}\Big|(\Delta_{I^1}+E_{I^1})\Delta_{I^2}
		\sum_{\substack{J^3\in \calD^3
				\\ \ell(J^3)=\ell(I^2)}} E_{J^3}\Delta_{I^3}f\Big|^2                                  \\
		                 & = \sum_{I^3} \sum_{\substack{ I^1 \times I^2 \times J^3 \in \calD_{=}
				                                \\ \ell(I^2) < \ell(I^3)}}
		|(\Delta_{I^1}+E_{I^1})\Delta_{I^2}E_{J^3} \Delta_{I^3} f|^2                            \\
		                 & \le \sum_{J^3} \sum_{I \in \calD_{=}}
		|(\Delta_{I^1}+E_{I^1})\Delta_{I^2}   E_{I^3}(\Delta_{J^3}f)|^2                         \\
		                 & =: \sum_{J^3}  (\wt S (\Delta_{J^3}f))^2,
	\end{align*}
	where $\wt S$ is again a variant of the one-parameter dyadic square function, allowing us to use Lemma \ref{lem:variantbounds}.
	Since $w\in A_p(\R^d)$, we deduce that
	\begin{align*}
		\| S_F^{=, <} f\|_{L^p(w)} & \le \Big\|\Big(\sum_{J^3}  (\wt S (\Delta_{J^3}f))^2 \Big)^{1/2}\Big\|_{L^p(w)}                       \\
		                           & \lesssim  \Big\|\Big(\sum_{J^3}  |\Delta_{J^3}f|^2 \Big)^{1/2}\Big\|_{L^p(w)}\lesssim \|f\|_{L^p(w)},
	\end{align*}
	where the last inequality holds by the one-parameter square function estimate in $\R^{d_3}$ using also that $w(x_1, x_2, \cdot )\in A_{p}(\R^{d_3})$ uniformly.

	Next, we estimate $S_F^{<, =}f$. This is quite similar to the above case. Indeed, we have
	\begin{align*}
		(S_F^{<, =} f)^2 & = \sum_{I \in \calD_{<, =}}\Big|\Delta_{I^1} \sum_{\substack{J^2\times J^3\in \calD^{23}
				                                                                \\ \ell(J^2)=\ell(J^3)=\ell(I^1)}}E_{J^2}E_{J^3}(\Delta_{I^2}+E_{I^2})\Delta_{I^3} f\Big|^2 \\
		                 & \le \sum_{J^{23} \in \calD^{23}_{=}} \sum_{I \in \calD_{=}}
		| \Delta_{I^1}  E_{I^2}E_{I^3}[ (\Delta_{J^2}+E_{J^2})\Delta_{J^3} f] |^2                                                                                     \\
		                 & =:  \sum_{J^{23} \in \calD^{23}_{=}} (\wt S ((\Delta_{J^2}+E_{J^2})\Delta_{J^3}f))^2,
	\end{align*}
	where $\wt S$ is again a variant of the one-parameter dyadic square function to which Lemma \ref{lem:variantbounds} applies (not exactly the same as above
	despite the same notation).
	We have using the vector-valued estimate of $\wt S$ and the fact that
	$w \in A_p(\R^d)$ that
	\begin{align*}
		\| S_F^{<, =} f\|_{L^p(w)} & \le \Big\|\Big(\sum_{J^{23} \in \calD^{23}_{=}}
		(\wt S ((\Delta_{J^2}+E_{J^2})\Delta_{J^3}f))^2\Big)^{1/2} \Big\|_{L^p(w)}        \\
		                           & \lesssim \Big\|\Big(\sum_{J^{23} \in \calD^{23}_{=}}
		|(\Delta_{J^2}+E_{J^2})\Delta_{J^3}f|^2\Big)^{1/2} \Big\|_{L^p(w)}\lesssim \|f\|_{L^p(w)},
	\end{align*}
	where the last step used the one-parameter dyadic square function estimate in $\R^{d_{23}}$
	and the fact that $w(x_1, \cdot)\in A_{p}(\R^{d_{23}})$ uniformly.

	Finally, the estimate
	of $S_F^{<, <} f$  is also similar in spirit but requires two applications of the above ideas.
	Indeed, similar calculus first gives us
	\begin{align*}
		\| S_F^{<, <} f\|_{L^p(w)} & \le \Big\|\Big(\sum_{I^{23} \in \calD_{<}^{23}}
		(\wt S (\Delta_{I^2}\Delta_{I^3}f))^2\Big)^{1/2} \Big\|_{L^p(w)}                                      \\
		                           & \lesssim \Big\|\Big(\sum_{I^{23} \in \calD_{<}^{23}}
		                                            |\Delta_{I^2}\Delta_{I^3}f|^2\Big)^{1/2} \Big\|_{L^p(w)}.
	\end{align*}
	Repeating the process, we get
	\begin{align*}
		\Big\|\Big(\sum_{I^{23} \in \calD_{<}^{23}} & |\Delta_{I^2}\Delta_{I^3}f|^2\Big)^{1/2} \Big\|_{L^p(w)} \\
		 & \le
		\Big\|\Big(\sum_{I^3} (\wt S_{23}(\Delta_{I^3}f))^2\Big)^{1/2} \Big\|_{L^p(w)}                         \\
		 & \lesssim
		\Big\|\Big(\sum_{I^3} |\Delta_{I^3}f|^2\Big)^{1/2} \Big\|_{L^p(w)}
		\lesssim \|f\|_{L^p(w)}.
	\end{align*}
	Here $\wt S_{23}$ is some one-parameter square function in $\R^{d_{23}}$ (to which an easier version of Lemma \ref{lem:variantbounds} applies)
	and the estimates used $w(x_1, \cdot)\in A_{p}(\R^{d_{23}})$
	and later $w(x_1, x_2, \cdot) \in A_p(\R^{d_3})$.

	As we have now proved the upper bound, it remains to prove the lower bound.
	By Remark \ref{rem:weight} and the one-parameter lower bound of the
	dyadic square function on $\R^{d_3}$, we have
	\begin{equation*}
		\|f\|_{L^q(v)}\lesssim \Big\|\Big(\sum_{  I^3 } |\Delta_{I^3} f|^2\Big)^{1/2}
		\Big\|_{L^q(v)}.
	\end{equation*}
	Similarly, we may use the vector-valued lower bound for the dyadic square function on $\R^{d_{23}}$ to get
	\begin{align*}
		\|f\|_{L^q(v)}\lesssim \Big\|\Big(\sum_{  I^3 }\sum_{J^{23}\in \calD^{23}_{=}} |\Delta_{J^{23}}\Delta_{I^3} f|^2\Big)^{1/2}
		\Big\|_{L^q(v)}.
	\end{align*}
	Here, recall that the one-parameter martingale difference can be written as
	\begin{equation*}
		\Delta_{J^{23}}=(\Delta_{J^2}+E_{J^2})\Delta_{J^3}+ \Delta_{J^2}E_{J^3}.
	\end{equation*}
	This gives that
	$$
		\Delta_{J^{23}} \Delta_{I^3} f =
		\begin{cases} (\Delta_{J^2} + E_{J^2})\Delta_{I^3} f & \textup{if } I^3 = J^3,          \\
		              \Delta_{J^2} E_{J^3} \Delta_{I^3} f    & \textup{if } J^3 \subsetneq I^3, \\
		              0                                      & \textup{otherwise}.
		\end{cases}
	$$
	Thus, we have
	\begin{align*}
		 & \sum_{  I^3  }\sum_{J^{23} \in \calD^{23}_{=}} |\Delta_{J^{23}}\Delta_{I^3} f|^2                                                                \\
		 & = \sum_{J^2\times I^3\in \calD^{23}_{=}}|(\Delta_{J^2}+E_{J^2})\Delta_{I^3} f|^2+
		\sum_{  I^3 }\sum_{\substack{J^{23}\in \calD^{23}_{=}                                \\ J^3\subsetneq I^3}} | \Delta_{J^2}E_{J^3}\Delta_{I^3} f|^2 \\
		 & =  \sum_{I^{23} \in \calD^{23}_{=}}|(\Delta_{I^2}+E_{I^2})\Delta_{I^3} f|^2+
		\sum_{I^{23} \in \calD^{23}_{<}} | \Delta_{I^2} \Delta_{I^3} f|^2.
	\end{align*}
	Since $v\in A_{\infty}(\R^{d})$, we can use the vector-valued lower bound of the dyadic square function on $\R^{d}$
	to get that
	\begin{equation}\label{eq:L1}
		\begin{split}
			\|f\|_{L^q(v)} & \lesssim \Big\| \Big( \sum_{L \in \calD_{=}}
			\sum_{I^{23}\in \calD^{23}_{=}}|\Delta_L(\Delta_{I^2}+E_{I^2})\Delta_{I^3} f|^2 \\
			               & \hspace{3cm}+\sum_{L \in \calD_{=}}
			\sum_{I^{23} \in \calD^{23}_{<}} | \Delta_L\Delta_{I^2} \Delta_{I^3} f|^2\Big)^{1/2}\Big\|_{L^q(v)}.
		\end{split}
	\end{equation}
	We recall that for $L = L^1 \times L^2 \times L^3 = L^1 \times L^{23}$ we have
	\[
		\Delta_L = (\Delta_{L^1}+E_{L^1})\Delta_{L^{23}}+ \Delta_{L^1}E_{L^{23}}.
	\]
	For the moment, let $\wt \Delta_{I^{23}} = (\Delta_{I^2} + E_{I^2})\Delta_{I^3}$.
	Notice that for $L \in \calD_{=}$ and $I^{23} \in \calD^{23}_=$ we have
	$$
		\Delta_{L^{23}} \wt \Delta_{I^{23}} f =
		\begin{cases}
			\wt \Delta_{I^{23}} f & \textup{if } L^{23} = I^{23}, \\
			0                     & \textup{otherwise},
		\end{cases}
	$$
	and
	$$
		E_{L^{23}} \wt \Delta_{I^{23}} f =
		\begin{cases}
			E_{L^{23}} (\Delta_{I^2} + E_{I^2})\Delta_{I^3} & \textup{if } L^3 \subsetneq I^3, \\
			0                                               & \textup{otherwise}.              \\
		\end{cases}
	$$
	Moreover, notice that in the penultimate case $\ell(L^2) = \ell(L^3) < \ell(I^3) = \ell(I^2)$
	so we must also have $L^2 \subsetneq I^2$.
	Using these observations we see that the first term in the RHS of \eqref{eq:L1} equals
	\begin{align*}
		\sum_{L \in \calD_{=}} \sum_{I^{23}\in \calD^{23}_{=}}|\Delta_L(\Delta_{I^2}+E_{I^2})\Delta_{I^3} f|^2
		 & =
		\sum_{I \in \calD_{=}}|(\Delta_{I^1}+E_{I^1})(\Delta_{I^2}+E_{I^2})\Delta_{I^3} f|^2                                                                                                   \\
		 & + \sum_{I \in \calD_{<, =}} \sum_{ \substack{ L^{23}: I^1\times L^{23} \in \calD_{=} \\ L^{23} \subsetneq I^{23} }} |\Delta_{I^1} E_{L^{23}}(\Delta_{I^2}+E_{I^2})\Delta_{I^3} f|^2 \\
		 & = (S_F^{=}f)^2+ (S_F^{<, =}f)^2.
	\end{align*}

	Finally, we analyze the second term in the RHS of \eqref{eq:L1}.
	Notice that for $L \in \calD_{=}$ and $I^{23} \in \calD^{23}_{<}$ we have
	that $\Delta_{L^{23}} \Delta_{I^2} \Delta_{I^3} f\neq 0$ right away
	implies that $L^2 \subset I^2$ and $L^3 \subset I^3$. But if
	$L^2 \subsetneq I^2$, we first of all have
	$$
		\Delta_{L^{23}} \Delta_{I^2} \Delta_{I^3} f = E_{L^2} \Delta_{L^3} \Delta_{I^2} \Delta_{I^3} f.
	$$
	Moreover, we have that here $\ell(L^3) = \ell(L^2) < \ell(I^2) < \ell(I^3)$ so this actually vanishes.
	So we must have $L^2 = I^2$ and then $\Delta_{L^{23}} \Delta_{I^2} \Delta_{I^3} f =
		\Delta_{I^2} E_{L^3} \Delta_{I^3} f$, since $\ell(L^3) < \ell(I^3)$ implies that terms
	with $\Delta_{L^3}$ do not survive. Notice also that
	$E_{L^{23}} \Delta_{I^2} \Delta_{I^3} f \ne 0$ implies $L^2 \subsetneq I^2$
	and $L^3 \subsetneq I^3$.
	Thus, we get (summing the averages as above)
	\begin{align*}
		 & \sum_{L \in \calD_{=}}
		\sum_{I^{23} \in \calD^{23}_{<}} | \Delta_L\Delta_{I^2} \Delta_{I^3} f|^2                                    \\
		 & = \sum_{I^1 \times I^2 \times L^3 \in \calD_{=}} \sum_{\substack{I^3 \colon I^3 \supsetneq L^3
				                                                    \\ \ell(I^3) > \ell(I^2) } }
		|(\Delta_{I^1} + E_{I^1}) \Delta_{I^2} E_{L^3} \Delta_{I^3} f|^2                                             \\
		 & \hspace{3cm} + \sum_{I^1 \times L^2 \times L^3 \in \calD_=} \sum_{\substack{ I^2 \colon I^2 \supsetneq L^2
				                                                               \\ \ell(I^2) > \ell(I^1) }}
		\sum_{\substack{ I^3 \colon I^3 \supsetneq L^3
				\\ \ell(I^3) > \ell(I^2) }} |\Delta_{I^1} E_{L^2} E_{L^3} \Delta_{I^2} \Delta_{I^3} f|^2                   \\
		 & = \sum_{I \in \calD_{=, <}} | (\Delta_{I^1}+E_{I^1})\Delta_{I^2} \Delta_{I^3} f|^2
		+ \sum_{I \in \calD_{<, <}} | \Delta_{I^1}\Delta_{I^2} \Delta_{I^3} f|^2                                     \\
		 & = (S_F^{=, <}f)^2+ (S_F^{<}f)^2.
	\end{align*}
	Hence, we have proved
	\begin{align*}
		\|f\|_{L^q(v)}\lesssim \big\| \big((S_F^{<} f)^2 + (S_F^{<, =} f)^2 + (S_F^{=, <} f)^2 + (S_F^{=} f)^2\big)^{1/2}\big\|_{L^q(v)}= \|S_Ff\|_{L^q(v)}.
	\end{align*}
	We are done.
\end{proof}

\subsubsection{Haar functions}
For an interval $I \subset \R$ we denote by $I_{l}$ and $I_{r}$ the left and right
halves of $I$, respectively. We define $h_{I}^0 := |I|^{-1/2}1_{I}$ and $h_{I}^1 := |I|^{-1/2}(1_{I_{l}} - 1_{I_{r}})$
-- these are the non-cancellative and cancellative, respectively, $L^2$ normalized Haar functions related to the interval $I$.
In higher dimensions there are more than one cancellative Haars in a given cube.
Indeed, let now $I = I_1 \times \cdots \times I_d$ be a cube in $\R^d$. We
define the Haar function $h_I^{\eta}$, $\eta = (\eta_1, \ldots, \eta_d) \in \{0,1\}^d$, by setting
\begin{displaymath}
	h_I^{\eta} = h_{I_1}^{\eta_1} \otimes \cdots \otimes h_{I_d}^{\eta_d}.
\end{displaymath}
If $\eta \ne 0$ the Haar function is cancellative: $\int h_I^{\eta} = 0$.
It is important that $\langle h_{I}^{\eta}, h_J^{\kappa} \rangle = \delta_{I, J} \delta_{\eta, \kappa}$
for $I, J$ in some fixed dyadic grid. We also note the fundamental decomposition
of martingale differences
$$
	\Delta_I f
	= \sum_{\eta \ne 0} \langle f, h_I^{\eta}\rangle h_I^{\eta}
$$
and the useful fact that $\langle f, h_I^{\eta}\rangle = \langle \Delta_I f, h_I^{\eta} \rangle$,
$\eta \ne 0$.

We sometimes find it useful to have the following type of Haar-like but more general cancellative functions.
Let $I, J$ be two dyadic cubes of the same side length $\ell(I) = \ell(J)$. Then $H_{I, J}$ generally
denotes any function in $\R^{d}$ that is supported on $I \cup J$, is constant on the children
of $I$ and $J$, satisfies $|H_{I, J}| \le |I|^{-1/2}$ and has vanishing mean, i.e.,
$\int H_{I, J} = 0$. In practice, for us $H_{I, J} = h_I^0 - h_J^0$, $H_{I, J} = h_J^0 - h_I^0$ or,
simply, $H_{I, J} = h_I^{\eta}$ or $H_{I,J} = h_J^{\eta}$ for some $\eta \ne 0$.

We often exploit notation by suppressing the presence of $\eta$,
and write $h_I$ for some $h_I^{\eta}$, $\eta \ne 0$.
We also use the notation $u_I$ to stand for either $h_I$ or $h_I^0$.
Then, for instance, we have
$$
	(\Delta_I + E_I) f
	= \sum_{\eta \ne 0} \langle f, h_I^{\eta}\rangle h_I^{\eta} + \langle f, h_I^0\rangle h_I^0
$$
With the signature sum suppressed, we write this as $\langle f, u_I\rangle u_I$.
Similarly, we also note that rigorously, for instance,
$$
	\Delta_I b \Delta_I f = \sum_{\eta, \kappa \ne 0} \langle b, h_{I}^{\eta}\rangle \langle f, h_I^{\kappa}\rangle
	h_I^{\eta} h_I^{\kappa}
$$
but we will just write the RHS as
$$
	\langle b, h_I \rangle \langle f, h_I \rangle h_I h_I.
$$
It requires a bit of care to do all of this -- for instance, we do not write $h_I h_I = h_I^2 = 1_I / |I|$, since, as we saw,
these can be two different cancellative Haar functions. But we always treat
$h_I h_I$ as non-cancellative and just utilize the size $|h_I h_I| = 1_I / |I|$.

When it comes to square functions the exploitation of notation works also fine, but
this is somewhat non-trivial. Let us make this point carefully. In what follows
let $I$ run over cubes in some fixed dyadic lattice of cubes.
Then one can consider the usual one-parameter square function
$$
	Sf = \Big( \sum_I |\Delta_I f|^2 \Big)^{1/2}
$$
and a Haar square function, say,
$$
	S_H f = \Big(\sum_{\eta \ne 0} \sum_I |\langle f, h_I^{\eta}\rangle|^2 \frac{1_I}{|I|}\Big)^{1/2}.
$$
While $Sf \lesssim S_H f$ even pointwise, there is also an $L^p$ comparison to the other direction in
a very strong sense:
$$
	\|S_H f\|_{L^p(w)} \lesssim \|Sf\|_{L^p(w)}
$$
for all $0 < p < \infty$ and $w \in A_{\infty}$. We omit these details now but our upcoming preprint \cite{IBLM-SF}
will, in particular, detail this briefly.
So, for all intents and purposes,
some Haar square function
$$
	\Big(\sum_I |\langle f, h_I\rangle|^2 \frac{1_I}{|I|}\Big)^{1/2}
$$
can be treated like the original martingale difference square function and one
does not mostly need to worry about where, if anywhere, the implicit $\eta$ summation is
(of course, going to the direction $Sf \lesssim S_H f$ it is necessary that all of the
cancellative Haars appear on the RHS).

In the tri-parameter flag setting we can define
that a cancellative flag Haar function $h_{I, F}$ in a (dyadic) flag rectangle
$I = I^1 \times I^2 \times I^3 \in \calD_F$ takes the form $h_{I^1} \otimes h_{I^2} \otimes h_{I^3}$
if $I \in \calD_<$, the form $u_{I^1} \otimes h_{I^2} \otimes h_{I^3}$ if
$I \in \calD_{=, <}$, the form $h_{I^1} \otimes u_{I^2} \otimes h_{I^3}$
if $I \in \calD_{<, =}$ and the form $u_{I^1} \otimes u_{I^2} \otimes h_{I^3}$
if $I \in \calD_=$. Then, with the obvious rigorous intepretation, we write
$$
	\Delta_{I, F}f = \langle f, h_{I, F}\rangle h_{I, F}.
$$
The above commentary about the equivalence of martingale difference
and suitable Haar square functions is a general fact and work also in the flag setting.

\subsection{Flag resolutions of bilinear forms}
Suppose $B(f, g)$ is a bilinear form. Of course, we could formally decompose
$$
	B(f, g) = \sum_{I, J \in \calD_F} B(\Delta_{I, F} f, \Delta_{J, F} g)
$$
simply by decomposing $f$ and $g$ in the flag sense like above. However, we
want to, in addition, have $\ell(I^i) = \ell(J^i)$ for all $i=1,2,3$ --
we denote this requirement by $\ell(I) = \ell(J)$. We can achieve this with the price of
losing some cancellation.

We make a preliminary decomposition that we apply a few times in what follows.
Suppose $s = 2^{-k}$ is some fixed dyadic scale. In what follows we
can take $V_1, V_2 \in \calD^i$ for some fixed $i = 1,2,3$.
Temporarily understand that $\Lambda_V$ means $\Delta_V + E_V$
when $\ell(V) = s$ and otherwise equals $\Delta_V$. With this understanding we can decompose
$$
	B(f, g) =  \sum_{\substack{V_1, V_2 \\ \ell(V_1),\, \ell(V_2) \le s }} B(\Lambda_{V_1} f, \Lambda_{V_2} g).
$$
Now, divide this into three sums over $V_1, V_2$, where $\ell(V_2) < \ell(V_1) \le s$,
$\ell(V_1) < \ell(V_2) \le s$ or $\ell(V_1) = \ell(V_2) \le s$. Using that
for a dyadic scale $m < s$ we have
$\sum_{m < \ell(V) \le s} \Lambda_V h = \sum_{\ell(V) = m} E_V h$, we get
\begin{align*}
	B(f, g) & = \sum_{\substack{ V_1, V_2 \\ \ell(V_1) = \ell(V_2) < s}} B(E_{V_1}f, \Delta_{V_2} g) \\
	        & + \sum_{\substack{ V_1, V_2 \\ \ell(V_1) = \ell(V_2) < s}} B(\Delta_{V_1}f, E_{V_2} g)
	+ \sum_{\substack{ V_1, V_2           \\ \ell(V_1) = \ell(V_2) \le s}} B(\Lambda_{V_1}f, \Lambda_{V_2} g).
\end{align*}

We apply this first in $\calD^3$ with $s = \infty$ -- this gives
$$
	B(f, g) = \sum_{\substack{ I^3, J^3 \\ \ell(I^3) = \ell(J^3)}}
	\big[ B(E_{I^3}f, \Delta_{J^3} g) + B(\Delta_{I^3}f, E_{J^3} g) + B(\Delta_{I^3}f, \Delta_{J^3} g)\big].
$$
Then, we fix $I^3, J^3$ with $\ell(I^3) = \ell(J^3)$ and one of the appearing three terms -- say, we look at $B(E_{I^3} f, \Delta_{J^3} g)$.
We apply the decomposition from above in $\calD^2$ with $s = \ell(I^3) = \ell(J^3)$ -- this gives
\begin{align*}
	B(E_{I^3} f, \Delta_{J^3} g) & = \sum_{\substack{ I^2, J^2 \\ \ell(I^2) = \ell(J^2) < \ell(I^3)}} B(E_{I^2} E_{I^3} f, \Delta_{J^2} \Delta_{J^3} g) \\
	                             & + \sum_{\substack{ I^2, J^2 \\ \ell(I^2) = \ell(J^2) < \ell(I^3)}} B(\Delta_{I^2} E_{I^3} f, E_{J^2} \Delta_{J^3} g) \\
	                             & + \sum_{\substack{ I^2, J^2 \\ \ell(I^2) = \ell(J^2) \le \ell(I^3)}}
	B(\Lambda_{I^2} E_{I^3} f, \Lambda_{J^2} \Delta_{J^3} g),
\end{align*}
where we recall that $\Lambda_{I^2}, \Lambda_{J^2}$ implicitly depend on $\ell(I^3) = \ell(J^3)$.

Finally, we now also fix $I^2, J^2$ with $\ell(I^2) = \ell(J^2)$ and look at one of the appearing three terms
-- say, we look at $B(E_{I^2} E_{I^3} f, \Delta_{J^2} \Delta_{J^3} g)$.
We apply the decomposition from above in $\calD^1$ with $s = \ell(I^2) = \ell(J^2)$ -- this gives
\begin{align*}
	B(E_{I^2} E_{I^3} f, \Delta_{J^2} \Delta_{J^3} g) & = \sum_{\substack{ I^1, J^1 \\ \ell(I^1) = \ell(J^1) < \ell(I^2)}} B(E_{I^1} E_{I^2} E_{I^3} f, \Delta_{J^1} \Delta_{J^2} \Delta_{J^3} g)           \\
	                                                  & + \sum_{\substack{ I^1, J^1 \\ \ell(I^1) = \ell(J^1) < \ell(I^2)}} B(\Delta_{I^1} E_{I^2} E_{I^3} f, E_{J^1}\Delta_{J^2} \Delta_{J^3} g)            \\
	                                                  & + \sum_{\substack{ I^1, J^1 \\ \ell(I^1) = \ell(J^1) \le \ell(I^2)}} B(\Lambda_{I^1} E_{I^2} E_{I^3} f, \Lambda_{J^1} \Delta_{J^2} \Delta_{J^3} g),
\end{align*}
where we again recall our usual flag style convention: $\Lambda_{I^1}, \Lambda_{J^1}$
implicitly depend on $\ell(I^2) = \ell(J^2)$.

In total, we get $3^3 = 27$ different sums -- one of them e.g. being
\begin{equation}\label{eq:ex1}
	\sum_{\substack{ I, J \in \calD_{<} \\ \ell(I) = \ell(J) }} B(\Delta_{I^1} E_{I^2} E_{I^3} f, E_{J^1}\Delta_{J^2} \Delta_{J^3} g).
\end{equation}
Recall that $I \in \calD_{<} = \calD_{<, <}$ means $\ell(I^1) < \ell(I^2) < \ell(I^3)$. All of the $27$ terms, like the one right above,
form our tri-parameter flag type multiresolution decomposition of $B(f, g)$. The similar decomposion
in the bi-parameter case has $9$ terms.

To be able to write all of the 27 terms, we introduce some notation. We will, for instance, denote
the sum in \eqref{eq:ex1} by $\Sigma_{(\Delta, E, E)}$. We explain the logic now.
A symbol $\Delta$ implies that there is a $\Delta$ hitting $f$ and an $E$ hitting $g$ in that parameter,
a symbol $D$ (for double $\Lambda$) means that there is a $\Lambda$ hitting both, and $E$ means that there is an $E$ hitting $f$ and a $\Delta$ hitting $g$.
So the symbol denotes the symbol in $f$ and the corresponding symbol in $g$ can be inferred --
the symbol $D$ is needed to denote the case where there is a $\Lambda$ in both.
Notice also that the sum can run over $\calD_F = \calD_{\le, \le}$, $\calD_< = \calD_{<, <}$,
$\calD_{<, \le}$ or $\calD_{\le, <}$. This
can also be inferred from the above notation -- if there is a $\Delta$ or $E$ in the slot $i \in \{1,2\}$, this
means we will have $<$ in the corresponding slot, and if there is a $D$, this mean we will have $\le$.

Let $\calS = \{(s_1, s_2, s_3) \colon s_i \in \{\Delta, E, D\}\}$ so that our
fundamental flag multiresolution can be written as
\begin{equation*}
	B(f, g) = \sum_{s \in \calS} \Sigma_s.
\end{equation*}
This can be applied with, e.g., $B(f, g) = \langle Tf, g\rangle$, where $T$ is a
flag singular integral, or for $B(f, g) = fg$ to obtain a flag style paraproduct decomposition of
products relevant in commutator theory. We will use both later.

\section{Flag kernels}
\subsection{Tri-parameter flag Fourier multipliers}
We consider tri-parameter flag type Fourier multipliers. We will use them as a model
for formulating our general -- even non-convolution form -- framework for flag SIOs.
Let $\wh{T_m f} = m\wh f$ be a Fourier multiplier whose symbol $m$ satisfies
\begin{equation}\label{eq:def}
	|\partial^\alpha m(\xi)| \lesssim
	(|\xi_1|+|\xi_2|+|\xi_3|)^{-|\alpha_1|}  (|\xi_2|+|\xi_3|)^{-|\alpha_2|}|\xi_3|^{-|\alpha_3|}
\end{equation}
for $\xi = (\xi_1, \xi_2, \xi_3) \in \R^d = \R^{d_1} \times \R^{d_2} \times \R^{d_3}$
and multi-indices $\alpha = (\alpha_1, \alpha_2, \alpha_3) \in \N^{d_1} \times \N^{d_2} \times \N^{d_3}$.

The condition \eqref{eq:def} is based on the typical example of a flag singular integral in the literature:
$R_{123} R_{23} R_3$, where $R_{123}, R_{23}, R_3$ are some Riesz transforms
in $\R^{d}, \R^{d_{23}}$ and $\R^{d_3}$, respectively, where $d = d_{123} := d_1+d_2+d_3$, $d_{23} := d_2 + d_3$.
Indeed, notice that in this case
$$
	m(\xi) = \frac{\pr_i \xi}{|\xi|} \frac{\pr_j \xi}{|\xi_{23}|} \frac{\pr_k \xi}{|\xi_3|},
$$
where $1 \le i \le d$, $d_1+1 \le j \le d$ and $d_1 + d_2 + 1 \le k \le d$, and
$\pr_m$ stands for the projection to the $m$th coordinate.

Our first question is, what kind of kernel estimates can be derived from \eqref{eq:def}?
This is actually already, up to a point, studied in
\cite{NRS2001}. However, we need to go further due to our intention to formulate
a general non-convolution form framework -- in particular, we need some
partial kernel estimates that are not considered in \cite{NRS2001}.
We provide the full details of all our considerations
to make this self-contained and note that our arguments are also different from those in \cite{NRS2001}.

Let $\psi_{123}$, $\psi_{23}$ and $\psi_{3}$ be standard resolutions of unity in
$\R^{d}\setminus \{0\}$, $\R^{d_{23}}\setminus \{0\}$ and $\R^{d_3}\setminus \{0\}$, respectively.
That is, for instance, $\psi_3(\xi_3) = \eta_3(\xi_3) - \eta_3(2\xi_3)$, where
$\eta_3$ is smooth with $\eta_3(\xi_3) = 1$ for $|\xi_3| \le 1$ and $\eta_3(\xi_3) = 0$
for $|\xi_3| \ge 2$. Let $m$ be a bounded function that is smooth in
$\R^{d_1} \times \R^{d_2} \times (\R^{d_3} \setminus \{0\})$ and satisfies \eqref{eq:def}.
For $\xi_3 \ne 0$ we can write
\begin{align*}
	1 & = \sum_{i\in \Z} \psi_3(2^{-i}\xi_3)=\sum_{i,j\in \Z} \psi_3(2^{-i}\xi_3)\psi_{23}(2^{-j}\xi_2, 2^{-j}\xi_3)                 \\
	  & =\sum_{i,j,k\in \Z} \psi_3(2^{-i}\xi_3)\psi_{23}(2^{-j}\xi_2, 2^{-j}\xi_3)\psi_{123}(2^{-k}\xi_1, 2^{-k}\xi_2, 2^{-k}\xi_3).
\end{align*}
Thus, we may write
\begin{align*}
	m(\xi)=\sum_{i,j,k\in \Z} \psi_3(2^{-i}\xi_3)\psi_{23}(2^{-j}\xi_2, 2^{-j}\xi_3)\psi_{123}(2^{-k}\xi_1, 2^{-k}\xi_2, 2^{-k}\xi_3)m(\xi)=:\sum_{i,j,k\in \Z} m_{ijk}(\xi).
\end{align*}
Moreover, we have that for $\xi\in \supp m_{ijk}$ that there holds that
\[
	2^i \sim |\xi_3|,\quad 2^j\sim  |\xi_2|+|\xi_3|,\quad 2^k \sim  |\xi_1|+|\xi_2|+|\xi_3|.
\]
In particular, we must have that $2^i\lesssim 2^j\lesssim 2^k$, and so
$$
	m(\xi) = \sum_{\substack{i, j, k \in \Z \\ 2^i \lesssim 2^j \lesssim 2^k}} m_{ijk}(\xi), \qquad \xi_3 \ne 0.
$$
Notice that each $m_{ijk}$ satisfies \eqref{eq:def} uniformly in $i,j,k$ and that $m_{ijk}$
is smooth and compactly supported so that $K_{ijk} := \check m_{ijk}$ is Schwartz.

\begin{lem}\label{lem:mfullk}
	Let $m$ satisfy \eqref{eq:def}. Then $K=\check m$ is smooth where $x_1 \ne 0$ and satisfies
	\begin{equation*}
		|   \partial^\beta K(x)|\lesssim \frac 1{ |x_1|^{d_1+|\beta_1|}(|x_1|+|x_2|)^{d_2+|\beta_2|}(|x_1|+|x_2|+|x_3|)^{d_3+|\beta_3|}}
	\end{equation*}
	for all multi-indices $\beta = (\beta_1, \beta_2, \beta_3) \in \N^d = \N^{d_1} \times \N^{d_2} \times \N^{d_3}$.
\end{lem}
\begin{proof}
	Fix the multi-index $\beta$. Since $m_{ijk}$ satisfies \eqref{eq:def} uniformly on $i,j,k$, we
	have for all multi-indices $\alpha$ that
	\begin{align*}
		|\partial^\alpha m_{ijk}(\xi)|\lesssim 2^{-i |\alpha_3|-j |\alpha_2|-k|\alpha_1|}.
	\end{align*}
	In particular, we have
	\[
		\|\partial^\alpha m_{ijk}\|_{L^1}\lesssim 2^{i (d_3-|\alpha_3|)+j (d_2-|\alpha_2|)+k(d_1-|\alpha_1|)}.
	\]
	Then for all multi-indices $\alpha$ we have
	\begin{align*}
		\| x^\alpha \partial^\beta K_{ijk}\|_{L^\infty}
		 & \lesssim \| \partial^\alpha (\xi^\beta m_{ijk} )\|_{L^1}\le  \sum_{\gamma\leq\alpha}\binom{\alpha}{\gamma}
		\| \partial^\gamma \xi^\beta \cdot \partial^{\alpha-\gamma}m_{ijk}\|_{L^1}                                        \\
		 & \lesssim\sum_{\gamma\leq\alpha\wedge\beta}\|  \xi^{\beta-\gamma} \cdot \partial^{\alpha-\gamma}m_{ijk}\|_{L^1} \\
		 & \lesssim \sum_{\gamma\leq\alpha\wedge\beta}2^{i (|\beta_3|-|\gamma_3|)
			            +j (|\beta_2|-|\gamma_2|)+k(|\beta_1|-|\gamma_1|)}\|   \partial^{\alpha-\gamma}m_{ijk}\|_{L^1}       \\
		 & \lesssim 2^{i (d_3+|\beta_3|-|\alpha_3|)+j (d_2+|\beta_2|-|\alpha_2|)+k(d_1+|\beta_1|-|\alpha_1|)}.
	\end{align*}
	Let $\zeta= (\zeta_1, \zeta_2, \zeta_3)$ be an arbitrary multi-index with $\zeta_i\in \N^{d_i}$ and $|\zeta_i|= |\beta_i| + d_i$.
	We now get that for all multi-indices $\alpha$ there holds that
	\begin{align*}
		\bigg| \frac{x^{\zeta} \partial^\beta K_{ijk}(x)}{x^{\zeta-\alpha}} \bigg|
		\le  \| x^\alpha \partial^\beta K_{ijk}\|_{L^\infty}
		\lesssim 2^{i (d_3+|\beta_3|-|\alpha_3|)+j (d_2+|\beta_2|-|\alpha_2|)+k(d_1+|\beta_1|-|\alpha_1|)}
	\end{align*}
	and so
	\[
		| x^{\zeta} \partial^\beta K_{ijk}(x)| \lesssim
		2^{i (d_3+|\beta_3|-|\alpha_3|)+j (d_2+|\beta_2|-|\alpha_2|)+k(d_1+|\beta_1|-|\alpha_1|)} |x^{\zeta-\alpha}|.
	\]
	Here the left hand side does not depend on $\alpha$, and the estimate is valid with any $\alpha$.
	We exploit this next.
	Write
	$$
		|x^{\zeta-\alpha}| = \prod_{i=1}^3 |x_i^{\zeta_i}||x_i^{\alpha_i}|^{-1}.
	$$
	We can estimate
	$$
		|x_i^{\zeta_i}| \le |x_i|^{|\beta_i| + d_i}.
	$$
	We also use that given $N_i > 0$ we have
	$$
		|x_i|^{N_i} \lesssim \max_{|\alpha_i| = N_i} |x_i^{\alpha_i}|
	$$
	so that for some multi-index $\alpha_i \in \N^{d_i}$ with $|\alpha_i| = N_i$ (which can depend on $x_i$) we have
	$$
		|x_i^{\alpha_i}|^{-1} \lesssim |x_i|^{-N_i}.
	$$
	So using multi-indices with $|\alpha_i| = 0$ or so that the above holds, we get
	\begin{align*}
		| x^{\zeta} \partial^\beta K_{ijk}(x)|
		\lesssim \min\{ & |2^k x_1|^{|\beta_1| + d_1}, |2^k x_1|^{|\beta_1|+d_1-N_1} \}                \\
		                & \times \min\{ |2^j x_2|^{|\beta_2| + d_2}, |2^j x_2|^{|\beta_2|+d_2-N_2} \}  \\
		                & \times \min\{ |2^i x_3|^{|\beta_3| + d_3}, |2^i x_3|^{|\beta_3|+d_3-N_3} \}.
	\end{align*}
	Below we will use this with some $N_i$ big enough (depending on the fixed $\beta$).

	Now, we get that
	\begin{align*}
		|x^{\zeta}| & \sum_{\substack{i,j,k                                                               \\ 2^i\lesssim 2^j\lesssim 2^k}}
		|\partial^\beta K_{ijk}(x)|                                                                                                        \\
		            & \lesssim \sum_{\substack{j,k
				                       \\   2^j\lesssim 2^k}} \min\{|2^j x_3|^{d_3+|\beta_3|}, 1\}
		\min\{|2^j x_2|^{d_2+|\beta_2|}, |2^j x_2|^{d_2+|\beta_2|-N_2}\}                                                                   \\
		            & \hspace{2cm}\times\min\{|2^k x_1|^{d_1+|\beta_1|}, |2^k x_1|^{d_1+|\beta_1|-N_1}\},
	\end{align*}
	where we have simply noted that if we take $N_3>d_3+|\beta_3|$ then
	\[
		\sum_{i\colon 2^i\lesssim 2^j}\min\{|2^i x_3|^{d_3+|\beta_3|}, |2^i x_3|^{d_3+|\beta_3|-N_3}\}
		\le \sum_{i\in \Z}\min\{|2^i x_3|^{d_3+|\beta_3|}, |2^i x_3|^{d_3+|\beta_3|-N_3}\}\lesssim 1
	\]
	and
	\[
		\sum_{i\colon 2^i\lesssim 2^j}\min\{|2^i x_3|^{d_3+|\beta_3|}, |2^i x_3|^{d_3+|\beta_3|-N_3}\}
		\le \sum_{i \colon 2^i\lesssim 2^j}|2^i x_3|^{d_3+|\beta_3|}\lesssim |2^j x_3|^{d_3+|\beta_3|}.
	\]
	Similarly, if we take $N_2>d_2+|\beta_2|+d_3+|\beta_3|$, we have
	\begin{equation}\label{eq:q1}
		\sum_{ j\colon 2^j\lesssim 2^k} \min\{|2^j x_2|^{d_2+|\beta_2|}, |2^j x_2|^{d_2+|\beta_2|-N_2}\} \lesssim \min\{|2^k x_2|^{d_2+|\beta_2|},1 \}
	\end{equation}
	and
	\begin{equation}\label{eq:q2}
		\begin{split}
			 & \sum_{ j\colon 2^j\lesssim 2^k} |2^j x_3|^{d_3+|\beta_3|}\min\{|2^j x_2|^{d_2+|\beta_2|}, |2^j x_2|^{d_2+|\beta_2|-N_2}\} \\
			 & = |x_3|^{d_3+|\beta_3|}|x_2|^{-d_3-|\beta_3|}\sum_{ j\colon 2^j\lesssim 2^k}
			\min\{|2^j x_2|^{d_3+|\beta_3|+d_2+|\beta_2|}, |2^j x_2|^{d_3+|\beta_3|+d_2+|\beta_2|-N_2}\}                                 \\
			 & \lesssim |x_3|^{d_3+|\beta_3|}|x_2|^{-d_3-|\beta_3|} \min\{|2^k x_2|^{d_3+|\beta_3|+d_2+|\beta_2|},1 \}.
		\end{split}
	\end{equation}
	Combining \eqref{eq:q1} and \eqref{eq:q2} we get that
	\begin{align*}
		\sum_{ j\colon 2^j\lesssim 2^k} & \min\{|2^j x_3|^{d_3+|\beta_3|}, 1\}\min\{|2^j x_2|^{d_2+|\beta_2|}, |2^j x_2|^{d_2+|\beta_2|-N_2}\}   \\
		                                & \lesssim \min\Big\{ |2^k x_2|^{d_2+|\beta_2|},1, (2^k|x_3|)^{d_3+|\beta_3|}(2^k|x_2|)^{d_2+|\beta_2|},
		\frac{|x_3|^{d_3+|\beta_3|}}{|x_2|^{d_3+|\beta_3|}}\Big \}.
	\end{align*}
	Finally, we can now take care of the $k$ summation similarly:
	\begin{align*}
		|x^{\zeta}| \sum_{\substack{i,j,k                                                                                \\ 2^i\lesssim 2^j\lesssim 2^k}}
		|\partial^\beta & K_{ijk}(x)|                                                                                                                                          \\
		                & \lesssim \sum_k \min\Big\{ |2^k x_2|^{d_2+|\beta_2|},1,
		                                                      (2^k|x_3|)^{d_3+|\beta_3|}(2^k|x_2|)^{d_2+|\beta_2|}, \frac{|x_3|^{d_3+|\beta_3|}}{|x_2|^{d_3+|\beta_3|}}\Big \} \\
		                & \hspace{2cm}\times \min\{|2^k x_1|^{d_1+|\beta_1|}, |2^k x_1|^{d_1+|\beta_1|-N_1}\}                                                                  \\
		                & \lesssim \min\Big\{\frac{|x_2|^{d_2+|\beta_2|}}{|x_1|^{d_2+|\beta_2|}}, 1,
		\frac{|x_3|^{d_3+|\beta_3|}}{|x_1|^{d_3+|\beta_3|}}\frac{|x_2|^{d_2+|\beta_2|}}{|x_1|^{d_2+|\beta_2|}},
		\frac{|x_3|^{d_3+|\beta_3|}}{|x_2|^{d_3+|\beta_3|}}\Big\}.
	\end{align*}
	The above required us to choose $N_1>d_1+|\beta_1|+d_2+|\beta_2|+d_3+|\beta_3|$.

	Now, we observe that
	\[
		\min\Big\{1,  \frac{|x_3|^{d_3+|\beta_3|}}{|x_2|^{d_3+|\beta_3|}}\Big\}\sim \frac{|x_3|^{d_3+|\beta_3|}}{(|x_2|+|x_3|)^{d_3+|\beta_3|}}
	\]
	and
	\[
		\min\Big\{\frac{|x_2|^{d_2+|\beta_2|}}{|x_1|^{d_2+|\beta_2|}},   \frac{|x_3|^{d_3+|\beta_3|}}{|x_1|^{d_3+|\beta_3|}}\frac{|x_2|^{d_2+|\beta_2|}}{|x_1|^{d_2+|\beta_2|}}\Big\}\sim \frac{|x_2|^{d_2+|\beta_2|}}{|x_1|^{d_2+|\beta_2|}} \frac{|x_3|^{d_3+|\beta_3|}}{(|x_1|+|x_3|)^{d_3+|\beta_3|}},
	\]
	so that
	\begin{align*}
		 & \min\Big\{\frac{|x_2|^{d_2+|\beta_2|}}{|x_1|^{d_2+|\beta_2|}}, 1,  \frac{|x_3|^{d_3+|\beta_3|}}{|x_1|^{d_3+|\beta_3|}}\frac{|x_2|^{d_2+|\beta_2|}}{|x_1|^{d_2+|\beta_2|}}, \frac{|x_3|^{d_3+|\beta_3|}}{|x_2|^{d_3+|\beta_3|}}\Big\} \\
		 & \sim \min\Big\{\frac{|x_3|^{d_3+|\beta_3|}}{(|x_2|+|x_3|)^{d_3+|\beta_3|}},
		\frac{|x_2|^{d_2+|\beta_2|}}{|x_1|^{d_2+|\beta_2|}} \frac{|x_3|^{d_3+|\beta_3|}}{(|x_1|+|x_3|)^{d_3+|\beta_3|}}\Big\}                                                                                                                   \\
		 & \sim  |x_3|^{d_3+|\beta_3|}|x_2|^{d_2+|\beta_2|}\min\Big\{\frac{1}{|x_2|^{d_2+|\beta_2|}(|x_2|+|x_3|)^{d_3+|\beta_3|}},
		\frac{1}{|x_1|^{d_2+|\beta_2|}(|x_1|+|x_3|)^{d_3+|\beta_3|}}\Big\}                                                                                                                                                                      \\
		 & \sim  |x_3|^{d_3+|\beta_3|}|x_2|^{d_2+|\beta_2|}\frac 1{(|x_1|+|x_2|)^{d_2+|\beta_2|} (|x_1|+|x_2|+|x_3|)^{d_3+|\beta_3|}}.
	\end{align*}

	We have proved that
	$$
		|x^{\zeta}| \sum_{\substack{i,j,k \\ 2^i\lesssim 2^j\lesssim 2^k}}
		|\partial^\beta K_{ijk}(x)| \lesssim
		|x_3|^{d_3+|\beta_3|}|x_2|^{d_2+|\beta_2|}\frac 1{(|x_1|+|x_2|)^{d_2+|\beta_2|} (|x_1|+|x_2|+|x_3|)^{d_3+|\beta_3|}}
	$$
	for all multi-indices $\zeta$ with $|\zeta_i| = |\beta_i| + d_i$.
	Now, with fixed $x$ choose $\zeta_1, \zeta_2, \zeta_3$ so that
	$$
		|x^{\zeta}| = |x_1^{\zeta_1}| |x_2^{\zeta_2}| |x_3^{\zeta_3}| \sim \prod_{i=1}^3 |x_i|^{|\beta_i| + d_i}.
	$$
	We have finally proved that
	\begin{equation}\label{eq:Kbound}
		\sum_{\substack{i,j,k \\ 2^i\lesssim 2^j\lesssim 2^k}}
		|\partial^\beta K_{ijk}(x)| \lesssim
		\frac 1{ |x_1|^{d_1+|\beta_1|}(|x_1|+|x_2|)^{d_2+|\beta_2|}(|x_1|+|x_2|+|x_3|)^{d_3+|\beta_3|}}.
	\end{equation}
	Actually, we proved this with the implicit assumption $x_i \ne 0$ for all $i = 1,2,3$. If $x_2 = 0$
	or $x_3 = 0$ we would need to modify the above argument a little bit.

	We have that $\sum_{i, j, k} K_{ijk}$ converges to $K$ in the sense of distributions. By considering test
	functions that are supported away from $x_1 = 0$ and using the estimate \eqref{eq:Kbound} together
	with the dominated convergence theorem we can see that the tempered distribution $K$ agrees
	with the function
	$$
		x \mapsto \sum_{\substack{i,j,k \\ 2^i\lesssim 2^j\lesssim 2^k}} K_{ijk}(x)
	$$
	that is smooth where $x_1 \ne 0$.
\end{proof}
These derivative estimates imply various estimates for the so-called full kernel $K$.
The abstract theory will not even require the full strength of them. We talk more about these consequences later.

Next, we need to get some idea of how to eventually formulate the so-called partial kernel assumptions
in our abstract framework. This has not been considered previously as far as we know.
So we will derive such partial kernel estimates next for these multipliers.
To make this technically less demanding, we prefer not to anymore work with the full multiplier $m$ like
we did above. This is because
this requires so much additional care with convergence arguments.
Instead, we will take a completely arbitrary truncation
\begin{equation*}
	m_L := \sum_{2^{L^1} \lesssim 2^i \lesssim 2^j \lesssim 2^k \lesssim 2^{L^2}}m_{i,j,k}, \qquad L = (L^1, L^2) \in \Z^2.
\end{equation*}
Notice that as each $m_{ijk}$ satisfies \eqref{eq:def} uniformly in $i,j,k$, then so does $m_L$.
This is because in some neighborhood of an arbitrary point only boundedly many terms in the sum can be non-zero. Also,
$m_L$ is smooth and supported in $|\xi_1| + |\xi_2| + |\xi_3| \lesssim 2^{L^2}$ so that
$K_L := \check m_L$ is Schwartz. We will derive partial kernel estimates for the multipliers uniformly in the truncation parameter $L$.
But we do not want to unnecessarily drag the $L$ along on every line -- so we exploit notation and denote these arbirary truncations with just $m$ and $K$.

First, we can write
$$
	T_m f(x) = \int_{\R^d} m(\xi) \wh f(\xi) e^{2\pi ix \cdot \xi} \ud \xi.
$$
Consider now $f_1, g_1$ to be nice functions in $\R^{d_1}$ and $f_{23}, g_{23}$ be nice functions in $\R^{d_{23}}$.
Then we can write
\begin{align*}
	\langle T_m(f_1 \otimes f_{23}), g_1 \otimes g_{23} \rangle
	 & = \int_{\R^{d_{23}}} \Big( \int_{\R^{d_{23}}} m_{f_1, g_1}(\xi_{23}) \wh f_{23}(\xi_{23})
	e^{2\pi i x_{23} \cdot \xi_{23}} \ud \xi_{23} \Big) g_{23}(x_{23}) \ud x_{23}                \\
	 & = \langle T_{m_{f_1, g_1}}f_{23}, g_{23}\rangle,
\end{align*}
where
\begin{align*}
	m_{f_1, g_1}(\xi_{23}) & := \int_{\R^{d_1}} \Big( \int_{\R^{d_1}} m(\xi) \wh f_1(\xi_1)
	e^{2\pi i x_1 \cdot \xi_1} \ud \xi_1\Big) g_1(x_1) \ud x_1
	= \langle T_{m(\cdot,\, \xi_{23})} f_1, g_1 \rangle.
\end{align*}
Notice that
$$
	\partial^{\alpha_2} \partial^{\alpha_3} m_{f_1, g_1}(\xi_{23})
	= \langle T_{\partial^{\alpha_2} \partial^{\alpha_3} m(\cdot,\, \xi_{23})} f_1, g_1 \rangle.
$$
On the other hand, here the symbol $\xi_1 \mapsto \partial^{\alpha_2} \partial^{\alpha_3} m(\xi_1,\, \xi_{23})$
satisfies
$$
	|\partial^{\alpha_1}\partial^{\alpha_2} \partial^{\alpha_3} m(\xi_1, \xi_{23})|
	\lesssim (|\xi_2| + |\xi_3|)^{-|\alpha_2|} |\xi_3|^{-|\alpha_3|} \cdot |\xi_1|^{-|\alpha_1|}
$$
so that by one-parameter multiplier theory
$$
	|\partial^{\alpha_2} \partial^{\alpha_3} m_{f_1, g_1}(\xi_{23})| \lesssim
	\|f_1\|_{L^p} \|g_1\|_{L^{p'}}(|\xi_2| + |\xi_3|)^{-|\alpha_2|} |\xi_3|^{-|\alpha_3|}, \qquad 1 < p < \infty.
$$
This means that $m_{f_1, g_1}$ is a bi-parameter flag multiplier in $\R^{d_{23}}$ and from the obvious bi-parameter
analogue of Lemma \ref{lem:mfullk} it follows that
$$
	T_{m_{f_1, g_1}}f_{23} = K_{f_1, g_1} * f_{23},
$$
where the kernel $K_{f_1, g_1}$ satisfies
$$
	|\partial^{\beta_2} \partial^{\beta_3} K_{f_1, g_1}(x_{23})|
	\lesssim \frac{\|f_1\|_{L^p}\|g_1\|_{L^{p'}}}{|x_2|^{d_2+|\beta_2|}(|x_2|+|x_3|)^{d_3+|\beta_3|}}, \qquad 1 < p < \infty.
$$

We next look at some other kind of partial kernel. For instance, suppose now that $f_{12}, g_{12}$ are
two functions in $\R^{d_{12}}$ and $f_3, g_3$ are two functions in $\R^{d_3}$. Then we again
write
\begin{align*}
	\langle T_m(f_{12} \otimes f_{3}), g_{12} \otimes g_{3} \rangle
	 & = \int_{\R^{d_{12}}} \Big( \int_{\R^{d_{12}}} m_{f_3, g_3}(\xi_{12}) \wh f_{12}(\xi_{12})
	e^{2\pi i x_{12} \cdot \xi_{12}} \ud \xi_{12} \Big) g_{12}(x_{12}) \ud x_{12}                \\
	 & = \langle T_{m_{f_3, g_3}}f_{12}, g_{12}\rangle,
\end{align*}
where
\begin{align*}
	m_{f_3, g_3}(\xi_{12}) & := \int_{\R^{d_3}} \Big( \int_{\R^{d_3}} m(\xi) \wh f_3(\xi_3)
	e^{2\pi i x_3 \cdot \xi_3} \ud \xi_3\Big) g_3(x_3) \ud x_3
	= \langle T_{m(\xi_{12},\, \cdot)} f_3, g_3 \rangle.
\end{align*}
Arguing like in the $m_{f_1, g_1}$ case we see that
$m_{f_3, g_3}$ is a bi-parameter flag multiplier in $\R^{d_{12}}$, and so
$$
	T_{m_{f_3, g_3}} f_{12} = K_{f_3, g_3} * f_{12},
$$
where
$$
	|\partial^{\beta_1} \partial^{\beta_2} K_{f_3, g_3}(x_{12})
	\lesssim \frac{\|f_3\|_{L^p}\|g_3\|_{L^{p'}}}{|x_1|^{d_1+|\beta_1|}(|x_1|+|x_2|)^{d_2+|\beta_2|}}, \qquad 1 < p < \infty.
$$
The key difference to the $m_{f_1, g_1}$ case is that we are not done with this yet. We now specialize to the case, where
a cube $I^3 \subset \R^{d_3}$ has been fixed and $|f_3|, |g_3| \le 1_{I^3}$. We notice that
if $|x_1| + |x_2| < \ell(I^3)$, then
$$
	\frac{|I^3|}{(|x_1| + |x_2| + \ell(I^3))^{d_3}} \sim \frac{|I^3|}{\ell(I^3)^{d_3}} = 1
$$
so that
\begin{align*}
	|\partial^{\beta_1} \partial^{\beta_2} K_{f_3, g_3}(x_{12})
	 & \lesssim \frac{\|f_3\|_{L^p}\|g_3\|_{L^{p'}}}{|x_1|^{d_1+|\beta_1|}(|x_1|+|x_2|)^{d_2+|\beta_2|}}             \\
	 & \lesssim \frac{|I^3|^2}{|x_1|^{d_1+|\beta_1|}(|x_1|+|x_2|)^{d_2+|\beta_2|}(|x_1| + |x_2| + \ell(I^3))^{d_3}}.
\end{align*}
So we have proved the above estimate whenever $|x_1| + |x_2| < \ell(I^3)$. We now want to see that
it always holds.

To this end, notice that
\begin{equation}\label{eq:parkereq}
	K_{f_3, g_3}(x_{12}) = \langle K(x_{12}, \cdot) * f_3, g_3\rangle.
\end{equation}
Indeed, one can, for instance, verify that given $x_3$ we have
\begin{align*}
	\calF\Big( x_{12} \mapsto & \int_{\R^{d_3}} K(x_{12}, x_3-y_3)f_3(y_3) \ud y_3 \Big)(\xi_{12})                                             \\
	                          & = \int_{\R^{d_3}} m(\xi) \wh f_3(\xi_3) e^{2\pi i x_3 \cdot \xi_3} \ud \xi_3 = T_{m(\xi_{12}, \cdot)} f_3(x_3)
\end{align*}
so that the Fourier transform of the LHS and RHS of \eqref{eq:parkereq} at $\xi_{12}$
both equal $m_{f_3, g_3}(\xi_{12}) = \langle T_{m(\xi_{12}, \cdot)} f_3, g_3\rangle$.
Now, we can estimate
$$
	\partial^{\beta_1} \partial^{\beta_2} K_{f_3, g_3}(x_{12}) = \int_{I^3} \int_{I^3}
	\partial^{\beta_1} \partial^{\beta_2} K(x_{12}, x_3-y_3)f_3(y_3)g_3(x_3) \ud y_3 \ud x_3
$$
using the already proved full kernel estimates -- we do this in the range of interest $\ell(I^3) \le |x_1| + |x_2|$.
So we have
\begin{align*}
	|\partial^{\beta_1} & \partial^{\beta_2} K_{f_3, g_3}(x_{12})|                                                                \\
	                    & \lesssim \frac{1}{|x_1|^{d_1+|\beta_1|}(|x_1|+|x_2|)^{d_2+|\beta_2|}}
	\iint_{I^3\times I^3}\frac{\ud y_3 \ud x_3}{(|x_1|+|x_2|+|x_3-y_3|)^{d_3}}                                                    \\
	                    & \sim \frac{1}{|x_1|^{d_1+|\beta_1|}(|x_1|+|x_2|)^{d_2+|\beta_2|}}
	\iint_{I^3\times I^3}\frac{\ud y_3 \ud x_3}{(|x_1|+|x_2|)^{d_3}}                                                              \\
	                    & \sim \frac {|I^3|^2}{|x_1|^{d_1+|\beta_1|}(|x_1|+|x_2|)^{d_2+|\beta_2|} (|x_1|+|x_2|+\ell(I^3))^{d_3}}.
\end{align*}
Overall, we have proved that
$$
	|\partial^{\beta_1} \partial^{\beta_2} K_{f_3, g_3}(x_{12})|
	\lesssim \frac {|I^3|^2}{|x_1|^{d_1+|\beta_1|}(|x_1|+|x_2|)^{d_2+|\beta_2|} (|x_1|+|x_2|+\ell(I^3))^{d_3}}
$$
whenever $I^3 \subset \R^{d_3}$ is a cube and $|f_3|, |g_3| \le 1_{I_3}$.

Now, we still have four partial kernels to cover. For instance, in our current situation
we can also write in the beginning that
\begin{align*}
	\langle T_m(f_{12} \otimes f_{3}), g_{12} \otimes g_{3} \rangle
	 & = \int_{\R^{d_{3}}} \Big( \int_{\R^{d_{3}}} m_{f_{12}, g_{12}}(\xi_{3}) \wh f_{3}(\xi_{3})
	e^{2\pi i x_{3} \cdot \xi_{3}} \ud \xi_{3} \Big) g_{3}(x_{3}) \ud x_{3}                       \\
	 & = \langle T_{m_{f_{12}, g_{12}}}f_{3}, g_{3}\rangle,
\end{align*}
where
\begin{align*}
	m_{f_{12}, g_{12}}(\xi_{3}) = \langle T_{m(\cdot,\, \xi_3)} f_{12}, g_{12} \rangle.
\end{align*}
This time $m_{f_{12}, g_{12}}$ is seen to be a one-parameter multiplier so that
$$
	T_{m_{f_{12}, g_{12}}} f_3 = K_{f_{12}, g_{12}} * f_3,
$$
where
$$
	|\partial^{\beta_3}K_{f_{12}, g_{12}}(x_3)| \lesssim \frac{\|f_{12}\|_{L^p} \|g_{12}\|_{L^{p'}}}{|x_3|^{d_3 + |\beta_3|}},
	\qquad 1 < p < \infty.
$$
We are directly happy with this.

In our initial setup, where $f_1, g_1$ are functions in $\R^{d_1}$ and $f_{23}, g_{23}$ are functions in $\R^{d_{23}}$,
we can also write
\begin{align*}
	\langle T_m(f_1 \otimes f_{23}), g_1 \otimes g_{23} \rangle
	 & = \langle T_{m_{f_{23}, g_{23}}}f_{1}, g_{1}\rangle, \qquad
	m_{f_{23}, g_{23}}(\xi_{1}) = \langle T_{m(\xi_1,\, \cdot)} f_{23}, g_{23} \rangle.
\end{align*}
Now, this $m_{f_{23}, g_{23}}$ can be seen as a one-parameter multiplier and
so
$$
	T_{m_{f_{23}, g_{23}}} f_1 = K_{f_{23}, g_{23}} * f_1, \qquad |\partial^{\beta_1} K_{f_{23}, g_{23}}(x_1)| \lesssim
	\frac{\|f_{23}\|_{L^p} \|g_{23}\|_{L^{p'}}}{|x_1|^{d_1 + |\beta_1|}}.
$$
Similarly to the $K_{f_3, g_3}$ case, we can add some flag structure to this estimate.
Fix $I^{23}$ with $\ell(I^2) \le \ell(I^3)$ and assume in what follows
$|f_{23}|, |g_{23}| \le 1_{I^{23}}$. Notice that if $|x_1| \le \ell(I^2)$ then
$$
	\frac{|I^{23}|}{(|x_1| + \ell(I^2))^{d_2}(|x_1| + \ell(I^2) + \ell(I^3))^{d_3}} \sim \frac{|I^{23}|}{\ell(I^2)^{d_2}\ell(I^3)^{d_3}} = 1
$$
so that
$$
	|\partial^{\beta_1} K_{f_{23}, g_{23}}(x_1)| \lesssim
	\frac{|I^{23}|}{|x_1|^{d_1 + |\beta_1|}} \sim \frac{|I^{23}|^2}{|x_1|^{d_1 + |\beta_1|}(|x_1| + \ell(I^2))^{d_2}
		(|x_1| + \ell(I^2) + \ell(I^3))^{d_3}}.
$$
Next, we show that this is true also for $|x^1| \ge \ell(I^2)$. To this end,
define $f_{23}^{y_2}(y_3) = f_{23}(y_2,y_3)$ and $g_{23}^{x_2}(x_3) = g_{23}(x_2, x_3)$ and
estimate using $|x^1| \ge \ell(I^2)$ and the already proved estimates for $K_{f_3, g_3}$ that
\begin{align*}
	|\partial^{\beta_1} K_{f_{23}, g_{23}} & (x_1)| = |\langle \partial^{\beta_1}K(x_1, \cdot) * f_{23}, g_{23}\rangle|                                             \\
	                                       & = \Big| \iint_{I^{23} \times I^{23}}
	\partial^{\beta_1} K(x_1, x_2-y_2, x_3-y_3)f_{23}(y_{23})g_{23}(x_{23}) \ud y_{23} \ud x_{23}\Big|                                                              \\
	                                       & \le \iint_{I^2 \times I^2} |\langle \partial^{\beta_1} K(x_1, x_2-y_2,\, \cdot) * f_{23}^{y_2}, g_{23}^{x_2} \rangle|
	\ud y_2 \ud x_2                                                                                                                                                 \\
	                                       & = \iint_{I^2 \times I^2} |\partial^{\beta_1} K_{f_{23}^{y_2}, g_{23}^{x_2}}(x_1, x_2-y_2)| \ud y_2 \ud x_2             \\
	                                       & \lesssim
	\iint_{I^2\times I^2} \frac {|I^3|^2}{|x_1|^{d_1+|\beta_1|}
		                      (|x_1|+|x_2-y_2|)^{d_2} (|x_1|+|x_2-y_2|+\ell(I^3))^{d_3}} \ud y_2 \ud x_2                                                               \\
	                                       & \sim \iint_{I^2 \times I^2} \frac{|I^3|^2}{|x_1|^{d_1+|\beta_1|}|x_1|^{d_2} (|x_1| + \ell(I^3))^{d_3}} \ud y_2 \ud x_2 \\
	                                       & = \frac{|I^{23}|^2}{|x_1|^{d_1+|\beta_1|}|x_1|^{d_2} (|x_1| + \ell(I^3))^{d_3}}                                        \\
	                                       & \sim \frac{|I^{23}|^2}{|x_1|^{d_1 + |\beta_1|}(|x_1| + \ell(I^2))^{d_2}
		                                              (|x_1| + \ell(I^2) + \ell(I^3))^{d_3}}.
\end{align*}
We have shown that
$$
	|\partial^{\beta_1} K_{f_{23}, g_{23}}(x_1)| \lesssim
	\frac{|I^{23}|^2}{|x_1|^{d_1 + |\beta_1|}(|x_1| + \ell(I^2))^{d_2}
		(|x_1| + \ell(I^2) + \ell(I^3))^{d_3}}
$$
whenever $\ell(I^2) \le \ell(I^3)$ and $|f_{23}|, |g_{23}| \le 1_{I^{23}}$.

We now move to the final two cases. Assume that $f_{13}, g_{13}$ are functions in $\R^{d_{13}}$
and $f_2, g_2$ are functions in $\R^{d_2}$. We can write
$$
	\langle T_m(f_{13} \otimes f_{2}), g_{13} \otimes g_{2} \rangle
	= \langle T_{m_{f_{2}, g_{2}}}f_{13}, g_{13}\rangle, \qquad
	m_{f_{2}, g_{2}}(\xi_{13}) = \langle T_{m(\xi_1,\, \cdot, \xi_3)} f_{2}, g_{2} \rangle.
$$
Now, $m_{f_2, g_2}$ is a bi-parameter flag multiplier in $\R^{d_{13}}$ and so
$$
	T_{m_{f_2, g_2}} f_{13} = K_{f_2, g_2} * f_{13}, \qquad
	|\partial^{\beta_1} \partial^{\beta_3} K_{f_2, g_2}(x_{13})| \lesssim \frac{\|f_2\|_{L^p}\|g_2\|_{L^{p'}}}
	{|x_1|^{d_1 + |\beta_1|} (|x_1| + |x_3|)^{d_3 + |\beta_3|} }.
$$
Let now $I^2 \subset \R^{d_2}$ be a cube and $|f_2|, |g_2| \le 1_{I^2}$.
If $|x_1| \le \ell(I^2)$ the above estimate directly implies
$$
	|\partial^{\beta_1} \partial^{\beta_3} K_{f_2, g_2}(x_{13})| \lesssim \frac{|I^2|^2}
	{|x_1|^{d_1 + |\beta_1|} (|x_1| + \ell(I^2))^{d_2} (|x_1| + |x_3|)^{d_3 + |\beta_3|} }.
$$
If $|x_1| \ge \ell(I^2)$ the same estimate follows from the full kernel estimates:
\begin{align*}
	|\partial^{\beta_1} \partial^{\beta_3} K_{f_2, g_2}(x_{13})| & = |\langle \partial^{\beta_1} \partial^{\beta_3} K(x_1, \cdot\,, x_3) * f_2, g_2\rangle|     \\
	                                                             & \lesssim
	\frac{1}{|x_1|^{d_1+|\beta_1|}(|x_1|+|x_3|)^{d_3+|\beta_3|}}
	\iint_{I^2\times I^2}\frac{\ud y_2 \ud x_2}{(|x_1|+|x_2-y_2|)^{d_2}}                                                                                        \\
	                                                             & \sim \frac{|I^2|^2}
	                                                                    {|x_1|^{d_1 + |\beta_1|} (|x_1| + \ell(I^2))^{d_2} (|x_1| + |x_3|)^{d_3 + |\beta_3|} }.
\end{align*}
Finally, we write
$$
	\langle T_m(f_{13} \otimes f_{2}), g_{13} \otimes g_{2} \rangle
	= \langle T_{m_{f_{13}, g_{13}}}f_{2}, g_{2}\rangle, \qquad
	m_{f_{13}, g_{13}}(\xi_{2}) = \langle T_{m(\cdot,\,\xi_2, \cdot\,)} f_{13}, g_{13} \rangle.
$$
This time $m_{f_{13}, g_{13}}$ is a one-parameter multiplier in $\R^{d_2}$ and so
$$
	T_{m_{f_{13}, g_{13}}} f_2 = K_{f_{13}, g_{13}} * f_2, \qquad |\partial^{\beta_2} K_{f_{13}, g_{13}}(x_2)|
	\lesssim \frac{\|f_{13}\|_{L^p}\|g_{13}\|_{L^{p'}}}
	{|x_2|^{d_2 + |\beta_2|}}.
$$
Now, suppose $I^1, I^3$ are cubes with $\ell(I^1) \le \ell(I^3)$ and $|f_{13}|, |g_{13}| \le 1_{I^{13}}$.
If $|x_2| \le \ell(I^3)$ the above estimate directly gives
$$
	|\partial^{\beta_2} K_{f_{13}, g_{13}}(x_2)| \lesssim
	\frac{|I^1||I^3|^2}
	{|x_2|^{d_2 + |\beta_2|}(|x_2| + \ell(I^3))^{d_3}}.
$$
For $|x_2| \ge \ell(I^3)$ we argue using the known estimate in the $K_{f_1, g_1}$ case:
\begin{align*}
	|\partial^{\beta_2} K_{f_{13}, g_{13}} & (x_2)| = |\langle \partial^{\beta_2}K(\cdot,\, x_2, \cdot\,) * f_{13}, g_{13}\rangle|                                 \\
	                                       & = \Big| \iint_{I^{13} \times I^{13}}
	\partial^{\beta_2} K(x_1-y_1, x_2, x_3-y_3)f_{13}(y_{13})g_{13}(x_{13}) \ud y_{13} \ud x_{13}\Big|                                                             \\
	                                       & \le \iint_{I^3 \times I^3} |\langle \partial^{\beta_2} K(\cdot,\, x_2, x_3-y_3) * f_{13}^{y_3}, g_{13}^{x_3} \rangle|
	\ud y_3 \ud x_3                                                                                                                                                \\
	                                       & = \iint_{I^3 \times I^3} |\partial^{\beta_2} K_{f_{13}^{y_3}, g_{13}^{x_3}}(x_2, x_3-y_3)| \ud y_3 \ud x_3            \\
	                                       & \lesssim
	\iint_{I^3\times I^3} \frac {|I^1|}{|x_2|^{d_2+|\beta_2|}
		                      (|x_2|+|x_3-y_3|)^{d_3}} \ud y_3 \ud x_3                                                                                                \\
	                                       & \sim \frac{|I^1||I^3|^2}{|x_2|^{d_2 + |\beta_2|}(|x_2| + \ell(I^3))^{d_3}}.
\end{align*}

We recap our estimates in the form of a lemma:
\begin{lem}\label{lem:par-2}
	Suppose that $m$ satisfies \eqref{eq:def}. Then for flag rectangles $I = I^1 \times I^2 \times I^3 \in \calR_F$ we have, using the notation
	from above for the partial kernels, that
	\begin{equation*}
		|\partial^{\beta_1}\partial^{\beta_2} K_{f_3, g_3}(x_{12})|
		\lesssim \frac {|I^3|^2}{|x_1|^{d_1+|\beta_1|}(|x_1|+|x_2|)^{d_2+|\beta_2|} (|x_1|+|x_2|+\ell(I^3))^{d_3}}
	\end{equation*}
	whenever $|f_3|, |g_3| \le 1_{I^3}$,
	\begin{equation*} 
		|\partial^{\beta_1}\partial^{\beta_3} K_{f_2, g_2}(x_{13})| \lesssim
		\frac {|I^2|^2}{|x_1|^{d_1+|\beta_1|}(|x_1|+\ell(I^2))^{d_2} (|x_1|+|x_3|)^{d_3+|\beta_3|}}
	\end{equation*}
	whenever $|f_2|, |g_2| \le 1_{I^2}$,
	\begin{equation*}
		|\partial^{\beta_2}\partial^{\beta_3} K_{f_1, g_1}(x_{23})| \lesssim
		\frac {|I^1|}{ |x_2|^{d_2+|\beta_2|} ( |x_2|+|x_3|)^{d_3+|\beta_3|}}
	\end{equation*}
	whenever $|f_1|, |g_1| \le 1_{I^1}$,
	\begin{equation*}
		|\partial^{\beta_1} K_{f_{23}, g_{23}}(x_1)| \lesssim
		\frac {|I^{23}|^2}{ |x_1|^{d_1+|\beta_1|}(|x_1|+\ell(I^2))^{d_2}(|x_1|+\ell(I^2)+\ell(I^3))^{d_3} }
	\end{equation*}
	whenever $|f_{23}|, |g_{23}| \le 1_{I^{23}}$,
	\begin{equation*}
		|\partial^{\beta_2} K_{f_{13}, g_{13}}(x_2)| \lesssim
		\frac {|I^1||I^{3}|^2}{ |x_2|^{d_2+|\beta_2|}(|x_2|+\ell(I^3))^{d_3} }
	\end{equation*}
	whenever $|f_{13}|, |g_{13}| \le 1_{I^{13}}$, and
	\begin{equation*}
		|\partial^{\beta_3} K_{f_{12}, g_{12}}(x_3)| \lesssim
		\frac {|I^{12}|}{ |x_3|^{d_3+|\beta_3|} }
	\end{equation*}
	whenever $|f_{12}|, |g_{12}| \le 1_{I^{12}}$.
\end{lem}

\subsection{Cancellative tri-parameter flag kernels}
Given the above kernel estimates in the concrete situation of tri-parameter flag multipliers, we are now
ready to formulate our abstract tri-parameter flag kernel estimate assumptions.
In the tri-parameter setting we are aiming for ``cancellative'' theory -- this means that
we will not build in any $\BMO$ philosophies to the partial kernels. However, we will do that in the
bi-parameter setting, where we will prove a completely general $T1$.
\subsubsection{Full kernel estimates}
The full kernel
is a function $K \colon (\R^{d}\times \R^{d})\setminus \Delta \to \C$, where
\begin{align*}
	\Delta=\big\{(x, y)\in\R^{d} & \times \R^{d}: x_1 = y_1 \big\}.
\end{align*}
If $\supp_{\R^{d_m}}f\cap  \supp_{\R^{d_m}}g=\emptyset$ for all $m=1,2,3$, then
the associated flag singular integral operator (SIO) $T$
must (among other things) have the following kernel representation
\[
	\langle Tf, g\rangle =\iint K(x,y) f(y) g(x)\ud y\ud x.
\]
\begin{rem}
	Flag SIOs will be linear operators (a priori defined in some appropriate manner, e.g., so that
	their bilinear form is defined on indicators of rectangles) that have all of the kernel estimates
	defined in section. In the tri-parameter setting our SIOs will always be ``cancellative''
	and in the bi-parameter setting general -- in the SIO level the difference is somewhat small, only
	on how generally certain partial kernel estimates are assumed. Such SIOs will be called
	flag CZOs (cancellative or general) if they satisfy additional cancellation or flag adapted $T1$ style assumptions.
	On the CZO level the difference between ``cancellative'' and ``general'' is much bigger
	than on the SIO level, but, again, there is already a difference in the SIO level.
\end{rem}

Define
$$
	\size_F(x,y) := \frac 1{ |x_1-y_1|^{d_1}(|x_1-y_1|+|x_2-y_2|)^{d_2}(|x_1-y_1|+|x_2-y_2|+|x_3-y_3|)^{d_3}}.
$$
Let $(x,y) \in (\R^{d }\times \R^{d })\setminus \Delta$. First, we assume that $K$ satisfies the size estimate
\begin{equation*}
	|K(x, y)|\lesssim \size_F(x,y).
\end{equation*}
Let $x'=(x_1',x_2',x_3')$ be such that $|x_i'-x_i| \le |x_i-y_i|/2$ for $i=1,2,3$. We assume that $K$ satisfies the mixed size and H\"older estimates
\begin{equation*}
	\begin{split}
		 & |K((x'_1,x_2,x_3),y)-K(x,y)|
		\lesssim  \Big(\frac{|x_1'-x_1|}{|x_1-y_1|}\Big)^{\alpha_1}\size_F(x, y),
	\end{split}
\end{equation*}
\begin{equation*}
	\begin{split}
		 & |K((x_1,x_2',x_3),y)-K(x,y)|
		\lesssim  \Big(\frac{|x_2'-x_2|}{  |x_2-y_2|}\Big)^{\alpha_2}\size_F(x,y),
	\end{split}
\end{equation*}
\begin{equation*}
	\begin{split}
		 & |K((x_1,x_2,x_3'),y)-K(x,y)|
		\lesssim  \Big(\frac{|x_3'-x_3|}{|x_3-y_3| }\Big)^{\alpha_3}\size_F(x,y) ,
	\end{split}
\end{equation*}
\begin{equation*}
	\begin{split}
		|K((x_1',x_2',x_3),y) & -K((x_1',x_2,x_3),y)-K((x_1,x_2',x_3),y)+K(x,y)|
		\\
		                      & \qquad  \lesssim  \Big(\frac{|x_1'-x_1|}{|x_1-y_1|}\Big)^{\alpha_1}\Big(\frac{|x_2'-x_2|}{  |x_2-y_2|}\Big)^{\alpha_2}\size_F(x,y),
	\end{split}
\end{equation*}

\begin{equation*}
	\begin{split}
		|K((x_1,x_2',x_3'),y) & -K((x_1,x_2',x_3),y)-K((x_1,x_2,x_3'),y)+K(x,y)|
		\\
		                      & \qquad  \lesssim  \Big(\frac{|x_2'-x_2|}{  |x_2-y_2|}\Big)^{\alpha_2}\Big(\frac{|x_3'-x_3|}{ |x_3-y_3|}\Big)^{\alpha_3}\size_F(x,y),
	\end{split}
\end{equation*}
and
\begin{equation*}
	\begin{split}
		|K((x_1',x_2,x_3'),y) & -K((x_1',x_2,x_3),y)-K((x_1,x_2,x_3'),y)+K(x,y)|
		\\
		                      & \qquad  \lesssim \Big(\frac{|x_1'-x_1|}{|x_1-y_1|}\Big)^{\alpha_1}\Big(\frac{|x_3'-x_3|}{ |x_3-y_3|}\Big)^{\alpha_3}\size_F(x,y),
	\end{split}
\end{equation*}
where $\alpha_m \in (0,1]$ for all $m=1,2,3$. Finally, we also require that  $K$ satisfies the H\"older estimate
\begin{equation*}
	\begin{split}
		\Big| & \big[K(x',y)-K((x_1',x_2',x_3),y)-K((x_1,x_2',x_3'),y)+K(x_1, x_2', x_3,y)\big]                                                                                              \\
		      & \qquad-\big[K((x_1',x_2,x_3'),y)-K((x_1',x_2,x_3),y)-K((x_1,x_2,x_3'),y)+K(x,y)\big] \Big|                                                                                   \\
		      & \lesssim  \Big(\frac{|x_1'-x_1|}{|x_1-y_1|}\Big)^{\alpha_1}\Big(\frac{|x_2'-x_2|}{ |x_2-y_2|}\Big)^{\alpha_2}\Big(\frac{|x_3'-x_3|}{ |x_3-y_3|}\Big)^{\alpha_3}\size_F(x,y).
	\end{split}
\end{equation*}
We emphasize that these assumptions are weaker than what we obtained for the flag Fourier multipliers.
First, this type of H\"older estimates are, of course, weaker than derivative bounds.
Second, for the regularity estimates in the second and third variables, we do not need to assume
the stronger decay factors of the form
\[
	\Big(\frac{|x_2'-x_2|}{ |x_1-y_1|+ |x_2-y_2|}\Big)^{\alpha_2}\quad \text{or}\quad \Big(\frac{|x_3'-x_3|}{ |x_1-y_1|+|x_2-y_2|+|x_3-y_3|}\Big)^{\alpha_3}
\]
that are true for the multipliers.

Define the adjoint kernels $K^*$, $K^*_1$ and $K^*_{23}$ via the following relations:
\begin{equation*}
	K^*(x,y)=K(y,x), \quad
	K^*_1(x,y)=K((y_1,x_2,x_3),(x_1,y_2,y_3))
\end{equation*}
and
\begin{equation*}
	K^*_{23}(x,y)=K((x_1,y_2,y_3),(y_1,x_2,x_3)).
\end{equation*}
Other natural adjoint kernels can be defined in a similar way.
We assume that each adjoint kernel satisfies the same estimates as the kernel $K$.

\subsubsection{Partial kernel estimates}
Since our tri-parameter theory is ``cancellative'', i.e., paraproduct free, we will only need
to require the partial kernel estimates in a more limited way than what is true for multipliers. For instance,
instead of having a priori control of $K_{f_1, g_1}$ for all cubes $I^1$ and $|f_1|, |g_1| \le 1_{I^1}$,
it will be enough to only demand the estimates for $K_{I^1} := K_{1_{I^1}, 1_{I^1}}$.

First, we assume that
for every cube $I^1\subset \R^{d_1}$ there exists a kernel
$$
	K_{I^1} \colon (\R^{d_{23}} \times \R^{d_{23}}) \setminus \{(x_{23}, y_{23}) \colon x_2 = y_2\} \to \C,
$$
so that if $\supp_{\R^{d_i}}f_{23}$ and   $\supp_{\R^{d_i}}g_{23}$ are disjoint for $i=2,3$, then
\begin{equation*}
	\langle T(1_{I^1} \otimes f_{23}), 1_{I^1} \otimes g_{23}\rangle
	= \iint K_{I^1}(x_{23}, y_{23}) f_{23}(y_{23})g_{23}(x_{23}) \ud y_{23} \ud x_{23}.
\end{equation*}
Define
$$
	\size_{I^1}(x_{23}, y_{23}) =\frac {|I^1|}{ |x_2-y_2|^{d_2} ( |x_2-y_2|+|x_3-y_3|)^{d_3 }}.
$$
The kernel $K_{I^1}$ is assumed to satisfy the following estimates.
Let $x_{23}, y_{23}$ be such that $x_2 \ne y_2$. First, we assume the size estimate
\begin{equation*}
	|K_{I^1}(x_{23}, y_{23}) | \lesssim \size_{I^1}(x_{23}, y_{23}).
\end{equation*}
Let $x'_{23}=(x_2',x_3')$ be such that $|x_i'-x_i| \le |x_i-y_i|/2$ for $i=2,3$.
We assume the mixed size and H\"older estimates
\begin{equation*}
	\begin{split}
		|K_{I^1}(x_2', x_3,y_{23})-K_{I^1}(x_{23},y_{23})|  \lesssim
		\Big( \frac{|x_2'-x_2|}{|x_2-y_2|}  \Big)^{\alpha_{2}} \size_{I^1}(x_{23}, y_{23})
	\end{split}
\end{equation*}
and
\begin{equation*}
	\begin{split}
		|K_{I^1}(x_2, x_3',y_{23})-K_{I^1}(x_{23},y_{23})|  \lesssim
		\Big( \frac{|x_3'-x_3|}{|x_3-y_3|}  \Big)^{\alpha_{3}}\size_{I^1}(x_{23}, y_{23}).
	\end{split}
\end{equation*}
Finally, we also assume the H\"older estimate
\begin{equation*}
	\begin{split}
		|K_{I^1}(x_{23}',y_{23}) & -K_{I^1}(x_2', x_3,y_{23})- K_{I^1}(x_2, x_3',y_{23}) +K_{I^1}(x_{23},y_{23})| \\
		                         & \lesssim
		\Big( \frac{|x_2'-x_2|}{|x_2-y_2|}  \Big)^{\alpha_{2}}\Big( \frac{|x_3'-x_3|}{|x_3-y_3|}  \Big)^{\alpha_{3}}\size_{I^1}(x_{23}, y_{23}).
	\end{split}
\end{equation*}
We assume that the adjoint kernels defined by $$
	K_{I^1}^*(x_{23}, y_{23})
	=K_{I^1}(y_{23}, x_{23}), \quad K_{I^1, 2}^*(x_{23}, y_{23})= K_{I^1}(y_{2}, x_3, x_{2}, y_3)
$$ and $ K_{I^1, 3}^*(x_{23}, y_{23})= K_{I^1}(x_{2}, y_3, y_{2}, x_3)$ satisfy the same estimates.

Next, for every cube $I^2\subset \R^{d_2}$ we assume that there exists a kernel
$$
	K_{I^2} \colon (\R^{d_{13}} \times \R^{d_{13}}) \setminus \{(x_{13}, y_{13}) \colon x_1 = y_1 \} \to \C,
$$
so that if $\supp_{\R^{d_i}}f_{13}$ and   $\supp_{\R^{d_i}}g_{13}$ are disjoint for $i=1,3$, then
\begin{equation*}
	\langle T(1_{I^2} \otimes f_{13}), 1_{I^2} \otimes g_{13}\rangle
	= \iint K_{I^2}(x_{13}, y_{13}) f_{13}(y_{13})g_{13}(x_{13}) \ud y_{13} \ud x_{13}.
\end{equation*}
The kernel $K_{I^2}$ is assumed to satisfy the following estimates.
Let $x_{13}, y_{13}$ be such that $x_1 \not= y_1$. First, we assume the size estimate
\begin{equation*}
	\begin{split}
		|K_{I^2}(x_{13}, y_{13}) |
		 & \lesssim \frac {|I^2|^2}{ |x_1-y_1|^{d_1}(|x_1-y_1|+\ell(I^2))^{d_2} ( |x_1-y_1|+|x_3-y_3|)^{d_3 }} \\
		 & =: \size_{I^2}(x_{13}, y_{13}).
	\end{split}
\end{equation*}
Let $x'_{13}=(x_1',x_3')$ be such that $|x_i'-x_i| \le |x_i-y_i|/2$ for $i=1,3$.
We also assume the mixed size and H\"older estimates
\begin{equation*}
	\begin{split}
		|K_{I^2}(x_1', x_3,y_{13})-K_{I^2}(x_{13},y_{13})|   \lesssim
		\Big( \frac{|x_1'-x_1|}{|x_1-y_1|}  \Big)^{\alpha_{1}} \size_{I^2}(x_{13}, y_{13}).
	\end{split}
\end{equation*}
and
\begin{equation*}
	\begin{split}
		|K_{I^2}(x_1, x_3',y_{13})-K_{I^2}(x_{13},y_{13})|  \lesssim
		\Big( \frac{|x_3'-x_3|}{|x_3-y_3|}  \Big)^{\alpha_{3}} \size_{I^2}(x_{13}, y_{13}).
	\end{split}
\end{equation*}
Finally, we also assume the H\"older estimate
\begin{equation*}
	\begin{split}
		|K_{I^2}(x_{13}',y_{13}) & -K_{I^2}(x_1', x_3,y_{13})- K_{I^2}(x_1, x_3',y_{13}) +K_{I^2}(x_{13},y_{13})| \\
		                         & \lesssim
		\Big( \frac{|x_1'-x_1|}{|x_1-y_1|}  \Big)^{\alpha_{1}}\Big( \frac{|x_3'-x_3|}{|x_3-y_3|}  \Big)^{\alpha_{3}}
		\size_{I^2}(x_{13}, y_{13}).
	\end{split}
\end{equation*}
As before, we also assume the natural symmetric estimates (i.e., the adjoint kernels satisfy the same estimates).

Next, for every cube $I^3\subset \R^{d_3}$ we assume that there exists a kernel
$$
	K_{I^3} \colon (\R^{d_{12}} \times \R^{d_{12}}) \setminus \{(x_{12}, y_{12}) \colon x_1 = y_1\} \to \C,
$$
so that if $\supp_{\R^{d_i}}f_{12}$ and   $\supp_{\R^{d_i}}g_{12}$ are disjoint for $i=1,2$, then
\begin{equation*}
	\langle T(f_{12}\otimes 1_{I^3} ), g_{12}\otimes 1_{I^3}\rangle
	= \iint K_{I^3}(x_{12}, y_{12}) f_{12}(y_{12})g_{12}(x_{12}) \ud y_{12} \ud x_{12}.
\end{equation*}
The kernel $K_{I^3}$ is assumed to satisfy the following estimates.
Let $x_{12}, y_{12}$ be such that $x_1\not=y_1$. First, we assume the size estimate
\begin{equation}\label{E:eq96}
	|K_{I^3}(x_{12}, y_{12}) |
	\lesssim \frac {|I^3|^2}{ |x_1-y_1|^{d_1}(|x_1-y_1|+|x_2-y_2| )^{d_2} ( |x_1-y_1|+|x_2-y_2|+\ell(I^3))^{d_3 }}.
\end{equation}
Again, we denote the RHS of \eqref{E:eq96} by $\size_{I^3}(x_{12}, y_{12})$.
Let $x'_{12}=(x_1',x_2')$ be such that $|x_i'-x_i| \le |x_i-y_i|/2$ for $i=1,2$.
We assume the mixed size and H\"older estimates
\begin{equation*}
	\begin{split}
		|K_{I^3}(x_1', x_2,y_{12})-K_{I^3}(x_{12},y_{12})|   \lesssim
		\Big( \frac{|x_1'-x_1|}{|x_1-y_1|}  \Big)^{\alpha_{1}} \size_{I^3}(x_{12}, y_{12}).
	\end{split}
\end{equation*}
and
\begin{equation*}
	\begin{split}
		|K_{I^3}(x_1, x_2',y_{12})-K_{I^3}(x_{12},y_{12})|   \lesssim
		\Big( \frac{|x_2'-x_2|}{|x_2-y_2|}  \Big)^{\alpha_{2}} \size_{I^3}(x_{12}, y_{12}).
	\end{split}
\end{equation*}
Finally, we also assume the H\"older estimate
\begin{equation*}
	\begin{split}
		|K_{I^3}(x_{12}',y_{12}) & -K_{I^3}(x_1', x_2,y_{12})- K_{I^3}(x_1, x_2',y_{12}) +K_{I^3}(x_{12},y_{12})| \\
		                         & \lesssim
		\Big( \frac{|x_1'-x_1|}{|x_1-y_1|}  \Big)^{\alpha_{1}}
		\Big( \frac{|x_2'-x_2|}{|x_2-y_2|}  \Big)^{\alpha_{2}} \size_{I^3}(x_{12}, y_{12}).
	\end{split}
\end{equation*}
Again, we assume the symmetric estimates.

It remains to define the three partial kernels, where we fix two out of the three parameters (instead of one like
in the above three cases).
Let $I^1 \times I^2 \times I^3$ be flag.
We assume the existence of the partial kernels $K_{I^{12}}(x_3, y_3)$, $K_{I^{13}}(x_2, y_2)$ and $K_{I^{23}}(x_1, y_1)$
modelling from the multiplier case as above. For instance, the kernel
\[
	K_{I^{12}} \colon (\R^{d_{3}} \times \R^{d_{3}}) \setminus \{(x_{3}, y_{3}) \colon x_3=y_3\} \to \C
\]
is a function such that if $\supp f_3 \cap \supp g_3 = \emptyset$ then
\begin{equation*}
	\langle T(1_{I^{12}} \otimes f_{3}), 1_{I^{12}} \otimes g_{3}\rangle
	= \iint K_{I^{12}}(x_{3}, y_{3}) f_{3}(y_{3})g_{3}(x_{3}) \ud y_{3} \ud x_{3}.
\end{equation*}
We assume that $K_{I^{12}}\in CZ_{\alpha_3}(\R^{d_3})$ is a classical $\alpha_3$-Calder\'on-Zygmund kernel with
the kernel constants being bounded $C|I^{12}|$.
For $K_{I^{13}} $ and $K_{I^{23}} $ we assume the obvious representation properties and the following estimates
together with their natural symmetric counterparts:
\begin{equation*}
	|K_{I^{13}}(x_2, y_2) |                        \lesssim \frac {|I^1||I^{3}|^2}{ |x_2-y_2|^{d_2 }(|x_2-y_2|+\ell(I^3))^{d_3} }
	=: \size_{I^{13}}(x_2, y_2),
\end{equation*}
\begin{equation*}
	|K_{I^{13}}(x_2', y_2)-K_{I^{13}}(x_2, y_2) |  \lesssim \Big( \frac{|x_2'-x_2|}{|x_2-y_2|} \Big)^{\alpha_{2}}
	\size_{I^{13}}(x_2, y_2),
\end{equation*}
\begin{equation*}
	\begin{split}
		|K_{I^{23}}(x_1, y_1) | & \lesssim \frac { |I^{23}|^2}{ |x_1-y_1|^{d_1 }(|x_1-y_1|+\ell(I^2))^{d_2} (|x_1-y_1|+\ell(I^2)+\ell(I^3))^{d_3}} \\
		                        & =: \size_{I^{23}}(x_1, y_1)
	\end{split}
\end{equation*}
and
\begin{equation*}
	|K_{I^{23}}(x_1', y_1)-K_{I^{23}}(x_1, y_1) |
	\lesssim \Big( \frac{|x_1'-x_1|}{|x_1-y_1|}  \Big)^{\alpha_{1}} \size_{I^{23}}(x_1, y_1),
\end{equation*}
where $|x_i'-x_i| \le |x_i-y_i|/2$ for $i=1,2$.

\subsection{General bi-parameter flag kernels}
The bi-parameter kernel estimates are highly similar than in the tri-parameter setting, but, since our
bi-parameter theory will involve paraproducts, we will
need partial kernels also for some atoms in addition to the indicators like in the tri-parameter
situation. While these estimates are simply the natural
bi-parameter analogs of the above tri-parameter estimates (with the
aforementioned difference in the partial kernels) we formulate things carefully to have a clear reference.
For the rest of this section, $d=d_1+d_2$ in order to capture the bi-parameter flag structure.

\subsubsection{Full kernel estimates}
This section is fully analogous to the tri-parameter cancellative case. The full kernel
is a function $K \colon (\R^{d}\times \R^{d})\setminus \Delta \to \C$, where
\begin{align*}
	\Delta=\big\{(x, y)\in\R^{d} & \times \R^{d}: x_1 = y_1 \big\}.
\end{align*}
If $\supp_{\R^{d_m}}f\cap  \supp_{\R^{d_m}}g=\emptyset$ for all $m=1,2$, then
the associated flag singular integral must have the following kernel representation
\[
	\langle Tf, g\rangle =\iint K(x,y) f(y) g(x)\ud y\ud x.
\]

Define
$$
	\size_F(x,y) := \frac 1{ |x_1-y_1|^{d_1}(|x_1-y_1|+|x_2-y_2|)^{d_2}}.
$$
Let $(x,y) \in (\R^{d }\times \R^{d })\setminus \Delta$. First, we assume that $K$ satisfies the size estimate
\begin{equation*}
	|K(x, y)|\lesssim \size_F(x,y).
\end{equation*}
Let $x'=(x_1',x_2')$ be such that $|x_i'-x_i| \le |x_i-y_i|/2$ for $i=1,2$. We assume that $K$ satisfies the mixed size and H\"older estimates
\begin{equation*}
	\begin{split}
		 & |K((x'_1,x_2),y)-K(x,y)|
		\lesssim  \Big(\frac{|x_1'-x_1|}{|x_1-y_1|}\Big)^{\alpha_1}\size_F(x, y),
	\end{split}
\end{equation*}
\begin{equation*}
	\begin{split}
		 & |K((x_1,x_2'),y)-K(x,y)|
		\lesssim  \Big(\frac{|x_2'-x_2|}{  |x_2-y_2|}\Big)^{\alpha_2}\size_F(x,y),
	\end{split}
\end{equation*}
where $\alpha_m \in (0,1]$ for all $m=1,2$. Finally, we also require that  $K$ satisfies the H\"older estimate
\begin{equation*}
	\begin{split}
		|K(x',y) & -K((x_1',x_2),y)-K((x_1,x_2'),y)+K(x,y)|
		\\
		         & \qquad  \lesssim  \Big(\frac{|x_1'-x_1|}{|x_1-y_1|}\Big)^{\alpha_1}\Big(\frac{|x_2'-x_2|}{  |x_2-y_2|}\Big)^{\alpha_2}\size_F(x,y).
	\end{split}
\end{equation*}

As in the cancellative tri-parameter case, define the adjoint kernels $K^*$, $K^*_1$ and $K^*_2$ via the following relations:
\begin{align*}
	K^*(x,y)&=K(y,x), \\
	K^*_1(x,y)&=K((y_1,x_2),(x_1,y_2)), \\
	K^*_2(x,y)&=K((x_1,y_2),(y_1,x_2)).
\end{align*}
We assume that each adjoint kernel satisfies the same estimates as the kernel $K$.

\subsubsection{Partial kernel estimates}
We start by recalling the tri-parameter partial kernel estimates and rewriting them in the bi-parameter setting. First, we assume that
for every cube $I^1\subset \R^{d_1}$ there exists a kernel
$$
	K_{I^1} \colon (\R^{d_{2}} \times \R^{d_{2}}) \setminus \{(x_{2}, y_{2}) \colon x_2 = y_2\} \to \C,
$$
so that if $\supp{f_{2}}$ and   $\supp g_{2}$ are disjoint, then
\begin{equation}\label{eq biparam partial kernel indicator}
	\langle T(1_{I^1} \otimes f_{2}), 1_{I^1} \otimes g_{2}\rangle
	= \iint K_{I^1}(x_{2}, y_{2}) f_{2}(y_{2})g_{2}(x_{2}) \ud y_{2} \ud x_{2}.
\end{equation}
Define
$$
	\size_{I^1}(x_{2}, y_{2}) =\frac {|I^1|}{ |x_2-y_2|^{d_2}}.
$$
The kernel $K_{I^1}$ is assumed to satisfy the following estimates.
Let $x_{2}\ne y_{2}$. First, we assume the size estimate
\begin{equation}\label{E:eq62}
	|K_{I^1}(x_{2}, y_{2}) | \lesssim \size_{I^1}(x_{2}, y_{2}).
\end{equation}
Let $x'_{2}$ be such that $|x_2'-x_2| \le |x_2-y_2|/2$.
We assume the H\"older estimate
\begin{equation}\label{eq partial holder 2}
	\begin{split}
		|K_{I^1}(x_2',y_{2})-K_{I^1}(x_{2},y_{2})|  \lesssim
		\Big( \frac{|x_2'-x_2|}{|x_2-y_2|}  \Big)^{\alpha_{2}} \size_{I^1}(x_{2}, y_{2}).
	\end{split}
\end{equation}
We assume that the adjoint kernels defined by $$
	K_{I^1}^*(x_{2}, y_{2})
	=K_{I^1}(y_{2}, x_{2})
$$ satisfy the same estimates.

Next, for every cube $I^2\subset \R^{d_2}$ we assume that there exists a kernel
$$
	K_{I^2} \colon (\R^{d_{1}} \times \R^{d_{1}}) \setminus \{x_1 = y_1 \} \to \C,
$$
so that if $\supp f_{1}$ and $\supp g_{1}$ are disjoint, then
\begin{equation*}
	\langle T(f_1\otimes 1_{I^2}), g_1\otimes 1_{I^2}\rangle
	= \iint K_{I^2}(x_{1}, y_{1}) f_{1}(y_{1})g_{1}(x_{1}) \ud y_{1} \ud x_{1}.
\end{equation*}
The kernel $K_{I^2}$ is assumed to satisfy the following estimates.
For $x_1 \not= y_1$, we assume the size estimate
\begin{equation*}
	\begin{split}
		|K_{I^2}(x_{1}, y_{1}) |
		 & \lesssim \frac {|I^2|^2}{ |x_1-y_1|^{d_1}(|x_1-y_1|+\ell(I^2))^{d_2}} \\
		 & =: \size_{I^2}(x_{1}, y_{1}).
	\end{split}
\end{equation*}
Let $x'_{1}$ be such that $|x_1'-x_1| \le |x_1-y_1|/2$. We also assume the H\"older estimate
\begin{equation*}
	\begin{split}
		|K_{I^2}(x_1',y_{1})-K_{I^2}(x_{1},y_{1})|   \lesssim
		\Big( \frac{|x_1'-x_1|}{|x_1-y_1|}  \Big)^{\alpha_{1}} \size_{I^2}(x_{1}, y_{1}).
	\end{split}
\end{equation*}
As before, we also assume the natural symmetric estimates (i.e., the adjoint kernels satisfy the same estimates).

Since we will not be assuming $T1 = 0$ style cancellation conditions in the bi-parameter theory,
we must also address partial kernel estimates associated to atoms,
rather than merely the case of indicator functions that was discussed above.
Let $a_{I^1}$ denote an $L^\infty$ atom supported on a cube $I^1$, i.e.,
$a_{I^1}=1_{I^1}a_{I^1}$, $\int a_{I^1}=0$ and $\|a_{I^1}\|_{L^{\infty}}\le 1$.
Abusing notation, we will again denote $K_{I^1}$ using the equation
\eqref{eq biparam partial kernel indicator}, but where $K_{I^1}$ is not the same as there,
instead being adapted to the atom $a_{I^1}$, that is
\begin{equation*}
	\langle T(a_{I^1} \otimes f_{2}), 1_{I^1} \otimes g_{2}\rangle
	= \iint K_{I^1}(x_{2}, y_{2}) f_{2}(y_{2})g_{2}(x_{2}) \ud y_{2} \ud x_{2}.
\end{equation*}
We require these partial kernels $K_{I^1}$ and their adjoints to satisfy \eqref{E:eq62} and \eqref{eq partial holder 2}.
Moreover, the partial kernels defined by
\begin{equation*}
	\langle T(1_{I^1} \otimes f_{2}), a_{I^1} \otimes g_{2}\rangle
	= \iint K_{I^1}(x_{2}, y_{2}) f_{2}(y_{2})g_{2}(x_{2}) \ud y_{2} \ud x_{2}.
\end{equation*}
must also satisfy these estimates as well.

Finally, we must also assume that the partial kernels $K_{I^2}$ exist
and satisfy the usual estimates when one of the indicators $1_{I^2}$
is replaced by an atom $a_{I^2}$.

\section{Flag \texorpdfstring{$\BMO$}{BMO} spaces}\label{sec:BMO}
We study these $\BMO$ spaces in the bi-parameter setting
$\R^d = \R^{d_1} \times \R^{d_2}$, where we prove the full $T1$ theorem
and give a direct multiresolution proof of the commutator estimate.

\subsection{Product flag $\BMO$ and necessity considerations}
We discuss a subtle point regarding the product $\BMO$ type assumptions that will
appear in our representation theorem for bi-parameter flag CZOs.
In particular, we discuss their necessity and how they compare to the
$\BMO$ assumptions in the standard pure bi-parameter setting.
We start by carefully defining all the $\BMO$ spaces
and discussing some known properties.

Let $A := (a_I)_{I \in \calD}$ be a sequence of scalars indexed by all the dyadic
rectangles $I \in \calD = \calD^1 \times \calD^2$.
For $\Omega \subset \R^d$ we define the following square functions
$$
	S_A(x) := \Big(\sum_{I \in \calD} |a_I|^2 \frac{1_I(x)}{|I|} \Big)^{1/2} \qquad \textup{and} \qquad
	S_{A, \Omega}(x) := \bigg(\sum_{ \substack{ I \in \calD \\ I \subset \Omega}}
	|a_I|^2 \frac{1_I(x)}{|I|} \bigg)^{1/2}.
$$
For $p \in (0,\infty)$ we define
$$
	\|A\|_{\BMO_{\pro}(p)} = \sup_{\Omega} \frac{1}{|\Omega|^{1/p}} \|S_{A, \Omega}\|_{L^p},
$$
where $\sup$ runs over $\Omega \subset \R^d$ with $0 < |\Omega| < \infty$
and satisfying that for all $x \in \Omega$ there exists $I \in \calD$
so that $x \in I \subset \Omega$.
We set $\|A\|_{\BMO_{\pro}} := \|A\|_{\BMO_{\pro}(2)}$ and note that the well-known
bi-parameter John-Nirenberg tells us that $\|A\|_{\BMO_{\pro}} \sim \|A\|_{\BMO_{\pro}(p)}$,
$0 < p < \infty$. It is also well-known that given another sequence $B = (b_I)$ we have
\begin{equation}\label{eq:h1bmoseq}
	\sum_{I \in \calD} |a_I| |b_I| \lesssim \|A\|_{\BMO_{\pro}} \|S_B\|_{L^1},
\end{equation}
which we refer to as the $H^1$--$\BMO$ duality estimate in this context.
In fact, we also conversely have that if $A$ is such that for all $B$ we have
$$
	\sum_{I \in \calD} |a_I| |b_I| \le C\|S_B\|_{L^1},
$$
then $\|A\|_{\BMO_\pro} \le C$. Therefore, we can define $\BMO$ in either equivalent way --
we mostly have use for the dual estimate and not the Carleson type definition.

If a sequence $A = (a_I)_{I \in \calF}$ is only indexed over a subset
$\calF \subset \calD$ (e.g., $\calF = \calD_F$) then we understand the condition $A \in \BMO_{\pro}$ via the requirement
that the sequence $(\wt a_I)_{I \in \calD} \in \BMO_\pro$, where $\wt a_I = a_I$ if $I \in \calF$
and $\wt a_I = 0$ otherwise.

Let now $b \in L^1_{\loc}(\R^d)$. We say that $b \in \BMO_{\pro}$ if, uniformly on all
dyadic lattices $\calD$,
the sequence $(\langle b, h_{I^1} \otimes h_{I^2}\rangle)_{I \in \calD} \in \BMO_{\pro}$ (rigorously
the sequence $(\langle b, h_{I^1}^{\eta_1} \otimes h_{I^2}^{\eta_2}\rangle)_{\substack{ I \in \calD \\ \eta_i \ne 0}}$),
	i.e.,
	\begin{align*}
		\|b\|_{\BMO_{\pro}(p)} & = \sup_{\calD} \sup_{\Omega} \frac{1}{|\Omega|^{1/p}} \bigg\| \Big(
		                                                                                       \sum_{\substack{ I \in \calD \\ I \subset \Omega }} |\langle b, h_{I^1} \otimes h_{I^2}\rangle|^2
		                                                                                       \frac{1_I}{|I|} \Big)^{1/2}\bigg\|_{L^p} \\
		                       & \sim \sup_{\calD} \sup_{\Omega} \frac{1}{|\Omega|^{1/p}} \bigg\| \Big(
		                                                                                          \sum_{\substack{ I \in \calD \\ I \subset \Omega }} |\Delta_{I^1}\Delta_{I^2} b|^2\Big)^{1/2}
		\bigg\|_{L^p} < \infty,
	\end{align*}
	where we used the equivalence of the Haar and martingale difference square functions in $L^p$,
$0 < p < \infty$ (this was discussed previously). Next, $b$ belongs to the flag product $\BMO$,
	denoted $b \in \BMO_{\pro, F}$, if, uniformly on all dyadic lattices $\calD$,
	the sequence $(\langle b, h_{I, F}\rangle)_{I \in \calD_F} \in
\BMO_{\pro}$. Therefore, we have
	\begin{align*}
		\|b\|_{\BMO_{\pro, F}(p)} & = \sup_{\calD} \sup_{\Omega} \frac{1}{|\Omega|^{1/p}} \bigg\| \Big(
		                                                                                          \sum_{\substack{ I \in \calD_F \\ I \subset \Omega }} |\langle b, h_{I, F}\rangle|^2
		                                                                                          \frac{1_I}{|I|} \Big)^{1/2}\bigg\|_{L^p}                  \\
		                          & \sim \sup_{\calD} \sup_{\Omega} \frac{1}{|\Omega|^{1/p}} \bigg\| \Big(
		                                                                                             \sum_{\substack{ I \in \calD_F \\ I \subset \Omega }} |\Delta_{I, F} b|^2\Big)^{1/2} \Big\|_{L^p}
	\end{align*}

	So the standard product $\BMO$ and the flag product $\BMO$ are not directly comparable --
	while it is true that $\calD_F \subset \calD$, the coefficients $\langle b, h_{I, F} \rangle$
	are not quite always the same than $\langle b, h_{I^1} \otimes h_{I^2}\rangle$. Indeed,
	recall that if $\ell(I^1) = \ell(I^2)$ we have $h_{I, F} = u_{I^1} \otimes h_{I^2}$,
	where $u_{I^1}$ can also be the non-cancellative Haar. If $T$ is an
$L^2$ bounded flag SIO, it is, in particular, an $L^2$ bounded standard bi-parameter SIO.
	Therefore, by the known theory using Journ\'e's covering theorem \cites{Jou1985, Journe1986}
	we know that $T1 \in \BMO_{\pro}$. But, in our representation theorem,
	we want to require that $T1 \in \BMO_{\pro, F}$ --
	is this a necessary condition as well? The answer is yes, but it is non-trivial.
	By the fact that $T1 \in \BMO_{\pro}$
	for $L^2$ bounded flag SIOs, we only need to check the following (we discuss
	this and other points more carefully after the theorem).
	\begin{thm}\label{thm:T1nec}
	Assume that $T$ is a bi-parameter flag SIO satisfying the cube testing condition
	\[
	\|1_Q T1_Q\|_{L^2} \lesssim |Q|^{1/2}
\]
for all cubes $Q$. Then we have
\[
	\sup_{Q \in \calD_=} |Q|^{-\frac 12}
	\Big(\sum_{\substack{I\in \calD_= \\ I\subset Q} }|\langle T1, h_{I^1}^0\otimes h_{I^2}\rangle|^2
	\Big)^{1/2}<\infty.
\]
\end{thm}
\begin{proof}
	Notice first that for $Q \in \calD_=$ we have
	\begin{align*}
		\Big(\sum_{\substack{I\in \calD_= \\ I\subset Q} }
		|\langle T1_{3Q}, h_{I^1}^0\otimes h_{I^2}\rangle|^2 \Big)^{1/2}
		 & \le
		\Big(\sum_{I\in \calD_= }
		|I| \langle |E_{I^1} \Delta_{I^2}(1_{3Q}T1_{3Q})| \rangle_I^2 \Big)^{1/2}   \\
		 & \le \Big(\sum_{I\in \calD_=}
		\| E_{I^1} \Delta_{I^2}(1_{3Q}T1_{3Q})\|_{L^2}^2  \Big)^{1/2}               \\
		 & \lesssim \|1_{3Q} T1_{3Q}\|_{L^2} \lesssim |3Q|^{1/2} \lesssim |Q|^{1/2}
	\end{align*}
	by the assumed testing condition and a one-parameter square function estimate.
	Therefore, we only need to estimate
	$$
		\Big(\sum_{\substack{I\in \calD_= \\ I\subset Q} }|\langle T1_{(3Q)^c}, h_{I^1}^0\otimes h_{I^2}\rangle|^2
		\Big)^{1/2}
		\le S_{c1} + S_{c2} + S_{c12},
	$$
	where ``$c$'' stands for complement and, e.g., $c1$ means complement in the first parameter,
	and we are referring to the division $(3Q)^c = ((3Q^1)^c \times 3Q^2) \cup (3Q^1 \times (3Q^2)^c)
		\cup ((3Q^1)^c \times (3Q^2)^c)$. Thus, for instance, we mean
	$$
		S_{c1} :=
		\Big(\sum_{\substack{I\in \calD_= \\ I\subset Q} }|\langle T(1_{(3Q^1)^c} \otimes 1_{3Q^2}), h_{I^1}^0\otimes h_{I^2}\rangle|^2
		\Big)^{1/2}.
	$$
	We deal with the above $S_{c1}$ term first and begin by estimating
	\begin{equation}\label{eq:S1csplit}
		\begin{split}
			|\langle T(1_{(3Q^1)^c} & \otimes 1_{3Q^2} ), h_{I^1}^0\otimes h_{I^2}\rangle|                                           \\
			                        & \le |\langle T(1_{(3Q^1)^c} \otimes 1_{I^2} ), h_{I^1}^0\otimes h_{I^2}\rangle|
			+|\langle T(1_{(3Q^1)^c} \otimes 1_{(3I^2\setminus I^2)} ), h_{I^1}^0\otimes h_{I^2}\rangle|                             \\
			                        & +|\langle T(1_{(3Q^1)^c} \otimes 1_{(3Q^2\setminus 3I^2)} ), h_{I^1}^0\otimes h_{I^2}\rangle|.
		\end{split}
	\end{equation}
	For the second term we estimate using the size estimate
	(throwing away the factor $|x_2-y_2| \lesssim \ell(I) \le \ell(Q) \lesssim |x_1-y_1|$
	and using the easy estimate $\int_{w \in \R^m\colon |w-w_0| > c} |w-w_0|^{-m-\eps} \lesssim c^{-\eps}$
	) that
	\begin{align*}
		|\langle T(1_{(3Q^1)^c} & \otimes 1_{(3I^2\setminus I^2)} ), h_{I^1}^0\otimes h_{I^2}\rangle|                                                \\
		                        & \lesssim
		|I|^{-1/2}\int_{I}\int_{(3Q^1)^c\times (3I^2\setminus I^2)} \frac{\ud y \ud x}{|x_1-y_1|^{d_1+d_2}}                                          \\
		                        & \lesssim |I|^{-1/2} |I^2|^2 \int_{I^1} \int_{(3Q^1)^c} \frac{\ud y_1 \ud x_1}{|x_1-y_1|^{d_1+d_2}}                 \\
		                        & \lesssim |I|^{-1/2} |I^2|^2 \int_{I^1} \ell(Q)^{-d_2} \ud x_1 = |I|^{1/2} \Big(\frac{\ell(I)}{\ell(Q)}\Big)^{d_2}.
	\end{align*}
	Next, for the first term in \eqref{eq:S1csplit} we estimate
	$$
		|\langle T(1_{(3Q^1)^c} \otimes 1_{I^2} ), h_{I^1}^0\otimes h_{I^2}\rangle|
		\le |I|^{-1/2} \sum_{L^2, P^2 \in \ch(I^2)}
		|\langle T(1_{(3Q^1)^c} \otimes 1_{P^2}), 1_{I^1}\otimes 1_{L^2}\rangle|.
	$$
	For $L^2 \ne P^2$ this is exactly the same estimate as in the above case.
	For $P^2 = L^2$ we have by the size estimate of the partial kernel that
	$$
		|\langle T(1_{(3Q^1)^c} \otimes 1_{P^2}), 1_{I^1}\otimes 1_{P^2}\rangle|
		\lesssim |I^2|^2 \int_{I^1} \int_{(3Q^1)^c} \frac{\ud y_1 \ud x_1}{|x_1-y_1|^{d_1+d_2}}
		\lesssim |I| \Big(\frac{\ell(I)}{\ell(Q)}\Big)^{d_2}
	$$
	and so we again have
	$$
		|\langle T(1_{(3Q^1)^c} \otimes 1_{I^2} ), h_{I^1}^0\otimes h_{I^2}\rangle|
		\lesssim |I|^{1/2} \Big(\frac{\ell(I)}{\ell(Q)}\Big)^{d_2}.
	$$
	For the last term in \eqref{eq:S1csplit} we use the mixed H\"older and size estimate
	of the full kernel to get
	\begin{align*}
		|\langle & T(1_{(3Q^1)^c} \otimes 1_{(3Q^2\setminus 3I^2)} ), h_{I^1}^0\otimes h_{I^2}\rangle|                                                                                         \\
		         & \lesssim  |I|^{-1/2} \int_{I}\int_{(3Q^1)^c\times (3Q^2\setminus 3I^2) }\frac{\ell(I)^{\alpha_2}}{|x_1-y_1|^{d_1}(|x_1-y_1|+|x_2-y_2|)^{d_2}|x_2-y_2|^{\alpha_2}}\ud y\ud x \\
		         & \lesssim |I|^{-1/2} \ell(I)^{\alpha_2/2}
		\int_{I^1} \int_{(3Q^1)^c} \frac{\ud y_1 \ud x_1}{|x_1-y_1|^{d_1+d_2}}
		\times \int_{I^2} \int_{3Q^2 \setminus 3I^2} \frac{\ud y_2 \ud x_2}{|x_2-y_2|^{\alpha_2/2}},
	\end{align*}
	where we did the $\alpha_2/2$ business only for the reason that if $d_2 = 1$ it could
	theoretically be that $\alpha_2 = 1 = d_2$, which we prefer not to have to deal with in the next integral
	(now $\alpha_2/2 \le 1/2 < d_2$ for sure). Next, we have by polar coordinates (or divide to dyadic
	annuli) that
	\begin{align*}
		\int_{y_2\colon |y_2-x_2| \lesssim \ell(Q)}
		\frac{\ud y_2}{|x_2-y_2|^{\alpha_2/2}}
		\sim \int_{0}^{c \ell(Q)} r^{d_2 - \alpha_2/2 - 1} \ud r
		\sim \ell(Q)^{d_2-\alpha_2/2}.
	\end{align*}
	We also already calculated above the $I^1 \times (3Q^1)^c$ integral -- thus, in total we have
	\begin{align*}
		|\langle & T(1_{(3Q^1)^c} \otimes 1_{(3Q^2\setminus 3I^2)} ), h_{I^1}^0\otimes h_{I^2}\rangle|                          \\
		         & \lesssim |I|^{-1/2} \ell(Q)^{-d_2}|I^1| \Big( \frac{\ell(I)}{\ell(Q)} \Big)^{\alpha_2/2} \ell(Q)^{d_2} |I^2|
		= |I|^{1/2}\Big(\frac{\ell(I)}{\ell(Q)} \Big)^{\alpha_2/2}.
	\end{align*}
	Now, combining the above, we get
	$$
		S_{c1}^2 \le \sum_{k=0}^{\infty} 2^{-\alpha_2 k} \sum_{I^{(k)} = Q} |I| = |Q| \sum_{k=0}^{\infty}
		2^{-\alpha_2 k} \lesssim |Q|.
	$$

	Next, we move on study
	$$
		S_{c2} :=
		\Big(\sum_{\substack{I\in \calD_= \\ I\subset Q} }|\langle T(1_{3Q^1} \otimes 1_{(3Q^2)^c}, h_{I^1}^0\otimes h_{I^2}\rangle|^2
		\Big)^{1/2}.
	$$
	We estimate
	\begin{align*}
		|\langle T(1_{3Q^1\times (3Q^2)^c} ), h_{I^1}^0\otimes h_{I^2}\rangle|
		 & \le|\langle T(1_{  I^1\times (3Q^2)^c} ), h_{I^1}^0\otimes h_{I^2}\rangle|              \\
		 & +|\langle T(1_{(3Q^1\setminus I^1)\times (3Q^2)^c} ), h_{I^1}^0\otimes h_{I^2}\rangle|.
	\end{align*}
	By the partial kernel estimates we have
	\begin{align*}
		|\langle T(1_{  I^1\times (3Q^2)^c} ), h_{I^1}^0\otimes h_{I^2}\rangle|
		 & \lesssim \int_{I^2}\int_{ (3Q^2)^c } \frac{|I^1|^{\frac 12}|I^2|^{-\frac 12}\ell(I)^{\alpha_2}}{|x_2-y_2|^{d_2+\alpha_2}}\ud y_2 \ud x_2
		\le |I|^{1/2}\Big( \frac{\ell(I)}{\ell(Q)}\Big)^{\alpha_2}.
	\end{align*}
	By the full kernel estimates there holds that
	\begin{align*}
		 & |\langle T(1_{(3Q^1\setminus I^1)\times (3Q^2)^c} ), h_{I^1}^0\otimes h_{I^2}\rangle|                                                                                 \\
		 & \lesssim \int_{I} \int_{(3Q^1\setminus I^1)\times (3Q^2)^c}\frac{|I|^{-\frac 12}\ell(I)^{\alpha_2} }{|x_1-y_1|^{d_1} (|x_1-y_1|+|x_2-y_2|)^{d_2}|x_2-y_2|^{\alpha_2}}
		\ud y\ud x                                                                                                                                                               \\
		 & \le |I|^{-1/2} \ell(I)^{\alpha_2} \int_{I^1} \int_{3Q^1 \setminus I^1} \frac{\ud y_1 \ud x_1}{|x_1-y_1|^{d_1}}
		\times \int_{I^2} \int_{(3Q^2)^c} \frac{\ud y_2 \ud x_2}{|x_2-y_2|^{d_2 + \alpha_2}}                                                                                     \\
		 & \lesssim |I|^{-1/2} |I^2| \Big( \frac{\ell(I)}{\ell(Q)} \Big)^{\alpha_2}
		\int_{I^1} \int_{3Q^1 \setminus I^1} \frac{\ud y_1 \ud x_1}{|x_1-y_1|^{d_1}}.
	\end{align*}
	Next, notice that using the standard ``nearby'' integral estimate \eqref{eq:nearbyint} and
	polar coordinates we get
	\begin{align*}
		\int_{I^1} \int_{3Q^1 \setminus I^1} \frac{\ud y_1 \ud x_1}{|x_1-y_1|^{d_1}}
		 & \le \int_{I^1} \int_{3I^1 \setminus I^1}\frac{\ud y_1 \ud x_1}{|x_1-y_1|^{d_1}}
		+ \int_{I^1} \int_{3Q^1 \setminus 3I^1}\frac{\ud y_1 \ud x_1}{|x_1-y_1|^{d_1}}               \\
		 & \lesssim |I^1| + \int_{I^1} \int_{c_1 \ell(I)}^{c_2 \ell(Q)} r^{-d_1+d_1-1} \ud r \ud x_1
		\lesssim \Big(1 + \log \frac{\ell(Q)}{\ell(I)} \Big)|I^1|.
	\end{align*}
	Therefore, we have proved
	$$
		|\langle T(1_{(3Q^1\setminus I^1)\times (3Q^2)^c} ), h_{I^1}^0\otimes h_{I^2}\rangle|
		\lesssim |I|^{1/2}
		\Big(1 + \log \frac{\ell(Q)}{\ell(I)} \Big) \Big( \frac{\ell(I)}{\ell(Q)} \Big)^{\alpha_2}.
	$$
	We now get
	\begin{align*}
		S_{c2}^2
		= \sum_{k=0}^\infty\sum_{I^{(k)}= Q} |\langle T(1_{3Q^1} \otimes 1_{(3Q^2)^c} ),
		h_{I^1}^0\otimes h_{I^2}\rangle|^2
		\lesssim  \sum_{k=0}^\infty (1+k)^2 2^{-2\alpha_2 k} |Q|\lesssim |Q|.
	\end{align*}

	We move on to the final term
	$$
		S_{c12} :=
		\Big(\sum_{\substack{I\in \calD_= \\ I\subset Q} }|\langle T(1_{(3Q^1)^c} \otimes 1_{(3Q^2)^c}, h_{I^1}^0\otimes h_{I^2}\rangle|^2
		\Big)^{1/2}.
	$$
	Again, by the full kernel estimates we have
	\begin{align*}
		 & |\langle T(1_{(3Q^1)^c\times (3Q^2)^c} ), h_{I^1}^0\otimes h_{I^2}\rangle|                                                                            \\
		 & \le \int_{I} \int_{(3Q^1)^c\times (3Q^2)^c}\frac{|I|^{-\frac 12}\ell(I)^{\alpha_2} }{|x_1-y_1|^{d_1} (|x_1-y_1|+|x_2-y_2|)^{d_2}|x_2-y_2|^{\alpha_2}}
		\ud y\ud x                                                                                                                                               \\
		 & \le |I|^{-1/2} \ell(I)^{\alpha_2} \int_{I^2} \int_{(3Q^2)^c} \frac{1}{|x_2-y_2|^{d_2+\alpha_2}} \int_{I^1}
		\int_{\substack{ y^1 \in (3Q^1)^c \\ |x_1-y_1| \le |x_2-y_2|}}
		\frac{\ud y_1}{|x_1-y_1|^{d_1}} \ud x_1 \ud y_2 \ud x_2                                                                                                  \\
		 & + |I|^{-1/2} \ell(I)^{\alpha_2} \int_{I^2} \int_{(3Q^2)^c} \frac{1}{|x_2-y_2|^{\alpha_2}} \int_{I^1}
		\int_{y^1 \colon \, |x_1-y_1| > |x_2-y_2|}
		\frac{\ud y_1}{|x_1-y_1|^{d_1 + d_2}} \ud x_1 \ud y_2 \ud x_2                                                                                            \\
		 & =: A_{\le} + A_>.
	\end{align*}
	Now, we have
	\begin{align*}
		A_{\le} & \lesssim |I|^{-1/2} \ell(I)^{\alpha_2} \int_{I^2} \int_{(3Q^2)^c} \frac{1}{|x_2-y_2|^{d_2+\alpha_2}}
		\int_{I^1} \int_{c_1 \ell(Q)}^{|x_2-y_2|} r^{-1} \ud r \ud x_1 \ud y_2 \ud x_2                                                     \\
		        & \lesssim |I|^{-1/2} |I^1| \ell(I)^{\alpha_2} \int_{I^2} \int_{|y_2-x_2| \ge c\ell(Q)} \frac{1}{|y_2-x_2|^{d_2+\alpha_2}}
		\Big(C+\log \frac{|y_2-x_2|}{\ell(Q)}\Big) \ud y_2 \ud x_2                                                                         \\
		        & \lesssim |I|^{1/2} \ell(I)^{\alpha_2}
		\sum_{k=0}^{\infty}  (2^k\ell(Q))^{-d_2-\alpha_2} (k+1) (2^k\ell(Q))^{d_2}                                                         \\
		        & = |I|^{1/2} \Big( \frac{\ell(I)}{\ell(Q)}\Big)^{\alpha_2} \sum_{k=0}^{\infty} 2^{-\alpha_2 k} (k+1)
		\lesssim |I|^{1/2} \Big( \frac{\ell(I)}{\ell(Q)}\Big)^{\alpha_2}
	\end{align*}
	and
	\begin{align*}
		A_> \lesssim |I|^{-1/2}|I^1| \ell(I)^{\alpha_2} \int_{I^2} \int_{y^2\colon \, |x_2-y_2| \ge c\ell(Q)} \frac{\ud y_2 \ud x_2}{|x_2-y_2|^{d_2 + \alpha_2}}
		\lesssim |I|^{1/2} \Big( \frac{\ell(I)}{\ell(Q)}\Big)^{\alpha_2}.
	\end{align*}
	So, finally, we also have
	$$
		S_{c12}^2 \lesssim \sum_{k=0}^{\infty} 2^{-2\alpha_2 k} |Q| \lesssim |Q|.
	$$
	We are done.
\end{proof}
\begin{rem}
	The estimate also holds with $h_{I^1}$ in place of $h_{I^1}^0$ by the same proof,
	using the output-atom partial kernel in the $S_{c2}$ estimate.
	Recall that a cancellative one-parameter Haar $h_{I^1 \times I^2}$ is of one of the types
	$h_{I^1} \otimes h_{I^2}$, $h_{I^1}^0 \otimes h_{I^2}$ or $h_{I^1} \otimes h_{I^2}^0$. So the above is somehow saying
	that $T1$ satisfies some partial one-parameter $\BMO$ condition. Notice that we really
	appeared to exploit that the non-cancellation was in the first parameter and not
	the other way around. Notice also that the previous theorem implicitly tells
	us that even these non-standard pairings $\langle T1, h_{I^1}^0 \otimes h_{I^2}\rangle$
	can be meaningfully \emph{defined} using the flag kernel estimates for the non-local parts.
\end{rem}
\begin{rem}
	Notice that we had the $\sup$ running over $\calD_=$ in the theorem, and not over
	the more general $\Omega$ as in the definition of $\BMO_{\pro, F}$. On
	this one-parameter boundary $\ell(I^1) = \ell(I^2)$ these are the same thing by
	the following standard argument.
	Let
	$$
		C :=   \sup_{Q \in \calD_=} |Q|^{-\frac 12}
		\Big(\sum_{\substack{I\in \calD_= \\ I\subset Q} }|\langle T1, h_{I^1}^0\otimes h_{I^2}\rangle|^2
		\Big)^{1/2}.
	$$
	Let $\Omega \subset \R^d$ and
	choose the maximal $Q \in \calD_=$ that are contained in $\Omega$ --
	denote the collection of such $Q$ by $\calS$. Then
	\begin{align*}
		\Big(\sum_{\substack{I\in \calD_= \\ I\subset \Omega} }|\langle T1, h_{I^1}^0\otimes h_{I^2}\rangle|^2
		\Big)^{1/2} & =
		\Big(\sum_{Q \in \calS} \sum_{\substack{I\in \calD_= \\ I\subset Q} }|\langle T1, h_{I^1}^0\otimes h_{I^2}\rangle|^2
		\Big)^{1/2}                                                         \\
		\le
		\Big(\sum_{Q \in \calS}  C^2 |Q| \Big)^{1/2}
		                                                                                                       & = C \Big| \bigcup_{Q \in \calS} Q \Big|^{1/2} \le C|\Omega|^{1/2}.
	\end{align*}
\end{rem}

So Theorem \ref{thm:T1nec} combined with the necessity of $T1 \in \BMO_{\pro}$
(and the above remark)
tells us that it is reasonable to assume $T1 \in \BMO_{\pro, F}$ in our $T1$ style
representation theorem (since the condition is necessary for $L^2$ boundedness). Assuming it then
allows us to derive as a direct
corollary that our operators are weighted bounded with the optimal flag weights -- a better
result than in the pure bi-parameter setting (where strong bi-parameter weights are required).
But even if the conclusion is better and the assumption is reasonable (necessary), it might
bother the reader that this flag adapted condition $T1 \in \BMO_{\pro, F}$ is not strictly weaker than the $T1 \in \BMO_{\pro}$
that appears in the classical bi-parameter $T1$ \cite{Jou1985} or the bi-parameter representation theorem \cite{Ma1}.
Indeed, generally speaking the philosophy is that only the kernel estimates are stronger, and this allows
us to have the other assumptions to be flag adapted in a way that is weaker than the corresponding pure bi-parameter
assumptions (and, of course, to also reach the stronger conclusion of boundedness with flag weights).
But the previous theorem, in fact, tells us that $T1 \in \BMO_{\pro, F}$ follows if $T$ is a bi-parameter flag
SIO satisfying $T1 \in \BMO_{\pro, <}$ (we define this below) and the rather weak, one-parameter style, cube testing condition
$\|1_QT1_Q\|_{L^2} \lesssim |Q|^{1/2}$ (this testing condition already appears, e.g., in many local formulations
of \textbf{one-parameter} $T1$). Here $T1 \in \BMO_{\pro, <}$ is just defined using the grid $\calD_<$
and the coefficients $\langle T1, h_{I, F}\rangle = \langle T1, h_{I^1} \otimes h_{I^2}\rangle$
so this condition is weaker than $T1 \in \BMO_{\pro}$ (that would require the same Carleson condition
but in the whole $\calD$ instead of $\calD_<$). We feel that this makes the assumption
$T1 \in \BMO_{\pro, F}$ completely justified -- on top of the other important points discussed above,
it is also weaker, at least for flag SIOs with very mild one-parameter style testing, than $T1 \in \BMO_{\pro}$.

Finally, for added clarity to non-experts, we discuss a slightly different formulation of \eqref{eq:h1bmoseq} for functions.
So, let $b$, for instance, be in the dyadic product flag $\BMO$, $b \in \BMO_{\pro, F}(\calD)$ in the sense that
\begin{align*}
	\sum_{\substack{ I \in \calD_F \\ I \subset \Omega }} |\langle b, h_{I, F}\rangle|^2 \lesssim |\Omega|
\end{align*}
or, equivalently,
$$
	\sum_{I \in \calD_F} |\langle b, h_{I, F}\rangle| |c_I| \lesssim \Big\|
	\Big(\sum_{I \in \calD_F} |c_I|^2 \frac{1_I}{|I|} \Big)^{1/2}\Big\|_{L^1}.
$$
Notice now that if $f$ is a (reasonable) function and we write
$f = \sum_{I \in \calD_F} \langle f, h_{I, F}\rangle h_{I, F}$
then
$$
	|\langle b, f\rangle| = \Big| \sum_{I \in \calD_F} \langle b, h_{I, F} \rangle \langle f, h_{I, F}\rangle\Big|
	\lesssim \|S_{\calD_F} f\|_{L^1}.
$$
Conversely, if we assume that $\langle b, f\rangle$ satisfies the above estimate for all reasonable $f$, then
given $c_I$ we can choose $f$ so that $\langle f, h_{I, F}\rangle = \sign( \langle b, h_{I, F}\rangle c_I)c_I$
so that
$$
	\sum_{I \in \calD_F} |\langle b, h_{I, F}\rangle| |c_I|
	= |\langle b, f\rangle| \lesssim
	\Big\|\Big(\sum_{I \in \calD_F} |c_I|^2 \frac{1_I}{|I|} \Big)^{1/2}\Big\|_{L^1}.
$$
Therefore, flag product $\BMO$ can equivalently be formulated by demanding that
$$
	|\langle b, f\rangle| \lesssim \|S_{\calD_F} f\|_{L^1}
$$
uniformly in all dyadic grids $\calD$.
Notice that this argument required that $h_{I ,F}$, $I \in \calD_F$, forms a base.
So for $\calD_{\pro, <}$ we would have to be careful, since not all
functions could be represented using just $h_{I, F}$, $I \in \calD_<$.

\subsection{Little flag $\BMO$ and flag product $\BMO$}
We define that $b \in \bmo_F(\R^d)$ if $b$ is locally integrable and
$$
	\|b\|_{\bmo_F(\R^d)} := \sup_{R \in \calR_F} \frac{1}{|R|} \int_R |b-\langle b \rangle_R| < \infty.
$$
This is the ``little'' flag $\BMO$ space.
We define that $b \in \BMO^{(2)}(\R^d)$ if
$$
	\|b\|_{\BMO^{(2)}(\R^d)} := \esssup_{x_1 \in \R^{d_1}} \|b(x_1, \cdot)\|_{\BMO(\R^{d_2})} < \infty,
$$
where $\BMO(\R^{d_2})$ denotes the classical one-parameter $\BMO$ in $\R^{d_2}$. It turns out
that, see \cite{DLOPW-FLAG}, there is the easy description
$$
	\bmo_F(\R^d) = \BMO(\R^d) \cap \BMO^{(2)}(\R^d).
$$
The equivalent $H^1$-$\BMO$ style dual estimates are by far the most useful
for us. In particular, being in $\bmo_F(\R^d)$ is then equivalent to satisfying,
uniformly for all dyadic grids, the following two different one-parameter square function estimates:
\begin{equation}\label{eq:bmo1}
	|\langle b, f\rangle| \lesssim \Big\|\Big(\sum_{I \in \calD_{=}} |\Delta_I f|^2 \Big)^{1/2}\Big\|_{L^1}
\end{equation}
and
\begin{equation}\label{eq:bmo2}
	|\langle b, f\rangle| \lesssim \Big\|\Big(\sum_{I^2 \in \calD^2} |\Delta_{I^2} f|^2 \Big)^{1/2}\Big\|_{L^1}.
\end{equation}
This little flag $\BMO$ space is only relevant for us in the context of commutators.
The product flag $\BMO$ is more relevant in the representation theorem
context. However, their mutual relationship is also of interest in the commutator setting, since
different $\BMO$ conditions are useful for different paraproducts.

The inclusions of various BMO spaces are only clear on the dual side.
With this idea we get the following key inclusion.
\begin{thm}\label{thm:bmoembedding}
	We have that
	$$
		\BMO^{(2)} \subset \BMO_{\pro, F}.
	$$
	In particular, we have
	$$
		\bmo_F \subset \BMO_{\pro, F}.
	$$
\end{thm}
\begin{proof}
	Using \eqref{eq:bmo2} we have that
	$$
		|\langle b, f\rangle| \lesssim \Big\|\Big(\sum_{I \in \calD^2} |\Delta_{I^2} f|^2 \Big)^{1/2}\Big\|_{L^1}.
	$$
	We will apply the vector-valued lower flag square function estimate
	$$
		\Big\| \Big( \sum_i |g_i|^2 \Big)^{1/2}\Big\|_{L^q}
		\lesssim \Big\| \Big( \sum_i |S_{\calD_F} g_i|^2 \Big)^{1/2}\Big\|_{L^q}, \qquad 0 < q < \infty,
	$$
	to the above one-parameter square function estimate.
	The weighted estimate that we proved for flag square functions
	implies this vector-valued version by extrapolation.
	Now, we get
	$$
		|\langle b, f\rangle| \lesssim \Big\|\Big(\sum_{I \in \calD^2} |S_{\calD_F}\Delta_{I^2} f|^2 \Big)^{1/2}\Big\|_{L^1}
		= \|S_{\calD_F} f\|_{L^1}.
	$$
	This proves the claimed inclusion.
\end{proof}

\section{Flag Calder\'on--Zygmund operators}
\subsection{Cancellative tri-parameter flag CZOs}\label{sec:triparCZO}
So far, we have defined cancellative SIOs in the tri-parameter setting by stating various desirable full and partial kernel estimates. In this section, we will formulate additional conditions on the flag SIO $T$ itself;
SIOs satisfying the subsequent cancellation assumptions and weak boundedness property will be called cancellative tri-parameter flag Calder\'on--Zygmund operators
(cancellative flag CZOs).
\subsubsection{Cancellation assumptions}
We assume that $\calD = \calD^1 \times \calD^2 \times \calD^3$
for arbitrary dyadic lattices $\calD^i$ in $\R^{d_i}$ and
$I, J \in \calD_F$ with $\ell(I) = \ell(J)$.
We assume that
\begin{align*}
	\langle T(1 \otimes 1_{I^2} \otimes 1_{I^3}), h_{I^1} \otimes 1_{J^2} \otimes 1_{J^3}\rangle
	=  \langle T(h_{I^1} \otimes 1_{I^2} \otimes 1_{I^3}), 1 \otimes 1_{J^2} \otimes 1_{J^3}\rangle
	= 0,
\end{align*}
and that the four symmetric terms, where there is a ``$1$'' in parameters $i=2$ or $i=3$
instead, also vanish. The pairings appearing in the cancellation assumptions are easy to define:
\begin{align*}
	\langle T(1 \otimes & 1_{I^2} \otimes 1_{I^3}), h_{I^1} \otimes 1_{J^2} \otimes 1_{J^3}\rangle                               \\
	                    & := \langle T(1_{3I^1} \otimes 1_{I^2} \otimes 1_{I^3}), h_{I^1} \otimes 1_{J^2} \otimes 1_{J^3}\rangle \\
	                    &\quad + \langle T(1_{(3I^1)^c} \otimes 1_{I^2} \otimes 1_{I^3}), h_{I^1} \otimes 1_{J^2} \otimes 1_{J^3}\rangle.
\end{align*}
We assume that $T$ is such that the first term a priori makes sense (the bilinear form
of $T$ is given to be well-defined for indicators of rectangles). The second term can be \textbf{defined}
using natural kernel representations (either full or some partial ones depending on whether
$I^i = J^i$ or $I^i \ne J^i$ for $i=2,3$). Various mixed H\"older and size estimates
guarantee that such integrals are absolutely convergent using the fact that $\int h_{I^1} = 0$.

\subsubsection{Weak boundedness property}
The flag weak boundedness property (WBP) requires that, for every flag rectangle $I$,
\begin{equation*}
	| \langle T1_I, 1_I \rangle| \lesssim |I|.
\end{equation*}

\begin{defn}\label{defn:flagCZO}
	We say that a linear operator $T$ (such that its bilinear form is at least defined
	for indicators of rectangles) is a (cancellative) flag type Calder\'on-Zygmund operator
	if $T$ has the full kernel representation,
	the partial kernel representations (i.e., it is a cancellative flag SIO) and, in addition,
	satisfies the cancellation assumptions
	and the weak boundedness property.
\end{defn}

\subsubsection{Concrete examples}
The convolution type operator $T_{NRS}$ defined on $\R\times \R \times \R$ with kernel
\[
	K_{NRS}(x_1, x_2, x_3)= \frac{\sign (x_2) \sign (x_3)}{x_1 \sqrt{x_1^2+x_2^2}\sqrt{x_1^2+x_2^2+x_3^2}}
\]
was introduced in \cite{NRS2001} by Nagel, Ricci and Stein.

It is straightforward that $K_{NRS}$ satisfies all the full kernel estimates.
On the other hand, the kernel is antisymmetric in each variable
so that all the partial  kernels have to be zero and also that $\langle T_{NRS}(1_I), 1_I\rangle=0$ for all rectangles $I$. To conclude that $T_{NRS}$ belongs to our class we
show the following cancellation property.

\begin{prop}\label{prop:nrs-cancellation}
	The operator $T_{NRS}$ satisfies that
	\begin{equation*}
		\langle T_{NRS}(1\otimes 1_{I^{23}}), h_{I^1}\otimes 1_{J^{23}}\rangle =0,
	\end{equation*}
	where the $I^i$ and $J^i$ are arbitrary intervals. The symmetrical equalities (i.e. when $T$ is replaced by its duals or the pair $(1, h_{I^1})$ is in other slots) also hold.
\end{prop}
\begin{proof}
	First of all, since the kernel is antisymmetric in each variable, we may assume that $I^2$ and $J^2$, and also $I^3$ and $J^3$, are disjoint.
	Otherwise, we may, e.g., write
	\[
		1_{I^2}=1_{I^2 \cap J^2}+ 1_{I^2\setminus J^2},\quad 1_{J^2}=1_{I^2 \cap J^2}+ 1_{J^2\setminus I^2}.
	\]
	Then
	\[
		\langle T_{NRS}(1\otimes 1_{I^2 \cap J^2} \otimes 1_{I^{3}}), h_{I^1}\otimes 1_{I^2 \cap J^2} \otimes 1_{J^{3}}\rangle=0.
	\]
	As usual, write
	\begin{align*}
		\langle T_{NRS}(1\otimes 1_{I^{23}}), h_{I^1}\otimes 1_{J^{23}}\rangle & = \langle T_{NRS}(1_{\R\setminus 3I^1}\otimes 1_{I^{23}}), h_{I^1}\otimes 1_{J^{23}}\rangle   \\
		                                                                       & \hspace{2cm}+ \langle T_{NRS}(1_{ 3I^1}\otimes 1_{I^{23}}), h_{I^1}\otimes 1_{J^{23}}\rangle.
	\end{align*}
	Here the second term is already defined. Using the cancellation of $h_{I^1}$ the first term can be interpreted as
	\begin{equation*}
		\iint \big(K(x-y)-K((c_{I^1}, x_2, x_3)- y)\big) 1_{\R\setminus 3I^1}\otimes 1_{I^{23}} (y) h_{I^1}\otimes 1_{J^{23}}(x)\ud y\ud x,
	\end{equation*}
	where using the full kernel estimates the integral is absolutely convergent. Consider the set
	\[
		J_A^1:= AI^1 \cap (AI^1+\frac{|I^1|}{2}),
	\]
	where $A$ is sufficiently large (for instance $A\ge 100$). Then it is easy to see that we should have $3I^1\subset J_A^1$.
	Hence, we have
	\begin{align*}
		\langle T_{NRS}(1_{J_A^1}\otimes 1_{I^{23}}), h_{I^1}\otimes 1_{J^{23}}\rangle & = \langle T_{NRS}(1_{J_A^1\setminus 3I^1}\otimes 1_{I^{23}}), h_{I^1}\otimes 1_{J^{23}}\rangle \\
		                                                                               & \hspace{2cm}+ \langle T_{NRS}(1_{ 3I^1}\otimes 1_{I^{23}}), h_{I^1}\otimes 1_{J^{23}}\rangle.
	\end{align*}
	\[
		\langle T_{NRS}(1\otimes 1_{I^{23}}), h_{I^1}\otimes 1_{J^{23}}\rangle=\lim_{A\to \infty} \langle T_{NRS}(1_{J_A^1}\otimes 1_{I^{23}}), h_{I^1}\otimes 1_{J^{23}}\rangle.
	\]
	Now, the key observation is that $T_{NRS}$ is translation invariant (since it is a convolution operator), i.e.,
	\[
		\langle T(1_I), 1_J\rangle = \langle T(1_{I+t}), 1_{J+t}\rangle,\quad \forall t\in \R^3.
	\]
	Let $I_l^1$ and $I_r^1$ be the left and right halves of $I^1$ and $t=(|I^1|/2, 0,0)$. We have
	\begin{align*}
		\langle T(1_{AI^1}\otimes 1_{I^{23}}), 1_{I_l^1}\otimes 1_{J^{23}}\rangle = \langle T(1_{AI^1+|I^1|/2}\otimes 1_{I^{23}}), 1_{I_r^1}\otimes 1_{J^{23}}\rangle
	\end{align*}
	so that
	\begin{align*}
		\langle T_{NRS}(1_{J_A^1}\otimes 1_{I^{23}}), h_{I^1}\otimes 1_{J^{23}}\rangle
		 & = |I^1|^{-\frac 12}\langle T_{NRS}(1_{J_A^1}\otimes 1_{I^{23}}), (1_{I_l^1}-1_{I_r^1})\otimes 1_{J^{23}}\rangle               \\
		 & =  |I^1|^{-\frac 12}\langle T_{NRS}(1_{(AI^1+|I^1|/2)\setminus J_A^1}\otimes 1_{I^{23}}),  1_{I_r^1}\otimes 1_{J^{23}}\rangle \\
		 & \quad-  |I^1|^{-\frac 12}\langle T_{NRS}(1_{AI^1 \setminus J_A^1}\otimes 1_{I^{23}}),  1_{I_l^1}\otimes 1_{J^{23}}\rangle.
	\end{align*}
	Then invoking the full kernel estimates it is not difficult to see that the  two terms at the right hand side go to zero as $A\to \infty$, since
	\[
		|(AI^1+|I^1|/2)\setminus J_A^1|=|AI^1 \setminus J_A^1|=|I^1|/2.
	\]
	This completes the proof.
	All the symmetrical estimates follow in a similar way.
\end{proof}

In particular, we have proved the following.
\begin{thm}
	$T_{NRS}$ is a cancellative flag type Calder\'on-Zygmund operator.
	In particular, (by our upcoming representation theorem) for every $1<p<\infty$ and all $w\in A_{p, F}$
	we have
	\[
		\|T_{NRS}f\|_{L^p(w)}\lesssim \|f\|_{L^p(w)}.
	\]
\end{thm}
We note that $T_{NRS}$ is a flag type CZO that also showcases the necessity of
the condition $w\in A_{p, F}$.
\begin{thm}
	Suppose that $w$ is a weight for which
	\[
		\|T_{NRS}f\|_{L^p(w)}\lesssim \|f\|_{L^p(w)}
	\]
	for all $f \in L^p(w)$, where $1<p<\infty$. Then $w\in A_{p, F}$.
\end{thm}
\begin{proof}
	The proof is actually quite standard. For completeness we record the details.
	Fix an arbitrary rectangle $I \in \calR_F$, and denote $I^+= I+ (\ell(I^1), \ell(I^2), \ell(I^3))$. Let
	\[
		\sigma_{\eps}:= w^{-\frac 1{p-1}} 1_{\{\eps<w<\eps^{-1}\}},\quad 0<\eps<1.
	\]
	Using
	\[
		\|1_{I^+}T_{NRS}(\sigma_\eps 1_I)\|_{L^p(w)}\lesssim \|\sigma_\eps 1_I\|_{L^p(w)}
	\]
	we get
	\[
		\Big(\frac {\sigma_\eps(I)}{|I|}\Big)^{p-1} \frac{w(I^+)}{|I|}\lesssim 1.
	\]
	Then by the monotone convergence theorem we obtain
	\[
		\Big(\frac {\sigma (I)}{|I|}\Big)^{p-1} \frac{w(I^+)}{|I|}\lesssim 1,\quad \sigma:=w^{-\frac 1{p-1}}.
	\]
	Likewise, using
	\[
		\|1_{I}T_{NRS}(\sigma_\eps 1_{I^+})\|_{L^p(w)}\lesssim \|\sigma_\eps 1_{I^+}\|_{L^p(w)}
	\]
	we get that
	\[
		\Big(\frac {\sigma (I^+)}{|I|}\Big)^{p-1} \frac{w(I)}{|I|}\lesssim 1.
	\]
	Now, invoking  H\"older's inequality
	\[
		\Big(\frac {\sigma (I^+)}{|I|}\Big)^{p-1} \frac{w(I^+)}{|I|}\ge 1
	\]
	we have
	\begin{align*}
		\Big(\frac {\sigma (I)}{|I|}\Big)^{p-1} \frac{w(I)}{|I|}\le \Big(\frac {\sigma (I)}{|I|}\Big)^{p-1}\frac{w(I^+)}{|I|}\Big(\frac {\sigma (I^+)}{|I|}\Big)^{p-1} \frac{w(I)}{|I|}\lesssim 1.
	\end{align*}
	We are done.
\end{proof}

Next, we record that also the flag type Fourier multipliers $T = T_m$ with the symbol $m$ satisfying
\eqref{eq:def} are cancellative flag type Calder\'on-Zygmund operators in the sense of our definition.
Of course, we have already proved the kernel estimates before.
The weak boundedness property is trivial, since these operators are certainly bounded in $L^2$.
Thus, we only need to check the cancellation properties.
This goes similarly as in the proof of Proposition \ref{prop:nrs-cancellation}. For instance, we consider
$$\langle T(1\otimes 1_{I^{23}}), h_{I^1}\otimes 1_{J^{23}}\rangle.$$
If both $I^2, J^2$ and $I^3, J^3$ are disjoint, then everything is the same as in
Proposition \ref{prop:nrs-cancellation}. If e.g. $I^2\cap J^2 \neq \emptyset$ and $I^3\cap J^3 =\emptyset$, then we need to consider
\[
	\langle T(1\otimes 1_{I^2 \cap J^2} \otimes 1_{I^{3}}), h_{I^1}\otimes 1_{I^2 \cap J^2} \otimes 1_{J^{3}}\rangle.
\]
In Proposition \ref{prop:nrs-cancellation} this is trivial due to the anti-symmetry of the kernel.
Here we need to use the partial kernel representations and the translation invariance of the kernel
via
\begin{align*}
	 & \langle T(1\otimes 1_{I^2 \cap J^2} \otimes 1_{I^{3}}), h_{I^1}\otimes 1_{I^2 \cap J^2} \otimes 1_{J^{3}}\rangle                            \\
	 & = \langle T(1_{\R^{d_1}\setminus 3I^1}\otimes 1_{I^2 \cap J^2} \otimes 1_{I^{3}}), h_{I^1}\otimes 1_{I^2 \cap J^2} \otimes 1_{J^{3}}\rangle \\
	 & \hspace{3cm}+ \langle T(1_{3I^1}\otimes 1_{I^2 \cap J^2} \otimes 1_{I^{3}}), h_{I^1}\otimes 1_{I^2 \cap J^2} \otimes 1_{J^{3}}\rangle.
\end{align*}
We omit the full details. Therefore, the multipliers are also in our class.

Finally, notice that our arguments show that the translation invariance
\[
	\langle T(1_I), 1_J\rangle = \langle T(1_{I+t}), 1_{J+t}\rangle,\quad \forall t\in \R^3
\]
is a stronger assumption than our cancellation property. In particular,
more general convolution frameworks can also then be seen as a subclass of our
class.
\subsection{Bi-parameter flag CZOs}\label{sec:biparCZO}
As in the tri-parameter case above, a bi-parameter flag CZO is defined to be a bi-parameter flag SIO (satisfying the relevant estimates on all partial and full kernels) that also satisfies certain weak boundedness and BMO conditions. As before, we say that the operator $T$ satisfies the weak boundedness property if for all flag rectangles $I$, it holds that
\begin{equation*}
	| \langle T1_I, 1_I \rangle| \lesssim |I|.
\end{equation*}
\subsubsection*{Diagonal BMO assumptions}
A bi-parameter flag SIO $T$ satisfies the flag diagonal BMO assumption if, for $I^1\times I^2 \in \mathcal R_F$ and $L^{\infty}$-atoms $a_{I^i}$, we have
\begin{align*}
	|\langle T(a_{I^1}\otimes 1_{I^2}), 1_{I}\rangle|+|\langle T(1_{I^1}\otimes a_{I^2}), 1_{I}\rangle|+
	|\langle T(1_I), a_{I^1}\otimes 1_{I^2}\rangle|+|\langle T(1_I), 1_{I^1}\otimes a_{I^2}\rangle|\lesssim |I|.
\end{align*}
\begin{rem}
	Notice that the above is only assumed for flag rectangles making it weaker than in the
	standard bi-parameter theory. Notice also that all of the above conditions
	are obviously necessary for boundedness.
\end{rem}

\subsubsection*{Product BMO assumptions}
A bi-parameter flag SIO $T$ satisfies the product BMO assumption, if $T1, T^*1 \in \BMO_{\pro, F}$
and $T_11, T_1^*1 \in \BMO_{\pro, <}$. Here $T_1$ is the first partial adjoint, defined by
\[
	\langle T_1(f_1\otimes f_2), g_1 \otimes g_2\rangle= \langle T(g_1\otimes f_2), f_1\otimes g_2\rangle.
\]
Its adjoint $T_1^*$ is the second partial adjoint of $T$.

\begin{rem}
	There is a subtle difference in our multiresolution so that the full flag paraproducts
	related to $T1, T^*1$ are defined in $\calD_F$ while the paraproducts
	related to the partial adjoints are defined only in $\calD_<$. This
	explains the small difference in the assumptions. Recall that the $\BMO_{\pro, <}$ condition
	is clearly weaker than the corresponding $\BMO_{\pro}$ requirement in classical
	bi-parameter $T1$ theorems. Recall also that we showed that $T1 \in \BMO_{\pro, F}$
	can be deduced (for flag SIOs) from $T1 \in \BMO_{\pro, <}$ and the one-parameter local $T1$
	condition requiring that $\|1_Q T1_Q\|_{L^2} \lesssim |Q|^{1/2}$ for all cubes $Q$.
\end{rem}

\section{Coefficient estimates for flag CZOs}
\subsection{Bi-parameter coefficient estimates}
\subsubsection{Kernel size estimates}
\begin{lem}\label{lem:fullsize}
	Let $I, J \in \calD_F$
	be two dyadic flag rectangles with $\ell(I) = \ell(J)$ (meaning $\ell(I^m) = \ell(J^m)$ for $m=1,2$)
	and $I^m \ne J^m$ for $m=1,2$.
	Then we have
	\begin{align*}
		F_{IJ} :=
		\int_{I} \int_{J} & \frac{\ud x \ud y}
		                    {|x_1-y_1|^{d_1}(|x_1-y_1| + |x_2-y_2|)^{d_2}} \\
		                  & \lesssim
		\frac{|I|^2}{(\ell(I^1) + \dist(I^1,J^1))^{d_1} \max_{m=1,2}(\ell(I^m) + \dist(I^m,J^m))^{d_2}}.
	\end{align*}
\end{lem}
\begin{proof}
	\textbf{Case $\dist(I^m, J^m) > 0$ for $m=1,2$.}
	We directly get
	$$
		F_{IJ} \le \frac{|I|^2}{\dist(I^1,J^1)^{d_1}
			\max( \dist(I^1,J^1), \dist(I^2, J^2) )^{d_2} },
	$$
	which is the right estimate as $\dist(I^m, J^m) \ge \ell(I^m)$ here.

	\textbf{Case $\dist(I^1, J^1) > 0$ and $\dist(I^2, J^2) = 0$.}
	This time we estimate
	\begin{align*}
		F_{IJ} & \le \frac{|I^1|^2}{\dist(I^1,J^1)^{d_1}}
		\int_{I^2} \int_{J^2} \frac{\ud x_2 \ud y_2}{\max(\dist(I^1, J^1), |x_2-y_2|)^{d_2}} \\
		       & \lesssim \frac{|I^1|^2}{\dist(I^1,J^1)^{d_1}}
		\min\Big(\frac{|I^2|^2}{\dist(I^1, J^1)^{d_2}}, |I^2| \Big).
	\end{align*}
	Here we used the well-known estimate
	\begin{equation}\label{eq:nearbyint}
		\int_{3J^2 \setminus J^2} \int_{J^2} \frac{\ud x_2 \ud y_2}{|x_2-y_2|^{d_2}} \lesssim |J^2| = |I^2|.
	\end{equation}
	In the case $\dist(I^1, J^1) \ge \ell(I^2)$ we estimate
	$$
		F_{IJ} \lesssim \frac{|I|^2}{\dist(I^1, J^1)^{d_1} \dist(I^1, J^1)^{d_2}},
	$$
	which is the desired estimate in this case. In the case $\dist(I^1, J^1) \le \ell(I^2)$ we estimate
	\begin{align*}
		F_{IJ} \lesssim \frac{|I^1|^2}{\dist(I^1,J^1)^{d_1}} |I^2|
		= \frac{|I|^2}{\dist(I^1,J^1)^{d_1}\ell(I^2)^{d_2}},
	\end{align*}
	which is again the desired estimate in this case.

	\textbf{Case $\dist(I^1, J^1) = 0$ and $\dist(I^2, J^2) > 0$.}
	We estimate
	\begin{align*}
		F_{IJ} & \le \int_{I^1} \int_{J^1} \frac{\ud x_1 \ud y_1}{|x_1-y_1|^{d_1}}
		\int_{I^2} \int_{J^2} \frac{\ud x_2 \ud y_2}{\dist(I^2, J^2)^{d_2}}        \\
		       & \lesssim |I^1| \frac{|I^2|^2}{\dist(I^2, J^2)^{d_2}}
		= \frac{|I|^2}{\ell(I^1)^{d_1} \dist(I^2, J^2)^{d_2}},
	\end{align*}
	where we used \eqref{eq:nearbyint} in the first parameter.
	Since $\ell(I^1) \le \ell(I^2) \le \dist(I^2, J^2)$ this is the desired estimate.

	\textbf{Case $\dist(I^m, J^m) = 0$ for $m=1,2$.}
	We simply estimate using \eqref{eq:nearbyint} in both parameters that
	\begin{align*}
		F_{IJ} & \le \int_{I^1} \int_{J^1} \frac{\ud x_1 \ud y_1}{|x_1-y_1|^{d_1}}
		\int_{I^2} \int_{J^2} \frac{\ud x_2 \ud y_2}{|x_2-y_2|^{d_2}}
		\lesssim |I|,
	\end{align*}
	which is the desired estimate in this case, since here
	\begin{align*}
		 & \frac{|I|}{(\ell(I^1) + \dist(I^1,J^1))^{d_1} \max_{m=1,2}(\ell(I^m) + \dist(I^m,J^m))^{d_2}} \\
		 & = \frac{|I|}{\ell(I^1)^{d_1} \max(\ell(I^1), \ell(I^2))^{d_2}}
		= \frac{|I|}{\ell(I^1)^{d_1}\ell(I^2)^{d_2}} = 1.
	\end{align*}
\end{proof}
The above was the integral of $\size_F$ over some flag rectangles. We also record
some estimates on integrals that are related to $\size_{I^1}$ and $\size_{I^2}$.
\begin{lem}\label{lem:sizeI2}
	Let $I^1, J^1 \in \calD^1$ be two dyadic cubes with $\ell(I^1) = \ell(J^1)$ and $I^1 \ne J^1$.
	Let $s$ be a scalar satisfying $s \ge \ell(I^1)$. Then we have that
	\begin{align*}
		F_{I^1J^1} & := \int_{I^1} \int_{J^1} \frac{\ud x_1 \ud y_1}{|x_1-y_1|^{d_1} (|x_1-y_1| + s)^{d_2}}                    \\
		           & \lesssim \frac{|I^1|^2}{(\ell(I^1) + \dist(I^1, J^1))^{d_1} \max(\ell(I^1) + \dist(I^1, J^1), s )^{d_2}}.
	\end{align*}
\end{lem}
\begin{proof}
	\textbf{Case $\dist(I^1, J^1) > 0$.} We estimate
	\begin{align*}
		F_{I^1J^1} \le \frac{|I^1|^2}{\dist(I^1,J^1)^{d_1} \max(\dist(I^1, J^1), s)^{d_2}},
	\end{align*}
	which is the required estimate in this case.

	\textbf{Case $\dist(I^1, J^1) = 0$.} We estimate, using \eqref{eq:nearbyint}, that
	$$
		F_{I^1J^1} \le \frac{1}{s^{d_2}} \int_{I^1} \int_{J^1} \frac{\ud x_1 \ud y_1}{|x_1-y_1|^{d_1}}
		\lesssim \frac{|I^1|}{s^{d_2}} = \frac{|I^1|^2}{\ell(I^1)^{d_1} s^{d_2}},
	$$
	which is the desired estimate in this case as $s \ge \ell(I^1)$.
\end{proof}

Finally, we have the following, at this point, obvious result as well.
\begin{lem}\label{lem:sizeI1}
	Let $I^2, J^2 \in \calD^2$ be two dyadic cubes with $\ell(I^2) = \ell(J^2)$ and $I^2 \ne J^2$.
	Then we have that
	\begin{align*}
		F_{I^2J^2} := \int_{I^2} \int_{J^2} \frac{\ud x_2 \ud y_2}{|x_2-y_2|^{d_2}}
		\lesssim \frac{|I^2|^2}{(\ell(I^2) + \dist(I^2, J^2))^{d_2}}.
	\end{align*}
\end{lem}

\begin{thm}\label{thm:biparcoeff}
	Suppose that $T$ is a bi-parameter flag SIO satisfying the flag WBP. Let $I, J \in \calD_F$
	be two dyadic flag rectangles with $\ell(I) = \ell(J)$ (meaning $\ell(I^m) = \ell(J^m)$ for $m=1,2$). Let
	\begin{equation}\label{eq:biparcoefficient}
		a_{IJ} :=
		|\langle T(h_{I^1}^{\gamma_1} \otimes h_{I^2}^{\gamma_2}),
		h_{J^1}^{\kappa_1} \otimes h_{J^2}^{\kappa_2}\rangle|.
	\end{equation}
	If $(\gamma_1, \kappa_1) \ne (0,0)$ and $(\gamma_2, \kappa_2) \ne (0,0)$ then
	\begin{align*}
		a_{IJ} \lesssim \prod_{m=1}^2 \Big( & \frac{\ell(I^m)}{\ell(I^m) + \dist(I^m, J^m)} \Big)^{\alpha_m}                                        \\
		                                    & \times \frac{|I|}{(\ell(I^1) + \dist(I^1,J^1))^{d_1} \max_{j=1,2}(\ell(I^j) + \dist(I^j,J^j))^{d_2}}.
	\end{align*}
	If we only know $(\gamma_2, \kappa_2) \ne (0,0)$ then
	\begin{align*}
		a_{IJ} \lesssim \Big( & \frac{\ell(I^2)}{\ell(I^2) + \dist(I^2, J^2)} \Big)^{\alpha_2}                                        \\
		                      & \times \frac{|I|}{(\ell(I^1) + \dist(I^1,J^1))^{d_1} \max_{j=1,2}(\ell(I^j) + \dist(I^j,J^j))^{d_2}}.
	\end{align*}
\end{thm}
\begin{proof}
	We prove the case $(\gamma_1, \kappa_1) \ne (0,0)$ and $(\gamma_2, \kappa_2) \ne (0,0)$ first.
	By symmetry we can assume $\gamma_1 \ne 0$ and $\gamma_2 \ne 0$. We can then
	assume for small notational convenience that $\kappa_1 = 0$ and $\kappa_2 = 0$ --
	we write the proof so that it also works if these are cancellative Haars.
	So, we now estimate $a_{IJ} = |\langle T(h_{I^1} \otimes h_{I^2}), h_J^0 \rangle|$.

	A few remarks before we start.
	We notice that it is relevant that cancellation
	is only ever utilized in a given parameter $m$ if $\dist(I^m, J^m) > 0$.
	Also, some of the partial kernel estimates below could technically be
	done without splitting the Haars into sums of indicators of the children
	by considering, instead, the cancellative Haar function as an atom. However, this is not
	advisable -- it would make it look like the cancellation is relevant in the diagonal,
	it would not even be allowed if we would have two cancellative Haars in the
	same parameter (which is also a possible case in our theorem) and it
	would not generalize to our tri-parameter setup where we do not assume
	atomic versions of partial kernel estimates.

	\textbf{Case 1: $\dist(I^m, J^m) > 0$ for $m=1,2$.}
	In this setting, we can use one of the H\"older estimates of the full kernel:
	\begin{align*}
		a_{IJ} & \lesssim |I|^{-1} \int_J \int_I
		\bigg( \frac{|c_{I_1}-y_1|}{|x_1-y_1|} \bigg)^{\alpha_1}\bigg( \frac{|c_{I^2}-y_2|}{|x_2-y_2|}
		                                                        \bigg)^{\alpha_2}\size_F(x,y)\ud y \ud x  \\
		       & \lesssim |I|^{-1} \prod_{m=1}^2 \Big( \frac{\ell(I^m)}{\dist(I^m, J^m)} \Big)^{\alpha_m}
		\int_I \int_J \size_F(x,y)\ud x \ud y.
	\end{align*}
	Recalling that $\dist(I^m, J^m) \ge \ell(I^m)$ and using Lemma \ref{lem:fullsize}
	completes the proof of this case.

	\textbf{Case 2: $\dist(I^1, J^1) > 0$ and $\dist(I^2, J^2) = 0$, $I^2 \ne J^2$.}
	We use the full kernel mixed H\"older--size estimate as follows
	\begin{align*}
		a_{IJ} & = \Big| \int_J \int_I [K(x,y)-K(x, (c_{I^1}, y_2))] h_{I^1}(y_1)h_{I^2}(y_2)h_J^0(x) \ud y \ud x\Big| \\
		       & \lesssim |I|^{-1}\Big( \frac{\ell(I^1)}{\dist(I^1, J^1)} \Big)^{\alpha_1}
		\int_I \int_J \size_F(x,y)\ud x \ud y
	\end{align*}
	so that this case again follows from Lemma \ref{lem:fullsize}.

	\textbf{Case 3: $\dist(I^1, J^1) = 0$, $I^1 \ne J^1$ and $\dist(I^2, J^2) > 0$.}
	Having access to Lemma \ref{lem:fullsize} this goes similarly as the previous case.

	\textbf{Case 4: $\dist(I^m, J^m) = 0$, $I^m \ne J^m$, for $m=1,2$.}
	Thanks to Lemma \ref{lem:fullsize} this is the same as the previous two cases
	except we just use the size estimate of the full kernel.

	\textbf{Case 5: $I^1 = J^1$ and $\dist(I^2, J^2) > 0$.}
	We estimate
	$$
		a_{IJ} \le |I^1|^{-1} \sum_{L^1, P^1 \in \ch(I^1)}
		|\langle T(1_{L^1} \otimes h_{I^2}), 1_{P^1} \otimes h_{J^2}^0\rangle|.
	$$
	It is enough to estimate the summands with the desired estimate (summing
	over the children then produces a dimensional dependence).
	If $L^1 \ne P^1$ this is estimated essentially using Case 3 (remember
	that cancellation in the first parameter is not used in such an estimate
	so it does not matter that we have two indicators here).
	If $L^1 = P^1$, then
	we estimate using the partial kernel's H\"older estimate that
	\begin{align*}
		|\langle T(1_{L^1} \otimes h_{I^2}), 1_{L^1} \otimes h_{J^2}^0\rangle|
		 & = \Big| \int_{J^2} \int_{I^2} [K_{L^1}(x_2, y_2) - K_{L^1}(x_2, c_{I^2})]
		h_{I^2}(y_2)h_{J^2}^0(x_2) \ud y_2 \ud x_2 \Big|                             \\
		 & \lesssim \Big( \frac{\ell(I^2)}{\dist(I^2,J^2)} \Big)^{\alpha_2}
		|I^1| |I^2|^{-1} \int_{I^2} \int_{J^2} \frac{\ud x_2 \ud y_2}{|x_2-y_2|^{d_2}}.
	\end{align*}
	It remains to invoke Lemma \ref{lem:sizeI1} -- notice that the factor
	$$
		\frac{|I^1|}{\ell(I^1)^{d_1}} = 1
	$$
	and also that $\max(\ell(I^1), \ell(I^2) + \dist(I^2,J^2)) = \ell(I^2) + \dist(I^2, J^2)$
	as $\ell(I^1) \le \ell(I^2)$.

	\textbf{Case 6: $I^1 = J^1$ and $\dist(I^2, J^2) = 0$, $I^2 \ne J^2$.}
	Same as the previous Case 5 except we use Case 4 in the case
	$L^1 \ne P^1$ and the size estimate of the partial kernel in the case $L^1 = P^1$.

	\textbf{Case 7: $\dist(I^1, J^1) > 0$ and $I^2 = J^2$.}
	Well, similarly as in Case 5 we estimate
	$$
		a_{IJ} \le |I^2|^{-1} \sum_{L^2, P^2 \in \ch(I^2)}
		|\langle T(h_{I^1} \otimes 1_{L^2}), h_{J^1}^0 \otimes 1_{P^2}\rangle|.
	$$
	If $L^2 \ne P^2$ we rely on Case 2. If $L^2 = P^2$ we use the
	H\"older estimate of $K_{L^2}$ to get
	\begin{align*}
		|\langle T(h_{I^1} & \otimes 1_{L^2}), h_{J^1}^0 \otimes 1_{P^2}\rangle| \\
		                   & \lesssim |I^1|^{-1} |I^2|^2
		\Big( \frac{\ell(I^1)}{\dist(I^1,J^1)} \Big)^{\alpha_1}
		\int_{I^1} \int_{J^1} \frac{\ud x_1 \ud y_1}{|x_1-y_1|^{d_1}(|x_1-y_1| + \ell(I^2))^{d_2}}.
	\end{align*}
	It remains to invoke Lemma \ref{lem:sizeI2} (with $s = \ell(I^2) \ge \ell(I^1)$).

	\textbf{Case 8: $\dist(I^1, J^1) = 0$, $I^1 \ne J^1$ and $I^2=J^2$.}
	Same as Case 7 except we use Case 4 in the case $L^2 \ne P^2$
	and we use just the size estimate of the partial kernel in the case $L^2 = P^2$.

	\textbf{Case 9: $I^1 = J^1$ and $I^2=J^2$.}
	We estimate
	\begin{align*}
		a_{IJ} & \le |I|^{-1} \sum_{L^1, P^1 \in \ch(I^1)} \sum_{L^2, P^2 \in \ch(I^2)}
		|\langle T(1_{L^1} \otimes 1_{L^2}), 1_{P^1} \otimes 1_{P^2} \rangle|           \\
		       & = |I|^{-1}\Big( \sum_{\substack{L^1 \ne P^1 \\ L^2 \ne P^2 }}
		+ \sum_{\substack{L^1 \ne P^1 \\ L^2 = P^2 }}
		+ \sum_{\substack{L^1 = P^1 \\ L^2 \ne P^2 }}
		+\sum_{\substack{L^1 = P^1 \\ L^2 = P^2 }}
		\Big) |\langle T(1_{L^1} \otimes 1_{L^2}), 1_{P^1} \otimes 1_{P^2} \rangle|     \\
		       & = A_{\ne, \ne} + A_{\ne, =} + A_{=, \ne} + A_{=, =}.
	\end{align*}
	As $L^1 \times L^2$ is flag, the last term $A_{=, =}$ is handled by using the flag WBP -- notice
	that the estimate we are after here simply is $a_{IJ} \lesssim 1$ (as it is in all
	of the cases where $\dist(I^m, J^m) = 0$ for $m=1,2$).
	Finally, $A_{\ne, \ne}$ is covered by Case 4, $A_{\ne, =}$ by Case 8 and $A_{=, \ne}$
	by Case 6.

	The remaining case, where only $(\gamma_2, \kappa_2) \ne (0,0)$, follows in the same way.
	When $\dist(I^1, J^1) > 0$, we use only the size
	variant of the corresponding H\"older estimate in the first parameter.
\end{proof}

The following is a way to formulate the above theorem that is more
useful in connection with our dyadic model operators.
\begin{cor}\label{cor:biparcoeffcor}
	Suppose that $T$ is a bi-parameter flag SIO satisfying the flag WBP.
	Suppose $I, J \in \calD_F$ are two dyadic flag rectangles
	with $\ell(I) = \ell(J)$. Write $J^m = I^m \dotplus n^m := I^m + n^m \ell(I^m)$, $n^m \in \Z^{d_m}$. If
	$n^m \ne 0$ let $k^m \ge 2$ be such that $|n^m| \in (2^{k^m-3}, 2^{k^m-2}]$,
	and for $n^m = 0$ put $k^m = 0$.
	Let $a_{IJ}$ be defined as in \eqref{eq:biparcoefficient}.
	If $(\gamma_1, \kappa_1) \ne (0,0)$ and $(\gamma_2, \kappa_2) \ne (0,0)$ then
	\begin{align*}
		|a_{IJ}| \lesssim 2^{-\alpha_1 k^1} 2^{-\alpha_2 k^2} \frac{|I|}{(2^{k^1}\ell(I^1))^{d_1}
			                                                      \max(2^{k^1}\ell(I^1), 2^{k^2}\ell(I^2))^{d_2}}.
	\end{align*}
	If we only know $(\gamma_2, \kappa_2) \ne (0,0)$ then
	\begin{align*}
		|a_{IJ}| \lesssim 2^{-\alpha_2 k^2} \frac{|I|}{(2^{k^1}\ell(I^1))^{d_1}
			                                    \max(2^{k^1}\ell(I^1), 2^{k^2}\ell(I^2))^{d_2}}.
	\end{align*}
\end{cor}
\begin{proof}
	The point is that $\dist(I^m, J^m) \sim |n^m|\ell(I^m)
		\sim 2^{k^m}\ell(I^m)$ if $|n^m| \ge 2$,
	and for $|n^m| \le 1$ we have $\ell(I^m) \sim 2^{k^m}\ell(I^m)$.
	Thus, we always have $\ell(I^m) + \dist(I^m, J^m) \sim 2^{k^m}\ell(I^m)$.
	Therefore, this directly follows from Theorem \ref{thm:biparcoeff}.
\end{proof}
\begin{rem}\label{rem:biparKFcoerr}
	Given a dyadic rectangle $K \in \calD = \calD^1 \times \calD^2$ let $K_F\in \calD_F$ be
	the dyadic flag rectangle satisfying
	$K\subset K_F$, $K_F^1 = K^1$ and
	$\ell(K_F^2)= \max\{\ell(K^1), \ell(K^2)\}$.

	In the dyadic representation theorem we will have
	the situation above -- $J^m = I^m \dotplus n^m$,
	$|n^m| \sim 2^{k^m}$ (in the precise sense above)
	and that $I^m$ is a so-called $k^m$-good cube (
	essentially deep inside its dyadic $k^m$ parent).
	This will guarantee that $I^m$ and $J^m$ not only
	have a common dyadic parent but also that, precisely,
	$(I^m)^{(k^m)} = (J^m)^{(k^m)}$. Call this common parent $K^m$.
	As $\ell(K^m) = 2^{k^m}\ell(I^m)$ the first estimate of Corollary \ref{cor:biparcoeffcor}
	then reads
	$$
		|a_{IJ}| \lesssim 2^{-\alpha_1 k^1} 2^{-\alpha_2 k^2} \frac{|I|}{|K_F|}.
	$$
	This ratio $\frac{|I|}{|K_F|}$ will appear prominently later in the dyadic model operators.
\end{rem}

\subsubsection{Kernel $\BMO$ estimates}
\begin{lem}\label{lem:bmocoeff1}
	Let $I^1, J^1 \in \calD^1$ be two dyadic cubes with $\ell(I^1) = \ell(J^1)$. Let $I^2 \in \calD^2$
	with $\ell(I^2) \ge \ell(I^1)$. Suppose $a_{I^2}$ is an $L^{\infty}$ atom on $I^2$.
	Let $T$ be a bi-parameter flag SIO satisfying the diagonal flag $\BMO$ assumption.
	Then we have
	\begin{align*}
		 & |\langle T(h_{I^1} \otimes a_{I^2}), u_{J^1} \otimes 1 \rangle|               \\
		 & \lesssim \Big( \frac{\ell(I^1)}{\ell(I^1) + \dist(I^1, J^1)} \Big)^{\alpha_1}
		\frac{|I^1|}{(\ell(I^1) + \dist(I^1, J^1))^{d_1}} |I^2|
		\min\Big(1, \frac{\ell(I^2)}{\ell(I^1)+\dist(I^1,J^1)}\Big)^{\alpha_2/2}.
	\end{align*}
	The same is true if $T$ is replaced by any of its (partial) duals.
\end{lem}
\begin{proof}
	We begin by splitting
	\begin{equation*}
		\begin{split}
			\langle T(h_{I^1}\otimes a_{I^2}), u_{J^1}\otimes 1\rangle
			 & =\langle T(h_{I^1}\otimes a_{I^2}), u_{J^1}\otimes 1_{I^2}\rangle                     \\
			 & +\langle T(h_{I^1}\otimes a_{I^2}), u_{J^1}\otimes 1_{3I^2\setminus I^2}\rangle       \\
			 & + \langle T(h_{I^1}\otimes a_{I^2}), u_{J^1}\otimes 1_{\R^{d_2}\setminus 3I^2}\rangle
			=: V_= + V_N + V_S.
		\end{split}
	\end{equation*}
	We have to study each of these three terms in all of the cases
	1) $\dist(I^1, J^1) > 0$, 2) $\dist(I^1, J^1) = 0$, $I^1 \ne J^1$ and
	3) $I^1 = J^1$. We use $S$ to denote the case of positive distance,
	$N$ to denote the neighboring (distance zero) case and ``$=$'' to denote the equal case.
	For instance, the case $(V_N, S)$ refers to the term
	$V_N = \langle T(h_{I^1}\otimes a_{I^2}), u_{J^1}\otimes 1_{3I^2\setminus I^2}\rangle$
	in the case $\dist(I^1, J^1) > 0$. For $V_=$ and
	$V_N$ we could essentially refer to Theorem \ref{thm:biparcoeff}
	modulo easy details about normalization and, more importantly,
	the distinction between WBP and diagonal $\BMO$ -- but we can also just quickly calculate these by hand.

	\textbf{Case $(V_=, S)$.} Use the partial kernel $K_{I^2}$ (the atom version)
	to write
	\begin{align*}
		|V_=| & = \Big| \int_{J^1} \int_{I^1} \big[K_{I^2}(x_1, y_1) - K_{I^2}(x_1, c_{I^1}) \big]
		h_{I^1}(y_1) u_{J^1}(x_1) \ud y_1 \ud x_1 \Big|                                              \\
		      & \lesssim \Big( \frac{\ell(I^1)}{\dist(I^1, J^1)} \Big)^{\alpha_1} |I^1|^{-1} |I^2|^2
		\int_{J^1} \int_{I^1} \frac{\ud x_1 \ud y_1}{|x_1-y_1|^{d_1}(|x_1-y_1| + \ell(I^2))^{d_2}}.
	\end{align*}
	It remains to use Lemma \ref{lem:sizeI2} -- we even get the last decay factor involving the minimum
	with the bigger exponent $d_2$ instead of $\alpha_2/2$. In fact, only in $V_S$ do we start getting
	this decay factor with the smaller exponent $\alpha_2/2$.

	\textbf{Case $(V_=, N)$.} This is the same as the case $(V_=, S)$ except we use
	the size estimate of the partial kernel.

	\textbf{Case $(V_=, =)$.} We, as usual, first estimate
	$$
		|V_=| \le |I^1|^{-1} \sum_{L^1, P^1 \in \ch(I^1)}
		|\langle T(1_{L^1}\otimes a_{I^2}), 1_{P^1}\otimes 1_{I^2}\rangle|.
	$$
	If $L^1 \ne P^1$ this is essentially the same as the case $(V_=, N)$.
	If $L^1 = P^1$ we use the diagonal flag $\BMO$ assumption recalling that
	$\ell(L^1) < \ell(I^2)$. Thus, we have $|V_=| \lesssim |I^2|$ as desired.

	\textbf{Case $(V_N, S)$.} We use the full kernel to estimate
	\begin{align*}
		|V_N| & =  \Big| \int_{J^1 \times (3I^2 \setminus I^2)} \int_{I}
		\big[ K(x,y) - K(x, (c_{I^1}, y_2))\big] h_{I^1}(y_1)a_{I^2}(y_2)u_{J^1}(x_1) \ud y \ud x \Big| \\
		      & \le \Big( \frac{\ell(I^1)}{\dist(I^1, J^1)} \Big)^{\alpha_1}
		|I^1|^{-1} \int_{J^1 \times (3I^2 \setminus I^2)} \int_{I} \size_F(x,y) \ud y \ud x             \\
		      & \lesssim
		\Big( \frac{\ell(I^1)}{\dist(I^1, J^1)} \Big)^{\alpha_1}
		\frac{|I^1| |I^2|^2}{\dist(I^1,J^1)^{d_1} \max(\dist(I^1,J^1), \ell(I^2))^{d_2}},
	\end{align*}
	where the last estimate is by Lemma \ref{lem:fullsize}. This is the desired estimate.

	\textbf{Case $(V_N, N)$.}
	This is the same as the previous case except we use the size estimate of the full kernel.

	\textbf{Case $(V_N, =)$.}
	We estimate
	$$
		|V_N| \le |I^1|^{-1} \sum_{L^1, P^1 \in \ch(I^1)}
		|\langle T(1_{L^1}\otimes a_{I^2}), 1_{P^1}\otimes 1_{3I^2 \setminus I^2}\rangle|.
	$$
	The case $L^1 \ne P^1$ is essentially the above case $(V_N , N)$. If $L^1 = P^1$
	we use the partial kernel $K_{L^1}$ as follows:
	\begin{align*}
		|\langle T(1_{L^1}\otimes a_{I^2}), 1_{L^1}\otimes 1_{3I^2 \setminus I^2}\rangle|
		 & = \Big| \int_{3I^2 \setminus I^2} \int_{I^2} K_{L^1}(x_2, y_2)a_{I^2}(y_2) \ud y_2 \ud x_2\Big| \\
		 & \lesssim |I^1| \int_{3I^2 \setminus I^2} \int_{I^2} \frac{\ud y_2 \ud x_2}{|x_2-y_2|^{d_2}}
		\lesssim |I^1| |I^2|.
	\end{align*}
	The last estimate was by \eqref{eq:nearbyint}. So $|V_N| \lesssim |I^2|$,
	which is the desired estimate.

	\textbf{Case $(V_S, S)$.}
	We use the full H\"older estimate of the full kernel to dominate $|V_S|$ by
	\begin{align*}
		\Big( \frac{\ell(I^1)}{\dist(I^1, J^1)} \Big)^{\alpha_1} \frac{|I^1|}{\dist(I^1, J^1)^{d_1}}
		\int_{(3I^2)^c} \int_{I^2}
		\frac{1}{(\dist(I^1, J^1) + |x_2-y_2|)^{d_2}}
		\frac{\ell(I^2)^{\alpha_2/2}}{|x_2-y_2|^{\alpha_2/2}} \ud y_2 \ud x_2.
	\end{align*}
	Notice that with $y_2 \in I_2$ fixed we have
	$$
		\int_{(3I^2)^c} \frac{\ud x_2}{(\dist(I^1, J^1) + |x_2-y_2|)^{d_2} |x_2-y_2|^{\alpha_2/2}}
		\lesssim \min(\dist(I^1,J^1)^{-\alpha_2/2}, \ell(I^2)^{-\alpha_2/2}).
	$$
	This completes this case.

	\textbf{Case $(V_S, N)$.}
	We estimate using the mixed H\"older and size of the full kernel that
	$$
		|V_S| \lesssim |I^1|^{-1} \int_{J^1} \int_{I^1} \frac{\ud y_1 \ud x_1}{|x_1-y_1|^{d_1}}
		\int_{(3I^2)^c} \int_{I^2} \frac{\ell(I^2)^{\alpha_2}}{|x_2-y_2|^{d_2+\alpha_2}} \ud y_2 \ud x_2
		\lesssim |I^2|,
	$$
	which is the desired estimate in this case.

	\textbf{Case $(V_S, =)$.}
	We estimate
	$$
		|V_S| \le |I^1|^{-1} \sum_{L^1, P^1 \in \ch(I^1)}
		|\langle T(1_{L^1} \otimes a_{I^2}), 1_{P^1}\otimes 1_{(3I^2)^c}\rangle|.
	$$
	The case $L^1 \ne P^1$ is essentially the above case $(V_S , N)$. If $L^1 = P^1$
	we use H\"older estimate of the partial kernel $K_{L^1}$ to get
	\begin{align*}
		|\langle T(1_{L^1} \otimes a_{I^2}), 1_{L^1}\otimes 1_{(3I^2)^c}\rangle|
		\lesssim |I^1| \int_{(3I^2)^c} \int_{I^2} \frac{\ell(I^2)^{\alpha_2}}{|x_2-y_2|^{d_2 + \alpha_2}}
		\lesssim |I^1| |I^2|
	\end{align*}
	giving us the desired estimate $|V_S| \lesssim |I^2|$.
\end{proof}
Now, consider a situation where $I^1, J^1$ with $\ell(I^1) = \ell(J^1)$ are fixed
with $(I^1)^{(k^1)} = (J^1)^{(k^1)} =: K^1$ and $\ell(I^1) + \dist(I^1, J^1) \sim \ell(K^1)$
(as in Remark \ref{rem:biparKFcoerr} -- this is the basic configuration that arises in the representation theorem).
Then for any $I^2$ with $\ell(I^2) \ge \ell(I^1)$ the above
estimate reads
$$
	|\langle T(h_{I^1} \otimes a_{I^2}), u_{J^1} \otimes 1 \rangle|
	\lesssim 2^{-\alpha_1 k^1} \frac{|I^1|}{|K^1|} |I^2| \min\Big(1, \frac{\ell(I^2)}{\ell(K^1)} \Big)^{\gamma}
$$
for some $\gamma > 0$. It turns out that in our partial paraproducts, where such
coefficient estimates are relevant, we arrange
it so that actually $\ell(I^2) \ge \ell(K^1)$ and the last decay factor
does not play a role. So we do not attempt to keep it around anymore.

\begin{cor}
	Let $T$ be a bi-parameter flag SIO satisfying the diagonal flag $\BMO$ assumption.
	\textbf{Fix} $I^1, J^1 \in \calD^1$ with $\ell(I^1) = \ell(J^1)$ and define the sequence
	\begin{align*}
		\gamma_{I^2} = \gamma_{I^1J^1I^2} & := \langle T(h_{I^1} \otimes h_{I^2}), u_{J^1} \otimes 1\rangle
	\end{align*}
	indexed by all $I^2 \in \calD^2$ with $\ell(I^2) \ge \ell(I^1)$. It satisfies
	$$
		\|(\gamma_{I^2})_{\substack{ I^2 \in \calD^2 \\ \ell(I^2) \ge \ell(I^1)}}\|_{\BMO}
		\lesssim  \Big( \frac{\ell(I^1)}{\ell(I^1) + \dist(I^1, J^1)} \Big)^{\alpha_1}
		\frac{|I^1|}{(\ell(I^1) + \dist(I^1, J^1))^{d_1}} =: C = C_{I^1J^1}.
	$$
\end{cor}
\begin{proof}
	Recall that $I^1, J^1$ are fixed in this argument.
	Define the function $b$ in $\R^{d_2}$ with Haar coefficients
	\[
		b_{I^2}:=\langle b, h_{I^2}\rangle=\begin{cases}
			\langle T^*(u_{J^1}\otimes 1), h_{I^1}\otimes h_{I^2}\rangle, & \ell(I^2) \ge \ell(I^1), \\
			0,                                                            & \text{otherwise.}
		\end{cases}
	\]
	That is, set $b=\sum_{I^2 \in \calD^2} b_{I^2} h_{I^2}$. We have that the sequence $\{\gamma_{I^2} \}$
	is a $\BMO$ sequence if and only if $b$ is in the space $\BMO(\calD^2)$ -- the dyadic
	$\BMO$ in $\R^{d_2}$ with respect to the lattice $\calD^2$.

	Now, we will study the $L^1$ mean oscillation of $b$. Fix a cube $Q^2 \in \calD^2$.
	Notice that if $\ell(Q^2) < \ell(I^1)$, then we have that
	\[
		|Q^2|^{-1}\int_{Q^2}|b-\langle b\rangle_{Q^2}| \le
		\left(|Q^2|^{-1}\int_{Q^2}|b-\langle b\rangle_{Q^2}|^2\right)^{1/2}
	\]
	and
	\[
		\Vert 1_{Q^2}(b-\langle b\rangle_{Q^2})\Vert_{L^2}^2=\left\Vert \sum_{L^2\subset Q^2} \Delta_{L^2} b\right \Vert_{L^2}^2=\sum_{L^2\subset Q^2}\left\Vert  \Delta_{L^2} b\right \Vert_{L^2}^2=0,
	\]
	where we used the fact that $\ell(L^2) \le \ell(Q^2) < \ell(I^1)$.
	Hence, we may assume that $\ell(Q^2) \ge \ell(I^1)$.

	We want to reduce to the $L^1$ oscillations of the function
	\[
		a := \langle T^*(u_{J^1}\otimes 1),h_{I^1}\rangle
	\]
	over cubes $\ell(Q^2) \ge \ell(I^1)$ instead of the $L^1$ oscillations of $b$. But notice that
	$$
		b =  \sum_{I^2 \colon \ell(I^2) \ge \ell(I^1)} \Delta_{I^2} a =  \sum_{\ell(I^2) = \ell(I^1)/2} E_{I^2} a.
	$$
	We decompose and estimate as follows
	\begin{align*}
		\int_{Q^2} |b-\langle b\rangle_{Q^2}| & =
		\sum_{\substack{L^2\subset Q^2 \\ \ell(L^2)=\ell(I^1)/2}}\int_{L^2}|b-\langle b\rangle_{Q^2}| \\
		                                      & =
		\sum_{\substack{L^2\subset Q^2 \\ \ell(L^2)=\ell(I^1)/2}} |L^2| |\langle a \rangle_{L^2}
		-\langle a\rangle_{Q^2}|                                                                      \\
		                                      & \le
		\sum_{\substack{L^2\subset Q^2 \\ \ell(L^2)=\ell(I^1)/2}} |L^2| \langle |a-\langle a \rangle_{Q^2}| \rangle_{L^2}
		=  \int_{Q^2} |a - \langle a \rangle_{Q^2}|.
	\end{align*}
	Therefore, we have reduced to showing that
	$$
		\int_{I^2} |a - \langle a \rangle_{I^2}| \lesssim C|I^2|
	$$
	for all $I^2 \in \calD^2$ with $\ell(I^2) \ge \ell(I^1)$. But this is equivalent to showing that
	$$
		|\langle a, a_{I^2} \rangle| \lesssim C|I^2|
	$$
	for all $L^{\infty}$ atoms $a_{I^2}$. Since
	$$
		\langle a, a_{I^2}\rangle = \langle T(h_{I^1} \otimes a_{I^2}), u_{J^1} \otimes 1\rangle
	$$
	this follows directly from Lemma \ref{lem:bmocoeff1}.
\end{proof}

We have the following analog of Lemma \ref{lem:bmocoeff1} as well.
\begin{lem}\label{lem:bmocoeff2}
	Let $I^2, J^2 \in \calD^2$ be two dyadic cubes with $\ell(I^2) = \ell(J^2)$. Let $I^1 \in \calD^1$
	with $\ell(I^2) \ge \ell(I^1)$. Suppose $a_{I^1}$ is an $L^{\infty}$ atom on $I^1$.
	Let $T$ be a bi-parameter flag SIO satisfying the diagonal flag $\BMO$ assumption.
	Then we have
	\begin{align*}
		|\langle T(a_{I^1} \otimes h_{I^2}), 1 \otimes u_{J^2} \rangle|
		 & \lesssim \Big( \frac{\ell(I^2)}{\ell(I^2) + \dist(I^2, J^2)} \Big)^{\alpha_2}
		\frac{|I^2|}{(\ell(I^2) + \dist(I^2, J^2))^{d_2}} |I^1|.
	\end{align*}
	The same is true if $T$ is replaced by any of its (partial) duals.
\end{lem}

\begin{proof}
	Similarly as before, we write
	\begin{equation*}
		\begin{split}
			\langle T(a_{I^1}\otimes h_{I^2}), 1\otimes u_{J^2}\rangle
			 & =\langle T(a_{I^1}\otimes h_{I^2}),  1_{I^1}\otimes u_{J^2}\rangle                      \\
			 & +\langle T(a_{I^1}\otimes h_{I^2}),  1_{3I^1\setminus I^1} \otimes u_{J^2}\rangle       \\
			 & + \langle T(a_{I^1}\otimes h_{I^2}),   1_{\R^{d_1}\setminus 3I^1}\otimes u_{J^2}\rangle
			=: W_= + W_N + W_S.
		\end{split}
	\end{equation*}
	The estimates of $W_=$ and $W_N$ are essentially
	included in the proof of Theorem \ref{thm:biparcoeff} --
	the difference is that when we use WBP in Theorem \ref{thm:biparcoeff},
	here we need to use the diagonal flag BMO assumption. We also did the
	analogous calculations by hand in the proof of Lemma \ref{lem:bmocoeff1}
	so this distinction should be clear.

	It remains to estimate $W_S$. We need to consider the cases
	1) $\dist(I^2, J^2) > 0$, 2) $\dist(I^2, J^2) = 0$, $I^2 \ne J^2$ and
	3) $I^2 = J^2$. We will denote these cases by $(W_S, S)$, $(W_S, N)$ and $(W_S, =)$, respectively.

	\textbf{Case $(W_S, S)$.}   We use the full H\"older estimate of the full kernel to dominate $|W_S|$ by
	\begin{align*}
		\Big( \frac{\ell(I^2)}{\dist(I^2, J^2)} \Big)^{\alpha_2}  |I^2|
		\int_{(3I^1)^c} \int_{I^1}
		\frac{1}{(\dist(I^2, J^2) + |x_1-y_1|)^{d_2}}
		\frac{\ell(I^1)^{\alpha_1}}{|x_1-y_1|^{d_1+\alpha_1}} \ud y_1 \ud x_1.
	\end{align*}
	Simply replace $(\dist(I^2, J^2) + |x_1-y_1|)^{-d_2}$ with $\dist(I^2, J^2)^{-d_2}$
	to obtain the desired estimate.

	\textbf{Case $(W_S, N)$.}  We estimate using the mixed H\"older and size estimate
	of the full kernel that
	$$
		|W_S| \lesssim |I^2|^{-1} \int_{J^2} \int_{I^2} \frac{\ud y_2 \ud x_2}{|x_2-y_2|^{d_2}}
		\int_{(3I^1)^c} \int_{I^1} \frac{\ell(I^1)^{\alpha_1}}{|x_1-y_1|^{d_1+\alpha_1}} \ud y_1 \ud x_1
		\lesssim |I^1|,
	$$
	which is the desired estimate in this case.

	\textbf{Case $(W_S, =)$.} In this case we estimate
	$$
		|W_S| \le |I^2|^{-1} \sum_{L^2, P^2 \in \ch(I^2)}
		|\langle T(a_{I^1} \otimes 1_{L^2} ),   1_{(3I^1)^c}\otimes 1_{P^2}\rangle|.
	$$
	The case $L^2 \ne P^2$ is essentially the above case $(W_S , N)$. If $L^2 = P^2$
	we use the H\"older estimate of the partial kernel $K_{L^2}$ to get
	\begin{align*}
		|\langle T(a_{I^1} \otimes 1_{L^2} ),   1_{(3I^1)^c}\otimes 1_{P^2}\rangle|
		\lesssim |I^2| \int_{(3I^1)^c} \int_{I^1} \frac{\ell(I^1)^{\alpha_1}}{|x_1-y_1|^{d_1 + \alpha_1}}\ud y_1 \ud x_1
		\lesssim |I^1| |I^2|
	\end{align*}
	giving us the desired estimate $|W_S| \lesssim |I^1|$.

\end{proof}

\begin{lem}
	Let $T$ be a bi-parameter flag SIO satisfying the diagonal flag $\BMO$ assumption.
	\textbf{Fix} $I^2, J^2 \in \calD^2$ with $\ell(I^2) = \ell(J^2)$ and define the sequence
	\begin{align*}
		\rho_{I^1} = \langle T(h_{I^1}\otimes h_{I^2}), 1\otimes u_{J^2}\rangle
	\end{align*}
	indexed by all $I^1 \in \calD^1$ with $\ell(I^1) \le \ell(I^2)$. It satisfies
	$$
		\|(\rho_{I^1})_{\substack{ I^1 \in \calD^1 \\ \ell(I^1) \le \ell(I^2)}}\|_{\BMO}
		\lesssim  \Big( \frac{\ell(I^2)}{\ell(I^2) + \dist(I^2, J^2)} \Big)^{\alpha_2}
		\frac{|I^2|}{(\ell(I^2) + \dist(I^2, J^2))^{d_2}} =: C = C_{I^2J^2}.
	$$
\end{lem}
\begin{proof}
	As before, since $I^2,J^2$ are fixed, we define a function $b$ on $\mathbb{R}^{d_1}$ with Haar coefficients
	\[
		b_{I^1}:=\langle b, h_{I^1}\rangle=\begin{cases}
			\langle T(h_{I^1}\otimes h_{I^2}), 1\otimes u_{J^2}\rangle, & \ell(I^1)\le \ell(I^2), \\
			0,                                                          & \text{otherwise}.
		\end{cases}
	\]
	That is, $b=\sum_{I^1} b_{I^1} h_{I^1}$. As usual, we have that the sequence $\{ \rho_{I^1}\}_{I^1}$ is in $\BMO(\calD^1)$ if and only if $b$ is a dyadic BMO function in $\mathbb{R}^{d_1}$ with respect to $\calD^1$.

	We study the $L^1$ mean oscillation of $b$. Fix a cube $Q^1$.
	First, suppose that $\ell(Q^1)>\ell(I^2)$.
	Let $\{J^1_j\}_j$ be the collection in $\calD^1(Q^1)$
	(the dyadic system inside $Q_1$) with $\ell(J^1_j)=\ell(I^2)$ for all $j$. We have
	\begin{align*}
		|Q^1|^{-1}\int_{Q^1} |b-\langle b\rangle_{Q^1}| & =   |Q^1|^{-1}\sum_j \int_{J^1_j}|b-\langle b\rangle_{Q^1}|                                                                                      \\
		                                                & \le  |Q^1|^{-1}\sum_j \int_{J^1_j}|b-\langle b\rangle_{J^1_j}|+  |Q^1|^{-1}\sum_j \int_{J^1_j}|\langle b\rangle_{J^1_j}- \langle b\rangle_{Q^1}| \\
		                                                & = |Q^1|^{-1}\sum_j \int_{J^1_j}|b-\langle b\rangle_{J^1_j}|,
	\end{align*}
	where we have used that by the definition of $b$ we have
	\[
		\langle b\rangle_{J^1_j}- \langle b\rangle_{Q^1} = \sum_{m=1}^s\langle \Delta_{(J^1_j)^{(m)}} b
		\rangle_{J^1_j}=0,\qquad (J^1_j)^{(s)}=Q^1.
	\]
	Thus, to control the oscillation over $Q^1$, $\ell(Q^1) > \ell(I^2)$,
	it actually suffices to prove that
	\[
		|J^1_j|^{-1} \int_{J^1_j}|b-\langle b\rangle_{J^1_j}|\le C.
	\]
	Therefore, we have actually reduced completely
	to studying $L^1$ oscillations in cubes $Q^1$ with $\ell(Q^1)\le \ell(I^2)$.

	We then want to reduce to the $L^1$ oscillations of the function
	\[
		a := \langle T^*(1\otimes u_{J^2}),h_{I^2}\rangle
	\]
	over the cubes $Q^1$ with $\ell(Q^1) \le \ell(I^2)$ instead of the $L^1$ oscillations of $b$. But notice that
	$$
		b =  \sum_{I^1 \colon \ell(I^1) \le \ell(I^2)} \Delta_{I^1} a = a- \sum_{\ell(I^1) = \ell(I^2)} E_{I^1} a.
	$$
	We next notice that
	\[
		\Big(\sum_{\ell(I^1) = \ell(I^2)} E_{I^1} a\Big)1_{Q^1}=\Big\langle \sum_{\ell(I^1) = \ell(I^2)} E_{I^1} a\Big\rangle_{Q^1}1_{Q^1}
	\]
	and, hence,
	\begin{align*}
		\frac 1{|Q^1|}\int_{Q^1}|b-\langle b\rangle_{Q^1}|= \frac 1{|Q^1|}\int_{Q^1}|a-\langle a\rangle_{Q^1}|.
	\end{align*}
	Therefore, we have reduced to showing that
	$$
		\int_{I^1} |a - \langle a \rangle_{I^1}| \lesssim C|I^1|
	$$
	for all $I^1 \in \calD^1$ with $\ell(I^1) \le \ell(I^2)$. But this is equivalent to showing that
	$$
		|\langle a, a_{I^1} \rangle| \lesssim C|I^1|
	$$
	for all $L^{\infty}$ atoms $a_{I^1}$. Since
	$$
		\langle a, a_{I^1}\rangle = \langle T(a_{I^1} \otimes h_{I^2}), 1 \otimes u_{J^2}\rangle,
	$$
	this follows directly from Lemma \ref{lem:bmocoeff2}.
\end{proof}

In the partial paraproducts where the above coefficient estimates are relevant --
those where the paraproduct part is in the first parameter --
we have the summing restriction $\ell(I^1) < \ell(I^2)$ (so we are in the above situation)
and $(I^2)^{(k^2)} = (J^2)^{(k^2)} =: K^2$ with $\ell(I^2) + \dist(I^2,J^2) \sim \ell(K^2)$.
So in that setting the above upper bound reads in the form $2^{-\alpha_2 k^2} \frac{|I^2|}{|K^2|}$,
which is the natural way to write it in the context of our representation theorem.
This will become more clear when we define the partial paraproducts,
look at the coefficient assumptions needed there and see how the
partial paraproducts arise in the representation theorem.

\subsection{Tri-parameter coefficient estimates}
\begin{lem}\label{lem:trifullsize}
	Let $I, J \in \calD_F$
	be two dyadic flag rectangles with $\ell(I) = \ell(J)$ (meaning $\ell(I^m) = \ell(J^m)$ for $m=1,2,3$)
	and $I^m \ne J^m$ for $m=1,2,3$.
	Then we have
	\begin{align*}
		F_{IJ} :=
		\int_{I} \int_{J} & \frac{\ud x \ud y}
		                    {|x_1-y_1|^{d_1}(|x_1-y_1| + |x_2-y_2|)^{d_2}(|x_1-y_1| + |x_2-y_2|+|x_3-y_3|)^{d_3}} \\
		                  & \lesssim
		\frac{|I|^2}{\prod_{k=1}^3\max_{1\le m \le k}(\ell(I^m) + \dist(I^m,J^m))^{d_k}}.
	\end{align*}
\end{lem}
\begin{proof}
	This is not so different from the bi-parameter version of this lemma proved earlier, though now we have
	$2^3 = 8$ cases to consider based on whether or not the distances between $I^k$ and $J^k$ are positive or 0 for each $k\in \{1,2,3\}$. We'll use $S$ to denote separation (positive distance) and $N$ to denote the neighboring case (distance=0); for instance, $(N,S,N)$ denotes the case where $\dist (I^1,J^1)=0=\dist(I^3,J^3)$ and $\dist(I^2,J^2)>0$.

	\textbf{Case $(S,S,S)$.}
	We directly get
	$$
		F_{IJ} \le \frac{|I|^2}{\prod_{k=1}^3\max_{1\le m \le k}(\dist(I^m,J^m))^{d_k}},
	$$
	which is the right estimate as $\dist(I^m, J^m) \ge \ell(I^m)$ here.

	\textbf{Case $(S,S,N)$.}
	This time we estimate
	\begin{align*}
		F_{IJ} & \le \frac{|I^1|^2|I^2|^2}{\dist(I^1,J^1)^{d_1}(\dist(I^1,J^1)+\dist(I^2,J^2))^{d_2}}                                           \\
		       & \hspace{3cm} \times \int_{I^3} \int_{J^3} \frac{\ud x_3 \ud y_3}{\max(\dist(I^1, J^1),\dist(I^2,J^2), |x_3-y_3|)^{d_3}}        \\
		       & \lesssim \frac{|I|^2}{\dist(I^1,J^1)^{d_1}(\dist(I^1,J^1)+\dist(I^2,J^2))^{d_2}}                                               \\
		       & \hspace{3cm} \times \min\Big(\frac{1}{\dist(I^1, J^1)^{d_3}},\frac{1}{\dist(I^2, J^2)^{d_3}}, \frac{1}{\ell(I^3)^{d_3}} \Big),
	\end{align*}
	where we used \eqref{eq:nearbyint} in the third parameter. This is the required estimate.

	\textbf{Case $(S,N,S)$.}
	Notice that in this case $|x_2-y_2| \lesssim \ell(I^2) \le \ell(I^3) \le \dist(I^3, J^3)$
	so it makes sense to majorize this via
	\begin{align*}
		F_{IJ} & \lesssim \frac{|I^1|^2|I^3|^2}{\dist(I^1,J^1)^{d_1}\max(\dist(I^1,J^1),\dist(I^3,J^3))^{d_3}}
		\int_{I^{2}} \int_{J^{2}} \frac{\ud x_{2} \ud y_{2}}{\max(\dist(I^1, J^1),|x_2-y_2|)^{d_2}}            \\
		       & \lesssim \frac{|I|^2}{\dist(I^1,J^1)^{d_1}(\dist(I^1,J^1)+\dist(I^3,J^3))^{d_3}}
		\min\Big(\frac{1}{\dist(I^1, J^1)^{d_2}}, \frac{1}{\ell(I^2)^{d_2}} \Big),
	\end{align*}
	which is the desired estimate.

	\textbf{Case $(N,S,S)$.}
	Here $|x_1-y_1| \lesssim \ell(I^1) \le \ell(I^2) \le \dist(I^2, J^2)$ so we estimate
	\begin{align*}
		F_{IJ} & \lesssim \frac{|I^2|^2|I^3|^2}{\dist(I^2,J^2)^{d_2}\max(\dist(I^2,J^2),\dist(I^3,J^3))^{d_3}}
		\int_{I^1} \int_{J^1} \frac{\ud x_1 \ud y_1}{|x_1-y_1|^{d_1}}                                                 \\
		       & \lesssim \frac{|I|^2}{\ell(I^1)^{d_1}\dist(I^2,J^2)^{d_2}\max(\dist(I^2,J^2),\dist(I^3,J^3))^{d_3}}.
	\end{align*}
	This is exactly the desired bound.

	\textbf{Case $(S,N,N)$.}
	We have the bound
	\begin{equation*}
		F_{IJ}  \lesssim \frac{|I^1|^2}{\dist(I^1,J^1)^{d_1}}
		\int_{I^{23}} \int_{J^{23}} \frac{\ud x_{23} \ud y_{23}}{\max(\dist(I^1, J^1),|x_2-y_2|)^{d_2}
			\max(\dist(I^1, J^1),|x_3-y_3|)^{d_3}}.
	\end{equation*}
	The argument then follows by using that the integral tensorizes and using \eqref{eq:nearbyint} as usual.
	(To see that this results in the right estimate one again needs to invoke $\ell(I^3) \ge \ell(I^2)$.)

	\textbf{Case $(N,S,N)$.}
	We bound it via
	\begin{equation*}
		F_{IJ}  \lesssim \frac{|I^2|^2}{\dist(I^2,J^2)^{d_2}}
		\int_{I^{13}} \int_{J^{13}} \frac{\ud x_{13} \ud y_{13}}{|x_1-y_1|^{d_1}\max(\dist(I^2, J^2),|x_3-y_3|)^{d_3}}
	\end{equation*}
	Once again, this tensorizes and each integral factor can be bounded using the standard methods.

	\textbf{Case $(N,N,S)$.}
	We have
	\begin{equation*}
		F_{IJ}  \lesssim \frac{|I^3|^2}{\dist(I^3,J^3)^{d_3}}
		\int_{I^{12}} \int_{J^{12}} \frac{\ud x_{12} \ud y_{12}}{|x_1-y_1|^{d_1}(|x_1-y_1|+|x_2-y_2|)^{d_2}}.
	\end{equation*}
	This can be bounded using Lemma \ref{lem:fullsize}, for instance.

	\textbf{Case $(N,N,N)$.}
	We have (by using \eqref{eq:nearbyint} in all parameters) that
	\begin{equation*}
		F_{IJ}  \lesssim
		\int_{I} \int_{J} \frac{\ud x \ud y}{|x_1-y_1|^{d_1}|x_2-y_2|^{d_2}|x_3-y_3|^{d_3}}
		\lesssim |I|,
	\end{equation*}
	which is the desired bound.
\end{proof}
At this point it is very easy to prove the analogs of Lemmas \ref{lem:sizeI2} and \ref{lem:sizeI1}
(there are six as there are six partial kernels).
We state these next but omit the simple proofs.
Some of these integrals have already appeared exactly in the same form before (perhaps
with different parameters) but we state even those for ease
of reference.
\begin{lem}\label{lem:i12size}
	Let $I^1, J^1 \in \calD^1$ and $I^2, J^2 \in \calD^2$ be dyadic cubes with $\ell(I^k) = \ell(J^k)$ and $I^k \ne J^k$ for $k=1,2$. Assume also that $\ell(I^1)\le \ell(I^2)$. Let $s$ be a scalar satisfying $s \ge \ell(I^2)$. Then we have that
	\begin{align*}
		F_{I^{12}J^{12}} & := \int_{I^{12}} \int_{J^{12}} \frac{\ud x_{12} \ud y_{12}}{|x_1-y_1|^{d_1} (|x_1-y_1| + |x_2-y_2|)^{d_2}(|x_1-y_1| + |x_2-y_2|+s)^{d_3}} \\
		                 & \lesssim \frac{|I^{12}|^2}{(\ell(I^1) + \dist(I^1, J^1))^{d_1} \max(\ell(I^1) + \dist(I^1, J^1), \ell(I^2) + \dist(I^2, J^2) )^{d_2}}     \\
		                 & \times \frac{1}{\max(\ell(I^1) + \dist(I^1, J^1), \ell(I^2) + \dist(I^2, J^2),s )^{d_3}}.
	\end{align*}
\end{lem}
\begin{lem}\label{lem:i13size}
	Let $I^1, J^1 \in \calD^1$ and $I^3, J^3 \in \calD^3$ be dyadic cubes with $\ell(I^k) = \ell(J^k)$ and $I^k \ne J^k$ for $k=1,3$. Assume also
	that $\ell(I^1)\le \ell(I^3)$. Let $s$ be a scalar satisfying $s \ge \ell(I^1)$. Then we have that
	\begin{align*}
		F_{I^{13}J^{13}} & := \int_{I^{13}} \int_{J^{13}} \frac{\ud x_{13} \ud y_{13}}{|x_1-y_1|^{d_1} (|x_1-y_1| + s)^{d_2}(|x_1-y_1| + |x_3-y_3|)^{d_3}} \\
		                 & \lesssim \frac{|I^{13}|^2}{(\ell(I^1) + \dist(I^1, J^1))^{d_1} \max(\ell(I^1) + \dist(I^1, J^1), s )^{d_2}}                     \\
		                 & \times \frac{1}{\max(\ell(I^1) + \dist(I^1, J^1), \ell(I^3) + \dist(I^3, J^3))^{d_3}}.
	\end{align*}
\end{lem}

\begin{lem}\label{lem:i23size}
	Let $I^2, J^2 \in \calD^2$ and $I^3, J^3 \in \calD^3$ be dyadic cubes with $\ell(I^k) = \ell(J^k)$ and $I^k \ne J^k$ for $k=2,3$. Assume also that $\ell(I^2)\le \ell(I^3)$. Then we have that
	\begin{align*}
		F_{I^{23}J^{23}} & := \int_{I^{23}} \int_{J^{23}} \frac{\ud x_{23} \ud y_{23}}{|x_2-y_2|^{d_2}(|x_2-y_2| + |x_3-y_3|)^{d_3}} \\
		                 & \lesssim
		\frac{|I^{23}|^2}{(\ell(I^2) + \dist(I^2,J^2))^{d_2} \max_{m=2,3}(\ell(I^m) + \dist(I^m,J^m))^{d_3}}.
	\end{align*}
\end{lem}

We also have the other partial size estimates such as the integral of $\size_{I^{12}}(x_3,y_3)$ and the analogous versions.
\begin{lem}\label{lem:i3size}
	Let $I^3, J^3 \in \calD^3$ be dyadic cubes with $\ell(I^3) = \ell(J^3)$ and $I^3 \ne J^3$. Then we have that
	\begin{equation*}
		F_{I^{3}J^{3}} := \int_{I^{3}} \int_{J^{3}} \frac{\ud x_{3} \ud y_{3}}{|x_3-y_3|^{d_3}}
		\lesssim \frac{|I^3|^2}{(\ell(I^3) + \dist(I^3, J^3))^{d_3}}.
	\end{equation*}
\end{lem}
\begin{lem}\label{lem:i2size}
	Let $I^2, J^2 \in \calD^2$ be dyadic cubes with $\ell(I^2) = \ell(J^2)$ and $I^2 \ne J^2$. Assume also that $s\ge \ell(I^2)$. Then we have that
	\begin{align*}
		F_{I^2J^2} & := \int_{I^2} \int_{J^2} \frac{\ud x_2 \ud y_2}{|x_2-y_2|^{d_2}(|x_2-y_2| + s)^{d_3}}                  \\
		           & \lesssim \frac{|I^2|^2}{(\ell(I^2) + \dist(I^2, J^2))^{d_2}\max(\ell(I^2) + \dist(I^2, J^2),s)^{d_3}}.
	\end{align*}
\end{lem}

\begin{lem}\label{lem:i1size}
	Let $I^1, J^1 \in \calD^1$ be dyadic cubes with $\ell(I^1) = \ell(J^1)$ and $I^1 \ne J^1$.
	Assume also that $\ell(I^1)\le s_m$, $m=1,2$. Then we have that
	\begin{align*}
		 & F_{I^{1}J^{1}} := \int_{I^{1}} \int_{J^{1}} \frac{\ud x_{1} \ud y_{1}}{|x_1-y_1|^{d_1}(|x_1-y_1|+s_1)^{d_2}(|x_1-y_1| + s_2)^{d_3}} \\
		 & \lesssim
		\frac{|I^{1}|^2}{(\ell(I^1) + \dist(I^1,J^1))^{d_1} \max(\ell(I^1) + \dist(I^1,J^1),s_1)^{d_2}\max(\ell(I^1) + \dist(I^1,J^1),s_2)^{d_3}}.
	\end{align*}
\end{lem}

Now, we are ready to prove our main tri-parameter coefficient estimate.
In our representation theorem we will have that there is always cancellation
in the third parameter -- similarly as above in the bi-parameter situation we
always had cancellation in the second parameter. So in the language
of the next theorem we could assume  $3 \in \mathscr M$, because this is all
that we will need. However, on
this result this assumption does not need to be made and there
is no real simplification on making it here, either.
\begin{thm}\label{thm:tripar coeffs}
	Suppose that $T$ is a cancellative tri-parameter flag SIO satisfying the flag WBP. Let $I, J$
	be two flag rectangles with $\ell(I) = \ell(J)$ (meaning $\ell(I^m) = \ell(J^m)$ for $m=1,2,3$). Let
	\begin{equation*}
		a_{IJ} :=
		|\langle T(h_{I^1}^{\gamma_1} \otimes h_{I^2}^{\gamma_2}\otimes h_{I^3}^{\gamma_3}),
		h_{J^1}^{\kappa_1} \otimes h_{J^2}^{\kappa_2}\otimes h_{J^3}^{\kappa_3}\rangle|.
	\end{equation*}
	Let $\mathscr{M}$ denote the set of indices $m\in\{1,2,3\}$ such that $(\gamma_m,\kappa_m)\ne (0,0)$. Then
	\begin{align*}
		a_{IJ} \lesssim \prod_{m\in\mathscr{M}} & \Big( \frac{\ell(I^m)}{\ell(I^m) + \dist(I^m, J^m)} \Big)^{\alpha_m}                  \\
		                                        & \times \frac{|I|}{\prod_{k=1}^3\max_{1\le j\le k}(\ell(I^j) + \dist(I^j,J^j))^{d_k}}.
	\end{align*}
\end{thm}

\begin{proof}
	This proof breaks into a large number of cases based on a few factors:
	\begin{enumerate}
		\item Which of the $(\gamma_i,\kappa_i)$ pairs are not $(0,0)$.
		\item For each pair $(I^i,J^i)$, whether the two cubes are separated ($\dist(I^i,J^i)>0$), neighboring ($\dist(I^i,J^i)=0$ and $I^i\ne J^i$) or equal ($I^i=J^i$). We extend the notation from Lemma \ref{lem:trifullsize} to cover the various configurations; for instance, $(N,S,=)$ denotes the case where $I^1,J^1$ are neighboring, $I^2, J^2$ are separated and $I^3=J^3$.
	\end{enumerate}
	For the first point above, the configuration of nonzero pairs is exactly what leads to the factor
	\[
		\prod_{m\in\mathscr{M}} \Big( \frac{\ell(I^m)}{\ell(I^m) + \dist(I^m, J^m)} \Big)^{\alpha_m}
	\]
	in the final estimate; the basic philosophy is that whenever $\gamma^i=0=\kappa^i$, then we will always use size bounds rather than H\"older bounds for the $i$th parameter in the kernel estimates we utilize.

	Thus, we essentially reduce our
	analysis to assuming $\gamma^i\ne 0$ and $\kappa^i=0$ for all $i=1,2,3$.
	We are left with the $27$ cases based on the configurations of $I^i$s and $J^i$s.
	Fortunately, these cases proceed very similarly to the bi-parameter case which we discussed in detail.
	We give a sketch of how to broadly approach the various cases.
	We use the notation $(=,\ne,\ne)$, for instance, to denote that $I^1=J^1,I^2\ne J^2,I^3\ne J^3$.

	\textbf{Case $(\ne,\ne,\ne)$.} In this case, we use the full kernel estimate, using the H\"older-type bounds in each parameter which has $S$. Then Lemma \ref{lem:trifullsize} addresses the contribution from integrating the flag size function. These cases go similarly to cases 1, 2, and 3 from the proof of Theorem \ref{thm:biparcoeff}. For instance, in the case of $(N,S,N)$, we estimate
	\begin{align*}
		a_{IJ} & \lesssim \frac{1}{|I|} \int_J\int_I\left( \frac{|c_{I_2}-y_2|}{|x_2-y_2|}\right)^{\alpha_2}\size_F(x,y)\ud y\ud x  \\
		       & \lesssim \frac{1}{|I|} \left(\frac{\ell(I^2)}{\dist(I^2,J^2)}\right)^{\alpha_2} \int_J\int_I\size_F(x,y)\ud y\ud x
	\end{align*}
	Then using Lemma \ref{lem:trifullsize} completes the proof.

	\textbf{Case $(\ne,\ne,=)$ and permutations of this.}
	In this case, we need to use the partial kernel estimates for $K_{I^3}$. We exploit H\"older estimates in each of the first two parameters that has an $S$ in the slot and otherwise use size estimates. For example,
	in the case $(S,N,=)$ we get the integral
	\begin{align*}
		a_{IJ} & \lesssim \frac{1}{|I|}\int_{J^{12}}\int_{I^{12}}\left( \frac{|c_{I_1}-y_1|}{|x_1-y_1|}\right)^{\alpha_1}\size_{I^3}(x_{12},y_{12})\ud y_{12}\ud x_{12}   \\
		       & \lesssim \left(\frac{\ell(I^1)}{\dist(I^1,J^1)}\right)^{\alpha_1} \frac{1}{|I|} \int_{J^{12}}\int_{I^{12}}\size_{I^3}(x_{12},y_{12})\ud y_{12}\ud x_{12}
	\end{align*}
	and some easy full kernel integral coming from the need to break up the Haars into finite
	sums of indicators of the children in the third parameter
	to be able to use the partial kernel estimates. We did the latter multiple
	times in Theorem \ref{thm:biparcoeff} and omit this remark in what follows.
	The above integral can be evaluated using Lemma \ref{lem:i12size}.

	More generally, in the case of $(=,\ne,\ne)$, we do an analogous argument based on estimates for the partial kernel $K_{I^1}$ and Lemma \ref{lem:i23size}, and in the case of $(\ne,=,\ne)$, we instead use estimates for the partial kernel $K_{I^2}$ and Lemma \ref{lem:i13size}.

	\textbf{Case $(=,\ne,=)$ and permutations of this.} Now, we use partial kernel estimates for $K_{I^{13}}$. When the second parameter is $S$, we use the H\"older bound; otherwise, we use the size bound. For example, in the case $(=,N,=)$ we bound
	\begin{align*}
		a_{IJ} & \lesssim \frac{1}{|I|}\int_{J^{2}}\int_{I^{2}}\size_{I^{13}}(x_{2},y_{2})\ud y_{2}\ud x_{2}
	\end{align*}
	which we can then estimate using Lemma \ref{lem:i2size}.

	The permutation $(=,=,\ne)$ is similar but we use the partial kernel $K_{I^{12}}$ and Lemma \ref{lem:i3size}, and the permutation $(\ne,=,=)$ uses the partial kernel $K_{I^{23}}$ and Lemma \ref{lem:i1size}.

	\textbf{Case $(=,=,=)$.} This follows from the weak boundedness property as in Theorem \ref{thm:biparcoeff}.
\end{proof}
There are obvious tri-parameter versions of Corollary \ref{cor:biparcoeffcor} and
Remark \ref{rem:biparKFcoerr}.

\section{Bi-parameter flag shifts}
\subsection{Martingale block expansions}\label{subsec:blocks}
We will expand $f$ in certain pairings $\langle f, \varphi\rangle$ using specific martingale difference blocks
either in one-parameter style or in flag style depending on the situation and the cancellation properties
of $\varphi$.

We will assume that $k = (k^1, k^2)$, $k^m \ge 0$, is fixed and that $K \in \calD^k_F$, i.e.,
$K = K^1 \times K^2 \in \calD = \calD^1 \times \calD^2$ and
$2^{-k^1}\ell(K^1) \le 2^{-k^2}\ell(K^2)$. Our basic assumption
is that $\varphi = \varphi_K$ satisfies $\varphi_K = 1_K \varphi_K$
and $\varphi_K$ is constant on $I = I^1 \times I^2 \in \calD$ with
$\ell(I^1) < 2^{-k^1}\ell(K^1)$ and $\ell(I^2) < 2^{-k^2}\ell(K^2)$.
Depending on the situation, we assume different cancellation properties.

\subsubsection*{One-parameter expansion}
Here, we assume cancellation only in the weak form $\int_{\R^d} \varphi_K = 0$.
We also fix $\ell \ge 0$ so that $2^{-k^1}\ell(K^1) = 2^{-l}2^{-k^2}\ell(K^2)$.
We expand
$$
	\langle f, \varphi_K \rangle = \Big\langle \sum_{\substack{I \in \calD \\ \ell(I^1) = 2^{-l}\ell(I^2) }}
	\Delta_I f, \varphi_K \Big\rangle,
$$
where $\Delta_I$ is, of course, a one-parameter martingale difference.
In the summation we can assume $I \cap K \ne \emptyset$. Suppose $\ell(I^1) < 2^{-k^1}\ell(K^1)$.
Then we also have
$$
	\ell(I^2) = 2^l \ell(I^1) < 2^{l-k^1}\ell(K^1) = 2^{-k^2}\ell(K^2).
$$
This implies that in this case
$$
	\langle \Delta_I f, \varphi_K \rangle = \langle \varphi_K \rangle_I \int_{\R^d} \Delta_I f = 0.
$$
So we must have $\ell(I^1) \ge 2^{-k^1}\ell(K^1)$ and $\ell(I^2) \ge 2^{-k^2}\ell(K^2)$.

Next, we do a small case study. First, assume that $k^1 \ge k^2$. If $I^1 \supsetneq K^1$, then
$\ell(I^2) > 2^l\ell(K^1) = 2^{k^1-k^2}\ell(K^2)$. In particular, also $I^2 \supsetneq K^2$
and so
$$
	\langle \Delta_I f, \varphi_K \rangle = \langle \Delta_I f \rangle_K \int_{\R^d} \varphi_K = 0.
$$
So we must have $I \subset K^{(0, k^1-k^2)}$. Assume then $k^2 \ge k^1$. If $\ell(I^1) > 2^{k^2-k^1}\ell(K^1)$
(in particular, $I^1 \supsetneq K^1$), then $\ell(I^2) > \ell(K^2)$ and again
$$
	\langle \Delta_I f, \varphi_K \rangle = \langle \Delta_I f \rangle_K \int_{\R^d} \varphi_K = 0.
$$
So must have $I \subset K^{(k^2-k^1, 0)}$. So if we here define $\tau = \tau_k
	:= (0, k^1-k^2)$ if $k^1 \ge k^2$ and $\tau := (k^2-k^1, 0)$ if $k^2 \ge k^1$, then
we always have $I \subset K^{(\tau)}$.

We have shown that
$$
	\langle f, \varphi_K \rangle = \Big\langle 1_K\sum_{\substack{I \subset K^{(\tau)} \\ \ell(I^1) = 2^{-l}\ell(I^2)
			\\ \ell(K^m) \le 2^{k^m}\ell(I^m) }}
	\Delta_I f, \varphi_K \Big\rangle =: \langle \calY_{K, k, l} f, \varphi_K \rangle.
$$

\begin{lem}\label{lem:auxsfY}
	Let
	\[
		\calY_{k, l} f := \Big(\sum_{ K\in \calD} |\calY_{K, k, l}f|^2\Big)^{1/2}.
	\]
	Then for all $p\in (1, \infty)$ and $w\in A_{p, F}$
	we have $$
		\|\calY_{k, l} f\|_{L^p(w)} \lesssim (|k|+1) \|f\|_{L^p(w)}.
	$$
\end{lem}
\begin{proof}
	By extrapolation, it suffices to only prove the $p=2$ case. To this end, we have
	\begin{align*}
		\|\calY_{k, l} f\|_{L^2(w)}^2 & = \sum_{ K\in \calD} \|\calY_{K, k, l}f\|_{L^2(w)}^2                  \\
		                              & = \sum_{R\in \calD}\sum_{K^{(\tau)}=R}\|\calY_{K, k, l}f\|_{L^2(w)}^2
		=\sum_{R\in \calD}\Big\|\sum_{ K^{(\tau)}=R}\calY_{K, k, l}f\Big\|_{L^2(w)}^2.
	\end{align*}
	Notice then that
	\begin{align*}
		\Big\|\sum_{ K^{(\tau)}=R} \calY_{K, k, l}f\Big\|_{L^2(w)}^2 & =
		\Big\|\sum_{\substack{I \subset R                                       \\ \ell(I^1) = 2^{-l}\ell(I^2) \\
				      \ell(R^m) \le 2^{\max(k^1, k^2)}\ell(I^m) }}
		\Delta_{I}f\Big\|_{L^2(w)}^2                                            \\
		                                                             & \lesssim
		\sum_{\substack{I \subset R                                             \\ \ell(I^1) = 2^{-l}\ell(I^2) \\
				\ell(R^m) \le 2^{\max(k^1, k^2)}\ell(I^m)} }
		\| \Delta_I f\|_{L^2(w)}^2.
	\end{align*}
	In the last step we used the weighted one-parameter square function lower bound
	in the scaled grid $\ell(I^1) = 2^{-l}\ell(I^2)$ -- this uses the fact
	that $w \in A_{2, F}$ implies that $w$ is $A_2$ in this scaled grid uniformly in $l$.
	Of course, we also used the fact that $\Delta_I \Delta_J f = \delta_{I, J} \Delta_{I} f$
	for $I, J$ in the scaled grid.

	Therefore, we have
	\begin{align*}
		\|\calY_{k, l} f\|_{L^2(w)}^2 & \lesssim
		\sum_{R \in \calD} \sum_{\substack{I \subset R                                                       \\ \ell(I^1) = 2^{-l}\ell(I^2) \\
				                   \ell(R^m) \le 2^{\max(k^1, k^2)}\ell(I^m)} }
		\| \Delta_I f\|_{L^2(w)}^2                                                                                                                             \\
		                              & = \sum_{I \colon \ell(I^1) = 2^{-l}\ell(I^2)}
		\| \Delta_I f\|_{L^2(w)}^2
		\sum_{\substack{ R \supset I                                                                         \\ \ell(R^m) \le 2^{\max(k^1, k^2)}\ell(I^m) }} 1 \\
		                              & \le (\max(k^1, k^2)+1)^2 \sum_{I \colon \ell(I^1) = 2^{-l}\ell(I^2)}
		\| \Delta_I f\|_{L^2(w)}^2 \lesssim (|k|+1)^2 \|f\|_{L^2(w)}^2.
	\end{align*}
	The last estimate used the weighted one-parameter square function upper bound
	in the scaled grid.
\end{proof}

\subsubsection*{Flag expansion}
Here we assume cancellation in the stronger form $\int_{\R^{d_1}} \varphi_K = \int_{\R^{d_2}} \varphi_K = 0$.
Write
$$
	\langle f, \varphi_K \rangle = \Big\langle \sum_{I \in \calD_F} \Delta_{I, F} f, \varphi_K \Big\rangle.
$$
Again, we may assume $I \cap K \ne \emptyset$. Suppose $\ell(I^2) < 2^{-k^2}\ell(K^2)$. Then
we have
$$
	\langle \Delta_{I, F} f, \varphi_K \rangle = \int_{\R^{d_1}} \langle \varphi_K \rangle_{I^2}
	\int_{\R^{d_2}} \Delta_{I, F} f = 0.
$$
So we can assume $\ell(I^2) \ge 2^{-k^2}\ell(K^2)$. Suppose $\ell(I^1) = \ell(I^2)$. Notice that then
$$
	\ell(I^1) = \ell(I^2) \ge 2^{-k^2}\ell(K^2) \ge 2^{-k^1}\ell(K^1)
$$
as $K \in \calD^k_F$. We next show that this lower bound can be assumed also when $\ell(I^1) < \ell(I^2)$.
So assume $\ell(I^1) < \ell(I^2)$ and $\ell(I^1) < 2^{-k^1}\ell(K^1)$. Then we have
$$
	\langle \Delta_{I, F} f, \varphi_K \rangle = \int_{\R^{d_2}} \langle \varphi_K \rangle_{I^1}
	\int_{\R^{d_1}} \Delta_{I, F} f = 0.
$$
So we can always assume $\ell(I^1) \ge 2^{-k^1}\ell(K^1)$.

Now, let us think about the upper bounds. Suppose $I^2 \supsetneq K^2$. Then we have
$$
	\langle \Delta_{I, F} f, \varphi_K \rangle = \int_{\R^{d_1}} \langle \Delta_{I, F} f \rangle_{K^2}
	\int_{\R^{d_2}} \varphi_K = 0.
$$
Similarly, if $I^1 \supsetneq K^1$ then
$$
	\langle \Delta_{I, F} f, \varphi_K \rangle = \int_{\R^{d_2}} \langle \Delta_{I, F} f \rangle_{K^1}
	\int_{\R^{d_1}} \varphi_K = 0.
$$
So we may assume $I \subset K$. We have showed that
$$
	\langle f, \varphi_K \rangle = \Big\langle \sum_{\substack{I \in \calD_F
			\\ I \subset K \\ \ell(K^m) \le 2^{k^m}\ell(I^m) }} \Delta_{I, F} f, \varphi_K \Big\rangle
	=: \langle \calU_{K, k}f, \varphi_K \rangle.
$$

\begin{lem}\label{lem:auxsf}
	Let
	\[
		\calU_kf := \Big(\sum_{ K\in \calD} |\calU_{K, k}f|^2\Big)^{1/2}.
	\]
	Then for all $p\in (1, \infty)$ and $w\in A_{p, F}$
	we have $$
		\|\calU_k f\|_{L^p(w)}\lesssim (|k|+1) \|f\|_{L^p(w)}.
	$$
\end{lem}
\begin{proof}
	By extrapolation it suffices to only prove the $p=2$ case. To this end, we have
	\begin{align*}
		\|\calU_k f\|_{L^2(w)}^2 & = \sum_{ K\in \calD} \|\calU_{K, k}f\|_{L^2(w)}^2 \\
		                         & \lesssim \sum_{K \in \calD}
		\sum_{\substack{I \in \calD_F
				\\ I \subset K \\ \ell(K^m) \le 2^{k^m}\ell(I^m) }} \|\Delta_{I, F} f\|_{L^2(w)}^2.
	\end{align*}
	In the last estimate
	we used the flag square function lower bound \eqref{Eq:lower}
	and the orthogonality Lemma \ref{lem:orthogonal} for flag martingales.
	Noticing that
	$$
		\sum_{\substack{K \in \calD \\ K \supset I \\ \ell(K^m) \le 2^{k^m}\ell(I^m) }} 1
		\le (k^1+1)(k^2+1) \le (|k|+1)^2,
	$$
	and using the  flag square function upper bound \eqref{Eq:upper} ends the proof.
\end{proof}

\subsection{Abstract bilinear form summation argument}
Our shifts will have various configurations. We take care of some of the easier cases first.
\begin{lem}\label{lem:formsum1}
	Let $k = (k^1, k^2)$, $k^m \ge 0$, be fixed and $\calF \subset \calD^k_F$. Let $f_K, g_K \ge 0$ be functions and
	\begin{align*}
		A_1 & := \sum_{K \in \calF} \sum_{I^{(k)} = J^{(k)} = K} \frac{1}{|K_F|} \int_J f_K \int_J g_K, \\
		A_2 & := \sum_{K \in \calF} \sum_{I^{(k)} = J^{(k)} = K} \frac{1}{|K_F|}
		\int_{J^1 \times I^2} f_K \int_{J} g_K.
	\end{align*}
	Then we have
	$$
		A_i \le \sum_{K \in \calF} \frac{|K|}{|K_F|} \int_K M_{\calD_F} f_K g_K, \qquad i = 1,2.
	$$
\end{lem}
\begin{proof}
	Notice that
	$$
		A_1 \le \sum_{K \in \calF} \sum_{I^{(k)} = J^{(k)} = K} \frac{|I|}{|K_F|}  \int_J M_{\calD_F} f_K g_K
		= \sum_{K \in \calF} \frac{|K|}{|K_F|} \int_K M_{\calD_F} f_K g_K,
	$$
	where we used that $J \in \calD_F$ and $|I| = |J|$.
	For the term $A_2$ we have
	$$
		A_2 \le \sum_{K \in \calF} \sum_{(I^1)^{(k^1)} = (J^1)^{(k^1)} = K^1} \frac{|I^1| |K^2|}{|K_F|}
		\int_{J^1 \times K^2} \langle f_K\rangle_{J^1 \times K^2} g_K
		\le \sum_{K \in \calF} \frac{|K|}{|K_F|} \int_K M_{\calD_F} f_K g_K,
	$$
	where it is of importance that $J^1 \times K^2$ trivially remains flag as $\ell(K^2)
		\ge 2^{k^2}2^{-k^1}\ell(K^1) = 2^{k^2}\ell(J^1) \ge \ell(J^1)$.
\end{proof}
Now, the factor $|K| / |K_F|$ above offers some decay when $K$ is not flag. While
we do not always need it (that is, we can sometimes afford to bound it by $1$), it will sometimes be critical to us. In the next configuration, we have to capture this decay in a more delicate way.
\begin{lem}\label{lem:formsum2}
	Let $k = (k^1, k^2)$, $k^m \ge 0$, be fixed and $\calF \subset \calD^k_F$. Let $f_K, g_K \ge 0$ be functions and
	\begin{align*}
		B := \sum_{K \in \calF} \sum_{I^{(k)} = J^{(k)} = K} \frac{1}{|K_F|} \int_I f_K \int_J g_K.
	\end{align*}
	Then we have
	$$
		B \le \sum_{K \in \calF} \int_K \wt M_{\calF} f_K g_K,
	$$
	where
	$$
		\wt M_{\calF} f := \sup_{K \in \calF} \frac{1_K}{|K_F|} \int_K |f|.
	$$
\end{lem}
\begin{proof}
	Notice simply that
	$$
		B = \sum_{K \in \calF} \frac{1}{|K_F|} \int_K f_K \int_K g_K.
	$$
\end{proof}
\begin{rem}
	Notice that $\wt M_{\calF} f \le M_{\calD_F} f$. Performing such a crude estimate essentially throws
	away the potential extra decay of $|K|/|K_F|$ somehow hidden in $\wt M_{\calF}$.
	Similarly as we can often throw away $|K|/|K_F|$, we can often (but not always) afford to do this
	maximal function estimate.
	But even in the case where we can afford to do this, it
	was important to originally have $1/|K_F|$ here and not $1/|K|$.
	The extra decay of flag kernels compared to pure bi-parameter kernels is quite important.
\end{rem}
Finally, there will be a configuration that is not necessarily good for all $\calF \subset \calD^k_F$,
but is fine if we demand more of $\calF$.
\begin{lem}\label{lem:formsum3}
	Let $k = (k^1, k^2)$, $k^m \ge 0$, be fixed and $\calF \subset \{K \in \calD \colon
		\ell(K^1) \le 2^{-k^2}\ell(K^2)\}$. Let $f_K, g_K \ge 0$ be functions and
	\begin{align*}
		C := \sum_{K \in \calF} \sum_{I^{(k)} = J^{(k)} = K} \frac{1}{|K_F|} \int_{I^1 \times J^2} f_K \int_J g_K.
	\end{align*}
	Then we have
	$$
		C \le \sum_{K \in \calF} \frac{|K|}{|K_F|} \int_K M_{\calD_F} f_K g_K.
	$$
\end{lem}
\begin{proof}
	Notice that
	$$
		C = \sum_{K \in \calF}  \sum_{(I^2)^{(k^2)} = (J^2)^{(k^2)} = K^2} \frac{|K^1| |I^2|}{|K_F|}
		\int_{K^1 \times J^2} \langle f_K \rangle_{K^1 \times J^2} g_K
		\le \sum_{K \in \calF} \frac{|K|}{|K_F|} \int_K M_{\calD_F} f_K g_K,
	$$
	where we used that $K^1 \times J^2$ is flag as $\ell(J^2) = 2^{-k^2}\ell(K^2)
		\ge \ell(K^1)$ by our standing assumption on $\calF$.
\end{proof}

Now, notice that $|K| / |K_F| = 1$ if $K$ is flag and otherwise (i.e., if $\ell(K^1) > \ell(K^2)$) we have
$$
	\frac{|K|}{|K_F|} = \Big(\frac{\ell(K^2)}{\ell(K^1)} \Big)^{d_2}.
$$
Because of the internal details of our upcoming representation theorem (related
to the kernel estimates and the available cancellation) we will end up needing
this ratio only when $K \in \calD^k_=$ and not for $K \in \calD^k_<$. Therefore, we will
estimate $|K| / |K_F| \le 1$ if $K \in \calD^k_<$. For $K \in \calD^k_=$
the point is that $\ell(K^1) \ge \ell(K^2)$ precisely when $k^1 \ge k^2$, and so
\begin{equation}\label{eq:KKFratio}
	\frac{|K|}{|K_F|} = 2^{-d_2(k^1-k^2)_+}, \qquad K \in \calD^k_=.
\end{equation}
The version of this in the context of Lemma \ref{lem:formsum2} is that
while we can always estimate $\wt M_{\calF} f \le M_{\calD_F} f$, we also have the following
non-trivial estimate with extra decay.
\begin{lem}\label{lem:KperKFmaximal}
	For all $p \in (1, \infty)$ and $w \in A_{p, F}$ there is
	$\eta = \eta([w]_{A_{p, F}}) > 0$ so that for all $k = (k^1, k^2)$ we have
	$$
		\|\wt M_{\calD^k_=} f\|_{L^p(w)} \lesssim
		2^{-\eta(k^1-k^2)_+} \|f\|_{L^p(w)}.
	$$
\end{lem}
\begin{proof}
	Notice that $\wt M_{\calD^k_=} f \le 2^{-d_2(k^1-k^2)_+} M_{\calD} f$ and so
	we have the following unweighted estimate:
	$$
		\|\wt M_{\calD^k_=} f\|_{L^p} \le 2^{-d_2(k^1-k^2)_+} \|M_{\calD} f\|_{L^p} \lesssim
		2^{-d_2(k^1-k^2)_+} \|f\|_{L^p}.
	$$
	On the other hand, we have $\wt M_{\calD^k_=} f \le M_{\calD_F} f$
	and so we have the following flag weighted estimate:
	$$
		\|\wt M_{\calD^k_=} f\|_{L^p(w)} \le \|M_{\calD_F} f\|_{L^p(w)} \lesssim \|f\|_{L^p(w)}.
	$$

	As a consequence of Proposition \ref{prop:open}, for all $w \in A_{p, F}$
	there is $\eps = \eps([w]_{A_{p, F}}) > 0$ so that also
	$$
		\|\wt M_{\calD^k_=} f \|_{L^p(w^{1+\eps} )} \lesssim \|f\|_{L^p(w^{1+\eps} )}.
	$$
	Then by interpolation with a change of measure -- see Proposition \ref{prop:interpo} below --
	we get that for $\kappa = \kappa([w]_{A_p, F}) :=\frac \eps{1+\eps}\in (0,1)$ we have
	$$
		\|\wt M_{\calD^k_=} f \|_{L^p(w)}\lesssim (2^{-d_2(k^1-k^2)_+})^\kappa  \|f\|_{L^p(w)}
		= 2^{-\eta(k^1-k^2)_+} \|f\|_{L^p(w)},
	$$
	where $\eta := \kappa d_2 \in (0, d_2)$.
	Indeed, this is obtained by using the interpolation result with
	$p_0 = p_1 = p$, $w_0 = 1$, $w_1 = w^{1+\eps}$ and $t = 1 / (1+\eps)$
	so that $w_0^{1-t}w_1^{t} = 1^{1-t}w^{t(1+\eps)} = w$.
	We are done.
\end{proof}

\begin{prop}\cite{SW}*{Theorem 2.11}\label{prop:interpo}
	Suppose that $1 \le p_0,p_1 \le \infty$ and
	let $w_0$ and $w_1$ be positive weights. Suppose that $T$ is a sublinear operator that satisfies the estimates
	$$
		\| T f  \|_{L^{p_i}(w_i)} \le M_i \| f \|_{L^{p_i}(w_i)}, \quad i=0,1.
	$$
	Let $t \in (0,1)$ and define
	$
		1/p=(1-t)/p_0+t/p_1
	$
	and $w=w_0^{p(1-t)/p_0}w_1^{pt/p_1}$. Then $T$ satisfies the estimate
	$$
		\| T f   \|_{L^{p}(w)} \le M_0^{1-t}M_1^t \| f  \|_{L^{p}(w)}.
	$$
\end{prop}

\begin{lem}\label{lem:VecKperKFmaximal}
	For all $p \in (1, \infty)$ and $w \in A_{p, F}$ there is
	$\eta = \eta([w]_{A_{p, F}}) > 0$ so that for all $k = (k^1,k^2)$ we have
	$$
		\Big\|\Big(\sum_j [\wt M_{\calD^k_=} f_j ]^2\Big)^{1/2}\Big\|_{L^p(w)} \lesssim
		2^{-\eta(k^1-k^2)_+} \Big\|\Big(\sum_j |f_j|^2\Big)^{1/2}\Big\|_{L^p(w)}.
	$$
\end{lem}
\begin{proof}
	We usually say that vector-valued estimates, such as this, just follow from extrapolation.
	That is also true here, but
	we should be slightly careful with how the decay factor with the $\eta$ (that depends on the weight) is
	affected by the extrapolation. The point is that Lemma \ref{lem:KperKFmaximal} gives that
	in $L^2(w)$ we have the claimed vector-valued estimate with
	$$
		N_1([w]_{A_{2, F}}) 2^{-d_2 \frac{\eps}{1+\eps}(k^1-k^2)_+},
	$$
	where $\eps$ is the constant from Proposition \ref{prop:open} and so we know that
	$\eps \sim [w]_{A_{2, F}}^{-1}$, and $N_1$ is some increasing function. But then
	$$
		N_2(x) := N_1(x)2^{-d_2 \frac{C}{C+x}(k^1-k^2)_+}
	$$
	is again a non-decreasing function for $x \ge 1$, and extrapolation
	theorems actually guarantee that in $L^p(w)$ the constant in front is
	$$
		N_2(c_p [w]_{A_{p, F}}^{\alpha(p)}) \lesssim 2^{-d_2 \rho (k^1-k^2)_+},
	$$
	where still $\rho = \rho([w]_{A_p}) > 0$.

\end{proof}

\subsection{Different shifts}
We will next formulate different types of shifts and then bound them using the
tools from the previous subsections.
The bilinear forms of all of our shifts will have the form
$$
	(f,g) \mapsto \sum_{K \in \calF} \sum_{I^{(k)} = J^{(k)} = K} a_{IJK} \langle f, \varphi_{I,J}\rangle
	\langle g, \psi_{I, J}\rangle,
$$
where the coefficients always satisfy the normalization
$$
	|a_{IJK}| \le \frac{|I|}{|K_F|}.
$$
The functions $\varphi_{I, J}$, $\psi_{I, J}$ will also always be supported at least in $K$ -- however,
crucially, they will also have more specific support conditions in terms of $I$, $J$, $I^1 \times J^2$
and so on. The exact conditions depend on the situation.
They will have either one-parameter or bi-parameter cancellation depending on the situation.
Moreover, they will always
satisfy the size conditions $|\varphi_{I, J}|, \, |\psi_{I, J}| \lesssim |I|^{-1/2}$
and the following constancy condition: they are constant on $L^1 \times L^2 \in \calD$
with $\ell(L^m) < \ell(I^m) = 2^{-k^m}\ell(K^m)$.

Thus, what will vary is the exact form of $\calF \subset \calD^k_F$,
the more specific support conditions of the functions $\varphi_{I, J}$, $\psi_{I,J}$
and whether they have one-parameter or bi-parameter cancellation. To focus
on the essential differences we will only specify these properties in what
follows -- the other conditions stated above are always assumed (normalization of coefficients,
size and constancy properties of the functions).

Notice that there is the bad case $\int_{I^1 \times J^2} f_K \int_J g_K$ (together with symmetric cases)
of Lemma \ref{lem:formsum3}. This case can happen, say, if it is a possibility
that $\varphi_{I, J}$ has support in $I^1 \times J^2$ and $\psi_{I, J}$ has support in $J$.
We will eventually have to make sure in our multiresolution of singular integral forms
$\langle Tf, g\rangle$ that we arrange things so that
this only happens if $\calF$ has the restriction that appears in Lemma \ref{lem:formsum3}.
Achieving this will, in particular, require using the so-called ``almost one-parameter'' shifts
$R_{k, l}$ for $l = 1, \ldots, k^1$ (the $l = 0$, i.e., $\calD^k_=$ case would always appear regardless
and it has its own independent difficulties related to the necessity of exploiting the
ratio $|K| / |K_F|$ or Lemma \ref{lem:KperKFmaximal}).

In the context of Lemma \ref{lem:formsum1}, Lemma \ref{lem:formsum2} and Lemma \ref{lem:formsum3}
we will always have by our martingale block expansions that
$f_K = |\calY_{K, k, l} f|$ (this will be the
case in the almost one-parameter shifts where $l$ is also fixed) or $f_K = |\calU_{K ,k} f|$
for some fixed function $f$ (and similarly for $g_K$). Then in both of these cases we know by
Lemma \ref{lem:auxsfY} and Lemma \ref{lem:auxsf} that
$$
	\Big\|\Big(\sum_{ K\in \calD} |f_K|^2\Big)^{1/2}\Big\|_{L^p(w)} \lesssim (|k|+1)\|f\|_{L^p(w)}
$$
for all $1 < p < \infty$ and $w \in A_{p, F}$. Then one can finish the estimates, for instance,
in the following way
\begin{align*}
	\sum_{K \in \calF} \frac{|K|}{|K_F|} \int_K M_{\calD_F} f_K g_K
	 & \le \sum_{K \in \calF} \int M_{\calD_F} f_K g_K                                 \\
	 & \le \Big\|\Big(\sum_{ K\in \calD} |M_{\calD_F} f_K|^2\Big)^{1/2}\Big\|_{L^p(w)}
	\Big\|\Big(\sum_{ K\in \calD} |g_K|^2\Big)^{1/2}\Big\|_{L^{p'}(\wt w)}             \\
	 & \lesssim (|k|+1)^2 \|f\|_{L^p(w)} \|g\|_{L^{p'}(\wt w)},
\end{align*}
where we used the vector-valued estimate of $M_{\calD_F}$ (obtained by extrapolating).
As described before, it turns out we can afford to do this if $\calF \subset
	\calD^k_<$. If $\calF = \calD^k_=$ we do essentially the same except we use \eqref{eq:KKFratio} to get
$$
	\sum_{K \in \calD^k_=} \frac{|K|}{|K_F|} \int_K M_{\calD_F} f_K g_K
	\lesssim (|k|+1)^2 2^{-d_2(k^1-k^2)_+}\|f\|_{L^p(w)} \|g\|_{L^{p'}(\wt w)}.
$$
There is also the situation of Lemma \ref{lem:formsum2}. If $\calF \subset \calD^k_<$ we
estimate $\wt M_{\calF} \le M_{\calD_F} f$, and by the argument above,
we can directly bound
$$
	\sum_{K \in \calF} \int_K \wt M_{\calF} f_K g_K
	\lesssim (|k|+1)^2 \|f\|_{L^p(w)} \|g\|_{L^{p'}(\wt w)}.
$$
If $\calF = \calD^k_=$ we use Lemma \ref{lem:VecKperKFmaximal}
to estimate
$$
	\sum_{K \in \calD^k_{=}} \int_K \wt M_{\calD^k_=} f_K g_K
	\lesssim (|k|+1)^2 2^{-\eta (k^1-k^2)_+} \|f\|_{L^p(w)} \|g\|_{L^{p'}(\wt w)},
$$
where $\eta = \eta([w]_{A_{p, F}}) > 0$.

Finally, we discuss some symmetry considerations that show that we will
never encounter situations that are not covered by Lemma \ref{lem:formsum1},
Lemma \ref{lem:formsum2} or Lemma \ref{lem:formsum3}. The right way to think about this
is to count how many of the $I^1, I^2, J^1$ and $J^2$ do not appear
in either of the two averages with $f_K$ and $g_K$.
\begin{enumerate}\label{symm for mixing}
	\item There is the case that all $4$ appear -- all such cases are, after relabeling, covered
	      by the case $\int_I f_K \int_J g_K$ of Lemma \ref{lem:formsum2}.
	\item There is the case that only $3$ of the four cubes appear -- these are all, after relabeling,
	      reduced to either the case $\int_{J^1 \times I^2} f_K \int_J g_K$ of Lemma \ref{lem:formsum1}
	      or the case $\int_{I^1 \times J^2} f_K \int_J g_K$ of Lemma \ref{lem:formsum3}.
	\item Finally, there is the case that only $2$ of them appear -- all such cases are
	      after relabeling, covered by the case $\int_J f_K \int_J g_K$ of Lemma \ref{lem:formsum1}.
\end{enumerate}
Again, only the case of Lemma \ref{lem:formsum3} (the latter case in (2) above)
is something that needs to be avoided, unless
the additional structure of $\calF$ can be guaranteed.

\subsubsection*{Almost one-parameter shifts}
Given $k = (k^1, k^2)$ and $l$, where $k^m, l\ge 0$, define $\calD(k, l) \subset \calD^k_F$ by setting
$$
	\calD(k, l) := \{K \in \calD \colon 2^{-k^1}\ell(K^1) = 2^{-l}2^{-k^2}\ell(K^2) \}.
$$
\begin{defn}\label{def:almostone}
	Let $k = (k^1, k^2)$, $k^m \ge 0$, and $l \ge 0$ be integers. An almost one-parameter
	shift $R_{k, l}$ has the above general shift form with $\calF = \calD(k, l)$ and the functions
	$\varphi_{I,J}$ and $\psi_{I,J} = \psi_J$ having one-parameter cancellation and
	satisfying $\supp \varphi_{I,J} \subset I \cup J \cup (J^1 \times I^2)$
	and $\supp \psi_{J} \subset J$. The duals of these operators also qualify.
\end{defn}
\begin{rem}
	We know exactly how the functions $\varphi_{I,J}, \psi_J$ look like in our dyadic
	representation theorem -- e.g., one possibility is that
	$\varphi_{I,J} = h_I^0 - h_J^0$ and $\psi_J = h_{J, F}$. But it does
	not appear wortwhile to list all the possible variants as long as we keep
	track they are in this abstract framework. The same goes to the upcoming
	flag shifts $Q_k$ as well.
\end{rem}

\begin{thm}\label{thm:almost}
	For all $p \in (1, \infty)$, $w \in A_{p, F}$ and $k, l$ we have
	$$
		\|R_{k, l}f\|_{L^p(w)} \lesssim (|k|+1)^2\|f\|_{L^p(w)}.
	$$
	For $l = 0$ we have the enhanced estimate
	$$
		\|R_{k, 0}f\|_{L^p(w)} \lesssim (|k|+1)^2 2^{-\eta(k^1-k^2)_+}\|f\|_{L^p(w)},
		\qquad \eta = \eta([w]_{A_{p, F}}) > 0.
	$$
\end{thm}
\begin{proof}
	By our martingale block expansions we can replace $f$ with $\calY_{K, k, l} f$
	and $g$ with $\calY_{K, k, l} g$. This gives that
	\begin{align*}
		|\langle R_{k, l}f, g\rangle| \lesssim \sum_{K \in \calD(k,l)} \sum_{I^{(k)} = J^{(k)} = K}
		\frac{1}{|K_F|} \Big(\int_I |f_K| + \int_J |f_K| + \int_{J^1 \times I^2} |f_K|  \Big)
		\int_J g_K,
	\end{align*}
	where $f_K := |\calY_{K, k, l} f|$ and $g_K = |\calY_{K, k, l} g|$.
	The result follows from our abstract shift theory discussed above.
\end{proof}

\subsubsection*{Flag shifts}
\begin{defn}\label{def:flagshift}
	Let $k = (k^1, k^2)$, $k^m \ge 0$. 	A flag shift $Q_k$ has the general shift form
	with $\calF \subset \{K \in \calD \colon \ell(K^1) < 2^{-k^2}\ell(K^2)\}$
	and the functions $\varphi_{I, J}$ and $\psi_{I, J}$ having bi-parameter cancellation and
	satisfying $\supp \varphi_{I,J}, \, \supp \psi_{I,J} \subset I \cup J \cup (J^1 \times I^2)
		\cup (I^1 \times J^2)$.
\end{defn}

\begin{thm}\label{thm:flagshifts}
	For all $p \in (1, \infty)$, $w \in A_{p, F}$ and $k$ we have
	$$
		\|Q_k f\|_{L^p(w)} \lesssim (|k|+1)^2\|f\|_{L^p(w)}.
	$$
\end{thm}
\begin{proof}
	By our martingale block expansions we can replace $f$ with $\calU_{K,k} f$
	and $g$ with $\calU_{K, k} g$. The proof is then completed as in the $R_{k, l}$
	case, since we have already discussed why Lemma \ref{lem:formsum1},
	Lemma \ref{lem:formsum2} Lemma \ref{lem:formsum3} cover all possible configurations
	of averages.
\end{proof}

\subsection{Shifts in the representation theorem}
The shift part (there will also be paraproducts) of our dyadic representation will look like
$$
	\sum_{k^1,\, k^2 = 0}^{\infty} 2^{-\alpha_1 k^1} 2^{-\alpha_2 k^2}
	\Big[ \langle Q_k f, g \rangle + \sum_{l=1}^{k^1} \langle R_{k, l}f, g\rangle \Big]
	+ \sum_{k^1,\, k^2 = 0}^{\infty} 2^{-\alpha_2 k^2} \langle R_{k, 0} f, g\rangle.
$$
The exponential decay comes from kernel estimates -- there is worse
decay associated to $\calD^k_=$ due to having no cancellation in the first parameter.
Strictly speaking for every $k$ there can be \emph{finitely} many different model
operators, but we omit the finite summations.
The first sum easily converges by the exponential decay in both of the summation variables.
The second sum is also fine -- indeed, fix $w \in A_{p, F}$
and choose $\eta > 0$ (depending on the weight) so that for all $k$
$$
	\|R_{k, 0}f\|_{L^p(w)} \lesssim (|k|+1)^2 2^{-\eta(k^1-k^2)_+}\|f\|_{L^p(w)}.
$$
As this estimate only becomes worse as $\eta$ becomes smaller we can assume $\eta < \alpha_2$.
We need to bound
\begin{align*}
	\sum_{k^2} (k^2+1)^2 2^{-\alpha_2k^2}
	\sum_{k^1 \le k^2} (k^1+1)^2
	+ \sum_{k^1} (k^1+1)^2 2^{-\eta k^1}
	\sum_{k^2 \le k^1} (k^2+1)^2 2^{(\eta-\alpha_2)k^2}.
\end{align*}
The first term is bounded by
$$
	\sum_{k^2} (k^2+1)^5 2^{-\alpha_2k^2} < \infty.
$$
The second term can be bounded by
$$
	\sum_{k^1} (k^1+1)^2 2^{-\eta k^1} \cdot \sum_{k^2} (k^2+1)^2 2^{(\eta-\alpha_2)k^2} < \infty
$$
as $\eta < \alpha_2$. So shifts are fine in the representation theorem
and their weighted boundedness passes through to the singular integral.

\section{Bi-parameter flag full paraproducts}
These types of paraproducts will appear in two separate contexts: in decompositions
of products $bf$ and, independently, in dyadic representation theorems of flag singular integrals
(related to the behaviour of $T1$, $T^*1$ and some partial adjoints).
We start with the flag paraproducts coming from products $bf$ as their boundedness
theory will also cover the representation theorem context
as we will see afterwards.
\subsection{Flag paraproduct decomposition of products}
We take our usual flag multiresolution decomposition of
bilinear forms $B$ and apply it to the
very special case $B(f, g) := fg$ in the bi-parameter flag setting. Then the appearing double sums
over $I, J$ reduce to a single sum simply due to the supports of the martingale difference and
averaging operators. We use the shorthand $\pi_{B, s} := \pi_s$, when $B(f, g) = fg$
and $s \in \calS := \{s = (s_1, s_2)\colon s_m \in \{\Delta, E, D\}\}$.
Therefore, we have the bi-parameter flag paraproduct decomposition of products:
$$
	bf = \sum_{s \in \calS} \pi_s(b, f).
$$
We do not think about $b$ and $f$ as being in symmetric positions -- $b$
will carry some $\BMO$ philosophy while $f$ will be just an $L^p$ function.
The $\BMO$ philosophy of $b$ can be exploited if there is some cancellation --
the worst case is, therefore, $s = (E, E)$. But also in the case $s = (D,E)$
the term $E_{I^1} E_{I^2} b$ appears in $\pi_s(b,f)$ for $\ell(I^1) = \ell(I^2)$.
We want to carve out that piece of $\pi_{(D, E)}(b, f)$ and integrate it into
$\pi_{(E, E)}(b, f)$. With this \textbf{redefinition},
which we have to separately remember, we have that
\begin{equation*}
	\begin{split}
		\pi_{(D, E)}(b, f) & := \sum_{I \in \calD_F} \Delta_{I^1} E_{I^2} b\Lambda_{I^1} \Delta_{I^2} f; \\
		\pi_{(E, E)}(b, f) & := \sum_{I \in \calD_F} E_{I^1} E_{I^2} b \Lambda_{I^1} \Delta_{I^2} f.     \\
	\end{split}
\end{equation*}
All the other paraproducts can be written out using our usual rules.

We categorize the paraproducts into three classes depending on how much cancellation hits
$b$. Category $I$ paraproducts have $s_1, s_2 \in \{\Delta, D\}$, and these require the least
from $b$ to define bounded operators. There are four Category I flag paraproducts.
Category $II$ paraproducts have ($s_1 = E$ and $s_2 \in \{\Delta, D\}$) or
$(s_1 \in \{\Delta, D\}$ and $s_2=E$). These four Category II paraproducts
require more of $b$ to be bounded. And then we have the worst Category $III$ paraproduct
$\pi_{(E, E)}$, which we are not interested in bounding separately (it will play a special
role in the commutator estimates that are considered separately later).
Let $\calS_1$ stand for those $s \in \calS$
that produce a Category I paraproduct. Define $\calS_2$ similarly.

\begin{thm}\label{thm:paraproducts}
	Let $1 < p < \infty$ and $w \in A_{p, F}$.
	If $s \in \calS_1$ we have that
	$$
		\|\pi_s(b, f)\|_{L^p(w)} \lesssim \|b\|_{\BMO_{\pro, F}} \|f\|_{L^p(w)}.
	$$
	If $s \in \calS_2$ we have that
	$$
		\|\pi_s(b, f)\|_{L^p(w)} \lesssim \|b\|_{\bmo_F} \|f\|_{L^p(w)}.
	$$
	In particular, for $s \in \calS_1 \cup \calS_2$ we have
	$$
		\|\pi_s(b, f)\|_{L^p(w)} \lesssim \|b\|_{\bmo_F} \|f\|_{L^p(w)}.
	$$
\end{thm}
\begin{proof}
	We deal with the category I paraproducts first. The cases $s = (\Delta, \Delta)$ and
	$s = (D, D)$ are essentially symmetric/dual to each other (with the slight difference in summation). Indeed,
	recall that
	\begin{align*}
		\langle \pi_{(\Delta, \Delta)}(b, f), g\rangle & = \sum_{I \in \calD_<} \langle b, h_{I, F}\rangle
		\langle f \rangle_I \langle g, h_{I, F} \rangle                                                \textup{ and} \\
		\langle \pi_{(D, D)}(b, f), g\rangle           & = \sum_{I \in \calD_F} \langle b, h_{I, F}\rangle
		\langle f, h_{I, F} \rangle \langle g, h_{I, F}h_{I, F} \rangle,
	\end{align*}
	where, to be rigorous, one has to, as usual, understand that there is a hidden finite summation related
	to the Haar functions.

	Therefore, for either one of these $s$, $\|\pi_s(b, f)\|_{L^p(w)} \lesssim
		\|b\|_{\BMO_{\pro, F}} \|f\|_{L^p(w)}$ follows by $H^1$-$\BMO$ duality if we show that
	$$
		\Big\|\Big( \sum_{I \in \calD_F} |\langle f, h_{I, F}\rangle|^2 \langle |g| \rangle_I^2
		\frac{1_I}{|I|} \Big)^{1/2} \Big\|_{L^1} \lesssim \|f\|_{L^p(w)} \|g\|_{L^{p'}(\wt w)},
	$$
	where $\wt w := w^{1-p'} \in A_{p', F}$. But this estimate is easy:
	\begin{align*}
		\Big\|\Big( \sum_{I \in \calD_F} |\langle f, h_{I, F}\rangle|^2 \langle |g| \rangle_I^2
		      \frac{1_I}{|I|} \Big)^{1/2} \Big\|_{L^1}
		 & \lesssim
		\Big\|\Big( \sum_{I \in \calD_F} \langle |\Delta_{I, F} f| \rangle_I^2
		      1_I \Big)^{1/2} \Big\|_{L^p(w)} \|M_{\calD_F} g\|_{L^{p'}(\wt w)} \\
		 & \lesssim
		\Big\|\Big( \sum_{I \in \calD_F} [M_{\calD_F} \Delta_{I, F} f]^2
		      \Big)^{1/2} \Big\|_{L^p(w)} \|M_{\calD_F} g\|_{L^{p'}(\wt w)}     \\
		 & \lesssim
		\Big\|\Big( \sum_{I \in \calD_F} [\Delta_{I, F} f]^2
		      \Big)^{1/2} \Big\|_{L^p(w)} \|g\|_{L^{p'}(\wt w)}                 \\
		 & = \|S_{\calD, F} f\|_{L^p(w)} \|g\|_{L^{p'}(\wt w)}
		\lesssim \|f\|_{L^p(w)} \|g\|_{L^{p'}(\wt w)}.
	\end{align*}
	Above we explicitly performed the ``from Haar SF to usual SF trick'' via
	$|\langle f, h_{I, F}\rangle|^2 = |\langle \Delta_{I, F} f, h_{I, F}\rangle|^2
		\le \langle |\Delta_{I, F} f| \rangle_I^2 |I|$.

	We turn to the more interesting paraproducts.
	We now look at the Haar sum related to $\langle \pi_{(D, \Delta)}(b,f), g\rangle$, namely,
	$$
		\sum_{I \in \calD_F} \langle b, h_{I, F}\rangle \Big\langle f, u_{I^1} \otimes \frac{1_{I^2}}{|I^2|}
		\Big\rangle \langle g, u_{I^1}u_{I^1} \otimes h_{I^2}\rangle.
	$$
	Again, it is enough to bound
	$$
		\Big\|\Big(\sum_{I \in \calD_F}
		\Big|\Big\langle f, u_{I^1} \otimes \frac{1_{I^2}}{|I^2|} \Big\rangle\Big|^2
		\langle |\Delta_{I^2} g| \rangle_I^2 \frac{1_I}{|I^1|}
		\Big)^{1/2} \Big\|_{L^1}.
	$$
	We first consider the cancellative case
	\begin{align*}
		\Big\|\Big(\sum_{I \in \calD_F}
		      \Big|\Big\langle f, h_{I^1} & \otimes \frac{1_{I^2}}{|I^2|} \Big\rangle\Big|^2
		      \langle |\Delta_{I^2} g| \rangle_I^2 \frac{1_I}{|I^1|}
		      \Big)^{1/2} \Big\|_{L^1} \\
		 & \le
		\Big\|\Big(\sum_{I \in \calD_F}
		\langle |\Delta_{I^1}E^2_{\ell(I^1)} f| \rangle_I^2
		\langle |\Delta_{I^2} g| \rangle_I^2 1_I
		\Big)^{1/2} \Big\|_{L^1}.
	\end{align*}
	Here $E_s^2 f = \sum_{J^2 \colon \ell(J^2) = s} E_{J^2} f$, and we used that
	$\langle f \rangle_{V^2} = \langle E_s^2 f\rangle_{V^2}$ for any $s \le \ell(V^2)$.
	We then estimate this up by
	\begin{align*}
		\Big\|\Big(\sum_{I \in \calD}
		M_{\calD_F}(\Delta_{I^1} & E^2_{\ell(I^1)} f)^2
		M_{\calD_F}(\Delta_{I^2} g)^2
		\Big)^{1/2} \Big\|_{L^1}                        \\
		                         & =
		\Big\|\Big(\sum_{I^1} M_{\calD_F}(\Delta_{I^1}E^2_{\ell(I^1)} f)^2\Big)^{1/2}
		\Big( \sum_{I^2} M_{\calD_F}(\Delta_{I^2} g)^2
		\Big)^{1/2} \Big\|_{L^1}                        \\
		                         & \lesssim
		\Big\|\Big(\sum_{I^1} |\Delta_{I^1}E^2_{\ell(I^1)} f|^2\Big)^{1/2} \Big\|_{L^p(w)}
		\Big\|\Big( \sum_{I^2} |\Delta_{I^2} g|^2 \Big)^{1/2} \Big\|_{L^{p'}(\wt w)}.
	\end{align*}
	Here the $g$ term is clearly OK, since $x_2 \mapsto \wt w(x_1, x_2)$
	is in $A_{p'}(\R^{d_2})$ uniformly in $x_1$. The $f$ term is also OK, since
	$$
		\Big(\sum_{I^1} |\Delta_{I^1}E^2_{\ell(I^1)} f|^2\Big)^{1/2}
		= \Big(\sum_{L \in \calD_=} |\Delta_{L^1}E_{L^2} f|^2\Big)^{1/2}
	$$
	is a version of the one-parameter square function in $\calD_=$
	and flag weights are also one-parameter weights in the full product space $\R^d$.


	To finish the estimates related to $\pi_{(D, \Delta)}$ it remains to look at
	the non-cancellative case
	\begin{align*}
		\Big\|\Big(\sum_{I \in \calD_=}
		      \langle |f|\rangle_I^2
		      \langle |\Delta_{I^2} g| \rangle_I^2 1_I
		      \Big)^{1/2} \Big\|_{L^1}
		 & \le
		\Big\|M_{\calD_F} f \Big(\sum_{I \in \calD_=}
		                    M_{\calD_F}(\Delta_{I^2} g)^2 1_{I^1}
		\Big)^{1/2} \Big\|_{L^1},
	\end{align*}
	where
	$$
		\sum_{I \in \calD_=} M_{\calD_F}(\Delta_{I^2} g)^2 1_{I^1}
		= \sum_{I^2} M_{\calD_F}(\Delta_{I^2} g)^2 \sum_{I^1 \colon \ell(I^1) = \ell(I^2)} 1_{I^1}
		= \sum_{I^2} M_{\calD_F}(\Delta_{I^2} g)^2.
	$$
	It is now clear that this non-cancellative part is also in control.
	The paraproduct $\pi_{(\Delta, D)}$ is essentially symmetric -- except
	it does not have the non-cancellative part of the above paraproduct
	as we work directly in $\calD_<$.

	It is time to look at the ``little $\BMO$'' style, i.e., Category $II$ paraproducts.
	The cases $s = (E, \Delta)$ and $s = (E, D)$ are dual, so we look at the following Haar form
	$$
		\Big| \sum_{I \in \calD_<} \Big\langle b, \frac{1_{I^1}}{|I^1|} \otimes h_{I^2} \Big\rangle
		\Big\langle f, h_{I^1} \otimes \frac{1_{I^2}}{|I^2|} \Big\rangle
		\langle g, h_{I^1} \otimes h_{I^2} \rangle \Big|
	$$
	coming from $\pi_{(E, \Delta)}(b, \cdot)$.
	We abbreviate $a_I := \Big\langle f, h_{I^1} \otimes \frac{1_{I^2}}{|I^2|} \Big\rangle
		\langle g, h_{I^1} \otimes h_{I^2} \rangle$.
	By using \eqref{eq:bmo2} we get that our Haar sum can be dominated by
	$$
		\Big\| \Big( \sum_{I^2} \Big| \sum_{I^1\colon I \in \calD_{<}} a_I
		\frac{1_{I^1}}{|I^1|} \Big|^2 \frac{1_{I^2}}{|I^2|}  \Big)^{1/2}\Big\|_{L^1}.
	$$
	Using that $|a_I| \le \langle | \Delta_{I^1} E^2_{\ell(I^1)} f|\rangle_I \langle |\Delta_{I, F} g|\rangle_I
		|I^1| |I^2|^{1/2}$ we can bound this by
	\begin{align*}
		\Big\| \Big( & \sum_{I^2} \Big[  \sum_{I^1\colon I \in \calD_{<}}
			                          M_{\calD_F}(\Delta_{I^1} E^2_{\ell(I^1)} f) M_{\calD_F}(\Delta_{I, F} g)
			                          \Big]^2 \Big)^{1/2}\Big\|_{L^1} \\
		             & \le
		\Big\| \Big( \sum_{I^2} \Big[ \sum_{I^1}
			                        M_{\calD_F}(\Delta_{I^1} E^2_{\ell(I^1)} f)^2 \Big] \Big[ \sum_{I^1\colon I \in \calD_{<}}
			                                                                             M_{\calD_F}(\Delta_{I, F} g)^2 \Big]
		\Big)^{1/2}\Big\|_{L^1}                                                                            \\
		             & =
		\Big\| \Big( \sum_{I^1}
		M_{\calD_F}(\Delta_{I^1} E^2_{\ell(I^1)} f)^2 \Big)^{1/2}
		\Big( \sum_{I \in \calD_{<}}
		M_{\calD_F}(\Delta_{I, F} g)^2 \Big)^{1/2}
		\Big\|_{L^1}                                                                                       \\
		             & \lesssim
		\Big\| \Big( \sum_{I^1}
		|\Delta_{I^1} E^2_{\ell(I^1)} f|^2 \Big)^{1/2} \Big\|_{L^p(w)}
		\Big\| \Big( \sum_{I \in \calD_{<}}
		       |\Delta_{I, F} g|^2 \Big)^{1/2}
		\Big\|_{L^{p'}(\wt w)} \lesssim \|f\|_{L^p(w)} \|g\|_{L^{p'}(\wt w)},
	\end{align*}
	where the last step used the estimate for $f$ from above (see the case $s = (D, \Delta)$)
	and the $g$ term was just the known flag square function estimate.

	Finally, we need to look at the Haar form related to $\pi_{s}$ with $s = (D, E)$, that is,
	$$
		\Big| \sum_{I \in \calD_F} \Big\langle b, h_{I^1} \otimes \frac{1_{I^2}}{|I^2|} \Big\rangle
		\langle f, u_{I^1} \otimes h_{I^2} \rangle \langle g, h_{I^1}u_{I^1} \otimes h_{I^2} \rangle
		\Big|.
	$$
	Recall that this is one of the two paraproducts that were redefined previously.
	Also, the paraproduct with $s = (\Delta, E)$ is essentially dual to this, but does
	not have the potentially non-cancellative $u$ functions and the sum is over $\calD_<$.

	We write $a_I := \langle f, u_{I^1} \otimes h_{I^2} \rangle \langle g, h_{I^1}u_{I^1} \otimes h_{I^2} \rangle$
	so that applying \eqref{eq:bmo1} we need to bound
	$$
		\Big\|\Big( \sum_{J \in \calD_{=}} \Big[ \sum_{I \in \calD_F} |a_I|
			\Big|\Delta_J\Big(h_{I^1} \otimes \frac{1_{I^2}}{|I^2|} \Big)\Big| \Big]^2 \Big)^{1/2}\Big\|_{L^1}.
	$$
	Notice that $\Delta_J = \Delta_{J^1}\Delta_{J^2} + E_{J^1} \Delta_{J^2} + \Delta_{J^1} E_{J^2}$.
	The first two parts of $\Delta_J$ actually do not contribute anything in our special situation.
	Indeed, notice first that
	$$
		\langle h_{I^1} \otimes 1_{I^2}, h_{J^1} \otimes h_{J^2} \rangle \ne 0
	$$
	forces (as $\langle h_{I^1}, h_{J^1} \rangle \ne 0$) that $I^1 = J^1$ and so $\ell(I^2) \ge \ell(I^1) = \ell(J^1)$.
	But $\langle 1_{I^2}, h_{J^2}\rangle \ne 0$ forces $\ell(I^2) < \ell(J^2) = \ell(J^1)$. So this always vanishes,
	and $\Delta_{J^1}\Delta_{J^2}$ does not contribute here. To see that, critically, $E_{J^1}\Delta_{J^2}$
	does not contribute either, notice that
	$$
		\langle h_{I^1} \otimes 1_{I^2}, 1_{J^1} \otimes h_{J^2} \rangle \ne 0
	$$
	implies $\ell(J^1) < \ell(I^1)$ and $\ell(I^1) \le \ell(I^2) < \ell(J^2) = \ell(J^1)$, so this always
	vanishes.

	Now, notice that
	$$
		\Big|\Delta_{J^1}  E_{J^2} \Big(h_{I^1} \otimes \frac{1_{I^2}}{|I^2|} \Big)\Big|
		= \delta_{I^1, J^1} \Big\langle \frac{1_{I^2}}{|I^2|}, \frac{1_{J^2}}{|J^2|}\Big\rangle
		\frac{1_{J^1}}{|J^1|^{1/2}} \otimes 1_{J^2}.
	$$
	In our situation we have $\ell(I^2) \ge \ell(I^1) = \ell(J^1) = \ell(J^2)$ so $I^2 \supset J^2$.
	In particular, we always have
	$$
		\Big\langle \frac{1_{I^2}}{|I^2|}, \frac{1_{J^2}}{|J^2|}\Big\rangle = \frac{1}{|I^2|}.
	$$
	Therefore, we are left with bounding
	$$
		\Big\|\Big( \sum_{J \in \calD_{=}} \Big[ \sum_{I^2 \colon I^2 \supset J^2}
		\langle |\Delta_{J^1 \times I^2, F} f| \rangle_{J^1 \times I^2}
		\langle |\Delta_{I^2} g| \rangle_{J^1 \times I^2} \Big]^2 1_J \Big)^{1/2}\Big\|_{L^1}
	$$
	We notice that $J \subset J^1 \times I^2$
	and that $J^1 \times I^2$ is flag. Therefore, we can dominate up by
	\begin{align*}
		\Big\|\Big( \sum_{J \in \calD_{=}} \Big[  \sum_{I^2 \colon I^2 \supset J^2} &
			                                   \langle |\Delta_{J^1 \times I^2, F} f| \rangle_{J^1 \times I^2}^2 1_J \Big]
		\Big[ \sum_{I^2} M_{\calD_F}(\Delta_{I^2} g)^2 \Big] \Big)^{1/2}\Big\|_{L^1} \\
		 & \lesssim
		\Big\|\Big( \sum_{J \in \calD_{=}} \sum_{I^2 \colon I^2 \supset J^2}
		      \langle |\Delta_{J^1 \times I^2, F} f|\rangle_{J^1 \times I^2}^2
		      1_{J} \Big)^{1/2}\Big\|_{L^p(w)}\|g\|_{L^{p'}(\wt w)},
	\end{align*}
	where we used the maximal function estimate and the square function estimate in the second parameter to handle the $g$ term.
	It remains to notice that
	\begin{align*}
		\sum_{J \in \calD_{=}} \sum_{I^2 \colon I^2 \supset J^2}
		\langle |\Delta_{J^1 \times I^2, F} f|\rangle_{J^1 \times I^2}^2
		                                      1_{J}
		 & = \sum_{J^1 \times I^2 \in \calD_F}
		\langle |\Delta_{J^1 \times I^2, F} f|\rangle_{J^1 \times I^2}^2 1_{J^1}
		\sum_{\substack{ J^2\colon \ell(J^2) = \ell(J^1)                          \\ J^2 \subset I^2}} 1_{J^2} \\
		 & = \sum_{J^1 \times I^2 \in \calD_F}
		\langle |\Delta_{J^1 \times I^2, F} f|\rangle_{J^1 \times I^2}^2 1_{J^1 \times I^2}                    \\
		 & \le \sum_{J^1 \times I^2 \in \calD_F} M_{\calD_F}(\Delta_{J^1 \times I^2, F} f)^2
	\end{align*}
	and use the vector-valued $M_{\calD_F}$ estimate and the flag square function estimate to dominate the $f$ term
	by $\|f\|_{L^p(w)}$.
\end{proof}

\subsection{Full paraproducts in the representation theorem}\label{sec:fullpararep}
Importantly, only Category I full paraproducts will appear in the representation theorem.
The proof of the representation theorem starts by applying our usual flag multiresolution
to the bilinear form $B(f, g) := \langle Tf, g\rangle$, where $T$ is an appropriate flag SIO.
So we will have
we have
\begin{equation*}
	B(f, g) = \sum_{s \in \calS} \Sigma_s,
\end{equation*}
where $\calS = \{(s_1, s_2) \colon s_i \in \{\Delta, E, D\}\}$ and our usual rules are followed --
for instance,
$$
	\Sigma_{(\Delta, D)} =
	\sum_{\substack{I,J\in \calD_<\\ \ell(I)=\ell(J)}}
	B(\Delta_{I^1} \Delta_{I^2} f, E_{J^1} \Delta_{J^2 }  g).
$$
We will have only $4$ full paraproducts in the representation theorem,
but there are $9$ different basic symmetries here.
So where exactly do the full paraproducts arise? Well, at least in $\calD_<$, the requirement is that there
is an averaging operator somewhere in the first parameter and in the second parameter.
So, first of all, we will get one full paraproduct from each $\Sigma_{(s_1, s_2)}$, where $s_m \ne D$ for $m=1,2$.
These will all be defined just in $\calD_<$. That is not quite the full story -- indeed,
we can also have too many averaging operators, but only in $\calD_=$,
in $\Sigma_{(D, s_2)}$, $s_2 \in \{\Delta, E\}$ -- and these turn out
to produce paraproducts in $\calD_=$.

We save the details to the proof of the representation theorem and just summarize here.
It turns out that $\Sigma_{(E, E)}$ produces the full paraproduct
$$
	\sum_{I \in \calD_<} \langle T1, h_{I, F}\rangle \langle f \rangle_I h_{I, F}.
$$
The $\calD_=$ part of this comes from $\Sigma_{(D, E)}$. In total they produce

$$
	\sum_{I \in \calD_F} \langle T1, h_{I, F}\rangle \langle f \rangle_I h_{I, F}.
$$
This is essentially $\pi_{(D, D)}^*(T1, f)$ with the immaterial difference
that here we have $1_I / |I|$ while there we have $h_{I, F} h_{I, F}$ (of
which we only ever use the fact the absolute value of that is $1_I / |I|$).
So this is surely bounded if $T1 \in \BMO_{\pro, F}$ by Theorem \ref{thm:paraproducts}.
But this is perhaps unwieldy notation in the current context. So we will
just call this $\Pi_1(T1, f)$. Similarly, then, $\Sigma_{(\Delta, \Delta)}$
and $\Sigma_{(D, \Delta)}$ together produce $\Pi_1^*(T^*1, f)$.

Next, $\Sigma_{(\Delta, E)}$ produces the full paraproduct
$$
	\sum_{I \in \calD_<} \langle T_1 1, h_{I, F} \rangle
	\Big\langle f, h_{I^1} \otimes \frac{1_{I^2}}{|I^2|} \Big\rangle
	\frac{1_{I^1}}{|I^1|} \otimes h_{I^2}.
$$
This is essentially like $\pi_{(\Delta, D)}^*(T_11, f)$ with the same minor difference as above
-- $h_{I^2}h_{I^2}$ versus directly having $1_{I^2}/|I^2|$. So this is again surely bounded if $T_1 1 \in \BMO_{\pro, <}$.
Similarly as above we just call this $\Pi_2(T_11, f)$.
Finally, $\Sigma_{(E, \Delta)}$ produces $\Pi_2^*(T_1^*1, f)$. So the full paraproduct
part of the representation theorem will be
$$
	\langle \Pi_1(T1, f) + \Pi_1^*(T^*1, f) + \Pi_2(T_11, f) + \Pi_2^*(T_1^*1, f), g\rangle.
$$
There is the interesting feature that the paraproducts associated with $T1$ and $T^*1$
are defined in $\calD_F$ while the paraproducts associated with $T_1 1, T_1^* 1$ are
defined in $\calD_<$.

\section{Bi-parameter flag partial paraproducts}
These are hybrid operators of shifts and paraproducts, first introduced in the pure bi-parameter
setting in \cite{Ma1}. They are usually the toughest model operator to handle and, in fact,
have never been successfully bounded with optimal weights in entangled situations. Indeed, even the
recent result in the Zygmund dilation setting \cite{HLMV} was for cancellative
singular integrals because the partial paraproducts were too tough; the full paraproducts
in the Zygmund setting were later handled in \cite{ALM24}. In this paper we are finally
able to handle partial paraproducts in our entangled-like setting, the flag setting.
\begin{defn}\label{def:partialp}
	Flag partial paraproducts take two different forms. In the first variant the shift is
	in the first parameter:
	\begin{equation*}
		\langle P_{k^1}^1 f, g\rangle :=
		\sum_{K\in \calD_<}\sum_{(I^1)^{(k^1)}=(J^1)^{(k^1)}=K^1}
		a_{KI^1J^1}\langle f, H_{I^1, J^1}\otimes h_{K^2}\rangle \Big\langle g, H_{I^1, J^1}\otimes \frac{1_{K^2}}{|K^2|}\Big\rangle,
	\end{equation*}
	where for all such fixed $K^1, I^1, J^1$ we have
	\[
		\|(a_{KI^1J^1})_{K^2}\|_{\BMO(\calD^2)}
		\le \frac{|I^1|}{|K^1|}.
	\]
	In the second variant the shift is in the second parameter:
	\begin{equation*}
		\langle P_{k^2}^2 f, g\rangle :=
		\sum_{K \in \calD^{(0, k^2)}_< }\sum_{(I^2)^{(k^2)}=(J^2)^{(k^2)}=K^2}
		a_{KI^2J^2}\langle f, h_{K^1}\otimes H_{I^2, J^2}\rangle \Big\langle g, \frac{1_{K^1}}{|K^1|}\otimes H_{I^2, J^2}\Big\rangle,
	\end{equation*}
	where for all such fixed $K^2, I^2, J^2$ we have
	\[
		\|(a_{KI^2J^2})_{K^1}\|_{\BMO(\calD^1)}
		\le \frac{|I^2|}{|K^2|}.
	\]
	Of course, one can also interchange $h_{K^2}$ and $1_{K^2}/|K^2|$ and
	$h_{K^1}$ and $1_{K^1}/|K^1|$.
\end{defn}
\begin{rem}
	This is a bit subtle. It is critical here that in $P_{k^1}^1$ we are summing over $\calD_<$
	(which is the $k^2=0$ version of the condition $\ell(K^1) < 2^{-k^2}\ell(K^2)$
	appearing in flag shifts $Q_k$).
	In the other partial paraproduct we have the more obvious
	condition $\calD^{(0, k^2)}_<$, which is the $k^1 = 0$ variant of $\calD^k_<$. Now, it is also importat that all the $H$ functions
	are cancellative here (we are not in some ``$=$'' type grid). These restrictions
	are achieved by our use of the almost one-parameter shifts in our dyadic multiresolution of
	singular integrals. It turns out that even with the key restriction $\calD_<$ in the
	partial paraproducts $P^1_{k^1}$, they are still subtly more difficult than the variant $P_{k^2}^2$.
	As we will see this is essentially because $M_{\calD^2}$ is fine while $M_{\calD^1}$ is not
	(as $x_2 \mapsto w(x_1, x_2)$ is in $A_p$ for flag weights but not the other way around).

	Finally, just like in our shifts, we could possibly abstractify further what $f$ and $g$
	are paired against. And already the current form is more general
	than what actually appears --
	in the partial paraproducts that arise in the representation theorem, there
	is only one true $H$ function (and the other one is a basic cancellative Haar).
	This formalism (where, say, the two $H_{I^1, J^1}$ do not have to be the same)
	simply makes this symmetric so that we do not have to consider different cases.
\end{rem}

\begin{thm}\label{thm:partial-para}
	Suppose that $1<p<\infty$ and $w\in A_{p, F}$. Then we have
	\[
		\|P_{k^i}^i f\|_{L^p(w)} \lesssim (k^i+1) \|f\|_{L^p(w)},\qquad i=1,2.
	\]
\end{thm}
\begin{proof}
	We start with the harder $i=1$ case and, at the end, comment on the
	$i=2$ case. Let $P = P^1_{k^1}$.
	Using the $H^1$--$\BMO$ duality we have
	\begin{align*}
		|\langle Pf, g\rangle| \lesssim \sum_{K^1} \sum_{(I^1)^{(k^1)}=(J^1)^{(k^1)}=K^1}
		\frac{|I^1|}{|K^1|}\Big\|
		\Big( & \sum_{K^2\colon \ell(K^2)>\ell(K^1)}|\langle f, H_{I^1, J^1}\otimes h_{K^2}\rangle|^2 \\
		      & \times \Big|\Big\langle g, H_{I^1, J^1}\otimes
		\frac{1_{K^2}}{|K^2|}\Big\rangle\Big|^2
		\frac{1_{K^2}}{|K^2|}\Big)^{\frac 12}\Big\|_{L^1(\R^{d_2})}.
	\end{align*}
	Using the results from Section \ref{subsec:blocks} we have
	$$
		|\langle f, H_{I^1, J^1}\otimes h_{K^2}\rangle|
		\le  |I^1|^{\frac 12}|K^2|^{\frac12}
		\big(\bla f_K \bra_{I^1\times K^2}+\bla f_K \bra_{J^1\times K^2}\big),
	$$
	where
	$$
		f_K := |\calU_{K, (k^1, 0)} f| = \Big|
		\sum_{\substack{ L = L^1 \times K^2 \in \calD_F \\ L^1 \subset K^1 \\ \ell(K^1) \le 2^{k^1}\ell(L^1) }}
		\Delta_{L, F} f \Big|
	$$
	and
	$$
		\Big\|\Big(\sum_{ K\in \calD} |f_K|^2\Big)^{1/2}\Big\|_{L^p(w)} \lesssim (k^1+1)^{1/2}\|f\|_{L^p(w)}.
	$$
	We also have
	\begin{align*}
		\Big\langle g, H_{I^1, J^1}\otimes \frac{1_{K^2}}{|K^2|}\Big\rangle
		 & = \Big\langle \sum_{\substack{ L^1 \subset K^1          \\ \ell(L^1) \ge 2^{-k^1}\ell(K^1) }}
		\Delta_{L^1} g, H_{I^1, J^1}\otimes
		\frac{1_{K^2}}{|K^2|}\Big\rangle                                                                 \\
		 & =
		\Big\langle \sum_{\substack{ L^1 \subset K^1               \\ \ell(L^1) \ge 2^{-k^1}\ell(K^1) }}
		\sum_{L^2 \colon \ell(L^2) = \ell(L^1) }
		\Delta_{L^1}E_{L^2} g, H_{I^1, J^1}\otimes
		\frac{1_{K^2}}{|K^2|}\Big\rangle                                                                 \\
		 & =
		\Big\langle \sum_{\substack{ L \in \calD_=                 \\ L^1 \subset K^1 \\ \ell(K^1) \le 2^{k^1}\ell(L^1) }}
		\Delta_{L^1}E_{L^2} g, H_{I^1, J^1}\otimes
		\frac{1_{K^2}}{|K^2|}\Big\rangle                                                                 \\
		 & =: \Big \langle \calW_{K^1, k^1} g, H_{I^1, J^1}\otimes
		\frac{1_{K^2}}{|K^2|}\Big\rangle.
	\end{align*}
	Here the averaging was inserted using $\ell(L^2) = \ell(L^1)  \le \ell(K^1) < \ell(K^2)$.
	Therefore, we have
	\begin{align*}
		\Big|\Big\langle g, H_{I^1, J^1}\otimes \frac{1_{K^2}}{|K^2|}\Big\rangle\Big|
		\le  |I^1|^{\frac 12}\big(\bla g_{K^1}\bra_{I^1\times K^2}
		+\bla g_{K^1} \bra_{J^1\times K^2}\big), \qquad g_{K^1} := |\calW_{K^1, k^1} g|.
	\end{align*}
	We can easily show using our techniques from Section \ref{subsec:blocks} that
	$$
		\Big\|\Big(\sum_{ K^1 \in \calD^1} |g_{K^1}|^2\Big)^{1/2}\Big\|_{L^{p'}(\wt w)}
		\lesssim (k^1+1)^{1/2}\|g\|_{L^{p'}(\wt w)}.
	$$
	Notice that combining all of the above we have
	\begin{align*}
		|\langle f, H_{I^1, J^1} & \otimes h_{K^2}\rangle|
		\Big|\Big\langle g, H_{I^1, J^1}\otimes \frac{1_{K^2}}{|K^2|}
		\Big\rangle\Big| \frac{1_{K^2}}{|K^2|^{1/2}}                                                                                  \\
		                         & \le |I^1|
		\big(\bla f_K \bra_{I^1\times K^2}+\bla f_K \bra_{J^1\times K^2}\big)
		\big(\bla g_{K^1}\bra_{I^1\times K^2}
		+\bla g_{K^1} \bra_{J^1\times K^2}\big) 1_{K^2}                                                                               \\
		                         & = \int_{I^1} \langle f_K \rangle_{I^1 \times K^2} \langle g_{K^1} \rangle_{I^1 \times K^2} 1_{K^2}
		+ \int_{J^1} \langle f_K \rangle_{J^1 \times K^2} \langle g_{K^1} \rangle_{J^1 \times K^2} 1_{K^2}                            \\
		                         & + \langle f_K \rangle_{I^1 \times K^2} \int_{J^1} \langle g_{K^1} \rangle_{K^2} 1_{K^2}
		+ \langle f_K \rangle_{J^1 \times K^2} \int_{I^1} \langle g_{K^1} \rangle_{K^2} 1_{K^2}.
	\end{align*}
	The first and the second term are symmetric and so are the third and the fourth.
	We estimate
	$$
		\int_{I^1} \langle f_K \rangle_{I^1 \times K^2} \langle g_{K^1} \rangle_{I^1 \times K^2} 1_{K^2}
		\le \int_{I^1} M_{\calD_F} f_K M_{\calD_F} g_{K^1}.
	$$
	Now, we bound the corresponding part of the square function using Minkowski's integral inequality
	as follows:
	\begin{align*}
		\sum_{K^1} \sum_{(I^1)^{(k^1)}=(J^1)^{(k^1)}=K^1}
		 & \frac{|I^1|}{|K^1|}\Big\|
		\Big( \sum_{K^2} \Big[
		\int_{I^1} M_{\calD_F} f_K M_{\calD_F} g_{K^1}\Big]^2
		\Big)^{\frac 12}\Big\|_{L^1(\R^{d_2})}                                        \\
		 & \le
		\sum_{K^1} \sum_{(I^1)^{(k^1)}=(J^1)^{(k^1)}=K^1}
		\frac{|J^1|}{|K^1|} \int_{\R^{d_2}}
		\int_{I^1} \Big(\sum_{K^2} [M_{\calD_F} f_K]^2 \Big)^{1/2} M_{\calD_F} g_{K^1}              \\
		 & = \sum_{K^1} \int_{\R^{d_2}}
		\int_{K^1} \Big(\sum_{K^2} [M_{\calD_F} f_K]^2 \Big)^{1/2} M_{\calD_F} g_{K^1}              \\
		 & \le
		\Big\|\Big(\sum_{ K\in \calD} |M_{\calD_F}f_K|^2\Big)^{1/2}\Big\|_{L^p(w)}
		\Big\|\Big(\sum_{ K^1 \in \calD^1} |M_{\calD_F} g_{K^1}|^2\Big)^{1/2}\Big\|_{L^{p'}(\wt w)} \\
		 & \lesssim (k^1+1)\|f\|_{L^p(w)} \|g\|_{L^{p'}(\wt w)}.
	\end{align*}

	Next, we consider the remaining part with
	$$
		\langle f_K \rangle_{I^1 \times K^2} \int_{J^1} \langle g_{K^1} \rangle_{K^2} 1_{K^2}.
	$$
	It would seem natural to estimate
	$\int_{J^1} \langle g_{K^1} \rangle_{K^2} 1_{K^2}
		= \int_{J^1} \langle g_{K^1} \rangle_{J^1 \times K^2} 1_{K^2}
		\le \int_{J^1} M_{\calD_F} g_{K^1}$. Essentially, this type of an estimate
	will be good enough in $P^2_{k^2}$, but would lead to trouble later on
	if we did it here. Instead, we estimate, using $\ell(K^1) < \ell(K^2)$, that
	\begin{align*}
		\langle g_{K^1} \rangle_{K^2} 1_{K^2} & = \sum_{\substack{ H^2 \colon \ell(H^2) = \ell(K^1)
				                                          \\  H^2 \subset K^2}} \langle g_{K^1} \rangle_{K^2} 1_{H^2}          \\
		                                      & = \sum_{\substack{H^2 \colon \ell(H^2) = \ell(K^1)  \\ H^2 \subset K^2}}
		\langle \langle g_{K^1} \rangle_{K^2} 1_{K^2} \rangle_{H^2} 1_{H^2}                                              \\
		                                      & \le  \sum_{H^2 \colon \ell(H^2) = \ell(K^1)}
		\langle \langle g_{K^1} \rangle_{K^2} 1_{K^2} \rangle_{H^2} 1_{H^2}                                              \\
		                                      & \le  \sum_{H^2 \colon \ell(H^2) = \ell(K^1)}
		\langle M_{\calD^2} g_{K^1} \rangle_{H^2} 1_{H^2}
		= E_{\ell(K^1)}^2 M_{\calD^2} g_{K^1}.
	\end{align*}
	Now, we bound the corresponding part of the square function as follows:
	\begin{align*}
		\sum_{K^1} & \sum_{(I^1)^{(k^1)}=(J^1)^{(k^1)}=K^1}
		\frac{|I^1|}{|K^1|}\Big\|
		\Big( \sum_{K^2} \Big[
			                 \langle f_K \rangle_{I^1 \times K^2}  \int_{J^1}
			                 E_{\ell(K^1)}^2 M_{\calD^2} g_{K^1} \Big]^2 1_{K^2}
		\Big)^{\frac 12}\Big\|_{L^1(\R^{d_2})}                           \\
		           & = \sum_{K^1} \sum_{(I^1)^{(k^1)}=(J^1)^{(k^1)}=K^1}
		\frac{|I^1|}{|K^1|} \int_{\R^{d_2}}
		\Big( \sum_{K^2} \langle f_K \rangle_{I^1 \times K^2} ^2 1_{K^2} \Big)^{\frac 12}
		\int_{J^1}
		E_{\ell(K^1)}^2 M_{\calD^2} g_{K^1}                              \\
		           & = \sum_{K^1} \sum_{(I^1)^{(k^1)} = K^1}
		\int_{I^1} \int_{\R^{d_2}}
		\Big( \sum_{K^2} \langle f_K \rangle_{I^1 \times K^2} ^2 1_{K^2} \Big)^{\frac 12}
		\langle E_{\ell(K^1)}^2 M_{\calD^2} g_{K^1} \rangle_{K^1}        \\
		           & \le \sum_{K^1}  \int_{K^1} \int_{\R^{d_2}}
		\Big( \sum_{K^2} [M_{\calD_F} f_K]^2 \Big)^{\frac 12}
		\langle E_{\ell(K^1)}^2 M_{\calD^2} g_{K^1} \rangle_{K^1}.
	\end{align*}
	Now, for an implicit $x_1$ variable in $K^1$ (from the above $\int_{K^1}$) we have
	\begin{align*}
		\langle E_{\ell(K^1)}^2 M_{\calD^2} g_{K^1} \rangle_{K^1}
		= \sum_{H^2 \colon \ell(H^2) = \ell(K^1)} \langle M_{\calD^2} g_{K^1} \rangle_{K^1 \times H^2} 1_{H^2}
		\le M_{\calD_=} M_{\calD^2} g_{K^1}.
	\end{align*}
	This is the place where we made critical use of having $E_{\ell(K^1)}^2 M_{\calD^2} g_{K^1}$
	instead of $M_{\calD_F} g_{K^1}$ -- for the latter term we would not have been able to dispose
	of the $K^1$ average (as $M_{\calD^1}$ is not fine with flag weights).
	Thus, we are left with
	\begin{align*}
		\Big\|\Big(\sum_{ K\in \calD} |M_{\calD_F}f_K|^2\Big)^{1/2} & \Big\|_{L^p(w)}
		\Big\|\Big(\sum_{ K^1 \in \calD^1} |M_{\calD_=} M_{\calD^2} g_{K^1}|^2\Big)^{1/2}\Big\|_{L^{p'}(\wt w)}             \\
		                                                            & \lesssim (k^1+1)\|f\|_{L^p(w)} \|g\|_{L^{p'}(\wt w)}.
	\end{align*}

	We are done with the boundedness proof for $P^1_{k^1}$. The proof
	for $P^2_{k^2}$ is essentially the same but one can use $M_{\calD_F}$ in
	the analogous place where we above used $E^2_{\ell(K^1)} M_{\calD^2}$ --
	this is because in the end there will be a $K^2$ average (instead of a $K^1$ average)
	that can be bounded with $M_{\calD^2}$. Similarly, the analogue of $\calW_{K^1, k^1}$
	simply is
	$$
		\sum_{\substack{ L^2 \subset K^2 \\ \ell(L^2) \ge 2^{-k^2}\ell(K^2)}} \Delta_{L^2},
	$$
	which is fine as a square function in the second parameter is directly OK (previously we needed
	an extra averaging term to turn a one-parameter square function in the first parameter into a one-parameter
	square function in $\calD_=$). We are done.
\end{proof}

\subsection{Partial paraproducts in the representation theorem}
This is easy -- the partial paraproducts will arise in the form
$$
	\sum_{k^1 = 0}^{\infty} 2^{-\alpha_1 k^1} \langle P_{k^1}^1f, g \rangle +
	\sum_{k^2 = 0}^{\infty} 2^{-\alpha_2 k^2} \langle P_{k^2}^2f, g \rangle,
$$
which is clearly summable with $f \in L^p(w)$, $g \in L^{p'}(\wt w)$, $1 < p < \infty$.

\section{Bi-parameter flag \texorpdfstring{$T1$}{T1} representation theorem}
We begin by some fundamental preliminaries on randomization.
Let $\calD_0$ be the standard dyadic grid in some $\R^m$.
For $\omega \in (\{0,1\}^m)^{\Z}$, $\omega = (\omega_i)_{i \in \Z}$,
we define the shifted dyadic lattice
$$
	\calD(\omega) := \Big\{L + \omega := L + \sum_{i\colon 2^{-i} < \ell(L)}
	2^{-i}\omega_i \colon L \in \calD_0\Big\}.
$$
Let $\bbP_{\omega}$ simply be the product probability measure on $ (\{0,1\}^m)^{\Z}$.

We recall the following notion of a $k$-good, $k \ge 2$, cube from \cite{GH}.
We say that $G \in \calD(\omega)$ is $k$-good, or $G \in \calD(\omega, k)$, $k \ge 2$,
if
\begin{equation*}
	d(G, \partial G^{(k)}) \ge \frac{\ell(G^{(k)})}{4} = 2^{k-2} \ell(G).
\end{equation*}
Notice that for all $L \in \calD_0$ and $k \ge 2$ we have
\begin{equation*}
	\bbP_{\omega}( \{ \omega\colon L + \omega \in \calD(\omega, k) \})  = \frac{1}{2^m}.
\end{equation*}
The key implication of $G \in \calD(\omega, k)$ is that for $n \in \Z^m$
with $|n| \le 2^{k - 2}$ (here we understand $|n|$ as $\max_i |n^i|$) we have
\begin{equation}\label{eq:kparent}
	(G \dotplus n)^{(k)} = G^{(k)}, \qquad G \dotplus n := G + n\ell(G).
\end{equation}
Notice that if $L \in \calD_0$,
the position of $L + \omega$ depends only on the values of
$\omega_i$ for those $i$ for which $2^{-i} < \ell(L)$.
On the other hand, the $k$-goodness $L + \omega \in \calD(\omega, k)$
of $L + \omega$ for some $k \ge 2$ depends on $\omega_i$ for those $i$ for which
$\ell(L) \le 2^{-i} < \ell(L^{(k)})$. In particular, the position is independent of
$k$-goodness for all $k \ge 2$.

In our current bi-parameter flag setting $\R^d = \R^{d_1}\times \R^{d_2}$ we define for
\[
	\sigma = (\sigma^1, \sigma^2) \in (\{0,1\}^{d_1})^{\Z} \times  (\{0,1\}^{d_2})^{\Z},
	\qquad \sigma^m = (\sigma^m_i)_{i \in \Z},
\]
that
$$
	\calD(\sigma) := \calD(\sigma^1) \times \calD(\sigma^2).
$$
Let $\bbP_{\sigma} := \bbP_{\sigma^1} \times \bbP_{\sigma^2}$.
For $k = (k^1, k^2)$, $k^m \ge 2$, we define
the corresponding collection of good rectangles by setting
$$
	\calD(\sigma, k) := \calD(\sigma^1, k^1) \times \calD(\sigma^2, k^2).
$$
We also, e.g., write
$$
	\calD(\sigma, (k^1, 0))
	= \calD(\sigma^1, k^1) \times \calD(\sigma^2),
$$
that is, a zero designates that we do not have goodness in that parameter.
As for most of the argument the grids are fixed,
it makes sense to mostly suppress $\sigma$ from the notation and
abbreviate
$\calD^m = \calD(\sigma^m)$ and $\calD^m(k^m) = \calD(\sigma^m, k^m)$, $m = 1,2$.
Then also
$
	\calD = \calD(\sigma) = \calD^1 \times \calD^2
$ and
$\calD(k) = \calD(\sigma, k)= \prod_{m=1}^2 \calD^m(k^m)$.
We use some related obvious notation, such as, $\calD_F(k)$ stands for the flag rectangles of $\calD(k)$.

\begin{thm}\label{thm:biparflagT1rep}
	Suppose $T$ is a bi-parameter flag CZO. Then we have
	\begin{align*}
		\langle Tf, g\rangle = C \E \Big(
		 & \sum_{k^1,\, k^2 = 0}^{\infty} 2^{-\alpha_1 k^1} 2^{-\alpha_2 k^2}
		\Big[ \langle Q_k f, g \rangle + \sum_{l=1}^{k^1} \langle R_{k, l}f, g\rangle \Big]
		+ \sum_{k^1,\, k^2 = 0}^{\infty} 2^{-\alpha_2 k^2} \langle R_{k, 0} f, g\rangle                  \\
		 & + \sum_{k^1 = 0}^{\infty} 2^{-\alpha_1 k^1} \langle P_{k^1}^1f, g \rangle +
		\sum_{k^2 = 0}^{\infty} 2^{-\alpha_2 k^2} \langle P_{k^2}^2f, g \rangle\Big)                     \\
		 & + \E \langle \Pi_1(T1, f) + \Pi_1^*(T^*1, f) + \Pi_2(T_11, f) + \Pi_2^*(T_1^*1, f), g\rangle.
	\end{align*}
	Here $C$ depends only on the kernel estimates and the various CZO assumptions.
	In particular, we have
	$$
		\|Tf\|_{L^p(w)} \lesssim \|f\|_{L^p(w)}
	$$
	for all $p \in (1, \infty)$ and $w \in A_{p, F}$.
\end{thm}
\begin{proof}
	We apply our usual flag multiresolution
	to the bilinear form $B(f, g) := \langle Tf, g\rangle$, which gives
	\begin{equation*}
		B(f, g) = \E \sum_{s \in \calS} \Sigma_s,
	\end{equation*}
	where $\calS = \{(s_1, s_2) \colon s_i \in \{\Delta, E, D\}\}$
	and the averaging is over the random dyadic lattices as explained
	previously. That is, $\E = \E_{\sigma}$ and $\Sigma_s = \Sigma_s(\sigma)$ in the sense that
	we have performed the dyadic flag multiresolution in the particular random grid $\calD = \calD(\sigma)$.

	\subsection{The case \texorpdfstring{$(D, D)$}{(D, D)}}
	Here we look at the simplest term
	\begin{align*}
		\Sigma_{(D, D)} & =
		\sum_{\substack{I,J\in \calD_F \\ \ell(I)=\ell(J)}}
		B(\Lambda_{I^1} \Delta_{I^2} f, \Lambda_{J^1} \Delta_{J^2} g).
	\end{align*}
	Now, the first step is to parametrize the cubes $J^i$ as shifts of the corresponding cubes
	$I^i$ -- this can be done as $\ell(I^i) = \ell(J^i)$.
	In what follows $n^i \in \Z^{d_i}$ for $i=1,2$ and
	$I^i \dot+ n^i := I^i + n^i\ell(I^i)$.
	We can then write
	$$
		\Sigma_{(D,D)} = \sum_{n=(n^1, n^2) \in \Z^d}
		\sum_{I \in \calD_F} c_{I, n},
	$$
	where
	\[
		c_{I, n} := B\big(\Lambda_{I^1}\Delta_{I^2}f,
		\Lambda_{I^1\dot+ n^1}\Delta_{I^2\dot+ n^2}g\big).
	\]
	Next, we divide these $n^m$ into dyadic blocks using a variable $k^m \ge 2$
	-- this is simply motivated by \eqref{eq:kparent}.
	So for $n^m \ne 0$ there exist a unique $k^m\ge 2$
	so that $|n^m|\in (2^{k^m-3}, 2^{k^m-2}]$.
	Using this we write
	\begin{align*}
		\sum_{n=(n^1, n^2) \in \Z^d} \sum_{I \in \calD_F} c_{I, n}
		 & = \sum_{k^1, k^2 = 2}^\infty \sum_{\substack{n = (n^1, n^2) \in \Z^{d} \\
		|n^m|\in (2^{k^m-3}, 2^{k^m-2}]}} \sum_{I \in \calD_F} c_{I, n}           \\
		 & + \sum_{k^1=2}^\infty \sum_{\substack{n^1 \in \Z^{d_1}                 \\ |n^1|\in (2^{k^1-3}, 2^{k^1-2}]}}
		\sum_{I \in \calD_F} c_{I, (n^1,0)}                                       \\
		 & + \sum_{k^2=2}^\infty \sum_{\substack{n^2 \in \Z^{d_2}                 \\ |n^2|\in (2^{k^2-3}, 2^{k^2-2}]}}
		\sum_{I \in \calD_F} c_{I, (0, n^2)}
		+ \sum_{I \in \calD_F} c_{I, (0,0)}.
	\end{align*}
	Using the independence of $k^m$-goodness and position we then get
	\begin{align*}
		\E \Sigma_{(D, D)}
		 & = 2^d \E \sum_{k^1, k^2 = 2}^\infty \sum_{\substack{n = (n^1, n^2) \in \Z^{d} \\
		|n^m|\in (2^{k^m-3}, 2^{k^m-2}]}} \sum_{I \in \calD_F(k)} c_{I, n}            \\
		 & + 2^{d_1}\E \sum_{k^1=2}^\infty \sum_{\substack{n^1 \in \Z^{d_1}              \\ |n^1|\in (2^{k^1-3}, 2^{k^1-2}]}}
		\sum_{I \in \calD_F(k^1, 0)} c_{I, (n^1,0)}                                   \\
		 & + 2^{d_2} \E \sum_{k^2=2}^\infty \sum_{\substack{n^2 \in \Z^{d_2}             \\ |n^2|\in (2^{k^2-3}, 2^{k^2-2}]}}
		\sum_{I \in \calD_F(0, k^2)} c_{I, (0, n^2)}
		+ \E \sum_{I \in \calD_F} c_{I, (0,0)}.
	\end{align*}
	We consider, for instance, with a fixed $k^1 \ge 2$ the part
	\begin{align*}
		A & := \sum_{\substack{n^1 \in \Z^{d_1}     \\ |n^1|\in (2^{k^1-3}, 2^{k^1-2}]}}
		\sum_{I \in \calD_F(k^1, 0)} c_{I, (n^1,0)} \\
		  & = \sum_{\substack{n^1 \in \Z^{d_1}      \\ |n^1|\in (2^{k^1-3}, 2^{k^1-2}]}}
		\sum_{I \in \calD_<(k^1, 0)} c_{I, (n^1,0)}
		+\sum_{\substack{n^1 \in \Z^{d_1}           \\ |n^1|\in (2^{k^1-3}, 2^{k^1-2}]}}
		\sum_{I \in \calD_=(k^1, 0)} c_{I, (n^1,0)}
		=: A_< + A_=.
	\end{align*}
	We write, using \eqref{eq:kparent}, that
	$$
		A_< = C2^{-\alpha_1 k^1} \sum_{K \in \calD^{(k^1, 0)}_<} \sum_{(I^1)^{(k^1)} = (J^1)^{(k^1)} = K^1} a_{I^1J^1K}
		\langle f, h_{I^1} \otimes h_{K^2} \rangle \langle g, h_{J^1} \otimes h_{K^2} \rangle,
	$$
	where $a_{I^1J^1K} = B(h_{I^1} \otimes h_{K^2}, h_{J^1} \otimes h_{K^2}) / (C2^{-\alpha_1 k^1})$
	whenever $J^1 = I^1 \dotplus n^1$ for $|n^1| \in (2^{k^1-3}, 2^{k^1-2}]$ and $I^1$ is $k^1$-good,
	and $a_{I^1J^1K}$ vanishes otherwise. Our coefficient estimates tell us that $|a_{I^1J^1K}| \le |I^1 \times K^2|/|K_F|$.
	Therefore, we have $A_< = C2^{-\alpha_1 k^1} (\langle Q_{(k^1, 0)}f, g\rangle+\sum_{l=1}^{k^1}\langle R_{(k^1, 0),l}f, g\rangle)$, where $Q_{(k^1, 0)}$ is a flag
	shift and $R_{(k^1, 0), l}$
	is an almost one-parameter shift. This is done by splitting the sum over $K^2$ into $\ell(K^2)>\ell(K^1)$ and $\ell(K^2)\le \ell(K^1)$. Similarly,
	\[
		A_= = C \langle R_{(k^1, 0), 0}f, g\rangle,
	\]
	where $R_{(k^1, 0), 0}$ is an almost one-parameter shift. The other terms are similar, giving
	$$
		\E \Sigma_{(D,D)} =
		C \E \Big[ \sum_{k^1,\, k^2 = 0}^{\infty} 2^{-\alpha_1 k^1} 2^{-\alpha_2 k^2} \Big(\langle Q_k f, g \rangle+\sum_{l=1}^{k^1}\langle R_{k, l} f, g\rangle\Big)
			+ \sum_{k^1,\, k^2 = 0}^{\infty} 2^{-\alpha_2 k^2} \langle R_{k, 0} f, g\rangle \Big].
	$$

	\subsection{The case \texorpdfstring{$(\Delta, D)$}{(Delta, D)}}
	We then look at the, still comparatively simple, term
	$$
		\Sigma_{(\Delta, D)} =
		\sum_{\substack{I,J\in \calD_< \\ \ell(I)=\ell(J)}}
		B(\Delta_{I^1} \Delta_{I^2} f, E_{J^1} \Delta_{J^2} g).
	$$
	While this term does require a paraproduct correction, the need is only in the first parameter --
	this is the easier case as it will not lead to a situation like in Lemma \ref{lem:formsum3}
	that would require further summation restrictions (a paraproduct correction
	in the second parameter, needed in some other terms, will, however, have this effect
	and require extra arrangements).

	For the paraproduct correction
	we write $E_{J^1} g = h_{J^1}^0(\langle g, h_{J^1}^0 \rangle - \langle g, h_{I^1}^0 \rangle) + 1_{J^1}\langle g \rangle_{I^1}$.
	This splits
	$$
		\Sigma_{(\Delta, D)} = \Sigma_{(\Delta, D)}^S + \Sigma_{(\Delta, D)}^{PP},
	$$
	where the superscript $S$ refers to a shift and the superscript $PP$ to a partial paraproduct, and we have
	for $H_{I^1, J^1} := h_{J^1}^0 - h_{I^1}^0$ that
	\begin{align*}
		\Sigma_{(\Delta, D)}^S =
		\sum_{\substack{I,J\in \calD_< \\ \ell(I)=\ell(J)}}
		B(h_{I^1} \otimes h_{I^2}, h_{J^1}^0 \otimes h_{J^2})
		\langle f, h_{I^1} \otimes h_{I^2}\rangle \langle g, H_{I^1, J^1} \otimes h_{J^2} \rangle
	\end{align*}
	and
	\begin{align*}
		\Sigma_{(\Delta, D)}^{PP} & =
		\sum_{\substack{I,J\in \calD_<                                                                               \\ \ell(I)=\ell(J)}}
		B(h_{I^1} \otimes h_{I^2}, 1_{J^1} \otimes h_{J^2})
		\langle f, h_{I^1} \otimes h_{I^2}\rangle \Big \langle g, \frac{1_{I^1}}{|I^1|} \otimes h_{J^2}\Big \rangle                             \\
		                          & =
		\sum_{I\in \calD_<} \sum_{\substack{J^2 \in \calD^2                                                          \\ \ell(J^2) = \ell(I^2)  } }
		B(h_{I^1} \otimes h_{I^2}, 1 \otimes h_{J^2})
		\langle f, h_{I^1} \otimes h_{I^2}\rangle \Big \langle g, \frac{1_{I^1}}{|I^1|} \otimes h_{J^2}\Big \rangle                             \\
		                          & = \sum_{K^1} \sum_{\substack{ I^2, J^2 \colon \ell(I^2) = \ell(J^2)              \\ \ell(I^2) > \ell(K^1)}}
		B(h_{K^1} \otimes h_{I^2}, 1 \otimes h_{J^2})
		\langle f, h_{K^1} \otimes h_{I^2}\rangle \Big \langle g, \frac{1_{K^1}}{|K^1|} \otimes h_{J^2}\Big \rangle.                            \\
	\end{align*}

	Now, to dig out the shift structure in both terms,
	we must perform the usual business with the $n^m, k^m$ and goodness -- in $\Sigma^S_{(\Delta, D)}$
	we will have $n^1 \in \Z^{d_1} \setminus \{0\}$ (as $H_{I^1, J^1} = 0$ if $I^1 = J^1$)
	and $n^2 \in \Z^{d_2}$, after that we split the summation into $\ell(I^2)>2^{k^1}\ell(I^1) $ and $\ell(I^2)\le 2^{k^1}\ell(I^1) $. In the $\Sigma^{PP}_{(\Delta, D)}$ we will only have
	$n^2 \in \Z^{d_2}$ (it is a paraproduct in the first parameter and readily $n^1 = 0$).
	Therefore, using our coefficient estimates we will get
	$$
		\E \Sigma_{(\Delta, D)}^S = C \E \sum_{k^1 = 2}^{\infty} \sum_{k^2 = 0}^{\infty}
		2^{-\alpha_1 k^1} 2^{-\alpha_2 k^2} \Big(\langle Q_k f, g \rangle+\sum_{l=1}^{k^1}\langle R_{k, l} f, g\rangle\Big)
	$$
	and
	$$
		\E \Sigma_{(\Delta, D)}^{PP} = C \E \sum_{k^2 = 0}^{\infty} 2^{-\alpha_2 k^2} \langle P_{k^2}^2 f, g\rangle.
	$$

	\subsection{The case \texorpdfstring{$(D, \Delta)$}{(D, Delta)}}
	Now, we look at
	$$
		\Sigma_{(D, \Delta)} =
		\sum_{\substack{I,J\in \calD_F \\ \ell(I)=\ell(J)}}
		B(\Lambda_{I^1} \Delta_{I^2} f, \Lambda_{J^1} E_{J^2} g).
	$$
	There is the apparent need to do a paraproduct correction in the second parameter
	-- this is problematic as getting $H_{I^2, J^2} = h_{J^2}^0 - h_{I^2}^0$ directly in the shift part
	would lead to the case of Lemma \ref{lem:formsum3} but without the required
	restriction on the summation. Therefore, we must avoid doing this correction on finitely many scales
	-- this is where the almost one parameter shifts $R_{k, l}$
	with $l \in \{0, \ldots, k^1\}$ will enter the picture.
	This term will, in addition, yield a $\calD_=$ type contribution to full paraproducts
	due to us also having to do a certain \emph{one-parameter} paraproduct correction (too little
	cancellation due to the $\Lambda$ operators). Importantly, such a one-parameter
	correction will not cause the situation of Lemma \ref{lem:formsum3} to appear. We get to the details.

	First, we write
	\begin{align*}
		\Sigma_{(D, \Delta)} & =
		\sum_{\substack{I,J\in \calD_<                          \\ \ell(I)=\ell(J)}}
		B(\Delta_{I^1} \Delta_{I^2} f, \Delta_{J^1} E_{J^2} g) +
		\sum_{\substack{I,J\in \calD_=                          \\ \ell(I)=\ell(J)}}
		B(\Lambda_{I^1} \Delta_{I^2} f, \Delta_{J^1} E_{J^2} g)                      \\
		                     & + \sum_{\substack{I,J\in \calD_= \\ \ell(I)=\ell(J)}}
		B(\Lambda_{I^1} \Delta_{I^2} f, E_J g)
		=:
		\Sigma_{(D, \Delta)}^< +
		\Sigma_{(D, \Delta)}^{=, S} +
		\Sigma_{(D, \Delta)}^{=, E}.
	\end{align*}

	Here $S$ again refers to a shift and $E$ just indicates that the term $\Sigma_{(D, \Delta)}^{=, E}$
	corresponds to the averaging part of $\Lambda_{J^1}g = \Delta_{J^1}g + E_{J^1}g$.
	We take care of the ``$=$'' terms first. In $\Sigma_{(D, \Delta)}^{=, S}$ we have
	one-parameter cancellation in both slots so, with our usual techniques and coefficient estimates,
	we have
	$$
		\E \Sigma_{(D, \Delta)}^{=, S} = C \E \sum_{k^1, k^2 = 0}^{\infty} 2^{-\alpha_1 k^1}
		2^{-\alpha_2 k^2} \langle R_{k, 0} f, g\rangle.
	$$
	(The exponential $k^1$ decay is not really needed but happens to be available in this term.)
	It is critical that we did not attempt
	to do some paraproduct correction in parameter $2$ here (which would produce the situation
	of Lemma \ref{lem:formsum3} without summing restrictions) but rather saw this
	directly as an almost one-parameter shift. Next, in $\Sigma_{(D, \Delta)}^{=, E}$
	we do not have enough cancellation even for a one-parameter shift and we need
	to do a paraproduct correction -- but only in a one-parameter sense in the full $\R^d$ which is not
	problematic. Indeed, we write $E_J g = h_J^0(\langle g, h_{J}^0 \rangle - \langle g, h_I^0 \rangle)
		+ 1_J \langle g \rangle_I$.
	This splits
	$$
		\Sigma_{(D, \Delta)}^{=, E} = \Sigma_{(D, \Delta)}^{=, E, S} + \Sigma_{(D, \Delta)}^{=, E, FP},
	$$
	where $FP$ stands for a full paraproduct.
	Now, in $\Sigma_{(D, \Delta)}^{=, E, S}$ there is the required one-parameter cancellation
	in both slots and the support conditions are allowable ($I$ and $I\cup J$) -- therefore, this is an
	almost one-parameter shift and we have
	$$
		\E \Sigma_{(D, \Delta)}^{=, E, S} =   C \E \sum_{k^1, k^2 = 0}^{\infty}
		2^{-\alpha_2 k^2} \langle R_{k, 0} f, g\rangle.
	$$
	Finally, we have
	$$
		\Sigma_{(D, \Delta)}^{=, E, FP} = \sum_{I \in \calD_=}
		\langle T^*1, h_{I, F} \rangle \langle f, h_{I, F} \rangle \langle g \rangle_I.
	$$
	This is the $\calD_=$ part of the full paraproduct $\Pi_1^*(T^*1, f)$ -- the corresponding $\calD_<$
	part will come from $\Sigma_{(\Delta, \Delta)}$.

	We turn to consider $\Sigma_{(D, \Delta)}^<$. So here we must do a paraproduct
	correction in the second parameter but we cannot quite always do it. To be able to specify when we do it,
	we must first introduce part of the shift style structure -- importantly, the related $k^1$ summation
	in the first parameter. So we first do the usual
	\begin{align*}
		\Sigma_{(D, \Delta)}^< & =
		\sum_{\substack{ I \in \calD_<                                                                      \\ J^2 \colon \ell(J^2) = \ell(I^2)}}
		B(\Delta_{I^1} \Delta_{I^2}f, \Delta_{I^1} E_{J^2}g)                                                                                      \\
		                       & + \sum_{k^1=2}^{\infty} \sum_{\substack{ n^1 \in \Z^{d_1}                                                        \\
		|n^1| \in (2^{k^1-3}, 2^{k^1-2}]}} \sum_{\substack{ I \in \calD_<                                   \\ J^2 \colon \ell(J^2) = \ell(I^2)}}
		B(\Delta_{I^1} \Delta_{I^2}f, \Delta_{I^1 \dotplus n^1} E_{J^2}g)                                                                         \\
		                       & = \sum_{\substack{ I \in \calD_<                                           \\ J^2 \colon \ell(J^2) = \ell(I^2)}}
		B(\Delta_{I^1} \Delta_{I^2}f, \Delta_{I^1} E_{J^2}g)                                                                                      \\
		                       & + \sum_{k^1=2}^{\infty} \sum_{\substack{ n^1 \in \Z^{d_1}                                                        \\
		|n^1| \in (2^{k^1-3}, 2^{k^1-2}]}} \sum_{\substack{ I \colon \ell(I^1) < 2^{-k^1}\ell(I^2)          \\ J^2 \colon \ell(J^2) = \ell(I^2)}}
		B(\Delta_{I^1} \Delta_{I^2}f, \Delta_{I^1 \dotplus n^1} E_{J^2}g)                                                                         \\
		                       & + \sum_{k^1=2}^{\infty} \sum_{l=1}^{k^1} \sum_{\substack{ n^1 \in \Z^{d_1}                                       \\
		|n^1| \in (2^{k^1-3}, 2^{k^1-2}]}} \sum_{\substack{ I \colon \ell(I^1)= 2^{-l}\ell(I^2)             \\ J^2 \colon \ell(J^2) = \ell(I^2)}}
		B(\Delta_{I^1} \Delta_{I^2}f, \Delta_{I^1 \dotplus n^1} E_{J^2}g)                                                                         \\
		                       & =: \Sigma_{(D, \Delta)}^{\ll} + \Sigma_{(D, \Delta)}^{<, S},
	\end{align*}
	where $\Sigma_{(D, \Delta)}^{\ll}$ consists of the first two terms with
	$\ell(I^1) < 2^{-k^1}\ell(I^2)$ (thinking that in the first sum formally $k^1 = 0$).
	In $\Sigma_{(D, \Delta)}^{<, S}$ we do not perform a paraproduct correction and, instead, view
	it directly as almost one parameter shifts:
	$$
		\E \Sigma_{(D, \Delta)}^{<, S}
		= C \E \sum_{k^1 = 2}^{\infty} \sum_{k^2 = 0}^{\infty} \sum_{l=1}^{k^1} 2^{-\alpha^1k^1}
		2^{-\alpha^2 k^2} \langle R_{k, l}f, g\rangle.
	$$
	In $\Sigma_{(D, \Delta)}^{\ll}$ we do the paraproduct correction
	via the usual $E_{J^2} g = h_{J^2}^0 \langle g, H_{I^2, J^2} \rangle + 1_{J^2} \langle g \rangle_{I^2}$.
	This splits $\Sigma_{(D, \Delta)}^{\ll} = \Sigma_{(D, \Delta)}^{\ll, S} + \Sigma_{(D, \Delta)}^{\ll, PP}$,
	where
	\begin{align*}
		\E \Sigma_{(D, \Delta)}^{\ll, S} = C\E \sum_{k^1 = 0}^{\infty} \sum_{k^2 = 2}^{\infty}
		                                                               2^{-\alpha^1 k^1} 2^{-\alpha^2 k^2} \langle Q_k f, g\rangle
	\end{align*}
	and
	\begin{align*}
		\E \Sigma_{(D, \Delta)}^{\ll, PP} = C \E \sum_{k^1 = 0}^{\infty} 2^{-\alpha^1 k^1}
		\langle P_{k^1}^1 f, g\rangle.
	\end{align*}
	Here   the partial paraproducts are  of the more delicate form $P_{k^1}^1$ --
	the required summation restrictions are valid in both by construction.

	\subsection{The case \texorpdfstring{$(\Delta, \Delta)$}{(Delta, Delta)}}
	We have arrived at one of the most complicated terms
	$$
		\Sigma_{(\Delta, \Delta)} =
		\sum_{\substack{I,J\in \calD_< \\ \ell(I)=\ell(J)}}
		B(\Delta_{I^1} \Delta_{I^2} f, E_J g),
	$$
	but luckily we have seen the main techniques already. The first step
	is to do the paraproduct correction in the first parameter:
	\begin{align*}
		\Sigma_{(\Delta, \Delta)} & =
		\sum_{\substack{I,J\in \calD_<                                                                \\ \ell(I)=\ell(J)}}
		B(\Delta_{I^1} \Delta_{I^2} f, 1_J(\langle g \rangle_J - \langle g \rangle_{I^1 \times J^2}))                      \\
		                          & + \sum_{\substack{I,J\in \calD_<                                  \\ \ell(I)=\ell(J)}}
		B(\Delta_{I^1} \Delta_{I^2} f, 1_J \langle g \rangle_{I^1 \times J^2})
		=: \Sigma_{(\Delta, \Delta)}^1 + \Sigma_{(\Delta, \Delta)}^2.
	\end{align*}
	Now, notice that
	$$
		\Sigma_{(\Delta, \Delta)}^2 = \sum_{\substack{I \in \calD_< \\ J^2 \colon \ell(J^2) = \ell(I^2)}}
		B(\Delta_{I^1} \Delta_{I^2} f, (1 \otimes 1_{J^2}) \langle g \rangle_{I^1 \times J^2})
	$$
	In this term we can directly do the paraproduct correction in the second parameter as well:
	\begin{align*}
		\Sigma_{(\Delta, \Delta)}^2 & = \sum_{\substack{I \in \calD_< \\ J^2 \colon \ell(J^2) = \ell(I^2)}}
		B(\Delta_{I^1} \Delta_{I^2} f, (1 \otimes 1_{J^2})
		(\langle g \rangle_{I^1 \times J^2} - \langle g \rangle_I))                                         \\
		                            & +\sum_{\substack{I \in \calD_<  \\ J^2 \colon \ell(J^2) = \ell(I^2)}}
		B(\Delta_{I^1} \Delta_{I^2} f, 1 \otimes 1_{J^2})\langle g \rangle_I
		=: \Sigma_{(\Delta, \Delta)}^{2, PP}
		+ \Sigma_{(\Delta, \Delta)}^{2, FP}.
	\end{align*}
	Notice that
	$$
		\Sigma_{(\Delta, \Delta)}^{2, FP} = \sum_{I \in \calD_<}
		\langle T^*1, h_{I, F} \rangle \langle f, h_{I, F}\rangle \langle g \rangle_I.
	$$
	This is the $\calD_<$ part of $\Pi_1^*(T^*1, f)$ -- recall that the
	$\calD_=$ part of this same full paraproduct was formed in $\Sigma_{(D, \Delta)}$.
	Next, notice that
	$$
		\Sigma_{(\Delta, \Delta)}^{2, PP}
		= \sum_{K^1} \sum_{\substack{ I^2, J^2 \colon \ell(I^2) = \ell(J^2)
				\\ \ell(K^1) < \ell(I^2) }} B(h_{K^1} \otimes h_{I^2}, 1 \otimes h_{J^2}^0)
		\langle f, h_{K^1} \otimes h_{I^2}\rangle
		\Big \langle g, \frac{1_{K^1}}{|K^1|} \otimes H_{I^2, J^2}\Big\rangle
	$$
	so that after invoking our coefficient estimates and doing the usual goodness business
	we get
	$$
		\E \Sigma_{(\Delta, \Delta)}^{2, PP} = C \E \sum_{k^2=2}^{\infty} 2^{-\alpha_2 k^2} \langle  P_{k^2}^2f, g\rangle.
	$$

	Now, moving on to $\Sigma_{(\Delta, \Delta)}^1$ -- here we cannot directly do the
	paraproduct correction in the second parameter and must follow the strategy from
	$\Sigma_{(D, \Delta)}^<$. Therefore, we write
	\begin{align*}
		 & \Sigma_{(\Delta, \Delta)}^1                                                \\
		 & =
		\sum_{k^1=2}^{\infty} \sum_{\substack{ n^1 \in \Z^{d_1}                       \\
		|n^1| \in (2^{k^1-3}, 2^{k^1-2}]}} \sum_{\substack{ I \colon \ell(I^1) < 2^{-k^1}\ell(I^2) \\ J^2 \colon \ell(J^2) = \ell(I^2)}}
		B(\Delta_{I^1} \Delta_{I^2} f, h_{I^1 \dotplus n^1}^0 \otimes h_{J^2}^0)
		\langle g, H_{I^1, I^1 \dotplus n^1} \otimes h_{J^2}^0\rangle                 \\
		 & + \sum_{k^1=2}^{\infty} \sum_{l=1}^{k^1} \sum_{\substack{ n^1 \in \Z^{d_1} \\
		|n^1| \in (2^{k^1-3}, 2^{k^1-2}]}} \sum_{\substack{ I \colon \ell(I^1)= 2^{-l}\ell(I^2)    \\ J^2 \colon \ell(J^2) = \ell(I^2)}}
		B(\Delta_{I^1} \Delta_{I^2} f, h_{I^1 \dotplus n^1}^0 \otimes h_{J^2}^0)
		\langle g, H_{I^1, I^1 \dotplus n^1} \otimes h_{J^2}^0\rangle                 \\
		 & =: \Sigma_{(\Delta, \Delta)}^{1, \ll} + \Sigma_{(\Delta, \Delta)}^{1, S}.
	\end{align*}
	As in the $(D, \Delta)$ case we have
	$$
		\E \Sigma_{(\Delta, \Delta)}^{1, S}
		= C \E \sum_{k^1 = 2}^{\infty} \sum_{k^2 = 0}^{\infty} \sum_{l=1}^{k^1} 2^{-\alpha^1k^1}
		2^{-\alpha^2 k^2} \langle R_{k, l}f, g\rangle.
	$$
	In $\Sigma_{(\Delta, \Delta)}^{1, \ll}$ we do the paraproduct correction in the second parameter
	leading to the splitting
	$$
		\Sigma_{(\Delta, \Delta)}^{1, \ll} = \Sigma_{(\Delta, \Delta)}^{1, \ll, S}
		+ \Sigma_{(\Delta, \Delta)}^{1, \ll, PP},
	$$
	where
	$$
		\E \Sigma_{(\Delta, \Delta)}^{1, \ll, S} = C\E \sum_{k^1=2}^{\infty} \sum_{k^2=2}^{\infty}
		2^{-\alpha^1 k^1}
		2^{-\alpha^2 k^2} \langle Q_k f, g\rangle
	$$
	and
	$$
		\E \Sigma_{(\Delta, \Delta)}^{1, \ll, P} = C\E \sum_{k^1=2}^{\infty} 2^{-\alpha_1 k^1}
		\langle P_{k^1}^1 f, g\rangle.
	$$

	\subsection{The case \texorpdfstring{$(\Delta, E)$}{(Delta, E)}}
	This goes exactly like the case $(\Delta, \Delta)$. The steps are carried out in the same order,
	just in slightly different ``slots''. Therefore, for instance, the full paraproduct
	that this term produces is
	\begin{align*}
		\sum_{I^1 \times J^2 \in \calD_<} & B(\Delta_{I^1}\langle f \rangle_{J^2} \otimes 1,
		1 \otimes \Delta_{J^2} \langle g \rangle_{I^1})                                      \\
		                                  & = \sum_{I \in \calD_<}
		\langle T_1 1, h_{I, F}\rangle
		\Big\langle f, h_{I^1} \otimes \frac{1_{I^2}}{|I^2|}\Big\rangle
		\Big\langle g, \frac{1_{I^1}}{|I^1|} \otimes h_{I^2}\Big\rangle
		= \langle \Pi_2(T_11, f), g\rangle.
	\end{align*}

	\subsection{The remaining cases}
	The rest of the cases follow by duality as we explain. If
	$\Sigma_s$ stands for $\Sigma_s(T, f, g)$, then
	$\Sigma_{(E, E)} = \Sigma_{(\Delta, \Delta)}(T^*, g, f)$,
	$\Sigma_{(E, \Delta)} = \Sigma_{(\Delta, E)}(T^*, g, f)$
	and so on. We have completed the proof of the representation theorem.
\end{proof}

\section{Flag commutators}
\subsection{Flag commutator $[b,T]$ upper bound via multiresolutions}
The ``Cauchy integral trick'' allows one to derive the boundedness
of certain commutators directly from the weighted boundedness of the underlying operator.
This is limited to commutators the characterizing condition of which is of ``little $BMO$'' type,
since it requires some classical connections between weights and $\BMO$. In particular,
the flag adapted Cauchy integral trick from \cite{DLOPW-FLAG} combined with our
result giving the weighted boundedness of highly general flag CZOs, Theorem \ref{thm:biparflagT1rep},
directly gives the following.
\begin{thm}\label{thm:littlebmocommutator}
	Let $T$ be a general bi-parameter flag CZO (or a cancellative tri-parameter
	flag CZO). Then we have
	$$
		\|[b, T]f\|_{L^p(w)} \lesssim \|b\|_{\bmo_F}\|f\|_{L^p(w)}
	$$
	for all $1 < p < \infty$ and $w \in A_{p, F}$.
\end{thm}
We next give a direct multiresolution proof in the cancellative bi-parameter case, because
the multiresolution method is much more flexible than the Cauchy integral trick and can be used
to derive much more complicated commutator estimates, say, various bi-commutator estimates
and two-weight Bloom type estimates. That is why we feel it is important that we develop the previously
missing multiresolution line of attack to flag commutators.

For this multiresolution proof of Theorem \ref{thm:littlebmocommutator}
in the cancellative bi-parameter case, it is enough to consider commutators $[b, U]$, where $U$ is one
of the flag model operators appearing in Theorem \ref{thm:biparflagT1rep}. Since we
want to only showcase the methodology here, we do this only for cancellative
bi-parameter CZOs so we only need to consider commutators of flag shifts and almost one-parameter shifts.

\begin{lem}\label{lem:commu_almost}
	For all $1<p<\infty$, $w\in A_{p, F}$ and $k,l$ we have
	\[
		\|[b, R_{k,l} ]f\|_{L^p(w)}\lesssim (|k|+1)^3\|b\|_{\bmo_F}\|f\|_{L^p(w)}.
	\]
	And for $l=0$ we have the enhanced estimate
	\[
		\|[b, R_{k,0}] f\|_{L^p(w)}\lesssim (|k|+1)^32^{-\eta(k^1-k^2)_+}\|b\|_{\bmo_F}\|f\|_{L^p(w)},
	\]
	where $\eta = \eta([w]_{A_{p, F}}) > 0$.
\end{lem}
\begin{proof}
	We expand $bf$ as one-parameter paraproducts in the dilated grid
	$\calD(0,l)= \{L\in \calD\colon \ell(L^1) = 2^{-l}\ell(L^2)\}$ as follows:
	\begin{align*}
		bf= \sum_{ L\in \calD(0,l)} \Delta_L b \Delta_L f
		+ \sum_{ L\in \calD(0,l)} \Delta_L b E_L f
		+ \sum_{ L\in \calD(0,l)}E_L b \Delta_L f
		=:\sum_{s \in \{D, \Delta, E\}} \rho_s(b,f).
	\end{align*}
	The well-known one-parameter theory tells us that if $s \ne E$ then
	\[
		\|\rho_s (b,f)\|_{L^p(w)} \lesssim \|b\|_{\bmo_F}\|f\|_{L^p(w)}.
	\]
	Here we have used that $w\in A_{p, F}$ implies $w\in A_p(\calD(0,l))$, and a similar fact for
	the symbol $b$.

	Writing
	$$
		\langle [b, R_{k,l}]f, g \rangle = \langle R_{k,l}f, bg \rangle - \langle R_{k,l}(bf), g \rangle,
	$$
	expanding the products $bf$ and $bg$ like above, and using the boundedness
	of the above paraproducts and the model operator $R_{k,l}$
	we are left with the task of estimating
	\begin{align*}
		\bla R_{k,l} & f, \rho_E (b, g)\bra -
		\bla R_{k,l} (\rho_E(b,f)), g\bra                                                                                                                                 \\
		             & = \sum_{K \in \calD(k,l)} \sum_{I^{(k)} = J^{(k)} = K} a_{IJK}\big[ \langle f, \varphi_{I,J}\rangle
			                                                                             \langle \rho_E(b, g), \psi_{I, J}\rangle- \langle \rho_E(b, f), \varphi_{I,J}\rangle
			                                                                             \langle g, \psi_{I, J}\rangle\big],
	\end{align*}
	where we recall that  the functions
	$\varphi_{I,J}$ and $\psi_{I,J} = \psi_J$ have one-parameter cancellation and
	satisfy $\supp \varphi_{I,J} \subset I \cup J \cup (J^1 \times I^2)$
	and $\supp \psi_{J} \subset J$.

	Observe that
	\begin{align*}
		\rho_E (b,f)=\sum_{ L\in \calD(0,l)} \langle b \rangle_L \Delta_L f
		= \sum_{ L\in \calD(0,l)} (\langle b\rangle_L-\langle b\rangle_{K^{(\tau)}} )\Delta_L f
		+ \langle b\rangle_{K^{(\tau)}} f
	\end{align*}
	and we have the analogous equation for $\rho_E(b,g)$. Here $\tau = \tau_k$ is defined
	as in Section \ref{subsec:blocks} so that $K^{(\tau)}$ is, in particular, flag. Now, the cancellation
	of the commutator is simply used to kill the terms with $\langle b \rangle_{K^{(\tau)}}f$
	and $\langle b \rangle_{K^{(\tau)}}g$ so that
	\begin{align*}
		\bla R_{k,l} & f, \rho_E (b, g)\bra -
		\bla R_{k,l} (\rho_E(b,f)), g\bra                                                                                                                                     \\
		             & = \sum_{K \in \calD(k,l)} \sum_{I^{(k)} = J^{(k)} = K} a_{IJK} \langle f, \varphi_{I,J}\rangle
		\Big\langle \sum_{ L\in \calD(0,l)} (\langle b\rangle_L-\langle b\rangle_{K^{(\tau)}} )
		\Delta_L g, \psi_{I, J}\Big\rangle                                                                                                                                    \\
		             & -  \sum_{K \in \calD(k,l)} \sum_{I^{(k)} = J^{(k)} = K} a_{IJK}\Big\langle \sum_{ L\in \calD(0,l)} (\langle b\rangle_L-\langle b\rangle_{K^{(\tau)}} )
		\Delta_L f, \varphi_{I,J}\Big\rangle
		\langle g, \psi_{I, J}\rangle.
	\end{align*}
	These terms can be handled on their own now -- essentially we need to run a ``$b$-adapted''
	proof of the boundedness of $R_{k, l}$. Indeed,
	using Section \ref{subsec:blocks} we can insert the familiar restrictions to the above $L$ summations
	(and also replace the free $f$ and $g$ with the appropriate martingale difference blocks) and so we get
	\begin{align*}
		\bla R_{k,l} & f, \rho_E (b, g)\bra -
		\bla R_{k,l} (\rho_E(b,f)), g\bra                                                                                             \\
		             & = \sum_{K \in \calD(k,l)} \sum_{I^{(k)} = J^{(k)} = K} a_{IJK}\langle \calY_{K, k, l} f, \varphi_{I,J} \rangle
		\langle \calY_{K, k, l}^bg, \psi_J\rangle                                                                                     \\ &\hspace{2cm}
		-\sum_{K \in \calD(k,l)} \sum_{I^{(k)} = J^{(k)} = K} a_{IJK}\langle \calY_{K, k, l}^b f, \varphi_{I,J} \rangle
		\langle \calY_{K, k, l} g, \psi_J\rangle,
	\end{align*}
	where the $b$-adapted variant of $\calY_{K, k, l}$ is defined by
	\[
		\calY_{K, k, l}^b f:= 1_K\sum_{\substack{L \subset K^{(\tau)} \\ \ell(L^1) = 2^{-l}\ell(L^2)
				\\ \ell(K^m) \le 2^{k^m}\ell(L^m) }}
		(\langle b\rangle_L-\langle b\rangle_{K^{(\tau)}} )\Delta_L f.
	\]
	Now, by the proof of Theorem \ref{thm:almost} we know that we are done
	if we check the square function estimate for the $b$-adapted operators
	$\calY_{K, k, l}^b$. But the proof of that is almost the same as the proof of
	Lemma \ref{lem:auxsfY} -- we only need to use the soon to follow $\bmo_F$ estimate
	in an appropriate point of the argument. In this argument it will play a role that $K^{(\tau)}$
	is flag (since we need to use the flag $\BMO$ control)
	so we need to retain this information in $\calY_{k, l}^b$ by not summing
	over all $K \in \calD$ (which was fine for $\calY_{k, l}$) but, say, over
	$K \in \calD^k_F$ (or even just $K \in \calD(k, l)$ if desired).

	So write $K^{(\tau)} = L^{(\zeta)}$ so that we must have
	\[
		\zeta^1 +\zeta^2\lesssim (|k|+1).
	\]
	This gives
	\begin{align*}
		|\langle b\rangle_{L}- \langle b\rangle_{K^{(\tau)}}|
		 & \le |\langle b\rangle_{L}- \langle b\rangle_{L^{(0, 1)}}|+\cdots + |\langle b\rangle_{L^{(0, \zeta^2-1)}}- \langle b\rangle_{L^{(0, \zeta^2)}}| \\
		 & \quad+ |\langle b\rangle_{L^{(0, \zeta^2)}}- \langle b\rangle_{L^{(1, \zeta^2)}}|+\cdots
		+  |\langle b\rangle_{L^{(\zeta^1-1, \zeta^2)}}- \langle b\rangle_{L^{(\zeta^1, \zeta^2)}}|
		\\&\lesssim (|k|+1) \|b\|_{\bmo_F}.
	\end{align*}
	Here we used that $L, K^{(\tau)}\in \calD_F$ so that the appearing
	$L^{(0,j)}, L^{(i, \zeta^2)}\in \calD_F$. This completes the proof.
\end{proof}
We must then run a similar argument for the flag shifts $Q_k$ -- this will require
our flag paraproduct decompositions proved in Theorem \ref{thm:paraproducts}.

\begin{lem}\label{lem:shifts}
	For all $p \in (1, \infty)$, $w \in A_{p, F}$ and $k$ we have
	$$
		\|[b, Q_k] f\|_{L^p(w)} \lesssim (|k|+1)^3 \|b\|_{\bmo_F}\|f\|_{L^p(w)}.
	$$
\end{lem}
\begin{proof}
	The strategy is the same as for the almost one-parameter shifts, but this
	time we require the more complicated flag type expansion
	\begin{align*}
		bf= \sum_{s\neq (E, E)}\pi_s (b, f)+ \pi_{(E,E)}(b,f).
	\end{align*}
	Again, by the paraproduct theory, Theorem \ref{thm:paraproducts}, we only need to estimate
	\begin{align*}
		\bla Q_k f, \pi_{(E,E)} (b, g)\bra -
		\bla Q_k(\pi_{(E, E)}(b,f)), g\bra.
	\end{align*}
	Working as above, and recalling that
	\begin{align*}
		\pi_{(E,E)}(b,f)= \sum_{L\in \calD_F }\langle b\rangle_L \Delta_{L, F} f
		=  \sum_{L\in \calD_F }(\langle b\rangle_L-\langle b\rangle_{K^{(\tau)}}) \Delta_{L, F} f
		+ \langle b\rangle_{K^{(\tau)}}f,
	\end{align*}
	we get
	\begin{align*}
		&\bla Q_k f, \pi_{(E,E)} (b, g)\bra -
		\bla Q_k(\pi_{(E, E)}(b,f)), g\bra \\
		 &\quad =  \sum_{K \in \calF} \sum_{I^{(k)} = J^{(k)} = K} a_{IJK} \langle \calU_{K,k} f, \varphi_{I,J}\rangle
		\langle \calU_{K,k}^b g, \psi_{I, J}\rangle                                                               \\
		 &\qquad -\sum_{K \in \calF} \sum_{I^{(k)} = J^{(k)} = K} a_{IJK} \langle \calU_{K,k}^b f, \varphi_{I,J}\rangle
		\langle \calU_{K,k} g, \psi_{I, J}\rangle,
	\end{align*}
	where $\calF$ depends on the type of the flag shift and we have defined
	\[
		\calU_{K,k}^b f:=\sum_{\substack{L \in \calD_F
				\\ L \subset K \\ \ell(K^m) \le 2^{k^m}\ell(L^m) }}
		(\langle b\rangle_L-\langle b\rangle_{K^{(\tau)}}) \Delta_{L, F} f.
	\]
	The proof is now ended in the same way as in the almost one-parameter shift case --
	we rely on the proof of Theorem
	\ref{thm:flagshifts} via the $b$-adapted version of Lemma \ref{lem:auxsf}
	that follows exactly as the $b$-adapted version in the previous result.
\end{proof}
The above two results give another proof of Theorem \ref{thm:littlebmocommutator}
for cancellative bi-parameter CZOs (to get the general case, we would still
need to study commutators of partial paraproducts and full paraproducts).
We briefly demonstrate the necessity of the above $\bmo_F$ condition in our generality.

\subsection{Commutator lower bounds}
We give a very brief outline of the AWF arguments without details as
the flag adaptation of the AWF is almost identical to the, say,
the Zygmund adaptation of the AWF \cite{LiMarZygCom}. The AWF
method originates from \cite{HyLpq2021} and is an efficient way to show lower bounds.

\subsubsection{AWF}
We now specify what we mean with non-degeneracy in the flag setting.
First, for these considerations and approximate weak factorization (AWF)
we are going to need only the following upper bounds for our
$K$. First, we assume the size estimate
$$
	|K(x,y)| \lesssim \size_F(x,y),
$$
where we recall that
\[
	\size_F(x,y)=\frac{1}{|x_1-y_1|^{d_1}(|x_1-y_1|+|x_2-y_2|)^{d_2}}.
\]
Second, we need the following weak variants of a subset of the usual
continuity assumptions of flag kernels.
Let $\omega \colon [0,1] \to [0, \infty)$ be increasing with $\omega(t) \to 0$ as $t \to 0$.
Let $|x_i - x_i'| \le |x_i-y_i| / 2$ for $i = 1,2$. Then we demand that
\begin{equation*}
	|K((x_1', x_2), y) - K(x,y)| \lesssim \omega\Big(\frac{|x_1-x_1'|}{|x_1-y_1|}\Big)
	\size_F(x,y)
\end{equation*}
and
\begin{equation*}
	|K((x_1, x_2'), y) - K(x,y)| \lesssim \omega\Big(\frac{|x_2-x_2'|}{|x_2-y_2|}\Big)
	\size_F(x,y)
\end{equation*}
together with the symmetric conditions in the $y$-variable.
\begin{defn}[Non-degenerate flag kernels]
	Let $K$ be a kernel as specified above.
	We say that $K$ is a non-degenerate flag kernel if
	for every $y \in\R^{d}$ and $\delta_1, \delta_2>0$ with $\delta_1 \le \delta_2$ there exists
	$x \in \R^d$ with $|x_1-y_1| > \delta_1$ and $|x_2-y_2| > \delta_2$
	so that there holds
	\[
		|K(x, y)|\gtrsim \frac{1}{\delta_1^{d_1}\delta_2^{d_2}}.
	\]
\end{defn}
It is critical that the restriction $\delta_1 \le \delta_2$ appears above.
Now, it follows that, in the definition of non-degeneracy, the point $x$ will automatically
satisfy $|x_1-y_1| \lesssim \delta_1$ and $|x_2-y_2| \lesssim \delta_2$.

\subsubsection{Reflected flag rectangles}
Let $I = I^1 \times I^2$ be a flag rectangle in $\R^d$ --
that is, $\ell(I^1) \le \ell(I^2)$. Let $c_{I^i}$ be the center of the interval $I^i$
and set $c_I = (c_{I^1}, c_{I^2})$. We use non-degeneracy with $y := c_I$,
$\delta_1 := A^{1/d}\ell(I^1)$ and $\delta_2 := A^{1/d}\ell(I^2)$, where
$A > 0$ is a large but at this point unspecified parameter. We find the related $x$
so that, in particular, we have
$$
	|K(x, c_I)| \gtrsim \frac{1}{A} \frac{1}{\ell(I^1)^{d_1}\ell(I^2)^{d_2}}
	= \frac{1}{A|I|}.
$$
Let $\wt I^i$ be a cube in $\R^{d_i}$ that is centered at $c_{\wt I^i} := x_i$
and of side length $\ell(I^i)$. Then $c_{\wt I} := (c_{\wt I^1}, c_{\wt I^2})$
is the center of the flag rectangle $\wt I := \wt I^1 \times \wt I^2$.
This is the reflected version of $I$. Notice that
\begin{align*}
	|c_{I^1} - c_{\wt I^1}| & \sim A^{1/d}\ell(I^1) \sim \dist(I^1, \wt I^1), \\
	|c_{I^2} - c_{\wt I^2}| & \sim A^{1/d}\ell(I^2) \sim \dist(I^2, \wt I^2)
\end{align*}
and
$$
	|K(c_{\wt I}, c_I)| \sim \frac{1}{A|I|}.
$$

\begin{lem}
	Let $I$ be a flag rectangle, $y \in I$ and $x \in \wt I$. Then we have
	(for some constant $C$) that
	$$
		|K(x,y) - K(c_{\wt I}, c_I)| \lesssim \omega(CA^{-1/d}) \frac{1}{A|I|}.
	$$
	In particular, for large $A$ we have
	$$
		\Big|\int_I K(x,y)\ud y \Big| \sim \int_I |K(x,y)| \ud y \sim \frac{1}{A}, \,\, x \in \wt I,
	$$
	and
	$$
		\Big|\int_{\wt I} K(x,y)\ud x \Big| \sim \int_{\wt I} |K(x,y)| \ud x \sim \frac{1}{A}, \,\, y \in I.
	$$
\end{lem}
\begin{proof}
	The proof is almost verbatim the same as in existing AWF arguments and omitted here.
\end{proof}

\subsubsection{AWF for flag rectangles}
In what follows (for the AWF considerations and also for the necessity results
for commutator boundedness) we do not really need a flag singular integral.
Rather, we only need a non-degenerate kernel $K$, and then $T$ and $T^*$ are simply notation for the integrals
\begin{align*}
	Th(x) = \int_{\R^d} K(x, y) h(y)\ud y, \\
	T^* h(x) = \int_{\R^d} K(y, x)  h(y)\ud y.
\end{align*}
We only use these integrals in situations where there is plenty of separation,
such as, $x \in \wt I$ but $h$ is supported on $I$ (and so $y \in I$). That is,
we use these in off-diagonal situations.

We again omit the following standard proof, see e.g. \cite{LiMarZygCom}.
\begin{lem}[Single iteration AWF]\label{lem: one awf}
	Let $K$ be a non-degenerate flag kernel.
	Let $I = I^1 \times I^2$ be a flag rectangle and $f \in L^1$ be a function
	that is supported on $I$ with $\int_I f = 0$. Then
	$$
		f = h_I T^* 1_{\wt I} - 1_{\wt I} Th_I + e_{\wt I},
	$$
	where
	\begin{enumerate}
		\item $|h_I| \lesssim A|f|$,
		\item $|e_{\wt I}| \lesssim \omega(CA^{-1/d})\langle |f| \rangle_I 1_{\wt I}$,
		\item $\int e_{\wt I} = 0$.
	\end{enumerate}
\end{lem}

We want a few variants and corollaries of this for later use.
It is useful to iterate this one more time to make the error term $e$
be supported in $I$ instead of $\wt I$.
\begin{lem}[twice iterated AWF]
	Let $K$ be a non-degenerate flag kernel.
	Let $I = I^1 \times I^2$ be a flag rectangle, and let $f \in L^1$ be a function that is supported on $I$ with $\int_I f = 0$. Then
	\begin{align*}
		f & = [h_I T^* 1_{\wt I} - 1_{\wt I} Th_I] +
		      [h_{\wt I} T 1_{I} - 1_{I} T^*h_{\wt I}]
		+ e_{I},
	\end{align*}
	where
	\begin{enumerate}
		\item $|h_I| \lesssim A|f|$,
		\item $|h_{\wt I}| \lesssim A\langle |f|\rangle_I 1_{\wt I}$,
		\item $|e_{I}| \lesssim \omega(CA^{-1/d})\langle |f| \rangle_I 1_{I}$,
		\item $\int e_{I} = 0$.
	\end{enumerate}
\end{lem}
\begin{proof}
	Use the single iteration AWF twice: first, apply it to $f$ to get a decomposition as in Lemma \ref{lem: one awf}, and then apply it to the term $e_{\wt I}$ arising in that decomposition, with $1_{\wt I}$ replaced by $1_I$ and $T$ with $T^*$.
\end{proof}

Next, we will bound oscillations
$$
	\osc(b, I) := \frac{1}{|I|}\int_I |b-\langle b \rangle_I|,
$$
where $I$ is a flag rectangle, using the AWF.
\begin{lem}\label{lem:AWFCorBdd}
	Let $K$ be a non-degenerate flag kernel.
	Let $b \in L^{1}_{\loc}$, and let the constant $A$ be large enough. Then for all flag rectangles $I$ there exist functions
	$$
		\varphi_{I, 1}, \, \varphi_{I,2}, \,
		\psi_{\wt I, 1}, \, \psi_{\wt I, 2}
	$$
	so that for $i = 1,2$ we have
	$$
		|\varphi_{I, i}|\lesssim 1_{I} \qquad \textup{and} \qquad |\psi_{\wt I,i}|\lesssim |I|^{-1}1_{\wt I}
	$$
	and
	$$
		\osc(b, I) \lesssim \sum_{i=1}^2 |\langle [b,T]\varphi_{I,i}, \psi_{\wt I, i}\rangle|.
	$$
\end{lem}
\begin{proof}
	Choose $f$ with $\|f\|_{L^{\infty}} \le 1$, $\supp f \subset I$ and $\int_I f = 0$ so that
	$$
		|I|\osc(b, I) =
		\int_I |b-\langle b \rangle_I| \lesssim |\langle b, f\rangle|.
	$$
	Using the twice iterated AWF, we write
	$$
		\langle b, f\rangle = -\langle [b, T]h_I, 1_{\wt I}\rangle
		+ \langle [b,T]1_I, h_{\wt I}\rangle
		+ \langle b, e_I\rangle,
	$$
	where
	\begin{align*}
		|\langle b, e_I\rangle| & \lesssim \omega(CA^{-1/d}) \langle |f| \rangle_I \int_I |b-\langle b \rangle_I| \\
		                        & \le \omega(CA^{-1/d}) |I| \osc(b,I).
	\end{align*}
	This shows that
	\begin{align*}
		\osc(b,I) & \lesssim |\langle [b, T]h_I, |I|^{-1}1_{\wt I}\rangle| \\
		          & + |\langle [b,T]1_I, |I|^{-1}h_{\wt I}\rangle|
		+ \omega(CA^{-1/d}) \osc(b,I).
	\end{align*}
	Fixing $A$ large enough we get
	$$
		\osc(b,I) \lesssim |\langle [b, T]\varphi_{I, 1}, \psi_{\wt I, 1}\rangle|
		+ |\langle [b,T]\varphi_{I, 2}, \psi_{\wt I, 2}\rangle|,
	$$
	where
	\begin{align*}
		\varphi_{I, 1} := h_I \qquad                 & \textup{and} \qquad \varphi_{I, 2} := 1_I, \\
		\psi_{\wt I, 1} := |I|^{-1} 1_{\wt I} \qquad & \textup{and} \qquad \psi_{\wt I, 2}
		:= |I|^{-1} h_{\wt I}.
	\end{align*}
\end{proof}
The constant $A$ is now \textbf{fixed} large enough so that conclusions as in the lemma above hold.

\subsubsection{Necessity results for boundedness}
We define an off-diagonal constant that can be used instead of the full
norm of the commutator $[b,T]$. That is, for these boundedness considerations, we only need the kernel $K$ and the off-diagonal constants instead of an actual SIO and the related commutator.
\begin{defn}
	For $u, t \in (1,\infty)$ and $B(x,y) = b(x)-b(y)$ define
	\begin{align*}
		\Off_{u}^{t} =  \sup \frac{1}{|P_1|^{1+1/u-1/t}}
		\Big| \iint_{\R^{d}\times \R^d} B(x,y)
		K(x,y) f_1(y)f_2(x) \ud y \ud x \Big|,
	\end{align*}
	where the supremum is taken over flag rectangles $P_1 = J^1 \times J^2$
	and $P_2 = L^1 \times L^2$ with
	$$
		\ell(J^i) = \ell(L^i) \qquad \textup{and} \qquad
		\dist (J^i, L^i)\sim \ell(J^i)
	$$
	and over functions $f_i \in L^{\infty}(P_i)$ with $\|f_j\|_{L^{\infty}} \le 1$.
\end{defn}
\begin{lem}\label{lem:off_est}
	Let $K$ be a non-degenerate flag kernel and $b \in L^1_{\loc}$. Then we have
	$$
		\osc(b, I) \lesssim \Off_u^t |I|^{1/u-1/t}, \qquad 1 < u,t < \infty,
	$$
	for all flag rectangles $I$.
\end{lem}
\begin{proof}
	Using AWF as above, estimate
	$$
		\osc(b, I) \lesssim \sum_{i=1}^2 |I|^{-1}|\langle [b,T]\varphi_{I,i}, |I|\psi_{\wt I, i}\rangle|,
	$$
	where $\|\varphi_{I, i}\|_{L^{\infty}} \lesssim 1$ and $\||I| \psi_{\wt I, i}\|_{L^{\infty}} \lesssim 1$,
	and the functions $\varphi_{I, i}, \psi_{\wt I, i}$ are supported on $I$ and $\wt I$, respectively.
	It follows directly from the definition that
	$$
		|I|^{-1}|\langle [b,T]\varphi_{I,i}, |I|\psi_{\wt I, i}\rangle| \lesssim \Off_u^t |I|^{1/u-1/t}
	$$
	so we are done.
\end{proof}

We now use this to quickly obtain a necessary condition for $\Off_p^p < \infty$,
in particular, for $\|[b, T]\|_{L^p \to L^p} < \infty$, $1 < p < \infty$.
Let $\calR_F$ denote the collection of flag rectangles and recall that
$$
	\|b\|_{\bmo_F} = \sup_{I \in \calR_F} \osc(b,I).
$$
\begin{thm}\label{thm:ppbddnec}
	Let $K$ be a non-degenerate flag kernel and $b \in L^1_{\loc}$. Then we have
	$$
		\|b\|_{\bmo_F} \lesssim \Off_p^p, \qquad 1 < p < \infty.
	$$
\end{thm}
\begin{proof}
	By the previous lemma
	$$
		\osc(b, I) \lesssim \Off_p^p
	$$
	for all flag rectangles, and so we are done.
\end{proof}

\section{Cancellative tri-parameter flag representation theorem}
This is one of the heaviest sections in terms of technical details and notation. Although it is similar in spirit to the bi-parameter case (with certain details being easier because we are in the cancellative setting), the fact of the matter is that there are additional cases that need to be considered. While certain parts of the bi-parameter section were written in a more ad hoc style due to the smaller number of cases, we focus the presentation of this section to give the reader an idea of how everything generalizes to the case of $k$-parameter cancellative flag CZOs. We present below a mini-outline for how this section is organized, drawing on analogies with the bi-parameter case when possible.
\begin{itemize}
	\item In Section~\ref{subsec:tripar-blocks}, we set the stage by studying how different amounts of cancellation lead to different philosophies on the number of parameters we consider a shift to have. We saw in the bi-parameter case that studying quantities like $\Sigma_{\Delta,D}$ leads to the consideration of sums over terms of the form e.g.
	      \[
		      B(\Delta_{I^1}\Delta_{I^2}f, E_{J^1}\Delta_{J^2} g),
	      \]
	      where $\ell(I)=\ell(J)$ and they live in some dyadic lattice like $\mathcal{D}_<$. As we have seen before, in this case we will perform a paraproduct correction in the first parameter, we are left with a quantity like
	      \[
		      B(h_{I^1} \otimes h_{I^2}, h_{J^1}^0 \otimes h_{J^2})
		      \langle f, h_{I^1} \otimes h_{I^2}\rangle \langle g, H_{I^1, J^1} \otimes h_{J^2} \rangle.
	      \]
	      Although the function $H_{I^1, J^1} \otimes h_{J^2}$ has the bi-parameter style cancellation, we still split it into two terms: one with $\ell(I^2)>2^{k^1} \ell(I^1)$ and the other with $\ell(I^2)\le 2^{k^1} \ell(I^1)$, and we call them bi-parameter flag shifts and almost one-parameter shifts, respectively. This means that we track not only cancellation but also the size of the summands.   In our tri-parameter flag representation theorem we will study  four different cancellation philosophies.
	\item In Section~\ref{subsec:tripar-case-study}, we present a model case of how to bound a single term out of the 27 that arise in our decomposition.
	\item In Section~\ref{subsec:tripar-statement}, we state the representation theorem, whose form is motivated by the concrete calculation in Section~\ref{subsec:tripar-case-study}.
	\item In Section~\ref{sec:bilinforms}, we address how to deal with summation for abstract bilinear forms, which will be applied to terms that actually arise when proving the representation theorem.
	\item In Section~\ref{subsec:tripar-shift-details}, we discuss the shifts in more detail, focusing on the support conditions that arise for each of the $4$ types of shift operators that appear.
	\item In Section~\ref{subsec:tripar-algorithm}, we develop a set of ``legal moves" for getting sufficient cancellation among the various groupings of the parameters. With these established, we describe an algorithm that can address all 27 of the cases. The basic idea is to use paraproduct corrections to get enough cancellation, though some care must be taken as far as in which cases this can be done productively.

\end{itemize}

\subsection{Tri-parameter martingale block expansions}\label{subsec:tripar-blocks}
In this setting we will assume that $k = (k^1, k^2, k^3)$, $k^m \ge 0$, is fixed and that $K \in \calD^k_F$, which means that
$K = K^1 \times K^2 \times K^3\in \calD = \calD^1 \times \calD^2 \times \calD^3$ and
$2^{-k^1}\ell(K^1) \le 2^{-k^2}\ell(K^2)\le 2^{-k^3}\ell(K^3)$. Our basic assumption
is that $\varphi = \varphi_K$ satisfies $\varphi_K = 1_K \varphi_K$
and $\varphi_K$ is constant on $I = I^1 \times I^2 \times I^3\in \calD$ with
$\ell(I^1) < 2^{-k^1}\ell(K^1)$, $\ell(I^2) < 2^{-k^2}\ell(K^2)$ and $\ell(I^3) < 2^{-k^3}\ell(K^3)$.
Depending on the situation, we assume different cancellation properties.

\subsubsection*{One-parameter expansion} This means that we have $\int_{\R^d} \varphi_K=0 $. We will fix
$l =(l^1, l^2)\in \N^2$ so that $2^{-k^1}\ell(K^1) = 2^{-l^1 } 2^{-k^2} \ell(K^2)$
and $2^{-k^2}\ell(K^2) = 2^{-l^2 } 2^{-k^3} \ell(K^3)$. We expand
\[
	\langle f, \varphi_K\rangle= \left\langle \sum_{\substack{I\in \calD \\ \ell(I^1)=2^{-l^1} \ell(I^2) \\
			\ell(I^2)=2^{-l^2} \ell(I^3) }} \Delta_I f, \varphi_K\right\rangle,
\]
where $\Delta_I$ is a one-parameter martingale difference. Of course, we can assume that $I\cap K\neq \emptyset$
in the summation. If $\ell(I^1)<2^{-k^1}\ell(K^1)$, then we have
\[
	\ell(I^2)=2^{l^1}\ell(I^1)<2^{l^1-k^1}\ell(K^1)= 2^{-k^2} \ell(K^2)
\]
and
\[
	\ell(I^3)=2^{l^2}\ell(I^2)<2^{l^2-k^2}\ell(K^2)= 2^{-k^3} \ell(K^3).
\]
Hence,
\[
	\langle \Delta_I f, \varphi_K\rangle = \langle \varphi_K\rangle_I \int_{\R^d} \Delta_I f =0.
\]
Thus, we must have $\ell(I^1)\ge 2^{-k^1} \ell(K^1)$, $\ell(I^2)\ge 2^{-k^2} \ell(K^2)$ and
$\ell(I^3)\ge 2^{-k^3} \ell(K^3)$.

Next we show that $I$ cannot be too `big'. More precisely, we show that $I\subset K^{(\tau)}$, where $\tau=(\tau^1, \tau^2,
	\tau^3)$ with
\[
	\tau^i= \max\{k^1, k^2, k^3\}-k^i.
\]
We give a slight variant of the argument from the bi-parameter case. Since $I$ and $K^{(\tau)}$ are in the same one-parameter rectangular grid, suppose on the contrary that $I\not\subset K^{(\tau)}$.
Then we must have $I^j\supsetneq (K^{(\tau)})^j$ for each $j=1,2,3$, and we deduce
\[
	\langle \Delta_I f, \varphi_K\rangle =\langle \Delta_I f\rangle_K \int_{\R^d} \varphi_K =0.
\]

We have shown that
\[
	\langle f, \varphi_K\rangle = \left\langle 1_K \sum_{\substack{I\subset K^{(\tau)} \\ \ell(I^1)=2^{-l^1}\ell(I^2)=2^{-l^1-l^2}
			\ell(I^3) \\ \ell(K^m)\le 2^{k^m} \ell(I^m)}}\Delta_I f, \varphi_K\right\rangle =: \langle \calY_{K,k,l}f,
	\varphi_K\rangle.
\]
Similarly to Lemma \ref{lem:auxsfY}, we have the following result.
\begin{lem}\label{lem:triparauxsfY}
	Let
	\[
		\calY_{k, l} f := \Big(\sum_{ K\in \calD} |\calY_{K, k, l}f|^2\Big)^{1/2}.
	\]
	Then for all $p\in (1, \infty)$ and $w\in A_{p, F}$
	we have $$
		\|\calY_{k, l} f\|_{L^p(w)} \lesssim (|k|+1)^{3/2} \|f\|_{L^p(w)}.
	$$
\end{lem}
\begin{proof}
	The bulk of the proof is analogous to that of Lemma \ref{lem:auxsfY}, leading to an estimate like
	\[
		\|\calY_{k, l} f\|_{L^2(w)}^2 \lesssim \sum_{\substack{I\in\calD \\ \ell(I^1)=2^{-l^1}\ell(I^2)=2^{-l^1-l^2}
				\ell(I^3)}} \|\Delta_I f\|_{L^2(w)}^2\left( \sum_{\substack{R\supset I \\ \ell(R^m)\le 2^{\max_j k^j}\ell(I^m),\ m=1,2,3}} 1\right).
	\]
	For each $I$, each coordinate of $R$ has at most $1+\max_j k^j$ possible ancestor generations. Hence
	\[
		\sum_{\substack{R\supset I \\ \ell(R^m)\le 2^{\max_j k^j}\ell(I^m),\ m=1,2,3}}1
		\le (1+\max_j k^j)^3 \le (|k|+1)^3.
	\]
	The weighted one-parameter square function upper bound and taking square roots give the desired $L^2(w)$ estimate. Extrapolation completes the proof.
\end{proof}
\subsubsection*{Bi-parameter expansion: Case I}
The bi-parameter expansion contains two cases. We begin with the case when we have the cancellation
$\int_{\R^{d_{12}}}\varphi_K= \int_{\R^{d_3}} \varphi_K=0$. Fix $l^1$ so that $2^{-k^1}\ell(K^1)=
	2^{-l^1} 2^{-k^2}\ell(K^2)$. Write
\[
	\langle f, \varphi_K \rangle =\left\langle \sum_{\substack{I \in \calD \\ 2^{l^1} \ell(I^1)=\ell(I^2)\le \ell(I^3) }}\tilde\Delta_I f, \varphi_K \right\rangle,
\]
where $\tilde\Delta_I =\Lambda_{I^{12}}\Delta_{I^3} $. Here $\Lambda_{I^{12}}=\Delta_{I^{12}}$ is the one-parameter martingale difference if $\ell(I^2)<\ell(I^3)$, and $\Lambda_{I^{12}}=\Delta_{I^{12}}+E_{I^{12}}$ if $\ell(I^2)=\ell(I^3)$.

As before, we assume $I \cap K \neq \emptyset$. If $\ell(I^3)< 2^{-k^3}\ell(K^3)$, then
\[
	\langle \tilde\Delta_I f, \varphi_K\rangle = \int_{\R^{d_{12}}} \langle \varphi_K\rangle_{I^3}\int_{\R^{d_3}} \tilde\Delta_I f
	=0.
\]
Thus, we may assume $\ell(I^3)\ge 2^{-k^3}\ell(K^3)$. If $\ell(I^2)=\ell(I^3)$, then
\[
	\ell(I^2)=\ell(I^3)\ge 2^{-k^3}\ell(K^3) \ge 2^{-k^2}\ell(K^2)
\]
and
\[
	\ell(I^1) = 2^{-l^1} \ell(I^2)\ge 2^{-l^1} 2^{-k^2}\ell(K^2)= 2^{-k^1}\ell(K^1).
\]
If $\ell(I^2)<\ell(I^3)$, then we have that $\tilde\Delta_I f= \Delta_{I^{12}}\Delta_{I^3} f$. Thus, if
$\ell(I^1)<2^{-k^1}\ell(K^1)$, we have $\ell(I^2)=2^{l^1} \ell(I^1)< 2^{l^1} 2^{-k^1}\ell(K^1) =2^{-k^2}\ell(K^2)$.
It follows that
\[
	\langle \tilde\Delta_I f, \varphi_K\rangle =
	\int_{\R^{d_3}} \langle \varphi_K\rangle_{I^{12}}\int_{\R^{d_{12}}} \tilde\Delta_I f=0.
\]
Hence, the only nonzero contributions occur when we have that $\ell(I^m) \ge 2^{-k^m} \ell(K^m)$ holds for $m=1,2,3$.

Next, we show that $I \subset K^{(\eta)}$, where $\eta=(\eta^1, \eta^2, 0)$ with $\eta^i=\max\{k^1, k^2\}-k^i$. The containment $I^3\subset K^3$ is obvious, and by the same argument as the bi-parameter case we have
$I^{12}\subset (K^{12})^{(\eta^{12})}$. So far, we have shown that
\[
	\langle f, \varphi_K \rangle =\left\langle 1_K\sum_{\substack{I \subset K^{(\eta)} \\ 2^{l^1} \ell(I^1)=\ell(I^2)\le \ell(I^3) \\ \ell(K^m)\le 2^{k^m} \ell(I^m)}}\tilde\Delta_I f, \varphi_K \right\rangle=:\langle \calV_{K,k,l^1} f, \varphi_K\rangle.
\]
\begin{lem}\label{lem:triparauxsfV1}
	Let
	\[
		\calV_{k, l^1} f := \Big(\sum_{ K\in \calD} |\calV_{K, k, l^1}f|^2\Big)^{1/2}.
	\]
	Then for all $p\in (1, \infty)$ and $w\in A_{p, F}$
	we have $$
		\|\calV_{k, l^1} f\|_{L^p(w)} \lesssim (|k|+1)^{3/2} \|f\|_{L^p(w)}.
	$$
\end{lem}
\begin{proof}
	First of all, we claim that
	\begin{equation}\label{eq:Ebiparsf}
		\Big\|\Big(\sum_{\substack{I \in \calD \\ 2^{l^1} \ell(I^1)=\ell(I^2)\le \ell(I^3) }} |\tilde\Delta_I f|^2\Big)^{1/2}\Big\|_{L^p(w)}\lesssim \|f\|_{L^p(w)}.
	\end{equation}
	In fact, the proof is quite similar as in the case of the bi-parameter flag square function. We can write
	\begin{align*}
		 & \sum_{\substack{I \in \calD \\ 2^{l^1} \ell(I^1)=\ell(I^2)\le \ell(I^3) }} |\tilde\Delta_I f|^2             \\
		 & \hspace{2cm}= \sum_{\substack{I \in \calD \\ 2^{l^1} \ell(I^1)=\ell(I^2)=\ell(I^3) }} |\tilde\Delta_I f|^2+
		\sum_{\substack{I \in \calD \\ 2^{l^1} \ell(I^1)=\ell(I^2)< \ell(I^3) }} |\tilde\Delta_I f|^2                  \\
		 & \hspace{2cm}= \sum_{\substack{I \in \calD \\ 2^{l^1} \ell(I^1)=\ell(I^2)=\ell(I^3) }} |\tilde\Delta_I f|^2+
		\sum_{\substack{I \in \calD \\ 2^{l^1} \ell(I^1)=\ell(I^2)< \ell(I^3) }} \sum_{\substack{J^3\in \calD^3\\ \ell(J^3)=\ell(I^2)}} |\Delta_{I^{12}}E_{J^3} \Delta_{I^3} f|^2.
	\end{align*}
	The first term can be bounded suitably using an almost one-parameter square function in $\R^d$. For the second term, write $g_{I^3}=\Delta_{I^3}f$ and observe that
	\[
		|\Delta_{I^{12}}E_{J^3} g_{I^3}(x)|\lesssim \langle |g_{I^3}|\rangle_{I^{12}\times J^3}.
	\]
	Thus, we can bound
	\[
		\bigg\Vert \Big(\sum_{\substack{I \in \calD \\ 2^{l^1} \ell(I^1)=\ell(I^2)< \ell(I^3) }} \sum_{\substack{J^3\in \calD^3\\ \ell(J^3)=\ell(I^2)}} |\Delta_{I^{12}}E_{J^3} \Delta_{I^3} f|^2\Big)^{1/2}\bigg\Vert_{L^p(w)} \lesssim \bigg\Vert\Big(\sum_{I^3\in \calD^3} |\Delta_{I^3} f|^2\Big)^{1/2}\bigg\Vert_{L^p(w)}.
	\]
	via the almost one-parameter square function bound on $\mathbb{R}^d$; then this resulting term can be bounded using a one-parameter square function bound on $\mathbb{R}^{d_3}$. Thus, we get the claim \eqref{eq:Ebiparsf}. The rest of the argument is similar to the proof of Lemma \ref{lem:auxsfY}.
\end{proof}

\subsubsection*{Bi-parameter expansion: Case II}
Here, we have the cancellation $\int_{\R^{d_{1}}}\varphi_K= \int_{\R^{d_{23}}} \varphi_K=0$.
We fix $l^2$ so that we have $2^{-k^2}\ell(K^2)=2^{-l^2} 2^{-k^3} \ell(K^3)$. Then we can write
\begin{align*}
	\langle f, \varphi_K\rangle =  \left\langle \sum_{\substack{I \in \calD \\ \ell(I^1)\le \ell(I^2)=2^{-l^2}  \ell(I^3) }}\hat\Delta_I f, \varphi_K \right\rangle,
\end{align*}
where $\hat\Delta_I = \Lambda_{I^1}\Delta_{I^{23}}$, where $\Delta_{I^{23}}$ is the one-parameter martingale difference and $\Lambda_{I^1}=\Delta_{I^1}$ if $\ell(I^1)< \ell(I^2)$, and otherwise $\Lambda_{I^1}=\Delta_{I^1}+E_{I^1}$.

Again, assume that $I\cap K\neq \emptyset$. First of all, we have $\ell(I^3) \ge 2^{-k^3}\ell(K^3)$ as otherwise we have $\ell(I^m) < 2^{-k^m}\ell(K^m)$ for $m=2,3$ and
\[
	\langle \hat\Delta_I f, \varphi_K\rangle =
	\int_{\R^{d_1}} \langle \varphi_K\rangle_{I^{23}}\int_{\R^{d_{23}}} \hat\Delta_I f=0.
\]
Then we have
\[
	\ell(I^2)=2^{-l^2} \ell(I^3) \ge 2^{-l^2}2^{-k^3}\ell(K^3)= 2^{-k^2} \ell(K^2).
\]
If we suppose that $\ell(I^1)<\ell(I^2)$, then we should also have that $\ell(I^1) \ge 2^{-k^1}\ell(K^1)$ as otherwise
\[
	\langle \hat\Delta_I f, \varphi_K\rangle =
	\int_{\R^{d_{23}}} \langle \varphi_K\rangle_{I^{1}}\int_{\R^{d_{1}}} \hat\Delta_I f=0.
\]
If $\ell(I^1)=\ell(I^2)$, then trivially
\[
	\ell(I^1)=\ell(I^2)\ge 2^{-k^2} \ell(K^2)\ge 2^{-k^1} \ell(K^1).
\]
So again we have proved that $\ell(I^m)\ge  2^{-k^m} \ell(K^m).$

Next we show that $I\subset K^{(\rho)}$, where $\rho=(0, \rho^2, \rho^3)$ with $\rho^m= \max \{k^2, k^3\}-k^m$.  In fact, $I^1 \subset K^1$ is trivial, by the fact that $\int_{\R^{d_1}} \varphi_K=0$. To see that
$I^{23} \subset (K^{23})^{(\rho^2,\rho^3)}$, since they are in the same grid, we suppose for contradiction that $I^{23} \not\subset (K^{23})^{(\rho^2,\rho^3)}$. Then we must have  $I^{23}\supsetneq (K^{23})^{(\rho^2,\rho^3)}$ and
\[
	\langle \hat\Delta_I f, \varphi_K\rangle =
	\int_{\R^{d_{1}}} \langle \hat\Delta_I f\rangle_{I^{23}}\int_{\R^{d_{23}}} \varphi_K=0.
\]
Thus we have proved
\[
	\langle f, \varphi_K\rangle =\left\langle 1_K\sum_{\substack{I \subset K^{(\rho)} \\ \ell(I^1)\le\ell(I^2)=2^{-l^2}  \ell(I^3) \\ \ell(K^m)\le 2^{k^m} \ell(I^m)}}\hat\Delta_I f, \varphi_K \right\rangle=:\langle \calV_{K,k,l^2} f, \varphi_K\rangle.
\]
\begin{lem}\label{lem:triparauxsfV2}
	Let
	\[
		\calV_{k, l^2} f := \Big(\sum_{ K\in \calD} |\calV_{K, k, l^2}f|^2\Big)^{1/2}.
	\]
	Then for all $p\in (1, \infty)$ and $w\in A_{p, F}$
	we have $$
		\|\calV_{k, l^2} f\|_{L^p(w)} \lesssim (|k|+1)^{3/2} \|f\|_{L^p(w)}.
	$$
\end{lem}
\begin{proof}
	It will be enough to prove the following square function estimate
	\[
		\Big\|\Big(\sum_{\substack{I \in \calD \\ \ell(I^1)\le\ell(I^2)= 2^{-l^2} \ell(I^3) }} |\hat\Delta_I f|^2\Big)^{1/2}\Big\|_{L^p(w)}\lesssim \|f\|_{L^p(w)}.
	\]
	Write
	\begin{align*}
		 & \sum_{\substack{I \in \calD \\ \ell(I^1)\le\ell(I^2)= 2^{-l^2} \ell(I^3) }} |\hat\Delta_I f|^2              \\
		 & \hspace{1cm}= \sum_{\substack{I \in \calD \\ \ell(I^1)=\ell(I^2)= 2^{-l^2} \ell(I^3) }} |\hat\Delta_I f|^2
		+ \sum_{\substack{I \in \calD \\ \ell(I^1)<\ell(I^2)= 2^{-l^2} \ell(I^3) }} |\Delta_{I^1} \Delta_{I^{23}} f|^2 \\
		 & \hspace{1cm}=  \sum_{\substack{I \in \calD \\ \ell(I^1)=\ell(I^2)= 2^{-l^2} \ell(I^3) }} |\hat\Delta_I f|^2
		+ \sum_{\substack{I \in \calD \\ \ell(I^1)<\ell(I^2)= 2^{-l^2} \ell(I^3) }}\sum_{\substack{J^{23}\in \calD^{23}\\ \ell(J^2)=\ell(J^3)=\ell(I^1)}} |\Delta_{I^1} E_{J^{23}}\Delta_{I^{23}} f|^2.
	\end{align*}
	The remainder of the estimate follows along the lines of the proof of Lemma \ref{lem:triparauxsfV1}.
\end{proof}
\subsubsection*{Tri-parameter flag expansion}
Here we assume $\int_{\R^{d_1}} \varphi_K =\int_{\R^{d_2}} \varphi_K=\int_{\R^{d_3}} \varphi_K=0$. Write
\begin{align*}
	\langle f, \varphi_K\rangle =\left \langle \sum_{I \in \calD_F}\Delta_{I, F} f,  \varphi_K\right\rangle.
\end{align*}
Again, we assume $I \cap K \neq \emptyset$. First of all, it is clear that we have $I^m \subset K^m$. Suppose that $\ell(I^3)<2^{-k^3} \ell(K^3)$, then
\[
	\langle \Delta_{I, F} f, \varphi_K \rangle = \int_{\R^{d_{12}}} \langle \varphi_K\rangle_{I^3} \int_{\R^{d_3}} \Delta_{I, F} f=0.
\]
Hence $\ell(I^3)\ge 2^{-k^3} \ell(K^3)$. If $\ell(I^2)= \ell(I^3)$, then
\[
	\ell(I^2)= \ell(I^3)\ge 2^{-k^3} \ell(K^3)  \ge 2^{-k^2} \ell(K^2).
\]
While if $\ell(I^2)< \ell(I^3)$, this means $\int_{\R^{d_2}} \Delta_{I, F} f=0$ and suppose that
$\ell(I^2)<2^{-k^2} \ell(K^2)$ we also have
\[
	\langle \Delta_{I, F} f, \varphi_K \rangle = \int_{\R^{d_{13}}} \langle \varphi_K\rangle_{I^2} \int_{\R^{d_2}} \Delta_{I, F} f=0.
\]Thus $\ell(I^2)\ge 2^{-k^2} \ell(K^2)$.
Similarly we have $\ell(I^1)\ge 2^{-k^1} \ell(K^1)$. So far we have proved
\[
	\langle f, \varphi_K\rangle =\left \langle 1_K\sum_{\substack{I \in \calD_F\\ I \subset K \\ \ell(I^m) \ge 2^{-k^m} \ell(K^m)}}\Delta_{I, F} f,  \varphi_K\right\rangle=:\langle \calU_{K,k} f, \varphi_K\rangle.
\]
With the help of Theorem \ref{thm:square}, we record the following lemma, whose proof is completely similar to the one above.
\begin{lem}\label{lem:auxsftri-par}
	Let
	\[
		\calU_kf := \Big(\sum_{ K\in \calD} |\calU_{K, k}f|^2\Big)^{1/2}.
	\]
	Then for all $p\in (1, \infty)$ and $w\in A_{p, F}$
	we have $$
		\|\calU_k f\|_{L^p(w)}\lesssim (|k|+1)^{3/2} \|f\|_{L^p(w)}.
	$$
\end{lem}

\subsubsection{Sizes of flag rectangles}
Given $K \in \calD$, let $K_F\in \calD_F$ be the dyadic rectangle containing $K$ with
\[
	K_F^1 = K^1, \qquad
	\ell(K_F^m)=\max_{1\le j\le m}\ell(K^j), \quad m=2,3.
\]
In particular, for later use we record here that
\begin{align*}
	\frac{|K|}{|K_F|} = \max\bigg\{ \frac{\ell(K^1)}{\ell(K^2)}, 1 \bigg\}^{-d_2}
	\max\bigg\{ \frac{\ell(K^1)}{\ell(K^3)}, \frac{\ell(K^2)}{\ell(K^3)}, 1 \bigg\}^{-d_3}.
\end{align*}

In our shifts, we will require that the coefficients satisfy
$$
	|a_{IJK}| \le \frac{|I|}{|K_F|} = \frac{|K|}{|K_F|}\frac{|I|}{|K|}.
$$
One thing worth noticing is that the condition $I \in \calD_F$ can also be written as
\begin{equation}\label{eq:k}
	2^{-k^1}\ell(K^1) \le 2^{-k^2}\ell(K^2) \le 2^{-k^3}\ell(K^3),
\end{equation}
so this is the basic structure satisfied by the $K \in \calD$ that appear in
the outer summation of the flag shifts.
Recall that we denote the collection of rectangles satisfying \eqref{eq:k} by $\calD_{F}^k$. Then
$$
	\sum_{K \in \calD}
	\sum_{\substack{ I, J \in \calD_F \\ I^{(k)} = J^{(k)} = K}}
	= \sum_{K \in \calD_F^k}
	\sum_{\substack{ I, J \in \calD \\ I^{(k)} = J^{(k)} = K}}.
$$
We define $\calD_{=}^k$, $\calD_{=, <}^k$ $\calD_{<, =}^k$ and $\calD_{<}^k$
analogously to the collections $\calD_=$, $\calD_{=, <}$, $\calD_{<, =}$ and $\calD_<$
-- for instance, if $K \in \calD_{=}^k$ then $2^{-k^1}\ell(K^1) = 2^{-k^2}\ell(K^2) = 2^{-k^3}\ell(K^3)$.
Notice that for $K \in \calD_F^k$, we have
$$
	\max\bigg\{ \frac{\ell(K^1)}{\ell(K^2)}, 1 \bigg\} \le \max\{2^{k^1-k^2}, 1 \}
	= 2^{\max\{k^1 - k^2, 0 \}} = 2^{(k^1-k^2)_+},
$$
so that
$$
	\max\bigg\{ \frac{\ell(K^1)}{\ell(K^2)}, 1 \bigg\}^{-d_2} \ge 2^{-d_2(k^1-k^2)_+}.
$$
Equality holds when the first two input side lengths agree:
\[
	2^{-k^1}\ell(K^1) = 2^{-k^2}\ell(K^2).
\]
Otherwise, no decay from this factor is needed. We will therefore use
$$
	\max\bigg\{ \frac{\ell(K^1)}{\ell(K^2)}, 1 \bigg\}^{-d_2} \le
	\begin{cases} 2^{-d_2(k^1-k^2)_+} & \textup{if } 2^{-k^1}\ell(K^1) = 2^{-k^2}\ell(K^2), \\
	              1                   & \textup{otherwise}.
	\end{cases}
$$
Similarly, we estimate
$$
	\max\bigg\{ \frac{\ell(K^1)}{\ell(K^3)}, \frac{\ell(K^2)}{\ell(K^3)}, 1 \bigg\}^{-d_3} \le
	\begin{cases}
		2^{-d_3\max(k^1-k^3, k^2-k^3, 0)} & \textup{if } K \in \calD_=^k,      \\
		2^{-d_3(k^2-k^3)_+}               & \textup{if } K \in \calD_{<, =}^k, \\
		1                                 & \textup{otherwise}.
	\end{cases}
$$
In total, we have obtained that
\begin{equation}\label{eq:ratiodecay}
	\frac{|K|}{|K_F|} \le
	\begin{cases}
		2^{-d_2(k^1-k^2)_+} 2^{-d_3\max(k^1-k^3, k^2-k^3, 0)} & \textup{if } K \in \calD_=^k,      \\
		2^{-d_3(k^2-k^3)_+}                                   & \textup{if } K \in \calD_{<, =}^k, \\
		2^{-d_2(k^1-k^2)_+}                                   & \textup{if } K \in \calD_{=, <}^k, \\
		1                                                     & K \in \calD_<^k,
	\end{cases}
\end{equation}
which will be a sufficiently good estimate for us from the point of view of the
decay coming from the kernel estimates in the representation theorem of flag CZOs. That is,
here we (for instance) have no decay if $K \in \calD_<^k$ (i.e., $I \in \calD_<$), but we have more
cancellation in the representation theorem to obtain the required decay otherwise.
Conversely, we have some extra decay here in the other cases, but
the lack of cancellation when $\ell(I^1) = \ell(I^2)$ or $\ell(I^2) = \ell(I^3)$ (i.e., $2^{-k^1}\ell(K^1) = 2^{-k^2}\ell(K^2)$
or $2^{-k^2}\ell(K^2) = 2^{-k^3}\ell(K^3)$) will result in otherwise worse estimates
in the representation theorem. Thus, these things will eventually balance out nicely.

\subsection{A case study}\label{subsec:tripar-case-study}
For the bilinear form $B(f,g)=\langle Tf, g\rangle$, the flag multiresolution gives
\[
	B(f,g)=\E \sum_{s\in \calS}\Sigma_s,
\]
where $\calS=\{(s_1, s_2, s_3): s_i \in \{\Delta, E, D\}\}$ and the averaging is over the usual random
dyadic lattices. Now, to understand all of the $3^3 = 27$ cases, we need an algorithm, which we will present in Section~\ref{subsec:tripar-algorithm}.

First, to  get a feeling of how this should work,
we will run the argument in the explicit case $s = (D, E, \Delta)$.
Before writing the argument, we describe a flowchart type of process
that exactly predicts the type of operators we will eventually get
from this particular $\Sigma_s = \Sigma_{(D, E, \Delta)}$.
For the flowchart, $\pc_1 f$ denotes a paraproduct correction
in $\R^{d_1}$ and $\pc_{12} f$ a paraproduct correction in $\R^{d_1 + d_2}$ (in the one-parameter sense)
performed on the function $f$.
A paraproduct correction simply means the process of ``adding and subtracting'' a suitable average
that was used repeatedly in the bi-parameter proof; the above operations mean that, say,
$1_{I^1}\langle f \rangle_{I^1}$ is replaced by $1_{I^1}(\langle f \rangle_{I^1} - \langle f \rangle_{J^1})$
or $1_{I^{12}}\langle f \rangle_{I^{12}}$ is replaced by $1_{I^{12}}
	(\langle f \rangle_{I^{12}} - \langle f \rangle_{J^{12}})$) (we emphasize that here we require $\ell(I^1)=\ell(I^2)$).
Here we are working with a cancellative SIO
so the paraproduct correction can be done for free -- no residual paraproducts actually appear.

Thus, we will start with $\Sigma_s$ and do the following procedure.
\begin{enumerate}
	\item Split the sum into the terms where $ \ell(I^1)=\ell(I^2) $ and where $ \ell(I^1)<\ell(I^2) $.
	\item For summation over the terms where $\ell(I^1)=\ell(I^2) $, we perform the paraproduct correction $\pc_{12}$ when necessary (i.e. when we have a non-cancellative operator in $f$ or $g$). Otherwise, we keep it as is.
	\item For summation over the terms where $ \ell(I^1)<\ell(I^2) $, we further split it into terms where $ \ell(K^1)<\ell(I^2) $ and where $  \ell(K^1)\ge \ell(I^2)$. We perform a paraproduct correction $\pc_2$ to the former group.
	\item Split all the resulting terms above based on whether $\ell(I^3)>\max\{\ell(K^1), \ell(K^2)\}$ or $\ell(I^3)\le \max\{\ell(K^1), \ell(K^2)\}$. Then, perform a paraproduct correction $\pc_3$ to the first group of terms.
\end{enumerate}
A superscript $c1$ on $\Sigma_s$ means that
both in $f$ and $g$ the associated Haar type function is cancellative
in the first parameter. The superscript $n1$ means that
we can have a non-cancellative operator in $f$ or $g$.
For instance, if we have the superscript $c12$ (as in $\Sigma_s^{c12}$) it means that
in this term we have at least one-parameter style cancellation in $\R^{d_1+d_2}$ in both $f,g$ so we
can, for instance, have $f$ being paired with $h_{I^{12}}^0 - h_{J^{12}}^0$
and $g$ paired with $h_{J^1} \otimes h_{J^2}$. Of course, we can also view
$c_2$ (or $c_3$) as the weaker information $c23$ -- and we sometimes do this or something analogous.
But this is only done when we have arranged $\ell(I^2) \sim \ell(I^3)$.

It is evident that we need to track multiple things beyond cancellation in $f$ and $g$ --
we need to track the summation conditions, the support conditions and
the decay from the coefficients as well. This is not part of this high level notation
that only tracks the cancellation in $f$ and $g$.
But the summing conditions are recorded on the edges between the nodes
and the support conditions can be tracked by tracking the paraproduct corrections
up to the given node
(this process is the only thing that can enlarge the support of the associated Haar-like functions).
See figure \ref{fig:DEDelta} -- it makes sense to return to this after having read the corresponding written
argument.
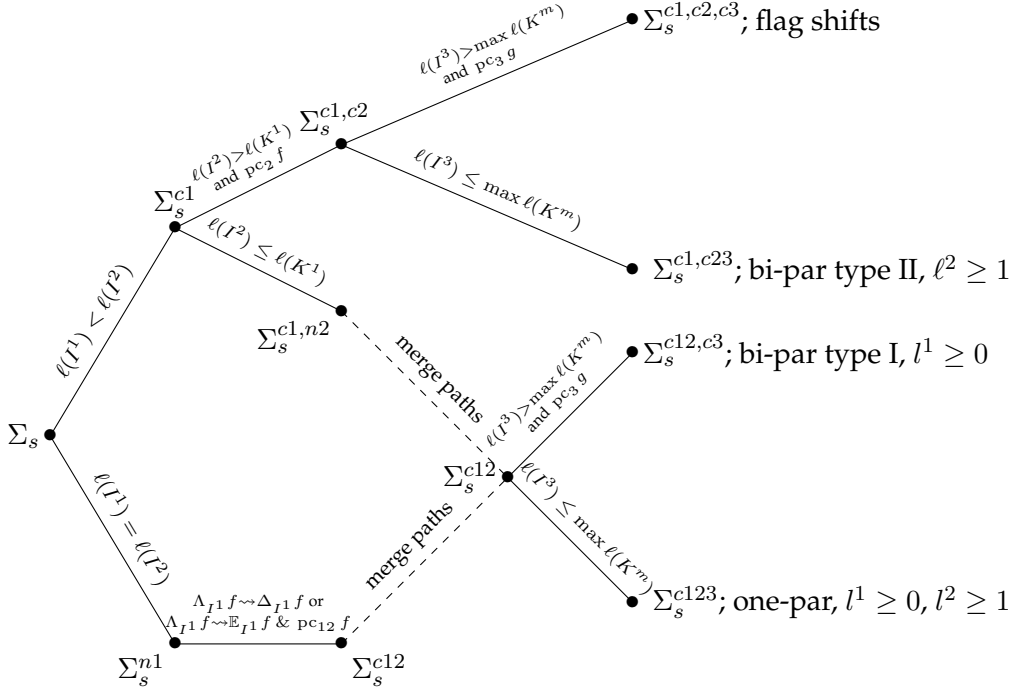
\begin{figure}[ht]
	\centering
	\begin{tikzpicture}[scale=1.1]
		\draw (-2.5, -2.5) node {$\bullet$} node[left] {$\Sigma_s$};
		\draw (-2.5,-2.5) -- node[sloped, above, font=\scriptsize] {$\ell(I^1) < \ell(I^2)$} (-1,0) node{$\bullet$}  node[above] {$\Sigma_s^{c1}$};
		\draw (-2.5,-2.5) -- node[sloped, above, font=\scriptsize] {$\ell(I^1) = \ell(I^2)$} (-1,-5) node{$\bullet$}  node[below left] {$\Sigma_s^{n1}$};

		\draw (-1,0) -- node[sloped, above, font=\scriptsize] {$\substack{\ell(I^2) > \ell(K^1) \\ \textup{and } \pc_2 f}$} (1,1) node{$\bullet$}  node[above] {$\Sigma_s^{c1,c2}\,\,\,$};
		\draw (-1,0) -- node[sloped, above, font=\tiny] {$\ell(I^2) \le \ell(K^1)$} (1,-1) node{$\bullet$}  node[below left] {$\Sigma_s^{c1,n2}$};

		\draw (-1,-5) -- node[sloped, above, font=\scriptsize] {$\substack{ \Lambda_{I^1} f \rightsquigarrow \Delta_{I^1}f \textup{ or} \\ \Lambda_{I^1} f \rightsquigarrow \E_{I^1}f \, \& \, \pc_{12}f  }$} (1,-5) node{$\bullet$}  node[below right] {$\Sigma_s^{c12}$};

		\draw[dashed] (1,-1) -- node[sloped, above, font=\scriptsize] {merge paths} (3,-3);
		\draw[dashed] (1, -5) -- node[sloped, above, font=\scriptsize] {merge paths} (3,-3) node{$\bullet$} node[left] {$\Sigma_s^{c12}$};

		\draw (3, -3) -- node[sloped, above, font=\scriptsize] {$\substack{\ell(I^3) > \max \ell(K^m) \\ \textup{and } \pc_3 g}$} (4.5, -1.5) node{$\bullet$} node[right] {$\Sigma_s^{c12, c3}$; bi-par type I, $l^1 \ge 0$};
		\draw (3, -3) -- node[sloped, above, font=\tiny] {$\,\,\ell(I^3) \le \max \ell(K^m)$} (4.5, -4.5) node{$\bullet$} node[right] {$\,\, \Sigma_s^{c123}$; one-par, $l^1 \ge 0$, $l^2 \ge 1$};

		\draw (1, 1) -- node[sloped, above, font=\scriptsize] {$\substack{\qquad \ell(I^3) > \max \ell(K^m) \\ \textup{and } \pc_3 g}$} (4.5, 2.5) node{$\bullet$} node[right] {$\Sigma_s^{c1, c2, c3}$; flag shifts};
		\draw (1, 1) -- node[sloped, above, font=\tiny] {$\,\,\ell(I^3) \le \max \ell(K^m)$} (4.5, -0.5) node{$\bullet$} node[right] {$\,\,\Sigma_s^{c1, c23}$; bi-par type II, $\ell^2 \ge 1$};
	\end{tikzpicture}
	\caption{The multiresolution of $\Sigma_s = \Sigma_{(D, E, \Delta)}$}
	\label{fig:DEDelta}
\end{figure}

We now follow the above flowchart to produce a written proof. We study
$$
	\Sigma_{(D,E,\Delta)}=
	\sum_{\substack{I, J\in \calD_{\le, < }\\ \ell(I)=\ell(J)}} B(\Lambda_{I^1}E_{I^2} \Delta_{I^3} f, \Lambda_{J^1}\Delta_{J^2} E_{J^3}g).
$$
In what follows we often also write $s$ in place of the explicit $(D, E, \Delta)$.
Actually, we would, as usual, need to study the expectation of this quantity and add
the relevant $k^m$-goodness to the cubes $I^m$ at the appropriate point of the argument --
we omit this below as it goes exactly like in the bi-parameter situation
and would now only serve as a technical distraction (but it is required to make sure that the
various $K^m$ cubes really are common $k^m$-parents of $I^m, J^m$). Throughout the proof, given any cubes $I,J$ we let $K(I,J)=K^1\times K^2\times K^3$ denote a common ancestor of $I$ and $J$ that will be defined more carefully below. By abuse of notation, we will suppress the dependence on $I,J$ and just write $K$ for $K(I,J)$.

First, we write
\begin{align*}
	\Sigma_{(D,E,\Delta)} & =\sum_{\substack{I, J\in \calD_{<, < }\\ \ell(I)=\ell(J)}} B(\Delta_{I^1}E_{I^2} \Delta_{I^3} f, \Delta_{J^1}\Delta_{J^2} E_{J^3}g)                \\
	                      & \hspace{2cm}+ \sum_{\substack{I, J\in \calD_{=, < }\\ \ell(I)=\ell(J)}} B(\Lambda_{I^1}E_{I^2} \Delta_{I^3} f, \Lambda_{J^1}\Delta_{J^2} E_{J^3}g) \\
	                      & =: \Sigma_{(D,E,\Delta)}^{c_1}+\Sigma_{(D,E,\Delta)}^{n1}.
\end{align*}
We consider $\Sigma_{(D,E,\Delta)}^{c_1}$ first. As in the bi-parameter case,
we perform the business with the translation parameter $J^1 = I^1 \dotplus n^1$ and the related dyadic
blocks captured by $k^1 \ge 2$. We explicitly write
\begin{equation*}
	K^1:= \begin{cases}
		(I^1)^{(k^1)},\qquad & 0\ne |n^1|\in (2^{k^1-3}, 2^{k^1-2}] \,\, \text{for some}\,\, k^1\ge 2, \\[1em]
		I^1,\qquad           & n^1=0.
	\end{cases}
\end{equation*}
With this notation we split to the parts $\ell(I^2) > \ell(K^1)$ and $\ell(I^2) \le \ell(K^1)$ giving
us the splitting
\begin{align*}
	\Sigma_{(D,E,\Delta)}^{c_1} = \Sigma_{(D,E,\Delta)}^{c1,c2} +\Sigma_{(D,E,\Delta)}^{c1,n2},
\end{align*}
where by the cancellation assumption, $\Sigma_{(D,E,\Delta)}^{c1,c2}$ equals (this
is the paraproduct correction $\pc_{2} f$)
\begin{align*}
	\sum_{\substack{I, J\in \calD_{<, < }\\ \ell(I)=\ell(J)\\ \ell(I^2)>\ell(K^1)}} & B(h_{I^1}\otimes 1_{I^2} \otimes h_{I^3}, h_{J^1}\otimes h_{J^2} \otimes h_{J^3}^0)\Big\langle f,   h_{I^1}\otimes\frac{1_{I^{2}}}{|I^{2}|}\otimes h_{I^3}\Big\rangle  \big\langle g,  h_{J^1}\otimes h_{J^2}\otimes h_{J^3}^0\big\rangle \\
	                                                                                & = \sum_{\substack{I, J\in \calD_{<, < }\\ \ell(I)=\ell(J)\\ \ell(I^2)>\ell(K^1)}} B(h_{I^1}\otimes h_{I^2}^0 \otimes h_{I^3}, h_{J^1}\otimes h_{J^2} \otimes h_{J^3}^0)                                                                   \\
	                                                                                & \hspace{4cm} \times \Big\langle f,   h_{I^1}\otimes H_{I^2, J^2}\otimes h_{I^3}\Big\rangle  \big\langle g,  h_{J^1}\otimes h_{J^2}\otimes h_{J^3}^0\big\rangle,
\end{align*}
where $H_{I^2, J^2}=h_{I^2}^0-h_{J^2}^0$.

To complete the second parameter related considerations we return
to $\Sigma_{(D,E,\Delta)}^{n1}$. We split it to the following two terms
\begin{align*}
	\sum_{\substack{I, J\in \calD_{=, < }\\ \ell(I)=\ell(J)}} B(\Delta_{I^1}E_{I^2} \Delta_{I^3} f, \Lambda_{J^1}\Delta_{J^2} E_{J^3}g)+ \sum_{\substack{I, J\in \calD_{=, < }\\ \ell(I)=\ell(J)}} B(E_{I^1}E_{I^2} \Delta_{I^3} f, \Lambda_{J^1}\Delta_{J^2} E_{J^3}g),
\end{align*}
where the first term is already acceptable and the second, by the cancellation assumption, equals
\begin{align*}
	 & \sum_{\substack{I, J\in \calD_{=, < }\\ \ell(I)=\ell(J)}} B(1_{I^{12}} \otimes h_{I^3}, u_{J^1}\otimes h_{J^2} \otimes h_{J^3}^0)\Big\langle f,   \frac{1_{I^{12}}}{|I^{12}|}\otimes h_{I^3}\Big\rangle  \big\langle g,  u_{J^1}\otimes h_{J^2}\otimes h_{J^3}^0\big\rangle \\
	 & =\sum_{\substack{I, J\in \calD_{=, < }\\ \ell(I)=\ell(J)}} B(h_{I^{12}}^0 \otimes h_{I^3}, u_{J^1}\otimes h_{J^2} \otimes h_{J^3}^0)\Big\langle f,   H_{I^{12}, J^{12}}\otimes h_{I^3}\Big\rangle  \big\langle g,  u_{J^1}\otimes h_{J^2}\otimes h_{J^3}^0\big\rangle,
\end{align*}
where $H_{I^{12}, J^{12}}= h_{I^{12}}^0- h_{J^{12}}^0$. This was the paraproduct correction $\pc_{12}f$.
These terms together are equal to $\Sigma_s^{c12}$ (which is just a modified way to write
$\Sigma_s^{n_1}$). We will now modify the term $\Sigma_s^{c12}$ to also include the
previously obtained $\Sigma_s^{c_1, n2}$ -- it is a similar term whose $c1$ cancellation also implies $c12$ cancellation, and while there we do not have $\ell(I^1) = \ell(I^2)$,
we still have $\ell(I^1) < \ell(I^2) \le \ell(K^1) = 2^{k^1}\ell(I^1)$.

Similarly as before, we define
\begin{equation*}
	K^2:= \begin{cases}
		(I^2)^{(k^2)},\qquad & 0\ne |n^2|\in (2^{k^2-3}, 2^{k^2-2}] \,\, \text{for some}\,\, k^2\ge 2, \\[1em]
		I^2,\qquad           & n^2=0.
	\end{cases}
\end{equation*}

As depicted in Figure \ref{fig:DEDelta}
both $\Sigma_s^{c1, c2}$ and $\Sigma_s^{c12}$ split into two terms according to whether
$\ell(I^3) > \max(\ell(K^1), \ell(K^2))$ or $\ell(I^3) \le \max(\ell(K^1), \ell(K^2))$.
In particular, performing the step $\pc_3 g$ leads to
\begin{align*}
	 & \Sigma_{(D,E,\Delta)}^{c1,c2}                                                                                                                     \\
	 & = \sum_{\substack{I, J\in \calD_{<, < }\\ \ell(I)=\ell(J)\\ \ell(I^2)> \ell(K^1)\\ \ell(I^3)\le \max\{\ell(K^1),\ell(K^2)\}}}
	a_{IJ} \Big\langle f,   h_{I^1}\otimes H_{I^2, J^2}\otimes h_{I^3}\Big\rangle  \big\langle g,  h_{J^1}\otimes h_{J^2}\otimes h_{J^3}^0\big\rangle    \\
	 & + \sum_{\substack{I, J\in \calD_{<, < }\\ \ell(I)=\ell(J)\\ \ell(I^2)> \ell(K^1)\\ \ell(I^3)> \max\{\ell(K^1), \ell(K^2)\}}}
	a_{IJ} \Big\langle f,   h_{I^1}\otimes H_{I^2, J^2}\otimes h_{I^3}\Big\rangle  \big\langle g,  h_{J^1}\otimes h_{J^2}\otimes H_{I^3, J^3}\big\rangle \\
	 & =: \Sigma_{(D, E, \Delta)}^{c1, c23} + \Sigma_{(D, E, \Delta)}^{c1, c2, c3},
\end{align*}
where $a_{IJ} := B(h_{I^1}\otimes h_{I^2}^0 \otimes h_{I^3}, h_{J^1}\otimes h_{J^2} \otimes h_{J^3}^0)$.
Analogously, we may also decompose $\Sigma_s^{c12} = \Sigma_s^{c12, c3} + \Sigma_s^{c123}$.

In total, we have four classes of terms appearing at this point:
\begin{itemize}
	\item Terms with $c1c2c3$ cancellation
	\item Terms with $c12c3$ cancellation where we have comparability among the first two parameters
	\item Terms with $c1c23$ cancellation where we have comparability among the last two parameters
	\item Terms with $c123$ cancellation where we have comparability among all three parameters
\end{itemize}

Recalling the coefficient estimates of Theorem \ref{thm:tripar coeffs}, we deduce that
\[
	\E \Sigma_{(D, E, \Delta)}^{c1, c23} = C\E \sum_{k^1=0}^\infty \sum_{k^2=2}^\infty \sum_{k^3=0}^\infty2^{-k^1 \alpha_1} 2^{-k^2 \alpha_2}2^{-k^3 \alpha_3}\sum_{l^2=1}^{k^2}\bla R_{k, l^2}^2f, g\bra
\]
and
\[
	\E \Sigma_{(D, E, \Delta)}^{c1, c2, c3} =
	C\E \sum_{k^1=0}^\infty \sum_{k^2=2}^\infty \sum_{k^3=0}^\infty2^{-k^1 \alpha_1} 2^{-k^2 \alpha_2}2^{-k^3 \alpha_3} \bla Q_{k}f, g\bra.
\]
Then, we have that $\E \Sigma_{(D,E,\Delta)}^{c12, c3}$ equals
\begin{align*}
	C\E \sum_{k^1=0}^\infty & \sum_{k^2=0}^\infty \sum_{k^3=2}^\infty
	                                                           2^{-k^1 \alpha_1} 2^{-k^2 \alpha_2}2^{-k^3 \alpha_3} \sum_{l^1=1}^{k^1}\bla R_{k, l^1}^1f, g\rangle \\
	                        & + C\E \sum_{k^1=0}^\infty \sum_{k^2=0}^\infty \sum_{k^3=2}^\infty
	                                                                                     2^{-k^2 \alpha_2}2^{-k^3 \alpha_3} \bla R_{k, 0 }^1f, g\bra.
\end{align*}
Finally, we can write $\E \Sigma_{(D,E,\Delta)}^{c123}$ as
\begin{align*}
	C\E \sum_{k^1=0}^\infty \sum_{k^2=0}^\infty & \sum_{k^3=0}^\infty
	                                                           2^{-k^1 \alpha_1} 2^{-k^2 \alpha_2}2^{-k^3 \alpha_3}
	\sum_{l^1=1}^{k^1}\sum_{l^2=1}^{\max\{k^1-l^1, k^2\}}\bla R_{k, l^1, l^2} f, g\bra                                                                                                                    \\
	                                            & +C\E \sum_{k^1=0}^\infty \sum_{k^2=0}^\infty \sum_{k^3=0}^\infty
	                                                                                                        2^{-k^2 \alpha_2}2^{-k^3 \alpha_3} \sum_{l^2=1}^{\max\{k^1, k^2\}}\bla R_{k, 0, l^2}f, g\bra.
\end{align*}
Here $l^2 \le \max \{k^1-l^1, k^2\}$ follows from that
\[
	2^{l^1+l^2}\ell(I^1) = \ell(I^3) \le \max\{\ell(K^1), \ell(K^2)\}= \max \{2^{k^1}, 2^{k^2+l^1}\}\ell(I^1).
\]
With this, we have addressed $1$ of the $27$ terms that arise.

\subsection{The representation theorem in general}\label{subsec:tripar-statement}
\begin{thm}\label{thm: triparam rep}
	Suppose that $T$ is a cancellative tri-parameter flag CZO. Then we have
	\begin{align*}
		\langle Tf, g\rangle & = C \E \bigg(\sum_{k^1, k^2, k^3=0}^\infty 2^{-\alpha_1 k^1-\alpha_2 k^2 -\alpha_3 k^3} \Big[\langle Q_k f, g\rangle+\sum_{l^1=1}^{k^1} \sum_{l^2=1}^{\max\{k^1-l^1, k^2\}} \langle R_{k, l^1, l^2}f, g\rangle\\
			                                                                                                               &\qquad+\sum_{l^1=1}^{k^1} \langle R_{k, l^1}^1 f, g\rangle + \sum_{l^2=1}^{k^2} \langle R_{k, l^2}^2 f, g\rangle\Big]
		\\ & \hspace{1.5cm} +\sum_{k^1, k^2, k^3=0}^\infty 2^{-\alpha_2 k^2-\alpha_3 k^3 }\Big[ \sum_{l^2=1}^{\max\{k^1, k^2\} }\langle R_{k, 0, l^2}f, g\rangle+ \langle R_{k, 0}^1 f, g\rangle
			                                                                                  \Big]\\ & \hspace{1.5cm} +\sum_{k^1, k^2, k^3=0}^\infty 2^{-\alpha_1 k^1-\alpha_3 k^3 }
		\Big[\sum_{l^1=1}^{k^1} \langle R_{k, l^1, 0}f, g\rangle+\langle R_{k, 0}^2 f, g\rangle  \Big]                                                                                                                                         \\
		                     & \hspace{1.5cm} +\sum_{k^1, k^2, k^3=0}^\infty 2^{-\alpha_3 k^3 }\langle R_{k, 0,0}f, g\rangle
		\bigg).
	\end{align*}
\end{thm}
We defer the final proof of this until we develop some machinery involving the various types of flag shifts in the next section. One useful aspect of this representation is the summability based on the bounds we have for each of the shift operators appearing in the formula.

\subsection{Abstract bilinear form summation}\label{sec:bilinforms}
As in the bi-parameter case, we first use the cancellation of both functions to insert the appropriate martingale blocks. We then take absolute values and expand the integrals over unions of cubes. Relabeling $I^m,J^m$ in each coordinate makes the integral involving $g_K$ be over $J$, without changing the common roots or side lengths. There are then $8$ possible cases for the integral involving $f_K$, which we estimate below.
\begin{enumerate}
	\item $J$; this holds by Lemma \ref{lem:triformsum1}
	\item $I$; this holds by Lemma \ref{lem:triformsum3}
	\item $I^1\times J^{23}$; this corresponds to either we are performing $\pc_2$ and $\pc_3$ simultaneously, or we are doing $\pc_{23}$, all of which only occur when $K^1\times J^{23}$ is flag, so this holds by Lemma \ref{lem:triformsum4}
	\item $J^1\times I^2\times J^3$; this corresponds to $\pc_1$ and $\pc_3$ simultaneously, which only occurs when $K^2\times J^3$ is flag, so this holds by Lemma \ref{lem:triformsum5}
	\item $J^{12}\times I^3$; this holds by Lemma \ref{lem:triformsum1}
	\item $I^{12}\times J^3$; this corresponds to $\pc_3$ which only occurs when $(K_F)^{12} \times J_3$ is flag, so this holds by Lemma \ref{lem:triformsum6}
	\item $I^1\times J^2\times I^3$; this corresponds to $\pc_2 $ which only occurs when $K^1\times J^2$ is flag, so this holds by Lemma \ref{lem:triformsum4}
	\item $J^1\times I^{23}$; this holds by Lemma \ref{lem:triformsum2}
\end{enumerate}
As in the bi-parameter case, certain cases above will only arise when additional restrictions apply to the subset of the lattics we are summing over. To capture this, we introduce the following notation. Note that some of the definitions below do not depend on every component of $k$, even though we always allow for such a dependence in the notation.
\begin{itemize}
	\item $\mathcal{D}_{1,<}(k):=\{K\in \calD_F^k: \ell (K^1) < 2^{-k^2} \ell(K^2)\}$ 
	\item $\mathcal{D}_{2,<}(k):=\{K\in\calD_F^k:\max\{\ell(K^1),\ell(K^2)\}<2^{-k^3}\ell(K^3)\} $ 
\end{itemize}
The two definitions above are straightforward enough: $\mathcal{D}_{1,<}$ corresponds to being able to perform paraproduct corrections in the second parameter, since $\ell(K^1)<\ell(I^2)$ implies that $K^1\times I^2$ is flag.  Similarly, $\mathcal{D}_{2,<}$ corresponds to being able to perform paraproduct corrections in the third parameter.  Let us also recall the notation $\calD_{=}^k$, $\calD_{=,<}^k$, $\calD_{<,=}^k$ and $\calD_{<}^k$, which will be relevant in below.

\begin{lem}\label{lem:triformsum1}
	Let $k = (k^1, k^2,k^3)$, $k^m \ge 0$, be fixed and $\calF \subset \calD^k_F$. Let $f_K, g_K \ge 0$ be functions and
	\begin{align*}
		A_1 & := \sum_{K \in \calF} \sum_{I^{(k)} = J^{(k)} = K} \frac{1}{|K_F|} \int_J g_K \int_J f_K, \\
		A_2 & := \sum_{K \in \calF} \sum_{I^{(k)} = J^{(k)} = K} \frac{1}{|K_F|}
		\int_{J^{12} \times I^{3}} f_K \int_{J} g_K.
	\end{align*}
	Then we have
	$$
		A_i \le \sum_{K \in \calF} \frac{|K|}{|K_F|} \int_K M_{\calD_F} f_K g_K, \qquad i = 1,2.
	$$
\end{lem}
\begin{proof}
	This is straightforward; see also the bi-parameter case Lemma \ref{lem:formsum1}.
\end{proof}

We also have:

\begin{lem}\label{lem:triformsum2}
	Let $k = (k^1, k^2,k^3)$, $k^m \ge 0$, be fixed and $\calF \subset \calD^k_F$. Let $f_K, g_K \ge 0$ be functions and
	\[
		A_3 := \sum_{K \in \calF} \sum_{I^{(k)} = J^{(k)} = K} \frac{1}{|K_F|}
		\int_{J^1 \times I^{23}} f_K \int_{J} g_K.
	\]
	Define the maximal function
	\[
		\wt M^{23}_{\mathcal{F}}(f)(x)=\sup_{K\in \mathcal{F}}\sup_{\ell(J^1)=2^{-{k^1}}\ell(K^1)}\frac{1_{J^1\times K^{23}}}{|J^1\times (K_F)^{23}|}\int_{J^1\times K^{23}} |f(x)|.
	\]
	Then
	\[
		A_3\le \sum_{K} \int_K \wt M^{23}_{\calF}(f_K)g_K.
	\]
\end{lem}
\begin{proof}
	This follows easily by summing in $I$ and $J$.
\end{proof}

Another case is
\begin{lem}\label{lem:triformsum3}
	Let $k = (k^1, k^2,k^3)$, $k^m \ge 0$, be fixed and $\calF \subset \calD^k_F$. Let $f_K, g_K \ge 0$ be functions and
	\begin{align*}
		A_4 := \sum_{K \in \calF} \sum_{I^{(k)} = J^{(k)} = K} \frac{1}{|K_F|} \int_I f_K \int_J g_K.
	\end{align*}
	Then we have
	$$
		A_4 \le \sum_{K \in \calF} \int_K \wt M_{\calF} f_K g_K,
	$$
	where
	$$
		\wt M_{\calF} f := \sup_{K \in \calF} \frac{1_K}{|K_F|} \int_K |f|.
	$$
\end{lem}
\begin{proof}
	It suffices to notice that
	$$
		A_4 = \sum_{K \in \calF} \frac{1}{|K_F|} \int_K f_K \int_K g_K.
	$$
\end{proof}

In certain cases, we will need to use additional information about $\mathcal{F}$.
\begin{lem}\label{lem:triformsum4}
	Let $k = (k^1, k^2,k^3)$, $k^m \ge 0$, be fixed and $\calF \subset \calD_{1,<}(k)$. Let $f_K, g_K \ge 0$ be functions and
	\begin{align*}
		A_5 & := \sum_{K \in \calF} \sum_{I^{(k)} = J^{(k)} = K} \frac{1}{|K_F|} \int_{I^1 \times J^{23}} f_K \int_J g_K,            \\
		A_6 & := \sum_{K \in \calF} \sum_{I^{(k)} = J^{(k)} = K} \frac{1}{|K_F|} \int_{I^1 \times J^{2}\times I^{3}} f_K \int_J g_K.
	\end{align*}
	Then we have
	$$
		A_i \le \sum_{K \in \calF} \frac{|K|}{|K_F|} \int_K M_{\calD_F} f_K g_K,\qquad i=5,6.
	$$
\end{lem}
\begin{proof}
	Both proofs are similar; as in the bi-parameter case, notice that
	$$
		A_i \le \sum_{K \in \calF} \frac{|K|}{|K_F|} \int_K M_{\calD_F} f_K g_K,
	$$
	for $i=5,6$, where we used that $K^1 \times J^2$ is flag (and hence, so are both $K^1\times J^{23}$ and $K^1\times J^2\times K^3$) by our standing assumption on $\calF$.
\end{proof}

\begin{lem}\label{lem:triformsum5}
	Let $k = (k^1, k^2,k^3)$, $k^m \ge 0$, be fixed and $\calF \subset \calD_{2,<}(k)$. Let $f_K, g_K \ge 0$ be functions and
	\begin{align*}
		A_7 & := \sum_{K \in \calF} \sum_{I^{(k)} = J^{(k)} = K} \frac{1}{|K_F|} \int_{J^1 \times I^{2}\times J^{3}} f_K \int_J g_K.
	\end{align*}
	Then we have
	$$
		A_7 \le \sum_{K \in \calF} \frac{|K|}{|K_F|} \int_K M_{\calD_F} f_K g_K.
	$$
\end{lem}
\begin{proof}
	As before, notice that
	$$
		A_7 \le \sum_{K \in \calF} \frac{|K|}{|K_F|} \int_K M_{\calD_F} f_K g_K,
	$$
	where we used that $K^2 \times J^3$ is flag (and hence, so is $J^1\times K^{2}\times J^3$) by our standing assumption on $\calF$.
\end{proof}

\begin{lem}\label{lem:triformsum6}
	Let $k = (k^1, k^2,k^3)$, $k^m \ge 0$, be fixed and $\calF \subset \calD_{2,<}(k)$. Let $f_K, g_K \ge 0$ be functions and
	\begin{align*}
		A_8 := \sum_{K \in \calF} \sum_{I^{(k)} = J^{(k)} = K} \frac{1}{|K_F|} \int_{I^{12} \times J^{3}} f_K \int_J g_K.
	\end{align*}
	Define the maximal function
	\[
		\wt M^{12}_{\calF}(f)(x)=\sup_{K\in \calF}\sup_{\ell(J^3)=2^{-{k^3}}\ell(K^3)}\frac{1_{K^{12}\times J^{3}}}{|(K_F)^{12}\times J^3|}\int_{K^{12}\times J^{3}} |f(x)|.
	\]
	Then
	$$
		A_8 \le \sum_{K \in \calF} \int_K \wt M^{12}_{\calF} (f_K) g_K.
	$$
\end{lem}
\begin{proof}
	This is immediate.
\end{proof}

\begin{lem}\label{lem:maxOp1}
	Let  $p \in (1, \infty)$ and $w \in A_{p, F}$, then there is
	$\eta = \eta([w]_{A_{p, F}}) > 0$ so that for all $k = (k^1, k^2,k^3)$ we have
	\[
		\|\wt M_{\calF}^{23} f\|_{L^p(w)} \lesssim
		\begin{cases}
			 & 2^{-\eta(k^2-k^3)_+} \|f\|_{L^p(w)}, \quad \calF \subset \calD_{<,=}^k, \\
			 & \|f\|_{L^p(w)}, \quad \calF \subset \calD_{<}^k.
		\end{cases}
	\]
\end{lem}
\begin{proof}
	The argument is very similar to the proof of Lemma \ref{lem:KperKFmaximal}, but we present the details nonetheless. First, if $\calF \subset \calD_{<,=}^k$, it is easy to observe the pointwise bound
	\[\wt M_{\calF}^{23} f \le 2^{-d_3(k^2-k^3)_+}M_{\calD} f,
	\]
	and so we have the following unweighted estimate:
	$$
		\|\wt M_{\calF}^{23} f\|_{L^p} \le 2^{-d_3(k^2-k^3)_+} \|M_{\calD} f\|_{L^p} \lesssim
		2^{-d_3(k^2-k^3)_+} \|f\|_{L^p}.
	$$
	On the other hand, we have $\wt M_{\calF}^{23} f \le M_{\calD_F} f$
	and so we have the following flag weighted estimate:
	$$
		\|\wt M_{\calF}^{23} f\|_{L^p(w)} \le \|M_{\calD_F} f\|_{L^p(w)} \lesssim \|f\|_{L^p(w)}.
	$$

	As a consequence of Proposition \ref{prop:open}, for all $w \in A_{p, F}$
	there is $\eps = \eps([w]_{A_{p, F}}) > 0$ so that also
	$$
		\|\wt M_{\calF}^{23} f \|_{L^p(w^{1+\eps} )} \lesssim \|f\|_{L^p(w^{1+\eps} )}.
	$$
	Then by interpolation with a change of measure (recall Proposition \ref{prop:interpo})
	we get that for $\kappa = \kappa([w]_{A_p, F}) :=\frac \eps{1+\eps}\in (0,1)$ we have
	$$
		\|\wt M_{\calF}^{23} f \|_{L^p(w)}\lesssim (2^{-d_3(k^2-k^3)_+})^\kappa  \|f\|_{L^p(w)}
		= 2^{-\eta(k^2-k^3)_+} \|f\|_{L^p(w)},
	$$
	where $\eta := \kappa d_3 \in (0, d_3)$.
	Indeed, this is obtained by using the interpolation result with
	$p_0 = p_1 = p$, $w_0 = 1$, $w_1 = w^{1+\eps}$ and $t = 1 / (1+\eps)$
	so that $w_0^{1-t}w_1^{t} = 1^{1-t}w^{t(1+\eps)} = w$. If $\calF \subset \calD_{<}^k$, then $\wt M_{\calF}^{23} f \le M_{\calD_F} f$ trivially yields the desired estimate.
	We are done.
\end{proof}
We have an analogous result for the $\wt M_{\calF}^{12}$ operator. We omit the proof since it is so similar.
\begin{lem}\label{lem:maxOp2}
	Let  $p \in (1, \infty)$ and $w \in A_{p, F}$, then there is
	$\eta = \eta([w]_{A_{p, F}}) > 0$ so that for all $k = (k^1, k^2,k^3)$ we have
	\[
		\|\wt M_{\calF}^{12} f\|_{L^p(w)} \lesssim
		\begin{cases}
			 & 2^{-\eta(k^1-k^2)_+} \|f\|_{L^p(w)}, \quad \calF \subset \calD_{=,<}^k\cap \calD_{2,<}(k), \\
			 & \|f\|_{L^p(w)}, \quad \calF \subset \calD_{<}^k\cap \calD_{2,<}(k).
		\end{cases}
	\]
\end{lem}

The following bound in terms of $|K|/|K_F|$ applies to both $\wt M_{\calF}$ and $\wt M_{\calF}^{23}$, including when the first two input side lengths agree.
\begin{lem}\label{lem:maxOp3}
	Let $p \in (1, \infty)$ and $w \in A_{p, F}$. There is
	$\eta = \eta([w]_{A_{p, F}}) > 0$ so that for all $k = (k^1, k^2,k^3)$ and $\calF\subset\calD_F^k$ we have
	$$
		\|\wt M_{\calF} f\|_{L^p(w)}+\|\wt M_{\calF}^{23} f\|_{L^p(w)} \lesssim
		\left(\sup_{K\in\calF}\frac{|K|}{|K_F|}\right)^\eta \|f\|_{L^p(w)}.
	$$
\end{lem}
\begin{proof}
	Use the pointwise bounds
	\[
		\wt M_{\calF} f\le
		\left(\sup_{K\in\calF}\frac{|K|}{|K_F|}\right)M_{\calD}f,
		\qquad \wt M_{\calF} f\le M_{\calD_F}f,
	\]
	and the same bounds for $\wt M_{\calF}^{23}$. For the latter, note that $K_F^1=K^1$, so
	\[
		\frac{|K^{23}|}{|(K_F)^{23}|}=\frac{|K|}{|K_F|},
		\qquad J^1\times (K_F)^{23}\in\calD_F,
	\]
	whenever $\ell(J^1)=2^{-k^1}\ell(K^1)$. The interpolation argument from Lemma \ref{lem:maxOp1} gives both estimates.
\end{proof}
If $\calF$ is contained in one of the four collections in \eqref{eq:ratiodecay}, inserting the corresponding bound gives the explicit decay in $k$.

Finally, we will need vector-valued analogues of all of these bounds.
\begin{lem}\label{lem:VecKperKFmaximalTri}
	Let $T$ be one of the maximal operators in Lemma \ref{lem:maxOp1}, \ref{lem:maxOp2} or \ref{lem:maxOp3}, and let $C(T)$ be the upper bound from that lemma. For all $p \in (1, \infty)$ and $w \in A_{p, F}$ there is
	$\alpha = \alpha([w]_{A_{p, F}}) > 0$ so that for all $k = (k^1,k^2,k^3)$ we have
	$$
		\Big\|\Big(\sum_j [T f_j ]^2\Big)^{1/2}\Big\|_{L^p(w)} \lesssim
		C(T)^{\alpha} \Big\|\Big(\sum_j |f_j|^2\Big)^{1/2}\Big\|_{L^p(w)}.
	$$
\end{lem}
We omit the proof which is analogous to the proof of Lemma \ref{lem:VecKperKFmaximal}.

\subsection{Details of the representation theorem}\label{subsec:tripar-shift-details}
In this section, we will carefully define the various shifts that appeared in the statement of the representation theorem but were not defined carefully, as well as bringing in some other notation that will be useful for our analysis. We will think of the vector $k=(k^1,k^2,k^3)$ as being fixed throughout.

\subsubsection{Different shifts in the tri-parameter case}
The bilinear forms of all our shifts will have the form
\[
	(f,g) \mapsto \sum_{K \in \calF} \sum_{I^{(k)} = J^{(k)} = K} a_{IJK} \langle f, \varphi_{I,J}\rangle
	\langle g, \psi_{I, J}\rangle,
\]
where the coefficients satisfy the normalization
\[
	|a_{IJK}|\le \frac{|I|}{|K_F|}.
\]
The functions $\varphi_{I,J}, \psi_{I,J}$ will be some specific `$\varphi_K$' from before; depending on the type of shift, it will be required that they satisfy some sort of cancellation condition. As in the bi-parameter case, we require the size bounds $|\varphi_{I,J}|,\,|\psi_{I,J}|\lesssim |I|^{-1/2}$. Once again, they satisfy the following constancy condition: they are constant on $L^1 \times L^2 \times L^3\in \calD$
with $\ell(L^m) < \ell(I^m) = 2^{-k^m}\ell(K^m)$. Below, we list the precise definition of different shifts.

\subsubsection*{Almost one-parameter shifts}
These arise in the representation theorem when all three parameters are comparable, i.e. for fixed $k$, we have that $K\notin \calD_{1,<}(k)\cup \calD_{2,<}(k)$. Additionally, we assume only the weakest one-parameter cancellation on the functions that appear, i.e.
\[
	\int_{\R^d} \varphi_{I,J}=\int_{\R^d} \psi_{I,J}=0.
\]
Thus, given $k=(k^1, k^2, k^3)\in \N^3$ and $l=(l^1, l^2)\in \N^2$ (where we actually have that $0\le l^1\le k^1$ and $0\le l^2\le \max\{k^1-l^1, k^2\}$), define  $\calD^{(1)}(k,l)\subset \calD_F^k$ as
\[
	\calD^{(1)}(k, l):= \{K\in \calD_F^k: 2^{l^1+l^2} 2^{-k^1} \ell(K^1)=2^{l^2} 2^{-k^2} \ell(K^2) =2^{-k^3}\ell(K^3)\}.
\]
Then the almost one-parameter shifts can be defined as:  $\calF= \calD^{(1)}(k, l)$ and the functions $\varphi_{I, J}$, $\psi_{I, J}$ having their respective supports inside (after relabeling)
\[
	(I\cup J\cup (J^1\times I^{23})\cup (J^{12}\times I^3),  J).
\]
Notice that these support conditions arise from paraproduct corrections $\pc_1, \pc_{12},$ or $\pc_{123}$ occurring in one function. The related model operators will be denoted by $R_{k, l^1, l^2}$. We have the following weighted boundedness result.
\begin{thm}\label{thm:almost-3par}
	For all $p \in (1, \infty)$, $w \in A_{p, F}$ and $k, l^1, l^2$ we have
	$$
		\|R_{k, l^1, l^2}f\|_{L^p(w)} \lesssim (|k|+1)^3\|f\|_{L^p(w)}.
	$$
	We have the enhanced estimate in different situations: for $l^1=0$:
	$$
		\|R_{k, 0, l^2}f\|_{L^p(w)} \lesssim (|k|+1)^3 2^{-\eta(k^1-k^2)_+}\|f\|_{L^p(w)};
	$$
	for $l^2 = 0$
	$$
		\|R_{k, l^1, 0}f\|_{L^p(w)} \lesssim (|k|+1)^3 2^{-\eta(k^2-k^3)_+}\|f\|_{L^p(w)};
	$$
	and for $l^1=l^2=0$
	$$
		\|R_{k, 0, 0}f\|_{L^p(w)} \lesssim (|k|+1)^3 2^{-\eta(k^1-k^2)_+} 2^{-\eta\max\{(k^1-k^3)_+, (k^2-k^3)_+\}}\|f\|_{L^p(w)},
	$$
	where $\eta = \eta([w]_{A_{p, F}}) > 0$.
\end{thm}
\begin{proof}
	Using our one-parameter expansion, we see that
	\begin{align*}
		 & |\langle R_{k, l^1, l^2}f, g\rangle|                                \\
		 & \lesssim \sum_{K \in \calD^{(1)}(k,l)} \sum_{I^{(k)} = J^{(k)} = K}
		\frac{1}{|K_F|} \Big(\int_I |f_K| + \int_{J} |f_K|+ \int_{J^1\times I^{23}} |f_K|+ \int_{J^{12}\times I^3} |f_K|\Big)
		\int_J g_K,
	\end{align*}
	where $f_K=|\mathcal{Y}_{K,k,l}f|$ and $g_K=|\mathcal{Y}_{K,k,l}g|$.
	These terms can be dealt with using Lemmas \ref{lem:triformsum1}, \ref{lem:triformsum2} and \ref{lem:triformsum3} to bound the bilinear forms in terms of $M_{\calD_F}$, $\wt M_{\calF}^{23}$ and $\wt M_{\calF}$. From there, we can bound it using Lemma \ref{lem:triparauxsfY}, as well as standard estimates for $M_{\calD_F}$ and Lemmas \ref{lem:maxOp1}, \ref{lem:maxOp3} and \ref{lem:VecKperKFmaximalTri} for $\wt M_{\calF}^{23}$ and $\wt M_{\calF}$.
\end{proof}

\subsubsection*{Almost bi-parameter shifts: type I}
Fix $k=(k^1, k^2, k^3)\in \N^3$ and $l^1\in \N$, where we in fact know that $l^1\le k^1$. These shifts arise in the representation theorem when we have the first two parameters grouped together (due to comparability) but not the third, due to having that $\ell(K^1)\ge\ell(I^2)$ and $\ell(I^3)>\max\{\ell(K^1),\ell(K^2)\}$. We assume cancellation conditions adapted to the grouping, i.e.
\[
	\int_{\R^{d_{12}}} \varphi_{I,J}=\int_{\R^{d_{12}}} \psi_{I,J}=\int_{\R^{d_3}} \varphi_{I,J}=\int_{\R^{d_3}} \psi_{I,J}=0.
\]
Thus, we define (note that the definition does not depend on $\ell^2$)
\[
	\mathcal{D}^{(2a)}(k,l) := \{K\in \calD_F^k: 2^{l^1} 2^{-k^1} \ell(K^1)= 2^{-k^2} \ell(K^2),  \max\{\ell(K^1), \ell(K^2)\}<2^{-k^3}\ell(K^3)\}.
\]
Paraproduct corrections $\pc_1,\pc_3$ or $\pc_{12},\pc_3$ may occur on the same function or on different functions. Accordingly, the supports of $\varphi_{I,J},\psi_{I,J}$ lie, respectively, in one of the following pairs, up to relabeling and taking the full adjoint:
\[
	\begin{aligned}
	&((I^1\cup J^1)\times I^2\times(I^3\cup J^3),\ J),\\
	&((I^{12}\cup J^{12})\times(I^3\cup J^3),\ J),\\
	&((I^1\cup J^1)\times I^2\times I^3,\ J^{12}\times(I^3\cup J^3)),\\
	&((I^{12}\cup J^{12})\times I^3,\ J^{12}\times(I^3\cup J^3)).
	\end{aligned}
\]
\begin{thm}\label{thm:albipar-type1}
	For all $p \in (1, \infty)$, $w \in A_{p, F}$ and $k, l^1\le k^1$ we have
	$$
		\|R_{k, l^1}^1f\|_{L^p(w)} \lesssim (|k|+1)^3\|f\|_{L^p(w)}.
	$$
	For $l^1=0$, we have the enhanced estimate:
	$$
		\|R_{k,  l^1}^1f\|_{L^p(w)} \lesssim (|k|+1)^3 2^{-\eta(k^1-k^2)_+}\|f\|_{L^p(w)},\quad \eta = \eta([w]_{A_{p, F}}) > 0.
	$$
\end{thm}
\begin{proof}
	Insert the blocks from Lemma \ref{lem:triparauxsfV1} before expanding the positive integrals. The last two support pairs give the same forms as the first two after relabeling $I^3,J^3$. Most terms follow as in Theorem \ref{thm:almost-3par}. The additional forms have supports $(I^{12}\times J^3,J)$ and $(J^1\times I^2\times J^3,J)$, and are covered by Lemmas \ref{lem:triformsum6} and \ref{lem:triformsum5}, since $\calD^{(2a)}(k,l)\subset\calD_{2,<}(k)$. The bound for $\wt M^{12}_{\calF}$ follows from Lemmas \ref{lem:maxOp2} and \ref{lem:VecKperKFmaximalTri}.
\end{proof}

\subsubsection*{Almost bi-parameter shifts: type II}
Now, fix $k=(k^1, k^2, k^3)\in \N^3$ and $l^2\in \N$. For type II bi-parameter shifts, we have comparability among the second and third parameters but not with the first, due to having that $\ell(K^1)<\ell(I^2)$ and $\ell(I^3)\le\max\{\ell(K^1),\ell(K^2)\}=\ell(K^2)$. Analogously to the type I bi-parameter shifts, we assume the cancellation conditions
\[
	\int_{\R^{d_{23}}} \varphi_{I,J}=\int_{\R^{d_{23}}} \psi_{I,J}=\int_{\R^{d_1}} \varphi_{I,J}=\int_{\R^{d_1}} \psi_{I,J}=0.
\]
Thus, we define (note that the definition does not depend on $\ell^1$)
\[
	\mathcal{D}^{(2b)}(k,l)=\{K\in \calD_F^k:  \ell(K^1)< 2^{-k^2} \ell(K^2),  2^{l^2}2^{-k^2}\ell(K^2)=2^{-k^3}\ell(K^3)\}.
\]
Here the corrections are $\pc_1,\pc_{23}$ or $\pc_1,\pc_2$. The corresponding support pairs, again up to relabeling and taking the full adjoint, are
\[
	\begin{aligned}
	&((I^1\cup J^1)\times(I^{23}\cup J^{23}),\ J),\\
	&((I^1\cup J^1)\times(I^2\cup J^2)\times I^3,\ J),\\
	&((I^1\cup J^1)\times I^{23},\ J^1\times(I^{23}\cup J^{23})),\\
	&((I^1\cup J^1)\times I^2\times I^3,\ J^1\times(I^2\cup J^2)\times J^3).
	\end{aligned}
\]
The related model operators will be denoted by $R_{k,l^2}^2$. Insert the blocks from Lemma \ref{lem:triparauxsfV2} first. After expanding the positive integrals, the last two support pairs reduce to the first two by relabeling $I^{23},J^{23}$ or $I^2,J^2$, respectively. The proof then follows as before, with the additional forms covered by Lemma \ref{lem:triformsum4}, since $\calD^{(2b)}(k,l)\subset\calD_{1,<}(k)$.
\begin{thm}\label{thm:albipar-type2}
	For all $p \in (1, \infty)$, $w \in A_{p, F}$ and $k, l^2\le k^2$ we have
	$$
		\|R_{k, l^2}^2f\|_{L^p(w)} \lesssim (|k|+1)^3\|f\|_{L^p(w)}.
	$$
	For $l^2=0$, we have the enhanced estimate:
	$$
		\|R_{k, 0}^2 f\|_{L^p(w)} \lesssim (|k|+1)^3 2^{-\eta(k^2-k^3)_+}\|f\|_{L^p(w)},\quad \eta = \eta([w]_{A_{p, F}}) > 0.
	$$
\end{thm}

\subsubsection*{Tri-parameter flag shifts} Finally, we move to define tri-parameter flag shifts. These occur when $\ell(I^2)>\ell(K^1)$ and $\ell(I^3)>\ell(K^2)$. As none of the parameters are grouped, we need to assume the strongest cancellation condition
\[
	\int_{\R^{d_{i}}} \varphi_{I,J}=\int_{\R^{d_{i}}} \psi_{I,J}=0,\qquad\text{ for }i=1,2,3.
\]
Thus, we consider (noting that this does not actually depend on $l$)
\[
	\mathcal{D}^{(3)}(k,l)=\mathcal{D}_{1,<}(k)\cap \mathcal{D}_{2,<}(k).
\]
Paraproduct corrections $\pc_1,\pc_2,\pc_3$ may occur on either function. We therefore require
\begin{equation*}
	\supp\varphi_{I,J},\ \supp\psi_{I,J}\subset
	(I^1\cup J^1)\times(I^2\cup J^2)\times(I^3\cup J^3).
\end{equation*}
We denote the related model operators by $Q_k$. After inserting $\calU_{K,k}$ in both pairings, expansion and relabeling give the eight forms of Section~\ref{sec:bilinforms}. All their hypotheses hold because $\calD^{(3)}(k,l)\subset\calD_{1,<}(k)\cap\calD_{2,<}(k)$, so the weighted bound follows from those estimates and Lemma \ref{lem:auxsftri-par}.
\begin{thm}\label{thm:tripar-type1}
	For all $p \in (1, \infty)$, $w \in A_{p, F}$ and $k$ we have
	$$
		\|Q_kf\|_{L^p(w)} \lesssim (|k|+1)^3\|f\|_{L^p(w)}.
	$$

\end{thm}

\subsection{Algorithm rules}\label{subsec:tripar-algorithm}
Putting all of this together, we deduce a set of legal ``moves" we may perform in order to facilitate getting enough cancellation in any term arising in our flag multiresolution.
For a general summand $\Sigma_s$, the rules are as follows:
\begin{enumerate}
	\item Split the summand into summation over $\calD_{=,=}$, $\calD_{=,<}$, $\calD_{<,=}$ and $\calD_{<,<}$. Of course, these 4 terms not necessarily all appear; for instance, in our concrete example $s=(D, E, \Delta)$, we only have $\calD_{=,<}$ and $\calD_{<,<}$.
	\item Group parameters whose side lengths are equal; initially leave the others separate. Thus $\calD_{=,=}$ gives the single group $\{123\}$, whereas $\calD_{<,=}$ gives $\{1,23\}$.
	\item Perform paraproduct corrections (if needed) to the first group, potentially expanding out the $\Lambda=\Delta+E$ terms as needed. For instance, if the first group is $\{1\}$, then perform $\pc_1$, and if it is $\{123\}$ arising from $(D,D,E)$, then we will need to perform $\pc_{123}f$ for the term arising from taking the $E$ part of $\Lambda$ in both of the first two parameters for $f$.
	\item Let $j$ be the first index that is not in the first grouping. For terms where $\ell(I^j)>\max\{\ell(K^1),\cdots, \ell(K^{j-1})\}$, then we may perform $\pc_j$ as needed to get cancellation. For the remaining terms, then we think of adding them to the group which contains $j-1$; by assumption, this group already has cancellation, so we can think of having $c1\ldots j$ cancellation well.
	\item Iterate; for terms where we have the inequality $\ell(I^j)>\max\{\ell(K^1),\cdots, \ell(K^{j-1})\}$, then we may perform $\pc_j$, and if it fails to hold, then we group it with an earlier grouping that already has cancellation. As a result of how we bifurcated to select our groupings, we may perform paraproduct corrections on any grouping, if needed. In the end, we find that we have sufficient cancellation for every term that arises, as desired.
\end{enumerate}

The above algorithm works with higher parameter flag structure, even in the non-cancellative case. In the sequel, we just use the cancellative  tri-parameter case as a model to illustrate  why everything is well-organized via this algorithm.

Now, given \textbf{any} of the 27 terms, we break up an arbitrary term, say $\Sigma_s$,  into 4 pieces based on the above algorithm:
\[
	\Sigma_s=\Sigma_s^{=,=}+\Sigma_s^{=,<}+\Sigma_s^{<,=}+\Sigma_s^{<,<}.
\]
Some of the above terms may be zero, depending on the value of $s$. Obviously, the term $\Sigma_s^{=,=}$ will be regarded as $\Sigma_s^{c123}$ (almost one-parameter shifts) after performing suitable paraproduct corrections. For the term $\Sigma_s^{=,<}$, by our algorithm we may perform the paraproduct correction $\pc_{12}$ as needed. Then terms where $\ell(I^3)\le\max\{\ell(K^1), \ell(K^2)\}$ are regarded as $\Sigma_s^{c123}$, and on terms where $\ell(I^3)>\max\{\ell(K^1), \ell(K^2)\}$, we perform $\pc_3$ and it is regarded as $\Sigma_s^{c12c3}$ (an almost bi-parameter shift of type I). So
\[
	\Sigma_s^{=,<}=\Sigma_s^{c123}+ \Sigma_s^{c12c3}.
\]
Similarly, for $\Sigma_s^{<,=}$ we perform $\pc_1$, then if $\ell(I^2)\le \ell(K^1)$ it gives $\Sigma_s^{c123}$, otherwise we may perform $\pc_{23}$ and we obtain $\Sigma_s^{c1c23}$ (almost bi-parameter shifts, type II). Hence
\[
	\Sigma_s^{<,=}=\Sigma_s^{c123}+ \Sigma_s^{c1c23}.
\]
Finally, we move to $\Sigma_{<,<}$. In this case, we may potentially have three groups appearing in some terms. Hence, we may formally write
\begin{align*}
	\Sigma_{s}^{<,<}
	 & = \sum_{\substack{I, J\in \calD_{<, < }\\ \ell(I)=\ell(J)\\ \ell(I^2)> \ell(K^1)\\ \ell(I^3)> \max\{\ell(K^1), \ell(K^2)\}}}
	+ \sum_{\substack{I, J\in \calD_{<,<}\\ \ell(I)=\ell(J)\\ \ell(I^2)\le \ell(K^1)\\ \ell(I^3)> \max\{\ell(K^1),\ell(K^2)\}}}                    \\
	 & \hspace{3cm} + \sum_{\substack{I, J\in \calD_{<, < }\\ \ell(I)=\ell(J)\\ \ell(I^2)> \ell(K^1)\\ \ell(I^3)\le \max\{\ell(K^1), \ell(K^2)\}}}
	+ \sum_{\substack{I, J\in \calD_{<,<}\\ \ell(I)=\ell(J)\\ \ell(I^2)\le \ell(K^1)\\ \ell(I^3)\le \max\{\ell(K^1),\ell(K^2)\}}}
	\\
	 & := \Sigma_s^{c1c2c3}+\Sigma_s^{c12c3}+\Sigma_s^{c1c23}+\Sigma_s^{c123},
\end{align*}
where $\Sigma_s^{c1c2c3}$ is obtained after $\pc_{1}, \pc_2$ and $\pc_3$, $\Sigma_s^{c12c3}$ is obtained after $\pc_{1}$ and $\pc_3$, $\Sigma_s^{c1c23}$ is obtained after $\pc_{1}$ and $\pc_2$, and $\Sigma_s^{c123}$ is obtained after $\pc_{1}$.

Thus, this results in a tri-parameter shift, a bi-parameter shift of type I, a bi-parameter shift of type II, and a one-parameter shift. The discussion in the above shows that the bounds for the sum are as in the representation theorem. This completes the proof of Theorem \ref{thm: triparam rep}.

Finally, observe that the weighted bounds for the shift operators in Theorems \ref{thm:almost-3par}, \ref{thm:albipar-type1}, \ref{thm:albipar-type2}, and \ref{thm:tripar-type1} combined with the representation theorem lead to the following corollary for any flag CZO.

\begin{cor}
	Suppose $T$ is a cancellative tri-parameter flag CZO. Then
	\[
		\Vert Tf\Vert_{L^p(w)}\lesssim \Vert f\Vert_{L^p(w)}
	\]
	for all $p\in (1,\infty)$ and all $w\in A_{p,F}$.
\end{cor}
\begin{proof}
	Let $g$ be normalized i.e. $\Vert g\Vert_{L^{p'}(w')}=1$. Recall the representation theorem.
	\begin{align*}
		\langle Tf, g\rangle & = C \E \bigg(\sum_{k^1, k^2, k^3=0}^\infty 2^{-\alpha_1 k^1-\alpha_2 k^2 -\alpha_3 k^3} \Big[\langle Q_k f, g\rangle+\sum_{l^1=1}^{k^1} \sum_{l^2=1}^{\max\{k^1-l^1, k^2\}} \langle R_{k, l^1, l^2}f, g\rangle\\
			                                                                                                               &\qquad+\sum_{l^1=1}^{k^1} \langle R_{k, l^1}^1 f, g\rangle + \sum_{l^2=1}^{k^2} \langle R_{k, l^2}^2 f, g\rangle\Big]
		\\ & \hspace{1.5cm} +\sum_{k^1, k^2, k^3=0}^\infty 2^{-\alpha_2 k^2-\alpha_3 k^3 }\Big[ \sum_{l^2=1}^{\max\{k^1, k^2\} }\langle R_{k, 0, l^2}f, g\rangle+ \langle R_{k, 0}^1 f, g\rangle
			                                                                                  \Big]\\ & \hspace{1.5cm} +\sum_{k^1, k^2, k^3=0}^\infty 2^{-\alpha_1 k^1-\alpha_3 k^3 }
		\Big[\sum_{l^1=1}^{k^1} \langle R_{k, l^1, 0}f, g\rangle+\langle R_{k, 0}^2 f, g\rangle  \Big]                                                                                                                                         \\
		                     & \hspace{1.5cm} +\sum_{k^1, k^2, k^3=0}^\infty 2^{-\alpha_3 k^3 }\langle R_{k, 0,0}f, g\rangle
		\bigg).
	\end{align*}
	Taking absolute values and using the triangle inequality, we appeal to the bounds for the various shift operators to deduce
	\begin{align*}
		|\langle Tf, g\rangle| & \le C\Vert f\Vert_{L^p(w)} \bigg(\sum_{k^1, k^2, k^3=0}^\infty 2^{-\alpha_1 k^1-\alpha_2 k^2 -\alpha_3 k^3} \Big[(|k|+1)^3+\sum_{l^1=1}^{k^1} \sum_{l^2=1}^{\max\{k^1-l^1, k^2\}} (|k|+1)^3 \\
			                                                                                                                                     &+\sum_{l^1=1}^{k^1} (|k|+1)^3 + \sum_{l^2=1}^{k^2} (|k|+1)^3\Big]
		\\ &+\sum_{k^1, k^2, k^3=0}^\infty 2^{-\alpha_2 k^2-\alpha_3 k^3 }\Big[ \sum_{l^2=1}^{\max\{k^1, k^2\} }(|k|+1)^32^{-\eta (k^1-k^2)_+}+ (|k|+1)^32^{-\eta (k^1-k^2)_+}
			                                                                  \Big]\\ &+\sum_{k^1, k^2, k^3=0}^\infty 2^{-\alpha_1 k^1-\alpha_3 k^3 }
		\Big[\sum_{l^1=1}^{k^1} (|k|+1)^32^{-\eta (k^2-k^3)_+}+(|k|+1)^32^{-\eta (k^2-k^3)_+}  \Big]                                                                                                                         \\
		                       & \hspace{1.5cm} +\sum_{k^1, k^2, k^3=0}^\infty 2^{-\alpha_3 k^3 }(|k|+1)^32^{-\eta (k^1-k^2)_+}2^{-\eta\max\{(k^1-k^3)_+,(k^2-k^3)_+\}}
		\bigg).
	\end{align*}
	All four sums can be estimated together. Set
	$\delta=\min\{\alpha_1,\alpha_2,\alpha_3,\eta\}>0$ and, for fixed $k$, let
	$m=\max\{k^1,k^2,k^3\}$. We have $k^1+k^2+k^3\ge m$ and
	\begin{align*}
		k^2+k^3+(k^1-k^2)_+ & =\max\{k^1,k^2\}+k^3\ge m, \\
		k^1+k^3+(k^2-k^3)_+ & =k^1+\max\{k^2,k^3\}\ge m, \\
		k^3+\max\{(k^1-k^3)_+,(k^2-k^3)_+\} & =m.
	\end{align*}
	Thus the decay factor in each group is at most $2^{-\delta m}$.
	For fixed $k$, each group contains at most $1+m^2+2m=(m+1)^2$ terms,
	and $(|k|+1)^3\lesssim(m+1)^3$. Finally,
	\[
		\#\{k\in\N^3:\max_j k^j=m\}=(m+1)^3-m^3\le 3(m+1)^2.
	\]
	Consequently,
	\[
		|\langle Tf,g\rangle|
		\lesssim \|f\|_{L^p(w)}\sum_{m=0}^{\infty}(m+1)^7 2^{-\delta m}
		\lesssim \|f\|_{L^p(w)},
	\]
	which proves the claim.
\end{proof}
\bibliography{references}
\end{document}